\pdfoutput=1
\documentclass[english, hidelinks, reqno]{amsart}
\usepackage{graphicx}
\usepackage{subfig}
\usepackage{tabularx}
\usepackage{booktabs}
\usepackage{array}
\usepackage{amsmath}
\usepackage{amsfonts}
\usepackage{amssymb}
\usepackage{amsthm}
\usepackage{enumitem}
\usepackage{permute}
\usepackage{url}
\usepackage[usenames,dvipsnames]{color}
\usepackage[pagebackref = true, colorlinks, linkcolor = Red, citecolor = Green, bookmarksdepth=2, linktocpage=true]{hyperref}
\usepackage[capitalize]{cleveref}
\usepackage{comment}
\DeclareMathOperator{\interior}{int}

\DeclareMathOperator{\Id}{Id}
\DeclareMathOperator{\tr}{tr}
\DeclareMathOperator{\Stab}{Stab}

\DeclareMathOperator{\supp}{supp}

\DeclareMathOperator{\GL}{GL}
\DeclareMathOperator{\SO}{SO}

\DeclareMathOperator{\Lip}{Lip}

\DeclareMathOperator{\proj}{proj}
\DeclareMathOperator{\Cay}{Cay}

\DeclareMathOperator{\T}{T}
\DeclareMathOperator{\F}{F}

\DeclareMathOperator{\Isom}{Isom}

\DeclareMathOperator{\diam}{diam}
\DeclareMathOperator{\Banana}{Banana}
\DeclareMathOperator{\len}{len}
\let\Pr\relax
\DeclareMathOperator{\Pr}{Pr}

\DeclareMathOperator{\Ad}{Ad}

\DeclareMathOperator{\Inj}{Inj}
\DeclareMathOperator{\Var}{Var}

\DeclareFontFamily{U}{mathb}{\hyphenchar\font45}
\DeclareFontShape{U}{mathb}{m}{n}{
      <5> <6> <7> <8> <9> <10> gen * mathb
      <10.95> mathb10 <12> <14.4> <17.28> <20.74> <24.88> mathb12
      }{}
\DeclareSymbolFont{mathb}{U}{mathb}{m}{n}
\DeclareMathSymbol{\bigast}{2}{mathb}{"06}

\def\XXint#1#2#3{{\setbox0=\hbox{$#1{#2#3}{\int}$}
     \vcenter{\hbox{$#2#3$}}\kern-.5\wd0}}

\theoremstyle{plain}
\newtheorem{theorem}{Theorem}[section]
\newtheorem{corollary}{Corollary}[theorem]
\newtheorem{lemma}[theorem]{Lemma}

\theoremstyle{definition}
\newtheorem{definition}[theorem]{Definition}

\newtheoremstyle{remark}
{}   % ABOVESPACE, \topsep for no space
{}   % BELOWSPACE, \topsep for no space
{\normalfont}  % BODYFONT
{}       % INDENT (empty value is the same as 0pt)
{\itshape} % HEADFONT
{.}         % HEADPUNCT
{5pt plus 1pt minus 1pt} % HEADSPACE
{}          % CUSTOM-HEAD-SPEC

\theoremstyle{remark}
\newtheorem*{remark}{Remark}

\setlist[enumerate,1]{ref=(\arabic*)}
\setlist[enumerate,2]{ref=(\theenumi)(\alph*)}
\setlist[enumerate,3]{ref=(\theenumi)(\theenumii)(\roman*)}
\setlist[enumerate,4]{ref=(\theenumi)(\theenumii)(\theenumiii)(\Alph*)}

\newlist{alternative}{enumerate}{4}     % this creates a dedicated counter named 'subtaski'
\setlist[alternative,1]{label=(\arabic*), ref=(\arabic*)}
\setlist[alternative,2]{label=(\alph*), ref=(\thealternativei)(\alph*)}
\setlist[alternative,3]{label=(\roman*), ref=(\thealternativei)(\thealternativeii)(\roman*)}
\setlist[alternative,4]{label=(\Alph*), ref=(\thealternativei)(\thealternativeii)(\thealternativeiii)(\Alph*)}

\Crefname{enumi}{Property}{Properties}
\Crefname{alternativei}{Alternative}{Alternatives}
\Crefname{subsection}{Subsection}{Subsections}

\begin{document}

%\setlength{\parindent}{0 in}
%\setlength{\mathindent}{0.5 in}
%\numberwithin{equation}{section}
%\renewcommand*{\abstractname}{Introduction}

\title[Uniform Exponential Mixing]{Uniform Exponential Mixing for Congruence Covers of Convex Cocompact Hyperbolic Manifolds}
\author{Pratyush Sarkar}
\address{Department of Mathematics, Yale University, New Haven, Connecticut 06511}
\email{pratyush.sarkar@yale.edu}
\date{\today}

\begin{abstract}
Let $\Gamma$ be a Zariski dense convex cocompact subgroup contained in an arithmetic lattice of $\SO(n, 1)^\circ$. We prove uniform exponential mixing of the geodesic flow for congruence covers of the hyperbolic manifold $\Gamma \backslash \mathbb H^n$ avoiding finitely many prime ideals. This extends the work of Oh-Winter who proved the result for the $n = 2$ case. Following their approach, we use Dolgopyat's method for the proof of exponential mixing of the geodesic flow. We do this uniformly over congruence covers by establishing uniform spectral bounds for the congruence transfer operators associated to the geodesic flow. This requires another key ingredient which is the expander machinery due to Bourgain-Gamburd-Sarnak extended by Golsefidy-Varj\'{u}.
\end{abstract}

\maketitle

\setcounter{tocdepth}{1}
\tableofcontents

\section{Introduction}
Let $\mathbb H^n$ be the $n$-dimensional hyperbolic space for $n \geq 2$. Let $\mathbb K \subset \mathbb R$ be a totally real number field, $\mathcal{O}_{\mathbb K}$ be the corresponding ring of integers, and $Q$ be a quadratic form of signature $(n, 1)$ defined over $\mathbb K$. Let $\mathbf{G} = \SO_Q(\mathbb C) < \GL_{n + 1}(\mathbb C)$ be an algebraic group defined over $\mathbb K$ such that $\mathbf{G}(\mathbb R) \cong \SO(n, 1)$ and $\mathbf{G}^\sigma(\mathbb R) \cong \SO(n + 1)$, which is compact, for all nontrivial embeddings $\sigma: \mathbb K \hookrightarrow \mathbb R$. Let $G = \mathbf{G}(\mathbb R)^\circ$ which we recognize as the group of orientation preserving isometries of $\mathbb H^n$. We identify $\mathbb H^n$ and its unit tangent bundle $\T^1(\mathbb H^n)$ with $G/K$ and $G/M$ respectively where $M \subset K$ are compact subgroups of $G$. Let $A = \{a_t: t \in \mathbb R\} < G$ be a one-parameter subgroup of semisimple elements such that its right translation action on $G/M$ corresponds to the geodesic flow.

Let $\Gamma < G$ be a Zariski dense convex cocompact subgroup. By the work of Stoyanov \cite{Sto11}, exponential mixing of the geodesic flow on the hyperbolic manifold $\T^1(\Gamma \backslash \mathbb H^n) = \Gamma \backslash G/M$ is known for the Bowen-Margulis-Sullivan measure. In this paper, when $\Gamma$ is defined arithmetically and satisfies the strong approximation property, we establish the \emph{uniform} exponential mixing for all congruence covers of $\Gamma \backslash \mathbb H^n$ corresponding to ideals $\mathfrak{q} \subset \mathcal{O}_{\mathbb K}$ avoiding finitely many prime ideals, extending the work of Oh-Winter \cite{OW16} for $n = 2$ to arbitrary dimensions $n \geq 2$.

%we prove the following theorem establishing \emph{uniform} exponential mixing of the geodesic flow, i.e., the right $A$-action on $\Gamma_\mathfrak{q} \backslash G/M$, with respect to the Bowen-Margulis-Sullivan (BMS) measure $m^{\mathrm{BMS}}_\mathfrak{q}$, for all nontrivial ideals $\mathfrak{q} \subset \mathcal{O}_{\mathbb K}$ avoiding finitely many primes.

Let $\tilde{\pi}: \tilde{\mathbf{G}} \to \mathbf{G}$ be a simply connected cover defined over $\mathbb K$. For all ideals $\mathfrak{q} \subset \mathcal{O}_{\mathbb K}$, let $\pi_{\mathfrak{q}}: \tilde{\mathbf{G}}(\mathcal{O}_{\mathbb K}) \to \tilde{\mathbf{G}}(\mathcal{O}_{\mathbb K}/\mathfrak{q})$ be the canonical quotient map. Let $\Gamma < \mathbf{G}(\mathbb K)$ be a Zariski dense torsion-free convex cocompact subgroup such that $\tilde{\pi}^{-1}(\Gamma)$ is contained in $\tilde{\mathbf{G}}(\mathcal{O}_{\mathbb K})$ and $\{\tr(\Ad_g): g \in \tilde{\pi}^{-1}(\Gamma)\}$ generates the ring $\mathcal{O}_{\mathbb K}$. We impose these conditions so that $\Gamma$ satisfies the strong approximation property. For all nontrivial ideals $\mathfrak{q} \subset \mathcal{O}_{\mathbb K}$, let $\Gamma_\mathfrak{q} < \Gamma$ be a congruence subgroup of level $\mathfrak{q}$, meaning that $\tilde{\pi}^{-1}(\Gamma_\mathfrak{q}) = \tilde{\pi}^{-1}(\Gamma) \cap \langle\ker(\pi_{\mathfrak{q}}), \{e, -e\}\rangle$. For all nontrivial ideals $\mathfrak{q} \subset \mathcal{O}_{\mathbb K}$, let $N_{\mathbb K}(\mathfrak{q})$ be the ideal norm and $m^{\mathrm{BMS}}_\mathfrak{q}$ be the Bowen-Margulis-Sullivan measure on $\Gamma_\mathfrak{q} \backslash G$ induced from the one on $\Gamma \backslash G/M$.

%Let $m^{\mathrm{BMS}}$ be the right $M$-invariant measure on $\Gamma \backslash G$ induced by the Bowen-Margulis-Sullivan measure on $\Gamma \backslash G/M$.

%Let $\Lambda(\Gamma)$ denote the limit set of $\Gamma$ and $\delta_\Gamma \in (0, n - 1]$ denote the critical exponent of $\Gamma$, which by convex cocompactness coincides with the Hausdorff dimension of $\Lambda(\Gamma)$. \\

\begin{theorem}
\label{thm:TheoremUniformExponentialMixingOfGeodesicFlow}
There exist $\eta > 0$, $C > 0$ and a nontrivial proper ideal $\mathfrak{q}_0 \subset \mathcal{O}_{\mathbb K}$ such that for all square free ideals $\mathfrak{q} \subset \mathcal{O}_{\mathbb K}$ coprime to $\mathfrak{q}_0$, and for all $M$-invariant functions $\phi, \psi \in C^1(\Gamma_\mathfrak{q} \backslash G, \mathbb R)$, we have
\begin{multline*}
\left|\int_{\Gamma_\mathfrak{q} \backslash G} \phi(xa_t)\psi(x) \, dm^{\mathrm{BMS}}_\mathfrak{q}(x) - \frac{1}{m^{\mathrm{BMS}}_\mathfrak{q}(\Gamma_\mathfrak{q} \backslash G)} m^{\mathrm{BMS}}_\mathfrak{q}(\phi) \cdot m^{\mathrm{BMS}}_\mathfrak{q}(\psi)\right| \\
\leq CN_{\mathbb K}(\mathfrak{q})^C e^{-\eta t} \|\phi\|_{C^1} \|\psi\|_{C^1}.
\end{multline*}
\end{theorem}

%For the case of $n = 2$, this has already been established by Oh-Winter \cite{OW16}. Moreover, if the critical exponent of $\Gamma$ satisfies $\delta_\Gamma > \frac{1}{2}$, then a related result was established by Bourgain-Gamburd-Sarnak \cite{BGS11}. Not only are mixing results of interest in itself in the study of dynamics, they are also of interest because they have a number of immediate applications to various problems such as counting, equidistribution and affine sieve.

\begin{remark}
\hfill
\begin{enumerate}
\item \cref{thm:TheoremUniformExponentialMixingOfGeodesicFlow} can be used to show the existence of a uniform resonance free strip for the resolvent of the Laplacian as well as a uniform zero-free strip of the Selberg zeta functions on the family of hyperbolic manifolds $\Gamma_\mathfrak{q} \backslash \mathbb H^n$ for nontrivial ideals $\mathfrak{q} \subset \mathcal{O}_{\mathbb K}$ (cf. \cite{OW16}).
\item When the critical exponent of $\Gamma$ satisfies $\delta_\Gamma > \frac{n - 1}{2}$, \cref{thm:TheoremUniformExponentialMixingOfGeodesicFlow} has been established by Mohammadi-Oh in \cite{MO15}.
\item In a forthcoming version, we extend \cref{thm:TheoremUniformExponentialMixingOfGeodesicFlow} to arbitrary functions which are not necessarily $M$-invariant.
\end{enumerate}
\end{remark}

Fix a Haar measure on $G$. This induces a left $G$-invariant measure on $\Gamma_\mathfrak{q} \backslash G$ for all ideals $\mathfrak{q} \subset \mathcal{O}_{\mathbb K}$. For all ideals $\mathfrak{q} \subset \mathcal{O}_{\mathbb K}$ and for all $\phi, \psi \in L^2(\Gamma_\mathfrak{q} \backslash G, \mathbb C)$, we define the matrix coefficient by the usual inner product
\begin{align*}
\langle a_t\phi, \psi \rangle = \int_{\Gamma_\mathfrak{q} \backslash G} \phi(xa_t)\overline{\psi(x)} \, dx
\end{align*}
where the use of the $G$-invariant measure is implicit. We denote by $m^{\mathrm{BR}}_\mathfrak{q}$ and $m^{\mathrm{BR}_*}_\mathfrak{q}$ the unstable Burger-Roblin measure and the stable Burger-Roblin measure on $\Gamma_\mathfrak{q} \backslash G$ compatible with the choice of the Haar measure for all ideals $\mathfrak{q} \subset \mathcal{O}_{\mathbb K}$.

\begin{corollary}
There exist $\eta > 0, C > 0$ and a nontrivial proper ideal $\mathfrak{q}_0 \subset \mathcal{O}_{\mathbb K}$ such that for all square free ideals $\mathfrak{q} \subset \mathcal{O}_{\mathbb K}$ coprime to $\mathfrak{q}_0$, and for all $M$-invariant functions $\phi, \psi \in C_c^1(\Gamma_\mathfrak{q} \backslash G, \mathbb R)$, there exists $C_{\phi, \psi} > 0$ such that we have
\begin{multline*}
\left|e^{(1 - \delta_\Gamma)t}\langle a_t\phi, \psi \rangle - \frac{1}{m^{\mathrm{BMS}}_\mathfrak{q}(\Gamma_\mathfrak{q} \backslash G)} m^{\mathrm{BR}}_\mathfrak{q}(\phi) \cdot m^{\mathrm{BR}_*}_\mathfrak{q}(\psi)\right| \\
\leq C_{\phi, \psi}N_{\mathbb K}(\mathfrak{q})^C e^{-\eta t} \|\phi\|_{C^1} \|\psi\|_{C^1}
\end{multline*}
where $C_{\phi, \psi}$ can be taken to depend only the supports of the functions $\phi$ and $\psi$.
\end{corollary}

\subsection{Outline of the proof of \texorpdfstring{\cref{thm:TheoremUniformExponentialMixingOfGeodesicFlow}}{\autoref{thm:TheoremUniformExponentialMixingOfGeodesicFlow}}}
First we recount the proof of exponential mixing of the geodesic flow on the single manifold $\T^1(\Gamma \backslash \mathbb H^n)$. From the works of Bowen and Ratner \cite{Bow70,Rat73}, it is well known that there are Markov sections for the geodesic flow on $\T^1(\Gamma \backslash \mathbb H^n)$. This allows us to model the geodesic flow as a suspension space of subshift of finite type on bi-infinite sequences $\Sigma$. For any function $\phi \in C^1(\Gamma \backslash G/M, \mathbb R)$, we can integrate out the strong stable direction of the Markov sections and take the Laplace transform in the flow direction to be left to deal with functions on one sided infinite sequences $\Sigma^+$ instead. Moreover, by Pollicott's observation which was used by many others \cite{AGY06,Dol98,Sto11}, the Ruelle-Perron-Frobenius theorem can be used to cleanly write the Laplace transform of the correlation function as an infinite sum involving the transfer operators $\mathcal{L}_\xi: C(\Sigma^+, \mathbb R) \to C(\Sigma^+, \mathbb R)$ for $\xi \in \mathbb C$ defined by
\begin{align*}
\mathcal{L}_\xi(h)(x) = \sum_{x' \in \sigma^{-1}(x)} e^{-(a + \delta_\Gamma - ib)\tau(x')}h(x').
\end{align*}
Now, by a Paley-Wiener type of analysis, we can extract the desired exponential decay using the inverse Laplace transform formula given that the Laplace transform has a holomorphic extension to the left of the imaginary axis. To obtain the holomorphic extension, we are lead to find spectral bounds for the transfer operators. A major advancement regarding this study is the work of Dolgopyat \cite{Dol98}. For transfer operators with large frequencies $|\Im(\xi)| \gg 1$, he was able to obtain spectral bounds by working explicitly on the strong unstable leaves of the Markov sections rather than the purely symbolic space in order to use the geometry of the manifold to obtain sufficiently strong bounds. The geometry provides a crucial local non-integrability condition which implies highly oscillating summands in the transfer operator and hence large cancellations which provide the required bounds. Moreover, he provided the right technical framework for the whole process to work in harmony. What is left are the transfer operators with small frequencies $|\Im(\xi)| \ll 1$ but they do not cause any problems due to the availability of the complex Ruelle-Perron-Frobenius theorem and a compactness argument which completes the proof.

For the proof of uniform exponential mixing, we can proceed as above, but we must use instead the \emph{congruence} transfer operators $\mathcal{M}_{\xi, \mathfrak{q}}: C(\Sigma^+, L^2(\Gamma_\mathfrak{q} \backslash \Gamma, \mathbb C)) \to C(\Sigma^+, L^2(\Gamma_\mathfrak{q} \backslash \Gamma, \mathbb C))$ for $\xi \in \mathbb C$ and ideal $\mathfrak{q} \subset \mathcal{O}_{\mathbb K}$ defined by
\begin{align*}
\mathcal{M}_{\xi, \mathfrak{q}}(H)(x) &= \sum_{x' \in \sigma^{-1}(x)} e^{-(a + \delta_\Gamma - ib)\tau(x')} \mathtt{c}_\mathfrak{q}(x')^{-1} H(x').
\end{align*}
Here the cocycle $\mathtt{c}_\mathfrak{q}$ ``keeps track of the coordinate in the fibers'' of the congruence cover $\T^1(\Gamma_\mathfrak{q} \backslash \mathbb H^n) \to \T^1(\Gamma \backslash \mathbb H^n)$. With this formulation, we now require spectral bounds for the congruence transfer operators uniform in the ideals $\mathfrak{q} \subset \mathcal{O}_{\mathbb K}$. A simple but crucial observation of Oh-Winter \cite{OW16} is that the cocycle is locally constant. Consequently, there is no interference with Dolgopyat's method for large frequencies $|\Im(\xi)| \gg 1$ and following \cite{OW16,Sto11}, it can be carried through uniformly in the ideals $\mathfrak{q} \subset \mathcal{O}_{\mathbb K}$. But now, since there are countably many ideals $\mathfrak{q} \subset \mathcal{O}_{\mathbb K}$, the compactness argument for small frequencies $|\Im(\xi)| \ll 1$ fails to hold uniformly in the ideals $\mathfrak{q} \subset \mathcal{O}_{\mathbb K}$.

Nevertheless, Bourgain-Gamburd-Sarnak demonstrated in their breakthrough work \cite{BGS11} that the required bounds are attainable using entirely new methods of expander graphs. Built on their work is the expander machinery of Golsefidy-Varj\'{u} \cite{GV12} which is what we use. The expander machinery captures the spectral properties of a certain family of Cayley graphs coming from the congruence setting. This time, we obtain large cancellations of the summands in the congruence transfer operator due to the cocyle. We follow \cite{OW16,MOW17} to use the expander machinery along with other techniques to take care of the spectral bounds for small frequencies $|\Im(\xi)| \ll 1$. In particular, Bourgain-Kontorovich-Magee has shown in the appendix of \cite{MOW17} a method to use the expander machinery more directly which we adapt in this paper. In their work however, Zariski density of certain subgroups of $\Gamma$ which appears in the argument are crucial for successfully using the expander machinery. This was a trivial point in their work in the Schottky semigroups and continued fractions settings but in the setting of this paper, this is not at all a triviality, although it is quite natural to expect it to be true. The proof of this fact is the main new ingredient in this paper and completes the necessary tools to obtain spectral bounds for small frequencies $|\Im(\xi)| \ll 1$ uniform in the ideals $\mathfrak{q} \subset \mathcal{O}_{\mathbb K}$ which completes the proof.

\subsection{Organization of the paper}
We start with reviewing the necessary background in \cref{sec:BackgroundAndNotations} and also go through important constructions for the rest of the paper. Then we first go through the expander machinery part of the argument in \cref{sec:GeodesicFlowSmall|b|} to obtain uniform spectral bounds for small frequencies. In \cref{sec:GeodesicFlowLarge|b|}, we use Dolgopyat's method to obtain uniform spectral bounds for large frequencies. Finally in \cref{sec:UniformExponentialMixingOfTheGeodesicFlow}, we use the obtained uniform spectral bounds to go through arguments by Pollicott along with Paley-Wiener theory to prove uniform exponential mixing.

\subsection*{Acknowledgements}
This work is part of my Ph.D. thesis at Yale University. First and foremost I thank my advisor Hee Oh for introducing the field of homogeneous dynamics and suggesting to think about this problem and helpful guidance throughout. I thank Dale Winter for explaining parts of his work. I also thank Wenyu Pan for helpful discussions early on.

%First and foremost I thank my advisor Hee Oh for introducing the field of homogeneous dynamics and suggesting to think about this problem and helpful guidance throughout. I thank Dale Winter and Michael Magee for explaining parts of their work. I also thank Minju Lee, Wenyu Pan, Jayameenakshi Venkatraman and Raghavendra Venkatraman for helpful conversations.

\section{Background and notations}
\label{sec:BackgroundAndNotations}
Fix an integer $n \geq 2$ and let $\mathbb H^n$ be the $n$-dimensional hyperbolic space, i.e., the unique complete simply connected $n$-dimensional Riemannian manifold with constant negative sectional curvature. We denote by $\langle \cdot, \cdot\rangle$ and $\|\cdot\|$ the inner product and norm respectively on any tangent space of $\mathbb H^n$ induced by the hyperbolic metric. Similarly, we denote by $d$ the distance function on $\mathbb H^n$ induced by the hyperbolic metric. Let $\mathbb K \subset \mathbb R$ be a totally real number field, $\mathcal{O}_{\mathbb K}$ be the corresponding ring of integers, and $Q$ be a quadratic form of signature $(n, 1)$ defined over $\mathbb K$. Let $\mathbf{G} = \SO_Q(\mathbb C) < \GL_{n + 1}(\mathbb C)$ be an algebraic group defined over $\mathbb K$ such that $\mathbf{G}(\mathbb R) \cong \SO(n, 1)$ and $\mathbf{G}^\sigma(\mathbb R) \cong \SO(n + 1)$, which is compact, for all nontrivial embeddings $\sigma: \mathbb K \hookrightarrow \mathbb R$. Let $G = \mathbf{G}(\mathbb R)^\circ$ which we recognize as the group of orientation preserving isometries of $\mathbb H^n$ by recalling that $\Isom^+(\mathbb H^n) \cong \SO(n, 1)^\circ$. Let $\Gamma < G$ be a torsion-free convex cocompact subgroup and Zariski dense in $\mathbf G$. Let $o \in \mathbb H^n$ be a reference point and $v_o \in \T^1(\mathbb H^n)$ be a reference tangent vector at $o$. Then we have the stabilizers $K = \Stab_G(o)$ and $M = \Stab_G(v_o) < K$. Note that $K \cong \SO(n)$ and it is a maximal compact subgroup of $G$ and $M \cong \SO(n - 1)$. Our base hyperbolic manifold is $X = \Gamma \backslash \mathbb H^n \cong \Gamma \backslash G/K$, its unit tangent bundle is $\T^1(X) \cong \Gamma \backslash G/M$ and its oriented orthonormal frame bundle is $\F_{\SO}(X) \cong \Gamma \backslash G$ which is a principal $\SO(n)$-bundle over $X$ and a principal $\SO(n - 1)$-bundle over $\T^1(X)$. There is a one-parameter subgroup of semisimple elements $A = \{a_t: t \in \mathbb R\} < G$, where $C_G(A) = AM$, parametrized such that its canonical right action on $G/M$ and $G$ corresponds to the geodesic flow and frame flow respectively. We choose any Riemannian metric on $G$ such that it is left $G$-invariant and right $K$-invariant \cite{Sas58,Mok78} and again use the notations $\langle \cdot, \cdot\rangle, \|\cdot\|$ and $d$ on $G$ and any of its quotient spaces. In particular, the metric descends down to $\mathbb H^n \cong G/K$ and coincides with the previous hyperbolic metric.

%\Gamma finitely generated is implied since it is convex cocompact (and hence geometrically finite).

%We also use the notations $\langle \cdot, \cdot\rangle, \|\cdot\|$ and $d$ on any Riemannian manifold where there is a canonical Riemannian metric induced by the hyperbolic metric on $\mathbb H^n$. In particular, there is a one parameter family of Riemannian metrics which are left $G$-invariant and right $K$-invariant on $\F_{\SO}(\mathbb H^n) \cong G$ called Sasaki-Mok metrics. The one parameter family is in one to one correspondence with biinvariant metrics on $\mathfrak{so}(n)$ which are all scalar multiples of each other and any one particular choice suffices. By right $K$ invariance, the Sasaki-Mok metric then decends to the Sasaki metric on $\T^1(\mathbb H^n)$. See Infinitesimal isometries of frame bundles with a natural Riemannian metric and Isometry types of frame bundles for details. By left $G$ invariance, these Riemannian metrics then decend to right quotients by $\Gamma$. \\

%and also on any Riemannian manifold where the Riemannian metric is induced by the Riemannian metric on $\mathbb H^n$, in particular the distance function induced by the Sasaki metric

To make use of the strong approximation theorem of Weisfeiler \cite{Wei84} later on, we need to work on the simply connected cover $\tilde{\mathbf G}$ endowed with the covering map $\tilde{\pi}: \tilde{\mathbf G} \to \mathbf G$ defined over $\mathbb K$. Let $\tilde{G} = \tilde{\mathbf G}(\mathbb R)$ which is connected and projects down to $\tilde{\pi}(\tilde{G}) = G$. Let $\tilde{\Gamma} < \tilde{G}$ be a convex cocompact subgroup containing $\ker(\tilde{\pi}) = \{e, -e\}$ as the only torsion elements and Zariski dense in $\tilde{\mathbf{G}}$. To be able to discuss the notion of congruence subgroups, let us suppose that there is a totally real number field $\mathbb K$ with ring of integers $\mathcal{O}_{\mathbb K}$ such that $\tilde{\Gamma} < \tilde{\mathbf G}(\mathcal{O}_{\mathbb K})$ which is indeed discrete in $\tilde{G}$ by our assumption that $\mathbf{G}^\sigma(\mathbb R)$ is compact for all nontrivial embeddings $\sigma: \mathbb K \hookrightarrow \mathbb R$. Another assumption we will make to use the strong approximation theorem is that the subring of $\mathbb K$ generated by $\{\tr(\Ad_g): g \in \tilde{\Gamma}\}$ is $\mathcal{O}_{\mathbb K}$. Then we take $\Gamma$ introduced previously to be $\Gamma = \tilde{\pi}(\tilde{\Gamma})$.

Let $\partial_\infty(\mathbb H^n)$ denote the boundary at infinity and $\Lambda(\Gamma) \subset \partial_\infty(\mathbb H^n)$ denote the limit set of $\Gamma$. We denote $\{\mu^{\mathrm{PS}}_x: x \in \mathbb H^n\}$ to be the \emph{Patterson-Sullivan density} of $\Gamma$ \cite{Pat76,Sul79}, i.e., the set of finite Borel measures on $\partial_\infty(\mathbb H^n)$ supported on $\Lambda(\Gamma)$ such that
\begin{enumerate}
\item	$g_*\mu^{\mathrm{PS}}_x = \mu^{\mathrm{PS}}_{gx}$ for all $g \in \Gamma$, for all $x \in \mathbb H^n$
\item	$\frac{d\mu^{\mathrm{PS}}_x}{d\mu^{\mathrm{PS}}_y}(\xi) = e^{\delta_\Gamma \beta_{\xi}(y, x)}$ for all $\xi \in \partial_\infty(\mathbb H^n)$, for all $x, y \in \mathbb H^n$
\end{enumerate}
where $\beta_{\xi}$ denotes the \emph{Busemann function} at $\xi \in \partial_\infty(\mathbb H^n)$ defined by $\beta_{\xi}(y, x) = \lim_{t \to \infty} (d(\xi(t), y) - d(\xi(t), x))$, where $\xi: \mathbb R \to \mathbb H^n$ is any geodesic such that $\lim_{t \to \infty} \xi(t) = \xi$. We allow tangent vector arguments for the Busemann function as well in which case we will use their basepoints in the definition. Since $\Gamma$ is convex cocompact, for all $x \in \mathbb H^n$, the measure $\mu^{\mathrm{PS}}_x$ is the $\delta_\Gamma$-dimensional Hausdorff measure on $\partial_\infty(\mathbb H^n)$ supported on $\Lambda(\Gamma)$ corresponding to the spherical metric on $\partial_\infty(\mathbb H^n)$ with respect to $x$, up to scalar multiples. Now, using the Hopf parametrization via the homeomorphism $G/M \cong \T^1(\mathbb H^n) \to \{(u^+, u^-) \in \partial_\infty(\mathbb H^n) \times \partial_\infty(\mathbb H^n): u^+ \neq u^-\} \times \mathbb R$ defined by $u \mapsto (u^+, u^-, t = \beta_{u^-}(o, u))$, we define the \emph{Bowen-Margulis-Sullivan (BMS) measure} $m^{\mathrm{BMS}}$ on $G/M$ \cite{Mar04,Bow71,Kai90} by
\begin{align*}
dm^{\mathrm{BMS}}(u) = e^{\delta_\Gamma \beta_{u^+}(o, u)} e^{\delta_\Gamma \beta_{u^-}(o, u)} \, d\mu^{\mathrm{PS}}_o(u^+) \, d\mu^{\mathrm{PS}}_o(u^-) \, dt
\end{align*}
Note that this definition only depends on $\Gamma$ and not on the choice of reference point $o \in \mathbb H^n$ and moreover $m^{\mathrm{BMS}}$ is left $\Gamma$-invariant. We now define induced measures on other spaces all of which we call the BMS measures and denote by $m^{\mathrm{BMS}}$ by abuse of notation. Since $M$ is compact, we can use any Haar measure on $M$ to lift $m^{\mathrm{BMS}}$ to a right $M$-invariant measure on $G$. By left $\Gamma$-invariance, $m^{\mathrm{BMS}}$ now descends to a measure on $\Gamma \backslash G$. By right $M$-invariance, $m^{\mathrm{BMS}}$ descends once more to a measure on $\Gamma \backslash G/M$. It can be checked that the BMS measures are invariant with respect to the geodesic flow or the frame flow as appropriate, i.e., right $A$-invariant. We denote $\Omega = \supp\left(m^{\mathrm{BMS}}\right) \subset \Gamma \backslash G/M$ which is compact since $\Gamma$ is convex cocompact, and it is also invariant with respect to the geodesic flow.

\subsection{Markov sections}
We use Markov sections on $\Omega \subset \T^1(X) \cong \Gamma \backslash G/M$, as developed by Bowen and Ratner \cite{Bow70,Rat73}, to obtain a symbolic coding of the dynamical system at hand. Let $W^{\mathrm{su}}(w) \subset \T^1(X)$ and $W^{\mathrm{ss}}(w) \subset \T^1(X)$ denote the leaves through $w \in \T^1(X)$ of the strong unstable and strong stable foliations, and $W_{\epsilon}^{\mathrm{su}}(w) \subset W^{\mathrm{su}}(w)$ and $W_{\epsilon}^{\mathrm{ss}}(w) \subset W^{\mathrm{ss}}(w)$ denote the open balls of radius $\epsilon > 0$ with respect to the induced distance functions $d_{\mathrm{su}}$ and $d_{\mathrm{ss}}$, respectively. We use similar notations for the weak unstable and weak stable foliations by replacing `su' with `wu' and `ss' with `ws' respectively. We recall that for all $w \in \T^1(X)$, for all $u \in W_{\epsilon_0}^{\mathrm{wu}}(w)$, for all $s \in W_{\epsilon_0}^{\mathrm{ss}}(w)$, there is a unique intersection denoted by $[u, s] = W_{\epsilon_0'}^{\mathrm{ss}}(u) \cap W_{\epsilon_0'}^{\mathrm{wu}}(s)$ and moreover $[\cdot, \cdot]$ defines a homeomorphism from $W_{\epsilon_0}^{\mathrm{wu}}(w) \times W_{\epsilon_0}^{\mathrm{ss}}(w)$ onto its image, where $\epsilon_0, \epsilon_0' > 0$ are constants from \cite{Rat73} which we use in this subsection. Subsets $U \subset W_{\epsilon_0}^{\mathrm{su}}(w) \cap \Omega$ and $S \subset W_{\epsilon_0}^{\mathrm{ss}}(w) \cap \Omega$ are called proper if $U = \overline{\interior^{\mathrm{su}}(U)}^{\mathrm{su}}$ and $S = \overline{\interior^{\mathrm{ss}}(S)}^{\mathrm{ss}}$ where the superscripts signify that the interiors and closures are taken in the topology of $W^{\mathrm{su}}(w) \cap \Omega$ and $W^{\mathrm{ss}}(w) \cap \Omega$, respectively. We will often drop the superscripts henceforth and include it whenever further clarity in notations is required. For any $w \in \Omega$ and proper sets $U \subset W_{\epsilon_0}^{\mathrm{su}}(w) \cap \Omega$ and $S \subset W_{\epsilon_0}^{\mathrm{ss}}(w) \cap \Omega$ both containing $w$, we call $R = [U, S] = \{[u, s] \in \Omega: u \in U, s \in S\} \subset \Omega$ a \emph{rectangle of size $\hat{\delta}$} if $\diam_{d_{\mathrm{su}}}(U_j), \diam_{d_{\mathrm{ss}}}(S_j) \leq \hat{\delta}$ for some $\hat{\delta} > 0$, and we call $w$ the \emph{center} of $R$. For any rectangle $R = [U, S]$, we generalize the notation and define $[v_1, v_2] = [u_1, s_2]$ for all $v_1 = [u_1, s_1] \in R$, for all $v_2 = [u_2, s_2] \in R$.

\begin{definition}
A set $\mathcal{R} = \{R_1, R_2, \dotsc, R_N\} = \{[U_1, S_1], [U_2, S_2], \dotsc, [U_N, S_N]\}$ for some $N \in \mathbb Z_{>0}$ consisting of rectangles is called a \emph{complete set of rectangles} of size $\hat{\delta}$ in $\Omega$ if
\begin{enumerate}
\item \label{itm:MarkovProperty1} $R_j \cap R_k = \varnothing$ for all integers $1 \leq j, k \leq N$ with $j \neq k$
\item \label{itm:MarkovProperty2} $\diam_{d_{\mathrm{su}}}(U_j), \diam_{d_{\mathrm{ss}}}(S_j) \leq \hat{\delta}$ for all integers $1 \leq j \leq N$
\item \label{itm:MarkovProperty3} $\Omega = \bigcup_{j = 1}^N \bigcup_{t \in [0, \hat{\delta}]} R_j a_t$.
\end{enumerate}
\end{definition}
%\item \label{itm:MarkovProperty3} for all integers $1 \leq j, k \leq N$ with $j \neq k$, we have $R_j \cap \bigcup_{t \in [0, \hat{\delta}]}R_k a_t = \varnothing$ or $R_k \cap \bigcup_{t \in [0, \hat{\delta}]}R_j a_t = \varnothing$.

We set $R = \bigsqcup_{j = 1}^N R_j$ and $U = \bigsqcup_{j = 1}^N U_j$. Let $\mathcal{P}: R \to R$ be the Poincar\'{e} first return map and $\sigma = ({\proj_U} \circ \mathcal{P})|_U: U \to U$ be its projection where $\proj_U: R \to U$ is the projection defined by $\proj_U([u, s]) = u$ for all $[u, s] \in R$. Let $\tau: R \to \mathbb R_{>0}$ be the first return time map defined by $\tau(u) = \inf\{t \in \mathbb R_{>0}: ua_t \in R\}$ so that $\mathcal{P}(u) = ua_{\tau(u)}$ for all $u \in R$. Note that $\tau$ is constant on $[u, S_j]$ for all $u \in U_j$, for all integers $1 \leq j \leq N$. For future convenience we define $\overline{\tau} = \sup_{u \in R} \tau(u)$ and $\underline{\tau} = \inf_{u \in R} \tau(u)$. Define the \emph{cores} $\hat{R} = \{u \in R: \mathcal{P}^k(u) \in \interior(R) \text{ for all } k \in \mathbb Z\}$ and $\hat{U} = \{u \in U: \sigma^k(u) \in \interior(U) \text{ for all } k \in \mathbb Z_{\geq 0}\}$. We note that the cores are both residual subsets (complements of meager sets) of $R$ and $U$ respectively.

\begin{definition}
We call the complete set $\mathcal{R}$ a \emph{Markov section} if in addition to \cref{itm:MarkovProperty1,itm:MarkovProperty2,itm:MarkovProperty3}, the following property
\begin{enumerate}
\setcounter{enumi}{3}
\item $[\interior(U_k), \mathcal{P}(u)] \subset \mathcal{P}([\interior(U_j), u])$ and $\mathcal{P}([u, \interior(S_j)]) \subset [\mathcal{P}(u), \interior(S_k)]$ for all integers $1 \leq i, j \leq N$ and for all $u \in R$ such that $u \in \interior(R_j) \cap \mathcal{P}^{-1}(\interior(R_k)) \neq \varnothing$
\end{enumerate}
called the \emph{Markov property}, is satisfied.
\end{definition}

The existence of Markov sections for Anosov flows (and hence in particular for geodesic flows) of arbitrarily small size was proved by by Bowen and Ratner \cite{Bow70,Rat73}. Henceforth, we fix a positive
\begin{align*}
\hat{\delta} < \min\left(1, \epsilon_0, \epsilon_0', \frac{1}{4}\Inj(\T^1(X))\right)
\end{align*}
where $\Inj(\T^1(X))$ denotes the injectivity radius of $\T^1(X)$. We also fix $\mathcal{R} = \{R_1, R_2, \dotsc, R_N\} = \{[U_1, S_1], [U_2, S_2], \dotsc, [U_N, S_N]\}$ to be a Markov section of size $\hat{\delta}$ in $\Omega$. We introduce the distance function $d$ on $U$ defined by
\begin{align*}
d(u, u') =
\begin{cases}
d_{\mathrm{su}}(u, u'), & u, u' \in U_j \text{ for some integer } 1 \leq j \leq N \\
1, & \text{otherwise.}
\end{cases}
\end{align*}

\begin{remark}
This definition makes sense since $\diam_{d_{\mathrm{su}}}(U_j) \leq \hat{\delta} < 1$ for all integers $1 \leq j \leq N$. This notation should not cause any confusion but we will use $d_{\mathrm{su}}$ whenever further clarity is required.
\end{remark}

\subsection{Symbolic dynamics}
Let $\mathcal A = \{1, 2, \dotsc, N\}$ be the alphabet for the symbolic coding corresponding to the Markov sections. Define the $N \times N$ \emph{transition matrix} $T$ by
\begin{align*}
T_{j, k} =
\begin{cases}
1, & \interior(R_j) \cap \mathcal{P}^{-1}(\interior(R_k)) \neq \varnothing \\
0, & \text{otherwise}
\end{cases}
\end{align*}
for all integers $1 \leq j, k \leq N$. The transition matrix $T$ is \emph{topologically mixing} \cite[Theorem 4.3]{Rat73}, i.e., there is a fixed $N_T \in \mathbb Z_{>0}$ such that all the entries of $T^{N_T}$ are positive. This definition is equivalent to the one in \cite{Rat73} in the setting of Markov sections. Define the spaces of bi-infinite and infinite \emph{admissible sequences} respectively by
\begin{align*}
\Sigma &= \{(\dotsc, x_{-1}, x_0, x_1, \dotsc) \in \mathcal A^{\mathbb Z}: T_{x_j, x_{j + 1}} = 1 \text{ for all } j \in \mathbb Z\} \\
\Sigma^+ &= \{(x_0, x_1, \dotsc) \in \mathcal A^{\mathbb Z_{\geq 0}}: T_{x_j, x_{j + 1}} = 1 \text{ for all } j \in \mathbb Z_{\geq 0}\}.
\end{align*}
We will use the term \emph{admissible sequences} for finite sequences as well in the natural way. For any $\theta \in (0, 1)$, we can endow $\Sigma$ with the distance function $d_\theta(x, y) = \theta^{\inf\{|j| \in \mathbb Z_{\geq 0}: x_j \neq y_j \text{ for } j \in \mathbb Z\}}$ for all $x, y \in \Sigma$. We can similarly endow $\Sigma^+$ with a distance function which we also denote by $d_\theta$.

\begin{definition}
For all $k \in \mathbb Z_{\geq 0}$, for all admissible sequences $x = (x_0, x_1, \dotsc, x_k)$, we define the corresponding \emph{cylinder} to be $\mathtt{C}[x] = \{u \in U: \sigma^j(u) \in \interior(U_{x_j}), 0 \leq j \leq k\}$ with \emph{length} $\len(\mathtt{C}[x]) = k$. We will denote cylinders simply by $\mathtt{C}$ (or other typewriter style letters) when we do not need to specify the corresponding admissible sequence. A \emph{closed cylinder} $\mathtt{C}$ will be the closure of some cylinder $\mathtt{C}'$, i.e., $\mathtt{C} = \overline{\mathtt{C}'}$.
\end{definition}

\begin{remark}
For all admissible pairs $(j, k)$, the restricted maps $\sigma|_{\mathtt{C}[j, k]}: \mathtt{C}[j, k] \to \interior(U_k), (\sigma|_{\mathtt{C}[j, k]})^{-1}: \interior(U_k) \to \mathtt{C}[j, k]$ and $\tau|_{\mathtt{C}[j, k]}: \mathtt{C}[j, k] \to \mathbb R_{>0}$ are Lipschitz.
\end{remark}

By a slight abuse of notation, let $\sigma$ also denote the shift map on $\Sigma$ or its subspaces. There are natural continuous surjections $\zeta: \Sigma \to R$ and $\zeta^+: \Sigma^+ \to U$ defined by $\zeta(x) = \bigcap_{j = -\infty}^\infty \overline{\mathcal{P}^{-j}(\interior(R_{x_j}))}$ for all $x \in \Sigma$ and $\zeta^+(x) = \bigcap_{j = 0}^\infty \overline{\sigma^{-j}(\interior(U_{x_j}))}$ for all $x \in \Sigma^+$. Define $\hat{\Sigma} = \zeta^{-1}(\hat{R})$ and $\hat{\Sigma}^+ = (\zeta^+)^{-1}(\hat{U})$. Then the restrictions $\zeta|_{\hat{\Sigma}}: \hat{\Sigma} \to \hat{R}$ and $\zeta^+|_{\hat{\Sigma}^+}: \hat{\Sigma}^+ \to \hat{U}$ are bijective and satisfy $\zeta|_{\hat{\Sigma}} \circ \sigma|_{\hat{\Sigma}} = \mathcal{P}|_{\hat{R}} \circ \zeta|_{\hat{\Sigma}}$ and $\zeta^+|_{\hat{\Sigma}^+} \circ \sigma|_{\hat{\Sigma}^+} = \sigma|_{\hat{U}} \circ \zeta^+|_{\hat{\Sigma}^+}$.

For $\theta \in (0, 1)$ sufficiently close to $1$, the maps $\zeta$ and $\zeta^+$ are Lipschitz \cite[Lemma 2.2]{Bow73}. We fix $\theta$ to be any such constant. We now introduce some function spaces corresponding to $\Sigma, \Sigma^+$ and $U$. Let $B(\Sigma, \mathbb R) \supset C(\Sigma, \mathbb R) \supset C^{\Lip(d_\theta)}(\Sigma, \mathbb R)$ denote the spaces of functions $f: \Sigma \to \mathbb R$ which are bounded, continuous and Lipschitz respectively. The first two are Banach spaces with the $L^\infty$ norm $\|f\|_\infty$ for all $f \in B(\Sigma, \mathbb R)$ or $f \in C(\Sigma, \mathbb R)$ and the last is a Banach space with the norm $\|f\|_{\Lip(d_\theta)} = \|f\|_\infty + \Lip_{d_\theta}(f)$ where
\begin{align*}
\Lip_{d_\theta}(f) = \sup\left\{\frac{|f(x) - f(y)|}{d_\theta(x, y)} \in \mathbb R_{>0}: x, y \in \Sigma \text{ such that } x \neq y\right\}
\end{align*}
is the Lipschitz seminorm, for all $f \in C^{\Lip(d_\theta)}(\Sigma, \mathbb R)$. We use similar notations for function spaces with domain space $\Sigma^+$ or $U$. We use similar notations again for target space $\mathbb C$.

Since $(\tau \circ \zeta)|_{\hat{\Sigma}}$ and $(\tau \circ \zeta^+)|_{\hat{\Sigma}^+}$ are Lipschitz, there are unique Lipschitz extensions which we denote by $\tau_\Sigma: \Sigma \to \mathbb R_{>0}$ and $\tau_{\Sigma^+}: \Sigma^+ \to \mathbb R_{>0}$ respectively. Note that the resulting maps are distinct from $\tau \circ \zeta$ and $\tau \circ \zeta^+$ because they may differ precisely on $x \in \Sigma$ for which $\proj_U(\zeta(x)) \in \partial(\mathtt{C})$ and $x \in \Sigma^+$ for which $\zeta^+(x) \in \partial(\mathtt{C})$ respectively, for some cylinder $\mathtt{C} \subset U$ with $\len(\mathtt{C}) = 1$. Then the previous properties extend to $\zeta(\sigma(x)) = \zeta(x)a_{\tau_\Sigma(x)}$ for all $x \in \Sigma$ and $\zeta^+(\sigma(x)) = \proj_U(\zeta^+(x)a_{\tau_{\Sigma^+}(x)})$ for all $x \in \Sigma^+$.

\subsection{Thermodynamics}
\begin{definition}
\label{def:Pressure}
For all $f \in C^{\Lip(d_\theta)}(\Sigma, \mathbb R)$, called the \emph{potential}, the \emph{pressure} is
\begin{align*}
\Pr_\sigma(f) = \sup_{\nu} \int_\Sigma f \, d\nu + h_\nu(\sigma)
\end{align*}
where the supremum is taken over all $\sigma$-invariant Borel probability measures $\nu$ on $\Sigma$ and $h_\nu(\sigma)$ is the measure theoretic entropy of $\nu$ with respect to $\sigma$.
\end{definition}

For all $f \in C^{\Lip(d_\theta)}(\Sigma, \mathbb R)$, there is in fact a unique $\sigma$-invariant Borel probability measure $\nu_f$ on $\Sigma$ which attains the supremum in \cref{def:Pressure} called the \emph{$f$-equilibrium state} \cite[Theorems 2.17 and 2.20]{Bow08} and it satisfies $\nu_f(\hat{\Sigma}) = 1$ \cite[Corollary 3.2]{Che02}.

In particular, we will consider the probability measure $\nu_{-\delta_\Gamma\tau_\Sigma}$ on $\Sigma$ which we will denote simply by $\nu_\Sigma$ and has corresponding pressure $\Pr_\sigma(-\delta_\Gamma\tau_\Sigma) = 0$. According to above, $\nu_\Sigma(\hat{\Sigma}) = 1$. Define the corresponding probability measure $\nu_R = \zeta_*(\nu_\Sigma)$ on $R$ and note that $\nu_R(\hat{R}) = 1$. Now consider the suspension space $R^\tau = (R \times \mathbb R_{\geq 0})/\mathord{\sim}$ where $\sim$ is the equivalence relation on $R \times \mathbb R_{\geq 0}$ defined by $(u, t + \tau(u)) \sim (\mathcal{P}(u), t)$. Then we have a bijection $R^\tau \to \Omega$ defined by $(u, t) \mapsto ua_t$. We can define the measure $\nu^\tau$ on $R^\tau$ as the product measure $\nu_R \times m^{\mathrm{Leb}}$ on $\{(u, t) \in R \times \mathbb R_{\geq 0}: 0 \leq t < \tau(u)\}$. Then using the aforementioned bijection we have the pushforward measure which, by abuse of notation, we also denote by $\nu^\tau$ on $\T^1(X)$ supported on $\Omega$.

\begin{theorem}
\label{thm:FlowEquilibriumStateEqualsBMS}
We have $\frac{\nu^\tau}{\nu_R(\tau)} = \frac{m^{\mathrm{BMS}}}{m^{\mathrm{BMS}}(\Gamma \backslash G/M)}$.
\end{theorem}

\begin{proof}
Sullivan \cite{Sul84} proved that $m^{\mathrm{BMS}}$ is the unique measure of maximal entropy for the geodesic flow on $\T^1(X)$. By \cite[Theorem 4.4]{Che02}, $\nu^\tau$ is also the unique measure of maximal entropy on $\T^1(X)$. Thus they are equal after normalization.
\end{proof}

Finally, we define the probability measure $\nu_U = (\proj_U)_*(\nu_R)$ and note that $\nu_U(\hat{U}) = 1$ and $\nu_U(\tau) = \nu_R(\tau)$. As a consequence of the RPF theorem, $\nu_U$ is in fact a Gibbs measure \cite{Bow08, PP90}, i.e., there are $c_1^U, c_2^U > 0$ such that
\begin{align}
\label{eqn:PropertyOfGibbsMeasures}
c_1^Ue^{-\delta_\Gamma \tau_k(x)} \leq \nu_U(\mathtt{C}) \leq c_2^Ue^{-\delta_\Gamma \tau_k(x)}
\end{align}
for all $x \in \mathtt{C}$, for all cylinders $\mathtt{C}$ with $\len(\mathtt{C}) = k \in \mathbb Z_{\geq 0}$, since $\lambda_0 = 1$. See the next subsection for terms and notations.

\subsection{Transfer operators}
First we recall the classic transfer operator and associated theorems.

\begin{definition}
For all $f \in B(\Sigma^+, \mathbb C)$, we define the \emph{transfer operator} $\mathcal{L}_f: B(\Sigma^+, \mathbb C) \to B(\Sigma^+, \mathbb C)$ by
\begin{align*}
\mathcal{L}_f(h)(x) = \sum_{x' \in \sigma^{-1}(x)} e^{f(x')}h(x')
\end{align*}
for all $x \in \Sigma^+$, for all $h \in B(\Sigma^+, \mathbb C)$.
\end{definition}

\begin{remark}
In fact, if $f \in C(\Sigma^+, \mathbb C)$ or $f \in C^{\Lip(d_\theta)}(\Sigma^+, \mathbb C)$, then the transfer operator $\mathcal{L}_f$ preserves the subspace $C(\Sigma^+, \mathbb C)$ or $C^{\Lip(d_\theta)}(\Sigma^+, \mathbb C)$ respectively.
\end{remark}

The following theorem is a consequence of the Ruelle-Perron-Frobenius (RPF) theorem and the theory of Gibbs measures \cite{Bow08,PP90}.

\begin{theorem}
\label{thm:RPF}
For all $f \in C^{\Lip(d_\theta)}(\Sigma^+, \mathbb R)$, the operator $\mathcal{L}_f: C(\Sigma^+, \mathbb C) \to C(\Sigma^+, \mathbb C)$ and its dual $\mathcal{L}_f^*: C(\Sigma^+, \mathbb C)^* \to C(\Sigma^+, \mathbb C)^*$ has eigenvectors with the following properties. There exist a unique positive function $h \in C^{\Lip(d_\theta)}(\Sigma^+, \mathbb R)$ and a unique Borel probability measure $\nu$ on $\Sigma^+$ such that
\begin{enumerate}
\item	$\mathcal{L}_f(h) = e^{\Pr_\sigma(f)}h$
\item	$\mathcal{L}_f^*(\nu) = e^{\Pr_\sigma(f)}\nu$
\item	the eigenvalue $e^{\Pr_\sigma(f)}$ is maximal simple and the rest of the spectrum of $\mathcal{L}_f|_{C^{\Lip(d_\theta)}(\Sigma^+, \mathbb C)}$ is contained in a disk of radius strictly less than $e^{\Pr_\sigma(f)}$
%\item	for all $k \in \mathbb Z_{\geq 0}$, we have
%\begin{align*}
%\left|e^{-k\Pr(f)} \mathcal{L}_f^k(\phi)(u) - \hat{\nu}(\phi)\hat{h}(u)\right| \leq c(1 - \epsilon)^k\|\phi\|%_{\Lip(d_\theta)}
%\end{align*}
%where $\hat{h}$ is normalized such that $\hat{\nu}(\hat{h}) = 1$
\item	$\nu(h) = 1$ and the Borel probability measure $\mu$ defined by $d\mu = h \, d\nu$ is $\sigma$-invariant and is the projection of the $f$-equilibrium state to $\Sigma^+$, i.e., $\mu$ is the pushforward of $\nu_f$ via the map $\Sigma \to \Sigma^+$ defined by $(\dotsc, x_{-1}, x_0, x_1, \dotsc) \mapsto (x_0, x_1, \dotsc)$.
\end{enumerate}
\end{theorem}

Similarly, we would like to define transfer operators on $B(U, \mathbb C)$ specifically corresponding to the function $\xi\tau$ for any $\xi \in \mathbb C$, but we need to deal with some technicalities due to the boundaries of the rectangles. For all admissible pairs $(j, k)$, it is not hard to see from the Markov section construction that there are unique natural Lipschitz maps $\sigma^{-(j, k)}: U_k \to \overline{\mathtt{C}[j, k]}$ which extends $(\sigma|_{\mathtt{C}[j, k]})^{-1}: \interior(U_k) \to \mathtt{C}[j, k]$ and $\tau_{(j, k)}: \overline{\mathtt{C}[j, k]} \to \mathbb R_{>0}$ which extends $\tau|_{\mathtt{C}[j, k]}: \mathtt{C}[j, k] \to \mathbb R_{>0}$. Note that $\tau_{(x_0, x_1)}(\zeta^+(x)) = \tau_{\Sigma^+}(x)$ for all $x \in \Sigma^+$.

\begin{definition}
\label{def:TransferOperatorOriginal}
For all $\xi \in \mathbb C$, we define the \emph{transfer operator} $\mathcal{L}_{\xi \tau}: B(U, \mathbb C) \to B(U, \mathbb C)$ by
\begin{align*}
\mathcal{L}_{\xi \tau}(h)(u) = \sum_{\substack{(j, k)\\ u' = \sigma^{-(j, k)}(u)}} e^{\xi\tau_{(j, k)}(u')}h(u')
\end{align*}
for all $u \in U_k$, for all $k \in \mathcal{A}$, for all $h \in B(U, \mathbb C)$, where it is understood that the sum is only over admissible pairs.
\end{definition}

\begin{remark}
Due to the careful treatment of the boundaries of the rectangles, it is clear that for all $\xi \in \mathbb C$, the transfer operator $\mathcal{L}_{\xi \tau}$ preserves the subspaces $C(U, \mathbb C)$ and $C^{\Lip(d)}(U, \mathbb C)$. Let $\xi \in \mathbb C$ and $h \in B(U, \mathbb C)$. Then for all $u \in \interior(U)$, we can in fact write $\mathcal{L}_{\xi \tau}(h)(u) = \sum_{u' \in \sigma^{-1}(u)} e^{\xi\tau(u')}h(u')$. With these observations, if $h \in C(U, \mathbb C)$, then we can equivalently define $\mathcal{L}_{\xi \tau}(h)$ by the previous expression on $\interior(U)$ and then extend continuously to a function on $U$ (where the unique existence is now known by \cref{def:TransferOperatorOriginal}). This will help us avoid the more cumbersome \cref{def:TransferOperatorOriginal}. Finally, note that $\mathcal{L}_{\xi\tau_{\Sigma^+}} \circ (\zeta^+)^* = (\zeta^+)^* \circ \mathcal{L}_{\xi\tau}$.
\end{remark}

The following is the RPF theorem in this setting.

\begin{theorem}
\label{thm:RPFonU}
For all $a \in \mathbb R$, the operator $\mathcal{L}_{a\tau}: C(U, \mathbb C) \to C(U, \mathbb C)$ and its dual $\mathcal{L}_{a\tau}^*: C(U, \mathbb C)^* \to C(U, \mathbb C)^*$ has eigenvectors with the following properties. There exist a unique positive function $h \in C^{\Lip(d)}(U, \mathbb R)$ and a unique Borel probability measure $\nu$ on $U$ such that
\begin{enumerate}
\item	$\mathcal{L}_{a\tau}(h) = e^{\Pr_\sigma(a\tau_{\Sigma})}h$
\item	$\mathcal{L}_{a\tau}^*(\nu) = e^{\Pr_\sigma(a\tau_{\Sigma})}\nu$
\item	the eigenvalue $e^{\Pr_\sigma(a\tau_{\Sigma})}$ is maximal simple and the rest of the spectrum of $\mathcal{L}_f|_{C^{\Lip(d)}(U, \mathbb C)}$ is contained in a disk of radius strictly less than $e^{\Pr_\sigma(a\tau_{\Sigma})}$
%\item	for all $k \in \mathbb Z_{\geq 0}$, we have
%\begin{align*}
%\left|e^{-k\Pr(f)} \mathcal{L}_f^k(\phi)(u) - \hat{\nu}(\phi)\hat{h}(u)\right| \leq c(1 - \epsilon)^k\|\phi\|%_{\Lip(d_\theta)}
%\end{align*}
%where $\hat{h}$ is normalized such that $\hat{\nu}(\hat{h}) = 1$
\item	$\nu(h) = 1$ and the Borel probability measure $\mu$ defined by $d\mu = h \, d\nu$ is $\sigma$-invariant and is the projection of the $a\tau_{\Sigma}$-equilibrium state to $U$, i.e., $\mu = (\proj_U \circ \zeta)_*(\nu_{a\tau_{\Sigma}})$.
\end{enumerate}
\end{theorem}

\begin{remark}
Recall $\mathcal{L}_{\xi\tau_{\Sigma^+}} \circ (\zeta^+)^* = (\zeta^+)^* \circ \mathcal{L}_{\xi\tau}$ for all $\xi \in \mathbb C$. Using this, it is clear from the proofs that the $h$ in \cref{thm:RPFonU} pulls back to the corresponding one in \cref{thm:RPF} and the $\nu$ in \cref{thm:RPFonU} is the pushforward of the corresponding one in \cref{thm:RPF}, both via the map $\zeta^+$. By the same property, the eigenvalues in \cref{thm:RPFonU} are the same as the corresponding ones in \cref{thm:RPF}.
\end{remark}

%\begin{remark}
%Using the identifications via $\zeta^+$, we can regard the transfer operators as defined on $B(U)$. Note that the %condition $f \in C_\theta(\Sigma^+)$ translates to requiring $f$ be essentially Lipschitz.
%\end{remark}

Now we normalize the transfer operators for convenience. Let $a \in \mathbb R$. Define $\lambda_a = e^{\Pr_\sigma(-(\delta_\Gamma + a)\tau_{\Sigma})}$ which is the largest eigenvalue of $\mathcal{L}_{-(\delta_\Gamma + a)\tau}$ and recall that $\lambda_0 = 1$. Define the eigenvectors, the unique positive function $h_a \in C^{\Lip(d)}(U, \mathbb R)$ and the unique probability measure $\nu_a$ on $U$ with $\nu_a(h_a) = 1$ such that $\mathcal{L}_{-(\delta_\Gamma + a)\tau}(h_a) = \lambda_a h_a$ and $\mathcal{L}_{-(\delta_\Gamma + a)\tau}^*(\nu_a) = \lambda_a \nu_a$, provided by \cref{thm:RPFonU}. Note that $d\nu_U = h_0 \, d\nu_0$. By abuse of notation, we will denote the pullback $(\zeta^+)^*(h_a)$, which is the corresponding eigenvector for $\mathcal{L}_{-(\delta_\Gamma + a)\tau_{\Sigma^+}}$, by $h_a$ as well. Moreover, by perturbation theory for operators as in \cite[Chapter 8]{Kat95} and \cite[Proposition 4.6]{PP90}, we can fix a $a_0' > 0$ such that the map $[-a_0', a_0'] \to \mathbb R$ defined by $a \mapsto \lambda_a$ and the map $[-a_0', a_0'] \to C(U, \mathbb R)$ defined by $a \mapsto h_a$ are Lipschitz. In particular, there is a $C > 0$ such that $|\lambda_a - 1| \leq C|a|$ and $|h_a(u) - h_0(u)| \leq C|a|$ for all $u \in U$, for all $|a| \leq a_0'$. For all $a \in \mathbb R$, we define
\begin{align}
\label{eqn:fa}
f^{(a)} = -(\delta_\Gamma + a)\tau + \log \circ h_0 - \log \circ (h_0 \circ \sigma) - \log(\lambda_a).
\end{align}
Then we can fix a $A_f > 0$ such that $|f^{(a)}(u) - f^{(0)}(u)| \leq A_f|a|$ for all $u \in U$, for all $|a| \leq a_0'$. For all $k \in \mathbb Z_{>0}$, we use the notation
\begin{align*}
f_k^{(a)}(u) &= \sum_{j = 0}^{k - 1} f^{(a)}(\sigma^j(u)) & \tau_k(u) &= \sum_{j = 0}^{k - 1} \tau(\sigma^j(u))
\end{align*}
for all $u \in U$, and when $k = 0$ we mean the empty sum which is $0$. By a slight abuse of notation, for all $\xi = a + ib \in \mathbb C$, we define $\mathcal{L}_\xi: C(U, \mathbb C) \to C(U, \mathbb C)$ by
\begin{align*}
\mathcal{L}_\xi(h)(u) &= \sum_{u' \in \sigma^{-1}(u)} e^{(f^{(a)} + ib\tau)(u')}h(u') \\
&= \frac{1}{\lambda_a h_0(u)} \sum_{u' \in \sigma^{-1}(u)} e^{-(a + \delta_\Gamma - ib)\tau(u')}(h_0h)(u')
\end{align*}
and for all $k \in \mathbb Z_{> 0}$, its $k$\textsuperscript{th} iteration is
\begin{align*}
\mathcal{L}_\xi^k(h)(u) = \sum_{u' \in \sigma^{-k}(u)} e^{(f_k^{(a)} + ib\tau_k)(u')}h(u')
\end{align*}
for all $u \in \interior(U)$, and then extend continuously to a function on $U$, for all $h \in C(U, \mathbb C)$. Then the transfer operators are normalized such that the largest eigenvalue of $\mathcal{L}_a$ is $1$ with normalized eigenvector $\frac{h_a}{h_0}$ for all $a \in \mathbb R$. For $\mathcal{L}_0$ in particular, the largest eigenvalue is $1$ with normalized eigenvector $\chi_U$ and $\mathcal{L}_0^*(\nu_U) = \nu_U$. We leave it to the reader to make the appropriate adjustments to define the operators on $B(U, \mathbb C)$ as in \cref{def:TransferOperatorOriginal}.

\begin{remark}
We can also define $f_{\Sigma^+}^{(a)}$ by replacing $\sigma|_U$ by $\sigma|_{\Sigma^+}$, $\tau$ by $\tau_{\Sigma^+}$ and using $h_0 \in C^{\Lip(d_\theta)}(\Sigma^+, \mathbb R)$ in \cref{eqn:fa}, which would serve to define, by abuse of notation, a corresponding normalized operator $\mathcal{L}_\xi: C(\Sigma^+, \mathbb C) \to C(\Sigma^+, \mathbb C)$ for all $\xi \in \mathbb C$. Then $\mathcal{L}_\xi \circ (\zeta^+)^* = (\zeta^+)^* \circ \mathcal{L}_\xi$ for all $\xi \in \mathbb C$ similar to the final remark after \cref{def:TransferOperatorOriginal}.
\end{remark}

%Let $D = D(\Gamma, o) \cap \CH(\Lambda(\Gamma))$ be the intersection of the Dirichlet domain $D(\Gamma, o) \subset \mathbb H^n$ and the convex hull $\CH(\Lambda(\Gamma))$.

Now fix a
\begin{align*}
T_0 >{}&\max\left(\|\tau\|_\infty, \Lip_d^{\mathrm{e}}(\tau), \Lip_{d_\theta}(\tau_{\Sigma^+}), \rule{0cm}{0.6cm}\right. \\
{}&\left.\sup_{|a| \leq a_0'} \|f^{(a)}\|_\infty, \sup_{|a| \leq a_0'} \Lip_d^{\mathrm{e}}(f^{(a)}), \sup_{|a| \leq a_0'} \Lip_{d_\theta}(f_{\Sigma^+}^{(a)})\right)
\end{align*}
where we define the essentially Lipschitz seminorm for any function $\varphi: U \to \mathbb R$ to be
\begin{align*}
&\Lip_d^{\mathrm{e}}(\varphi) \\
={}&\sup\left\{\frac{|\varphi(u) - \varphi(u')|}{d(u, u')} \in \mathbb R_{>0}: u, u' \in \mathtt{C}, \mathtt{C} \subset U \text{ is a cylinder with } \len(\mathtt{C}) = 1\right\}
\end{align*}
and call $\varphi$ essentially Lipschitz if $\Lip_d^{\mathrm{e}}(\varphi) < \infty$. This is possible by \cite[Lemma 4.1]{PS16}.

\subsection{Cocycles and congruence transfer operators}
\label{subsec:CocyclesAndCongruenceTransferOperators}
We recall some definitions from \cite{OW16} regarding the congruence setting. Noting that $\T^1(\mathbb H^n)$ is a locally isometric cover of $\T^1(X)$, for all $j \in \mathcal{A}$, choose homeomorphic lifts $\overline{R}_j = [\overline{U}_j, \overline{S}_j] \subset \T^1(\mathbb H^n) \cong G/M$ of $R_j$. Define $\overline{R} = \bigsqcup_{j = 1}^N \overline{R}_j$ and $\overline{U} = \bigsqcup_{j = 1}^N \overline{U}_j$. For all $u \in R$, let $\overline{u} \in \overline{R}$ denote the unique lift in $\overline{R}$.
%such that $\overline{R}_j \cap D \neq \varnothing$

\begin{definition}
The \emph{cocycle} $\mathtt{c}: R \to \Gamma$ is a map such that for all $u \in R$, we have $\overline{u}a_{\tau(u)} \in \mathtt{c}(u)\overline{R}$.
\end{definition}

\begin{lemma}
\label{lem:CocyclesLocallyConstant}
The cocycle $\mathtt{c}$ is locally constant, i.e., if $u_1, u_2 \in R_x \cap \mathcal{P}^{-1}(R_y)$ for some $x, y \in \mathcal{A}$, then $\mathtt{c}(u_1) = \mathtt{c}(u_2)$.
\end{lemma}

\begin{proof}
Let $u_1, u_2 \in R_x \cap \mathcal{P}^{-1}(R_y)$ for some $x, y \in \mathcal{A}$ but suppose $\mathtt{c}(u_1) \neq \mathtt{c}(u_2)$. Then $\overline{u}_j \in \overline{R}_x$ and $u_j = \Gamma \overline{u}_j$ for all $j \in \{1, 2\}$. By definition, $\overline{v}_j = \mathtt{c}(u_j)^{-1}\overline{u}_j a_{\tau(u_j)} \in \overline{R}_x$ for all $j \in \{1, 2\}$ and by local isometry, we have
\begin{align*}
0 \neq d(\mathtt{c}(u_1)\overline{v}_1, \mathtt{c}(u_2)\overline{v}_1) &\leq d(\overline{u}_1, \mathtt{c}(u_2)\overline{v}_1) + \hat{\delta} \leq d(\overline{u}_2, \mathtt{c}(u_2)\overline{v}_1) + 2\hat{\delta} \\
&\leq d(\mathtt{c}(u_2)\overline{v}_2, \mathtt{c}(u_2)\overline{v}_1) + 3\hat{\delta} \leq 4\hat{\delta} \leq \Inj(\T^1(X))
\end{align*}
and so there is a geodesic from $\mathtt{c}(u_1)\overline{v}_1$ to $\mathtt{c}(u_2)\overline{v}_1$ of positive length less than the injectivity radius. But such a geodesic would project isometrically to a closed geodesic on $\T^1(X)$ since $\mathtt{c}(u_j)\overline{v}_1$ projects to $\Gamma \mathtt{c}(u_j)\overline{v}_1 = \Gamma \overline{v}_1$ for all $j \in \{1, 2\}$ which is a contradiction.
\end{proof}

Justified by the lemma above, for all $k \in \mathbb Z_{>0}$, we use the notation
\begin{align*}
\mathtt{c}^k(u) = \prod_{j = 0}^{k - 1}\mathtt{c}(\sigma^j(u)) = \mathtt{c}(u)\mathtt{c}(\sigma(u)) \dotsb \mathtt{c}(\sigma^{k - 1}(u))
\end{align*}
for all $u \in U$, and when $k = 0$ we mean the empty product which is $e \in \Gamma$. Note that the order in the product is important in the definition above.

\begin{corollary}
\label{cor:CocyclesLocallyConstantCorollary}
If $u_1, u_2 \in \mathtt{C}$ for some cylinder $\mathtt{C} \subset U$ with $\len(\mathtt{C}) = k \in \mathbb Z_{>0}$, then $\mathtt{c}^k(u_1) = \mathtt{c}^k(u_2)$.
\end{corollary}

%It is clear from the definitions above that since the covering map $\tilde{G} \to G$ induces both the isomorphisms $\tilde{G}/\tilde{M} \simrightarrow G/M$ and $\tilde{\Gamma} \backslash \tilde{G}/\tilde{M} \simrightarrow \Gamma \backslash G/M$, so it also induces the map $\tilde{\Gamma} \to \Gamma$ given by $\tilde{\mathtt{c}}(u) \mapsto \mathtt{c}(u)$ for all $u \in R$. \\

For all ideals $\mathfrak{q} \subset \mathcal{O}_{\mathbb K}$, we have the canonical quotient map $\pi_{\mathfrak{q}}: \tilde{\mathbf{G}}(\mathcal{O}_{\mathbb K}) \to \tilde{\mathbf{G}}(\mathcal{O}_{\mathbb K}/\mathfrak{q})$ and we define the principal congruence subgroup of level $\mathfrak{q}$ to be $\ker(\pi_{\mathfrak{q}})$. We would like to define the congruence subgroup of $\tilde{\Gamma}$ of level $\mathfrak{q}$ to be the normal subgroup $\tilde{\Gamma}_{\mathfrak{q}} = \ker(\pi_{\mathfrak{q}}|_{\tilde{\Gamma}}) \lhd \tilde{\Gamma}$. However, we make a minor modification and assume as before that $\tilde{\Gamma}_{\mathfrak{q}} \lhd \tilde{\Gamma}$ contains $\ker(\tilde{\pi}) = \{e, -e\}$ as the only torsion elements, i.e., we define $\tilde{\Gamma}_{\mathfrak{q}} = \langle\ker(\pi_{\mathfrak{q}}|_{\tilde{\Gamma}}), \ker(\tilde{\pi})\rangle \lhd \tilde{\Gamma}$. Again we define $\Gamma_{\mathfrak{q}} = \tilde{\pi}(\tilde{\Gamma}_{\mathfrak{q}})$. For all \emph{nontrivial} $\mathfrak{q} \subset \mathcal{O}_{\mathbb K}$, also define the finite group $F_{\mathfrak{q}} = \Gamma_{\mathfrak{q}} \backslash \Gamma \cong \tilde{\Gamma}_{\mathfrak{q}} \backslash \tilde{\Gamma}$. By the strong approximation theorem of Weisfeiler \cite{Wei84}, there is a nontrivial proper ideal $\mathfrak{q}_0 \subset \mathcal{O}_{\mathbb K}$ such that for all ideals $\mathfrak{q} \subset \mathcal{O}_{\mathbb K}$ coprime to $\mathfrak{q}_0$, the map $\pi_{\mathfrak{q}}|_{\tilde{\Gamma}}$ is in fact surjective and hence induces the isomorphism $\overline{\pi_{\mathfrak{q}}|_{\tilde{\Gamma}}}: F_{\mathfrak{q}} \to \tilde{G}_{\mathfrak{q}}$ where we define the finite group $\tilde{G}_{\mathfrak{q}} = \{e, -e\} \backslash \tilde{\mathbf{G}}(\mathcal{O}_{\mathbb K}/\mathfrak{q})$, by a slight abuse of notation. Without loss of generality, we assume that for all $(y, z) \in \mathcal{A}^2$, the same ideal $\mathfrak{q}_0 \subset \mathcal{O}_{\mathbb K}$ is sufficient to apply both the strong approximation theorem and \cite[Corollary 6]{GV12} for the subgroups $\tilde{H}^p(y, z) < \tilde{\Gamma}$ which is to be introduced later. This condition will appear throughout the paper in order for the strong approximation theorem to apply. For all nontrivial ideals $\mathfrak{q} \subset \mathcal{O}_{\mathbb K}$, define the \emph{congruence cocycle} $\mathtt{c}_{\mathfrak{q}}: R \to F_{\mathfrak{q}}$ by $\mathtt{c}_{\mathfrak{q}} = \pi_{\Gamma_\mathfrak{q}} \circ \mathtt{c}$, i.e., $\mathtt{c}_{\mathfrak{q}}(x) = \Gamma_{\mathfrak{q}}\mathtt{c}(x)$ for all $x \in R$.

Let $\mathfrak{q} \subset \mathcal{O}_{\mathbb K}$ be a nontrivial ideal. Let $X_\mathfrak{q} = \Gamma_\mathfrak{q} \backslash \mathbb H^n$ be the \emph{congruence cover} of $X$ of level $\mathfrak{q}$. Note that we have the isometries $\T^1(X_\mathfrak{q}) \cong \Gamma_\mathfrak{q} \backslash G/M$ and $\F_{\SO}(X_\mathfrak{q}) \cong \Gamma_\mathfrak{q} \backslash G$. Recall that $m^{\mathrm{BMS}}|_G$ is left $\Gamma$-invariant, so in particular it is left $\Gamma_\mathfrak{q}$-invariant. Thus it descends to the measure $m^{\mathrm{BMS}}_\mathfrak{q}$ on $\Gamma_\mathfrak{q} \backslash G$ which, by right $M$-invariance, descends once more to the measure $m^{\mathrm{BMS}}_\mathfrak{q}$ on $\Gamma_\mathfrak{q} \backslash G/M$ by abuse of notation. We call both of these the \emph{congruence BMS measures}. Note that $m^{\mathrm{BMS}}_\mathfrak{q}(\Gamma_\mathfrak{q} \backslash G/M) = \#F_\mathfrak{q} \cdot m^{\mathrm{BMS}}(\Gamma \backslash G/M)$. Let $p_\mathfrak{q}: \T^1(X_\mathfrak{q}) \to \T^1(X)$ be the locally isometric covering map and $\Omega_\mathfrak{q} = {p_\mathfrak{q}}^{-1}(\Omega) = \supp\left(m^{\mathrm{BMS}}_\mathfrak{q}\right) \subset \Gamma_\mathfrak{q} \backslash G/M$. Now we need a Markov section on $\Omega_\mathfrak{q}$ compatible with the Markov section on $\Omega$. For all $x \in \mathcal{A}$, for all $g \in F_\mathfrak{q}$, we define $R_{x, g}^\mathfrak{q} = g\overline{R}_x \subset \T^1(X_\mathfrak{q})$ and similarly define $U_{x, g}^\mathfrak{q}$ and $S_{x, g}^\mathfrak{q}$. We also define $R^\mathfrak{q} = \bigsqcup_{x \in \mathcal{A}, g \in F_\mathfrak{q}} R_{x, g}^\mathfrak{q}$ and similarly define $U^\mathfrak{q}$ and their corresponding cores. To facilitate notation, we make the identification $R \times F_\mathfrak{q} \cong R^\mathfrak{q}$ via the isometry $(u, g) \mapsto g\overline{u}$ and also similar identifications for $U^\mathfrak{q}$ and their corresponding cores. Using this identification, define the measure $\nu_{R^\mathfrak{q}}$ on $R^\mathfrak{q}$ to be the pushforward of $\nu_R \times m_{F_\mathfrak{q}}$ where $m_{F_\mathfrak{q}}$ is simply the counting measure on $F_\mathfrak{q}$. Note that $\nu_{R^\mathfrak{q}}(R^\mathfrak{q}) = \#F_\mathfrak{q}$. It can be checked that
\begin{align*}
\mathcal{R}^\mathfrak{q} = \{R_{x, g}^\mathfrak{q}: x \in \mathcal{A}, g \in F_\mathfrak{q}\}
\end{align*}
is a Markov section of size $\hat{\delta}$. We have the natural Poincar\'{e} first return map $\mathcal{P}_\mathfrak{q}: R^\mathfrak{q} \to R^\mathfrak{q}$ defined by $\mathcal{P}_\mathfrak{q}(u, g) = (\mathcal{P}(u), g\mathtt{c}_\mathfrak{q}(u))$ for all $(u, g) \in R^\mathfrak{q}$. It is easy to check that the corresponding first return time map is $\tau_\mathfrak{q} = p_\mathfrak{q}^*(\tau) = \tau \circ p_\mathfrak{q}$. Note that $\tau_\mathfrak{q}$ is independent of the $F_\mathfrak{q}$ component. Like $R^\tau$ and the measure $\nu^\tau$, we define $R^{\mathfrak{q}, \tau}$ and the measure $\nu^{\mathfrak{q}, \tau}$ in a similar fashion. Again by abuse of notation, we have the measure $\nu^{\mathfrak{q}, \tau}$ on $\T^1(X_\mathfrak{q})$ supported on $\Omega_\mathfrak{q}$. Note that $\nu^{\mathfrak{q}, \tau}(\T^1(X_\mathfrak{q})) = \nu_{R^\mathfrak{q}}(\tau_\mathfrak{q}) = \#F_\mathfrak{q} \cdot \nu_R(\tau)$ and so together with \cref{thm:FlowEquilibriumStateEqualsBMS}, we have $\frac{\nu^{\mathfrak{q}, \tau}}{\#F_\mathfrak{q} \cdot \nu_R(\tau)} = \frac{m^{\mathrm{BMS}}_\mathfrak{q}}{m^{\mathrm{BMS}}_\mathfrak{q}(\Gamma_\mathfrak{q} \backslash G/M)}$. Finally, $\mathcal{P}_\mathfrak{q}$ induces the shift map $\sigma_\mathfrak{q}: U^\mathfrak{q} \to U^\mathfrak{q}$ defined by $\sigma_\mathfrak{q}(u, g) = (\sigma(u), g\mathtt{c}_\mathfrak{q}(u))$ for all $(u, g) \in U^\mathfrak{q}$.

\begin{definition}
For all $\xi \in \mathbb C$, for all nontrivial ideals $\mathfrak{q} \subset \mathcal{O}_{\mathbb K}$, we define the \emph{congruence transfer operator} $\mathcal{M}_{\xi\tau_\mathfrak{q}, \mathfrak{q}}: C(U^\mathfrak{q}, \mathbb C) \to C(U^\mathfrak{q}, \mathbb C)$ by
\begin{align*}
\mathcal{M}_{\xi\tau_\mathfrak{q}, \mathfrak{q}}(h)(u, g) &= \sum_{(u', g') \in {\sigma_\mathfrak{q}}^{-1}(u, g)} e^{\xi\tau_\mathfrak{q}(u', g')}h(u', g') \\
&= \sum_{u' \in \sigma^{-1}(u)} e^{\xi\tau_\mathfrak{q}(u', g\mathtt{c}_\mathfrak{q}(u')^{-1})}h(u', g\mathtt{c}_\mathfrak{q}(u')^{-1})
\end{align*}
for all $(u, g) \in \interior(U^\mathfrak{q})$, and then extend continuously to a function on $U^\mathfrak{q}$, for all $h \in C(U^\mathfrak{q}, \mathbb C)$.
\end{definition}

\begin{remark}
As before, for all $\xi \in \mathbb C$, for all nontrivial ideals $\mathfrak{q} \subset \mathcal{O}_{\mathbb K}$, the congruence transfer operator $\mathcal{M}_{\xi\tau_\mathfrak{q}, \mathfrak{q}}$ preserves $C^{\Lip(d)}(U^\mathfrak{q}, \mathbb C)$.
\end{remark}

Let $\mathfrak{q} \subset \mathcal{O}_{\mathbb K}$ be a nontrivial ideal. We make the identification $C(U^\mathfrak{q}, \mathbb C) \cong C(U, L^2(F_\mathfrak{q}, \mathbb C))$. Note that for all $u \in R$, we have the induced left $\mathbb C[F_\mathfrak{q}]$-module automorphism on $L^2(F_\mathfrak{q}, \mathbb C)$ given by $(\mathtt{c}_\mathfrak{q}(u)^{-1}\phi)(g) = \phi(g\mathtt{c}_\mathfrak{q}(u)^{-1})$ for all $g \in F_\mathfrak{q}$, for all $\phi \in L^2(F_\mathfrak{q}, \mathbb C)$ (where $L^2(F_\mathfrak{q}, \mathbb C)$ is viewed as a module using the left regular representation). Again for normalization, for all $\xi = a + ib \in \mathbb C$, we define $\mathcal{M}_{\xi, \mathfrak{q}}: C(U, L^2(F_\mathfrak{q}, \mathbb C)) \to C(U, L^2(F_\mathfrak{q}, \mathbb C))$ by
\begin{align}
\label{eqn:NormalizedCongruenceTransferOperator}
\mathcal{M}_{\xi, \mathfrak{q}}(H)(u) &= \sum_{u' \in \sigma^{-1}(u)} e^{(f^{(a)} + ib\tau)(u')} \mathtt{c}_\mathfrak{q}(u')^{-1} H(u') \\
&= \frac{1}{\lambda_a h_0(u)} \sum_{u' \in \sigma^{-1}(u)} e^{-(a + \delta_\Gamma - ib)\tau(u')} \mathtt{c}_\mathfrak{q}(u')^{-1} (h_0H)(u')
\end{align}
and for all $k \in \mathbb Z_{> 0}$, its $k$\textsuperscript{th} iteration is
\begin{align*}
\mathcal{M}_{\xi, \mathfrak{q}}^k(H)(u) = \sum_{u' \in \sigma^{-k}(u)} e^{(f_k^{(a)} + ib\tau_k)(u')} \mathtt{c}_\mathfrak{q}^k(u')^{-1} H(u')
\end{align*}
for all $u \in \interior(U)$, and then extend continuously to a function on $U$, for all $H \in C(U, L^2(F_\mathfrak{q}, \mathbb C))$.

\begin{remark}
As before, since $(\mathtt{c}_\mathfrak{q} \circ \zeta^+)|_{\hat{\Sigma}^+}$ is Lipschitz, there is a unique Lipschitz extension $\mathtt{c}_{\mathfrak{q}, \Sigma^+}: \Sigma^+ \to F_\mathfrak{q}$ (using the discrete metric on $F_\mathfrak{q}$) which we again note is distinct from $\mathtt{c}_\mathfrak{q} \circ \zeta^+$. Then replacing $\sigma|_U$ by $\sigma|_{\Sigma^+}$, $\tau$ by $\tau_{\Sigma^+}$, $f^{(a)}$ by $f_{\Sigma^+}^{(a)}$, $\mathtt{c}_\mathfrak{q}$ by $\mathtt{c}_{\mathfrak{q}, \Sigma^+}$ and using $h_0 \in C^{\Lip(d_\theta)}(\Sigma^+, \mathbb R)$ in \cref{eqn:NormalizedCongruenceTransferOperator}, we can similarly define, by abuse of notation, a corresponding normalized operator $\mathcal{M}_{\xi, \mathfrak{q}}: C(\Sigma^+, L^2(F_\mathfrak{q}, \mathbb C)) \to C(\Sigma^+, L^2(F_\mathfrak{q}, \mathbb C))$ for all $\xi \in \mathbb C$. Then $\mathcal{M}_{\xi, \mathfrak{q}} \circ (\zeta^+)^* = (\zeta^+)^* \circ \mathcal{M}_{\xi, \mathfrak{q}}$ for all $\xi \in \mathbb C$ similar to the final remark after \cref{def:TransferOperatorOriginal}. Again we leave it to the reader to make the appropriate adjustments to define the operators on $B(U, L^2(F_\mathfrak{q}, \mathbb C))$ as in \cref{def:TransferOperatorOriginal}.
\end{remark}

\begin{remark}
As above, for all $u \in R$, the cocycle $\mathtt{c}(u)^{-1}$ acts on $L^2(F_\mathfrak{q}, \mathbb C)$ by a left $\mathbb C[F_\mathfrak{q}]$-module automorphism and following definitions we see that the action is the same as that of $\mathtt{c}_\mathfrak{q}(u)^{-1}$. This justifies dropping the subscript of the cocycle in the definition of the transfer operator above whenever required. We prefer to keep the subscript in the \cref{sec:GeodesicFlowSmall|b|} so that the cocycle takes values in a finite group and we prefer to drop the subscript in \cref{sec:GeodesicFlowLarge|b|} where it is unnecessary.
\end{remark}

\begin{comment}
The cocycle $\tilde{\mathtt{c}}_q$ has a natural right action on $\tilde{G}_q$ via the isomorphism $\tilde{G}_q \cong \tilde{\Gamma}_q\backslash \tilde{\Gamma}$. This induces natural actions $\tilde{\mathtt{c}}_q: R \to \GL(\mathbb C[\tilde{G}_q])$ on the group algebra $\mathbb C[\tilde{G}_q]$ and also $\tilde{\mathtt{c}}_q: R \to \GL(L^2(\tilde{G}_q))$ on the $\mathbb C[\tilde{G}_q]$-module $L^2(\tilde{G}_q)$, which we also denote by $\tilde{\mathtt{c}}_q$ by slight abuse of notation, defined by its action on the standard basis elements, $\tilde{\mathtt{c}}_q(u)(g) = g\tilde{\mathtt{c}}_q(u)$ for all $u \in R$, for all $g \in \tilde{G}_q$.
\end{comment}

\subsection{Main technical theorem}
We introduce some inner products, norms and seminorms. Let $\mathfrak{q} \subset \mathcal{O}_{\mathbb K}$ be a nontrivial ideal and $H, H_1, H_2 \in C(U, L^2(F_\mathfrak{q}, \mathbb C))$. Let $\langle \cdot, \cdot \rangle$ and $\|\cdot\|_2$ denote the usual $L^2$ inner product and norm on any appropriate space respectively. Similarly, let $\|\cdot\|_\infty$ denote the usual $L^\infty$ norm (i.e., essential supremum norm) on any appropriate space. In particular, we have
\begin{align*}
\langle H_1, H_2 \rangle &= \int_U \langle H_1(u), H_2(u)\rangle \, d\nu_U(u) = \int_U \sum_{g \in F_\mathfrak{q}} H_1(u)(g) \cdot \overline{H_2(u)(g)} \, d\nu_U(u) \\
\|H\|_2 &= \left(\int_U {\|H(u)\|_2}^2 \, d\nu_U(u) \right)^{\frac{1}{2}}
\end{align*}
and $\|H\|_\infty = \sup_{u \in U} \|H(u)\|_2$ since $H \in C(U, L^2(F_\mathfrak{q}, \mathbb C)) \subset L^2(U, L^2(F_\mathfrak{q}, \mathbb C))$. For convenience we will also denote $|H| \in C(U, L^2(F_\mathfrak{q}, \mathbb R))$ and $\|H\| \in C(U, \mathbb R)$ to be the functions defined by $|H|(u) = |H(u)| \in L^2(F_\mathfrak{q}, \mathbb R)$ and $\|H\|(u) = \|H(u)\|_2$ for all $u \in U$, respectively. We use the notations $\Lip_{d_\theta}(H) = \Lip_{d_\theta}(H \circ \zeta^+)$ for the Lipschitz seminorm and $\|H\|_{\Lip(d_\theta)} = \|H \circ \zeta^+\|_{\Lip(d_\theta)} = \|H\|_\infty + \Lip_{d_\theta}(H)$ for the Lipschitz norm. Similarly, we use the notations
\begin{align*}
\Lip_d(H) = \sup\left\{\frac{\|H(u) - H(u')\|_2}{d(u, u')} \in \mathbb R_{>0}: u, u' \in U \text{ such that } u \neq u'\right\}
\end{align*}
for the Lipschitz seminorm and $\|H\|_{\Lip(d)} = \|H\|_\infty + \Lip_d(H)$ for the Lipschitz norm. We generalize this to another useful norm denoted by
\begin{align*}
\|H\|_{1, b} = \|H\|_\infty + \frac{1}{\max(1, |b|)} \Lip_d(H).
\end{align*}
We denote any operator norms simply by $\|\cdot\|_{\mathrm{op}}$.

\begin{remark}
We have $\|H\|_2 \leq \|H\|_\infty$ and $\|H\|_{\Lip(d_\theta)} \leq C_\theta \|H\|_{\Lip(d)}$ for some fixed $C_\theta > 0$.
\end{remark}

For all nontrivial ideals $\mathfrak{q} \subset \mathcal{O}_{\mathbb K}$, we define
\begin{align*}
\mathcal{V}_\mathfrak{q}(U) = C^{\Lip(d)}(U, L^2(F_\mathfrak{q}, \mathbb C)) = \{H \in C(U, L^2(F_\mathfrak{q}, \mathbb C)): \|H\|_{\Lip(d)} < \infty\}
\end{align*}
and when $\mathfrak{q}$ is coprime to $\mathfrak{q}_0$, we can similarly define
\begin{align*}
\mathcal{W}_\mathfrak{q}(U) = C^{\Lip(d)}(U, L_0^2(\tilde{G}_\mathfrak{q}, \mathbb C)) \subset \mathcal{V}_\mathfrak{q}(U)
\end{align*}
where $L_0^2(\tilde{G}_\mathfrak{q}, \mathbb C) = \left\{\phi \in L^2(\tilde{G}_\mathfrak{q}, \mathbb C): \sum_{g \in \tilde{G}_\mathfrak{q}} \phi(g) = 0\right\}$.

For all ideals $\mathfrak{q} \subset \mathcal{O}_{\mathbb K}$, we denote the ideal norm by $N_{\mathbb K}(\mathfrak{q}) = \#(\mathcal{O}_{\mathbb K}/\mathfrak{q})$ and we say $\mathfrak{q}$ is \emph{square free} if it is a nontrivial proper ideal without any square prime ideal factors.

Now we can state \cref{thm:TheoremGeodesicFlow} which is the main technical theorem regarding spectral bounds. The theorem is proved in \cref{sec:GeodesicFlowSmall|b|,sec:GeodesicFlowLarge|b|} as a simple consequence of \cref{thm:TheoremSmall|b|,thm:TheoremLarge|b|}.

\begin{theorem}
\label{thm:TheoremGeodesicFlow}
There exist $\eta > 0, C \geq 1, a_0 > 0$ and a nontrivial proper ideal $\mathfrak{q}_0' \subset \mathcal{O}_{\mathbb K}$ such that for all $\xi = a + ib \in \mathbb C$ with $|a| < a_0$, for all square free ideals $\mathfrak{q} \subset \mathcal{O}_{\mathbb K}$ coprime to $\mathfrak{q}_0\mathfrak{q}_0'$, for all integers $k \in \mathbb Z_{\geq 0}$, for all $H \in \mathcal{W}_\mathfrak{q}(U)$, we have
\begin{align*}
\left\|\mathcal{M}_{\xi, \mathfrak{q}}^k(H)\right\|_2 \leq CN_{\mathbb K}(\mathfrak{q})^C e^{-\eta k} \|H\|_{1, b}.
\end{align*}
\end{theorem}

\section{Spectral bounds for small $|b|$ using expander machinery}
\label{sec:GeodesicFlowSmall|b|}
In this section, we prove \cref{thm:TheoremSmall|b|}. We fix $b_0 > 0$ to be the one from \cref{thm:TheoremLarge|b|} where if we examine the proof of \cref{thm:Dolgopyat} it is clear that we can assume $b_0 = 1$.

\begin{theorem}
\label{thm:TheoremSmall|b|}
There exist $\eta > 0, C \geq 1, a_0 > 0$ and a nontrivial proper ideal $\mathfrak{q}_0' \subset \mathcal{O}_{\mathbb K}$ such that for all $\xi = a + ib \in \mathbb C$ with $|a| < a_0$ and $|b| \leq b_0$, for all square free ideals $\mathfrak{q} \subset \mathcal{O}_{\mathbb K}$ coprime to $\mathfrak{q}_0\mathfrak{q}_0'$, for all integers $k \in \mathbb Z_{\geq 0}$, for all $H \in \mathcal{W}_\mathfrak{q}(U)$, we have
\begin{align*}
\left\|\mathcal{M}_{\xi, \mathfrak{q}}^k(H)\right\|_2 \leq CN_{\mathbb K}(\mathfrak{q})^C e^{-\eta k} \|H\|_{\Lip(d)}.
\end{align*}
\end{theorem}
To prove this, we first make some reductions as in \cite{OW16}.

\subsection{Reductions}
Let $\mathfrak{q}, \mathfrak{q}' \subset \mathcal{O}_{\mathbb K}$ be ideals such that $\mathfrak{q} \subset \mathfrak{q}'$. We have the canonical quotient map $\pi_{\mathfrak{q}, \mathfrak{q}'}: \tilde{\mathbf{G}}(\mathcal{O}_{\mathbb K}/\mathfrak{q}) \to \tilde{\mathbf{G}}(\mathcal{O}_{\mathbb K}/\mathfrak{q}')$ which induces the pull back $\pi_{\mathfrak{q}, \mathfrak{q}'}^*: L^2(\tilde{\mathbf{G}}(\mathcal{O}_{\mathbb K}/\mathfrak{q}'), \mathbb C) \to L^2(\tilde{\mathbf{G}}(\mathcal{O}_{\mathbb K}/\mathfrak{q}), \mathbb C)$. Define $\hat{E}_{\mathfrak{q}'}^\mathfrak{q} = \pi_{\mathfrak{q}, \mathfrak{q}'}^*(L^2(\tilde{\mathbf{G}}(\mathcal{O}_{\mathbb K}/\mathfrak{q}'), \mathbb C)) \subset L^2(\tilde{\mathbf{G}}(\mathcal{O}_{\mathbb K}/\mathfrak{q}), \mathbb C)$. Define $\dot{E}_{\mathfrak{q}'}^\mathfrak{q} = \hat{E}_{\mathfrak{q}'}^\mathfrak{q} \cap \left(\bigoplus_{\mathfrak{q}' \subsetneq \mathfrak{q}''} \hat{E}_{\mathfrak{q}''}^\mathfrak{q}\right)^\perp$. Then, for all ideals $\mathfrak{q} \subset \mathcal{O}_{\mathbb K}$, we have the orthogonal decomposition
\begin{align*}
L_0^2(\tilde{\mathbf{G}}(\mathcal{O}_{\mathbb K}/\mathfrak{q}), \mathbb C) = \bigoplus_{\mathfrak{q} \subset \mathfrak{q}' \subsetneq \mathcal{O}_{\mathbb K}} \dot{E}_{\mathfrak{q}'}^\mathfrak{q}.
\end{align*}
\begin{remark}
We exclude $\mathfrak{q}' = \mathcal{O}_{\mathbb K}$ above because the subspace $\dot{E}_{\mathcal{O}_{\mathbb K}}^\mathfrak{q} \subset L^2(\tilde{\mathbf{G}}(\mathcal{O}_{\mathbb K}/\mathfrak{q}), \mathbb C)$ consists of constant functions.
\end{remark}
Similarly, using the same procedure with the induced quotient map $\pi_{\mathfrak{q}, \mathfrak{q}'}: \tilde{G}_\mathfrak{q} \to \tilde{G}_{\mathfrak{q}'}$ for ideals $\mathfrak{q}, \mathfrak{q}' \subset \mathcal{O}_{\mathbb K}$ such that $\mathfrak{q} \subset \mathfrak{q}'$, we can obtain a similar orthogonal decomposition
\begin{align*}
L_0^2(\tilde{G}_\mathfrak{q}, \mathbb C) = \bigoplus_{\mathfrak{q} \subset \mathfrak{q}' \subsetneq \mathcal{O}_{\mathbb K}} E_{\mathfrak{q}'}^\mathfrak{q}
\end{align*}
for all ideals $\mathfrak{q} \subset \mathcal{O}_{\mathbb K}$. Again let $\mathfrak{q}, \mathfrak{q}' \subset \mathcal{O}_{\mathbb K}$ be ideals such that $\mathfrak{q} \subset \mathfrak{q}'$. For all $\phi \in L^2(\tilde{G}_\mathfrak{q}, \mathbb C)$, define $\dot{\phi} \in L^2(\tilde{\mathbf{G}}(\mathcal{O}_{\mathbb K}/\mathfrak{q}), \mathbb C)$ by $\dot{\phi}(g) = \phi(\{e, -e\}g)$ for all $g \in \tilde{\mathbf{G}}(\mathcal{O}_{\mathbb K}/\mathfrak{q})$. Then we can describe $E_{\mathfrak{q}'}^\mathfrak{q}$ by
\begin{align*}
E_{\mathfrak{q}'}^\mathfrak{q} = \{\phi \in L^2(\tilde{G}_\mathfrak{q}, \mathbb C): \dot{\phi} \in \dot{E}_{\mathfrak{q}'}^\mathfrak{q}\}.
\end{align*}
In the relevant case when $\mathfrak{q}$ is coprime to $\mathfrak{q}_0$, the subspace $E_{\mathfrak{q}'}^\mathfrak{q} \subset L^2(\tilde{G}_\mathfrak{q}, \mathbb C)$ can also be thought of as consisting of ``new functions'' invariant under $\Gamma_{\mathfrak{q}'}$ but not invariant under $\Gamma_{\mathfrak{q}''}$ for any $\mathfrak{q}' \subsetneq \mathfrak{q}''$, using the isomorphism $\overline{\pi_\mathfrak{q}|_{\Gamma}}: F_\mathfrak{q} \to \tilde{G}_\mathfrak{q}$. Continuing the case when $\mathfrak{q}$ is coprime to $\mathfrak{q}_0$, we define $\mathcal{W}_{\mathfrak{q}'}^\mathfrak{q}(U) = \{H \in \mathcal{W}_\mathfrak{q}(U): H(u) \in E_{\mathfrak{q}'}^\mathfrak{q} \text{ for all } u \in U\}$, so that we have the orthogonal decomposition
\begin{align*}
\mathcal{W}_\mathfrak{q}(U) = \bigoplus_{\mathfrak{q} \subset \mathfrak{q}' \subsetneq \mathcal{O}_{\mathbb K}} \mathcal{W}_{\mathfrak{q}'}^\mathfrak{q}(U).
\end{align*}
Let $\mathfrak{q}, \mathfrak{q}' \subset \mathcal{O}_{\mathbb K}$ be ideals coprime to $\mathfrak{q}_0$ such that $\mathfrak{q} \subset \mathfrak{q}'$. We have the canonical projection operator $e_{\mathfrak{q}, \mathfrak{q}'}: \mathcal{W}_\mathfrak{q}(U) \to \mathcal{W}_{\mathfrak{q}'}^\mathfrak{q}(U)$ and we can define the canonical projection map $\proj_{\mathfrak{q}, \mathfrak{q}'} = (\pi_{\mathfrak{q}, \mathfrak{q}'}^*)^{-1}\big|_{E_{\mathfrak{q}'}^\mathfrak{q}}: E_{\mathfrak{q}'}^\mathfrak{q} \to E_{\mathfrak{q}'}^{\mathfrak{q}'}$ since $\pi_{\mathfrak{q}, \mathfrak{q}'}^*$ is injective and $E_{\mathfrak{q}'}^\mathfrak{q} \subset \pi_{\mathfrak{q}, \mathfrak{q}'}^*(L^2(\tilde{G}_{\mathfrak{q}'}, \mathbb C))$. This is equivalent to defining $(\proj_{\mathfrak{q}, \mathfrak{q}'}F)(g) = F(\tilde{g})$ for all $g \in \tilde{G}_{\mathfrak{q}'}$, for all $F \in E_{\mathfrak{q}'}^\mathfrak{q}$, where $\tilde{g}$ is any lift of $g$ with respect to quotient map $\pi_{\mathfrak{q}, \mathfrak{q}'}: \tilde{G}_\mathfrak{q} \to \tilde{G}_{\mathfrak{q}'}$. We use the same notation $\proj_{\mathfrak{q}, \mathfrak{q}'}: \mathcal{W}_{\mathfrak{q}'}^\mathfrak{q}(U) \to \mathcal{W}_{\mathfrak{q}'}^{\mathfrak{q}'}(U)$ for the induced projection map defined pointwise. The congruence transfer operator $\mathcal{M}_{\xi, \mathfrak{q}}$ preserves $\mathcal{W}_{\mathfrak{q}'}^\mathfrak{q}(U)$ for all $\xi \in \mathbb C$. The projection operator commutes with the congruence transfer operator, i.e., $e_{\mathfrak{q}, \mathfrak{q}'} \circ \mathcal{M}_{\xi, \mathfrak{q}} = \mathcal{M}_{\xi, \mathfrak{q}} \circ e_{\mathfrak{q}, \mathfrak{q}'}$ for all $\xi \in \mathbb C$, and the projection map is equivariant with respect to the congruence transfer operator, i.e., $\proj_{\mathfrak{q}, \mathfrak{q}'} \circ \mathcal{M}_{\xi, \mathfrak{q}} = \mathcal{M}_{\xi, \mathfrak{q}'} \circ \proj_{\mathfrak{q}, \mathfrak{q}'}$ for all $\xi \in \mathbb C$. By surjectivity of $\pi_{\mathfrak{q}, \mathfrak{q}'}$, we can denote $\spadesuit_{\mathfrak{q}, \mathfrak{q}'} = \#\ker(\pi_{\mathfrak{q}, \mathfrak{q}'}) = \frac{\#\tilde{G}_\mathfrak{q}}{\#\tilde{G}_{\mathfrak{q}'}}$ and by direct calculation it can be checked that $\|H\|_2 = \sqrt{\spadesuit_{\mathfrak{q}, \mathfrak{q}'}} \|\proj_{\mathfrak{q}, \mathfrak{q}'}(H)\|_2$ and $\|H\|_{\Lip(d)} = \sqrt{\spadesuit_{\mathfrak{q}, \mathfrak{q}'}} \|\proj_{\mathfrak{q}, \mathfrak{q}'}(H)\|_{\Lip(d)}$.

\begin{theorem}
\label{thm:Reduced'TheoremSmall|b|}
There exist $\eta > 0, C \geq 1, \eta > 0, a_0 > 0$ and $q_1 \in \mathbb Z_{>0}$ such that for all $\xi = a + ib \in \mathbb C$ with $|a| < a_0$ and $|b| \leq b_0$, for all square free ideals $\mathfrak{q} \subset \mathcal{O}_{\mathbb K}$ coprime to $\mathfrak{q}_0$ with $N_{\mathbb K}(\mathfrak{q}) > q_1$, for all integers $k \in \mathbb Z_{\geq 0}$, for all $H \in \mathcal{W}_\mathfrak{q}^\mathfrak{q}(U)$, we have
\begin{align*}
\left\|\mathcal{M}_{\xi, \mathfrak{q}}^k(H)\right\|_2 \leq CN_{\mathbb K}(\mathfrak{q})^C e^{-\eta k} \|H\|_{\Lip(d)}.
\end{align*}
\end{theorem}

\begin{proof}[Proof that \cref{thm:Reduced'TheoremSmall|b|} implies \cref{thm:TheoremSmall|b|}]
Fix $\eta > 0, C_1 \geq 1, a_0 > 0$ and $q_1 \in \mathbb Z_{>0}$ to be the $\eta, C, a_0$ and $q_1$ from \cref{thm:Reduced'TheoremSmall|b|} respectively. Fix $C = C_1 + \frac{1}{2}$. Set $\mathfrak{q}_0' \subset \mathcal{O}_{\mathbb K}$ to be the product of all nontrivial prime ideals $\mathfrak{p} \subset \mathcal{O}_{\mathbb K}$ with $N_{\mathbb K}(\mathfrak{p}) \leq q_1$ so that if $\mathfrak{q} \subset \mathcal{O}_{\mathbb K}$ is an ideal coprime to $\mathfrak{q}_0\mathfrak{q}_0'$ and $\mathfrak{q} \subset \mathfrak{q}' \subsetneq \mathcal{O}_{\mathbb K}$ is a proper ideal, then $N_{\mathbb K}(\mathfrak{q}') > q_1$. Let $\xi = a + ib \in \mathbb C$ with $|a| < a_0$ and $|b| \leq b_0$. Let $\mathfrak{q} \subset \mathcal{O}_{\mathbb K}$ be an ideal coprime to $\mathfrak{q}_0\mathfrak{q}_0'$, $k \in \mathbb Z_{\geq 0}$ and $H \in \mathcal{W}_\mathfrak{q}(U)$. We have
\begin{align*}
{\left\|\mathcal{M}_{\xi, \mathfrak{q}}^k(H)\right\|_2}^2 &= \sum_{\mathfrak{q} \subset \mathfrak{q}' \subsetneq \mathcal{O}_{\mathbb K}} {\left\|e_{\mathfrak{q}, \mathfrak{q}'}(\mathcal{M}_{\xi, \mathfrak{q}}^k (H))\right\|_2}^2 \\
&= \sum_{\mathfrak{q} \subset \mathfrak{q}' \subsetneq \mathcal{O}_{\mathbb K}} \spadesuit_{\mathfrak{q}, \mathfrak{q}'} {\left\|\proj_{\mathfrak{q}, \mathfrak{q}'}(e_{\mathfrak{q}, \mathfrak{q}'}(\mathcal{M}_{\xi, \mathfrak{q}}^k(H)))\right\|_2}^2 \\
&= \sum_{\mathfrak{q} \subset \mathfrak{q}' \subsetneq \mathcal{O}_{\mathbb K}} \spadesuit_{\mathfrak{q}, \mathfrak{q}'} {\left\|\mathcal{M}_{\xi, \mathfrak{q}'}^k(\proj_{\mathfrak{q}, \mathfrak{q}'}(e_{\mathfrak{q}, \mathfrak{q}'}(H)))\right\|_2}^2.
\end{align*}
Since $\proj_{\mathfrak{q}, \mathfrak{q}'}(e_{\mathfrak{q}, \mathfrak{q}'}(H)) \in \mathcal{W}_{\mathfrak{q}'}^{\mathfrak{q}'}(U)$, we can now apply \cref{thm:Reduced'TheoremSmall|b|} to get
\begin{align*}
{\left\|\mathcal{M}_{\xi, \mathfrak{q}}^k(H)\right\|_2}^2 &\leq \sum_{\mathfrak{q} \subset \mathfrak{q}' \subsetneq \mathcal{O}_{\mathbb K}} {C_1}^2N_{\mathbb K}(\mathfrak{q}')^{2C_1} e^{-2\eta k} \spadesuit_{\mathfrak{q}, \mathfrak{q}'} {\left\|\proj_{\mathfrak{q}, \mathfrak{q}'}(e_{\mathfrak{q}, \mathfrak{q}'}(H))\right\|_{\Lip(d)}}^2 \\
&= \sum_{\mathfrak{q} \subset \mathfrak{q}' \subsetneq \mathcal{O}_{\mathbb K}} {C_1}^2N_{\mathbb K}(\mathfrak{q}')^{2C_1} e^{-2\eta k} {\left\|e_{\mathfrak{q}, \mathfrak{q}'}(H)\right\|_{\Lip(d)}}^2 \\
&\leq {C_1}^2N_{\mathbb K}(\mathfrak{q})^{2C_1} e^{-2\eta k} {\left\|H\right\|_{\Lip(d)}}^2 \sum_{\mathfrak{q} \subset \mathfrak{q}' \subsetneq \mathcal{O}_{\mathbb K}} 1 \\
&\leq C^2N_{\mathbb K}(\mathfrak{q})^{2C}e^{-2\eta k} {\left\|H\right\|_{\Lip(d)}}^2.
\end{align*}
\end{proof}

\begin{theorem}
\label{thm:ReducedTheoremSmall|b|}
There exist $C_s > 0, a_0 > 0$ and $\kappa \in (0, 1)$ and $q_1 \in \mathbb Z_{>0}$ such that for all $\xi = a + ib \in \mathbb C$ with $|a| < a_0$ and $|b| \leq b_0$, for all square free ideals $\mathfrak{q} \subset \mathcal{O}_{\mathbb K}$ coprime to $\mathfrak{q}_0$ with $N_{\mathbb K}(\mathfrak{q}) > q_1$, there exists an integer $s \in (0, C_s\log(N_{\mathbb K}(\mathfrak{q})))$ such that for all integers $j \in \mathbb Z_{\geq 0}$, for all $H \in \mathcal{W}_\mathfrak{q}^\mathfrak{q}(U)$, we have
\begin{align*}
\big\|\mathcal{M}_{\xi, \mathfrak{q}}^{js}(H)\big\|_2 \leq N_{\mathbb K}(\mathfrak{q})^{-j\kappa} \|H\|_{\Lip(d_\theta)}.
\end{align*}
\end{theorem}

\begin{proof}[Proof that \cref{thm:ReducedTheoremSmall|b|} implies \cref{thm:Reduced'TheoremSmall|b|}]
Fix $C_s, a_0 > 0$ and $\kappa \in (0, 1)$ and $q_1 \in \mathbb Z_{>0}$ to be the ones from \cref{thm:ReducedTheoremSmall|b|}. Fix
\begin{align*}
C_\mathcal{M} &= \max\left(0, \sup_{|\Re(\xi)| \leq a_0, \{0\} \subsetneq \mathfrak{q} \subset \mathcal{O}_{\mathbb K}} \log\left\|\mathcal{M}_{\xi, \mathfrak{q}}\right\|_{\mathrm{op}}\right) \\
&\leq \max\left(0, \sup_{|\Re(\xi)| \leq a_0} \log\left\|\mathcal{L}_\xi\right\|_{\mathrm{op}}\right) \leq \max(0, \log(Ne^{T_0}))
\end{align*}
where we use operator norms for operators on $L^2(U, L^2(F_\mathfrak{q}, \mathbb C))$ and $L^2(U, \mathbb R)$ respectively. Fix $C = \max(C_\mathcal{M}C_s + 1, C_\theta) \geq 1$ and $\eta = \frac{\kappa}{C_s} > 0$. Let $\xi = a + ib \in \mathbb C$ with $|a| < a_0$ and $|b| \leq b_0$. Let $\mathfrak{q} \subset \mathcal{O}_{\mathbb K}$ be an ideal coprime to $\mathfrak{q}_0$ with $N_{\mathbb K}(\mathfrak{q}) > q_1$ and $s \in (0, C_s\log(N_{\mathbb K}(\mathfrak{q})))$ be the corresponding integer provided by \cref{thm:ReducedTheoremSmall|b|}. Then $\left\|\mathcal{M}_{\xi, \mathfrak{q}}\right\|^m \leq N_{\mathbb K}(\mathfrak{q})^{C_\mathcal{M}C_s}$ for all integers $0 \leq m < s$. Let $k \in \mathbb Z_{\geq 0}$ and $H \in \mathcal{W}_\mathfrak{q}^\mathfrak{q}(U)$. Writing $k = js + m$ for some integers $j \in \mathbb Z_{\geq 0}$ and $0 \leq m < s$, and using \cref{thm:ReducedTheoremSmall|b|}, we have
\begin{align*}
\left\|\mathcal{M}_{\xi, \mathfrak{q}}^k(H)\right\|_2 &\leq \left\|\mathcal{M}_{\xi, \mathfrak{q}}\right\|^m \cdot \big\|\mathcal{M}_{\xi, \mathfrak{q}}^{js}(H)\big\|_2 \\
&\leq N_{\mathbb K}(\mathfrak{q})^{C_\mathcal{M}C_s} N_{\mathbb K}(\mathfrak{q})^{-j\kappa} \|H\|_{\Lip(d_\theta)} \\
&\leq N_{\mathbb K}(\mathfrak{q})^{C_\mathcal{M}C_s + \frac{m\kappa}{s}} e^{-(js + m) \frac{\kappa}{s}\log(N_{\mathbb K}(\mathfrak{q}))} \|H\|_{\Lip(d_\theta)} \\
&\leq N_{\mathbb K}(\mathfrak{q})^{C_\mathcal{M}C_s + 1} e^{-\frac{\kappa}{C_s} \cdot k} \|H\|_{\Lip(d_\theta)} \\
&= CN_{\mathbb K}(\mathfrak{q})^C e^{-\eta k} \|H\|_{\Lip(d)}.
\end{align*}
\end{proof}

The rest of the section is devoted to obtaining strong bounds which are crucial for the proof of \cref{thm:ReducedTheoremSmall|b|} in \cref{subsec:SupremumAndLipschitzBounds}.

\subsection{Approximating the transfer operator}
The aim of this subsection is to approximate the transfer operator by a convolution with a measure as in \cite{OW16}.

Let $\mathfrak{q} \subset \mathcal{O}_{\mathbb K}$ be an ideal coprime to $\mathfrak{q}_0$. Recall the property $\mathcal{M}_{\xi, \mathfrak{q}} \circ (\zeta^+)^* = (\zeta^+)^* \circ \mathcal{M}_{\xi, \mathfrak{q}}$ for all $\xi \in \mathbb C$ from a remark in \cref{subsec:CocyclesAndCongruenceTransferOperators}, and also recall that $\zeta^+$ is surjective. In light of these observations, without losing any information we can abuse notation slightly to interpret $\mathcal{M}_{\xi, \mathfrak{q}}(H) \in C(\Sigma^+, L^2(\tilde{G}_\mathfrak{q}, \mathbb C))$ for any $H \in C(U, L^2(\tilde{G}_\mathfrak{q}, \mathbb C))$ as the operator $\mathcal{M}_{\xi, \mathfrak{q}}: C(\Sigma^+, L^2(\tilde{G}_\mathfrak{q}, \mathbb C)) \to C(\Sigma^+, L^2(\tilde{G}_\mathfrak{q}, \mathbb C))$ applied to the pullback $H \in C(\Sigma^+, L^2(\tilde{G}_\mathfrak{q}, \mathbb C))$ with the map $(\zeta^+)^*$ suppressed. We will do this throughout the rest of this section. From the same remark, recall the map $\mathtt{c}_{\mathfrak{q}, \Sigma^+}: \Sigma^+ \to \tilde{G}_\mathfrak{q}$. By \cref{cor:CocyclesLocallyConstantCorollary}, the map can be described simply by $\mathtt{c}_{\mathfrak{q}, \Sigma^+}(x) = \mathtt{c}_\mathfrak{q}(u)$ for all $x \in \Sigma^+$, for any choice of $u \in \mathtt{C}[x_0, x_1]$ and hence we see that there is dependence only on the first two entries. Thus it makes sense to introduce the notation $\mathtt{c}_{\mathfrak{q}, \Sigma^+}(j, k)$ for all admissible pairs $(j, k)$ in the natural way. In particular, with this notation we have
\begin{align*}
\mathtt{c}_{\mathfrak{q}, \Sigma^+}^k(x) = \mathtt{c}_{\mathfrak{q}, \Sigma^+}(x_0, x_1) \mathtt{c}_{\mathfrak{q}, \Sigma^+}(x_1, x_2) \dotsb \mathtt{c}_{\mathfrak{q}, \Sigma^+}(x_{k - 1}, x_k)
\end{align*}
for all $x \in \Sigma^+$. For the rest of this section we will henceforth drop the subscript $\Sigma^+$ in $\tau_{\Sigma^+}$, $f_{\Sigma^+}^{(a)}$ and $\mathtt{c}_{\mathfrak{q}, \Sigma^+}$ but we reiterate that it should not be confused with $\tau \circ \zeta^+$, $f^{(a)} \circ \zeta^+$ and $\mathtt{c}_\mathfrak{q} \circ \zeta^+$ respectively.

We introduce some notations for admissible sequences to save space. We define $\alpha^j = (\alpha_j, \alpha_{j - 1}, \dotsc, \alpha_1)$ for all admissible sequences $(\alpha_j, \alpha_{j - 1}, \dotsc, \alpha_1)$, for all $j \in \mathbb Z_{>0}$. Also, when sequences are themselves written in a sequence, they are to be concatenated. For all $y \in \mathcal A$, denote $\omega(y) \in \Sigma^+$ to be any sequence such that $(y, \omega(y))$ is admissible. We extend the notation naturally to admissible sequences as well so that $\omega(\alpha^j) = \omega(\alpha_1)$ for all admissible sequences $\alpha^j$, for all $j \in \mathbb Z_{>0}$. For the rest of the section, any written sums over sequences will be understood to mean sums \emph{only} over \emph{admissible} sequences.

Let $\mathfrak{q} \subset \mathcal{O}_{\mathbb K}$ be an ideal coprime to $\mathfrak{q}_0$. For any measure $\mu$ on $\tilde{G}_\mathfrak{q}$ and $\phi \in L^2(\tilde{G}_\mathfrak{q}, \mathbb C)$ the convolution $\mu * \phi \in L^2(\tilde{G}_\mathfrak{q}, \mathbb C)$ is defined by
\begin{align*}
(\mu * \phi)(g) = \sum_{h \in \tilde{G}_\mathfrak{q}} \mu(h) \phi(gh^{-1})
\end{align*}
for all $g \in \tilde{G}_\mathfrak{q}$.

Now we present a bound which will be used often.

\begin{lemma}
\label{lem:SumExpf^aBound}
There exists $C > 1$ such that for all $|a| < a_0'$, for all $x \in \Sigma^+$, for all $k \in \mathbb Z_{>0}$, we have
\begin{align*}
\sum_{\alpha^k} e^{f_k^{(a)}(\alpha^k, x)} \leq C.
\end{align*}
\end{lemma}

\begin{proof}
Fix $C = e^{A_fa_0'} > 1$. Let $|a| < a_0'$. Then $|f^{(a)} - f^{(0)}| \leq A_f|a| < A_fa_0'$ and so we have
\begin{align*}
\sum_{\alpha^k} e^{f_k^{(a)}(\alpha^k, x)} \leq e^{A_fa_0'}\sum_{\alpha^k} e^{f_k^{(0)}(\alpha^k, x)} = e^{A_fa_0'}\mathcal{L}_0^k(\chi_U) = e^{A_fa_0'} = C.
\end{align*}
\end{proof}

We will denote the constant $C$ provided by \cref{lem:SumExpf^aBound} by $C_f$.

For all $\xi = a + ib \in \mathbb C$, for all ideals $\mathfrak{q} \subset \mathcal{O}_{\mathbb K}$ coprime to $\mathfrak{q}_0$, for all $x \in \Sigma^+$, for all integers $0 < r < s$, for all admissible sequences $(\alpha_s, \alpha_{s - 1}, \dotsc, \alpha_{r + 1})$, we define the measures
\begin{align*}
\mu_{(\alpha_s, \alpha_{s - 1}, \dotsc, \alpha_{r + 1})}^{\xi, \mathfrak{q}, x} &= \sum_{\alpha^r} e^{(f_s^{(a)} + ib\tau_s)(\alpha^s, x)} \delta_{\mathtt{c}_\mathfrak{q}^{r + 1}(\alpha_{r + 1}, \alpha^r, x)} \\
\nu_0^{a, \mathfrak{q}, x, r} &= \sum_{\alpha^r} e^{f_r^{(a)}(\alpha^r, x)} \delta_{\mathtt{c}_\mathfrak{q}^{r + 1}(\alpha_{r + 1}, \alpha^r, x)} \\
\hat{\mu}_{(\alpha_s, \alpha_{s - 1}, \dotsc, \alpha_{r + 1})}^{a, \mathfrak{q}, x} &= \sum_{\alpha^r} e^{f_s^{(a)}(\alpha^s, x)} \delta_{\mathtt{c}_\mathfrak{q}^{r + 1}(\alpha_{r + 1}, \alpha^r, x)} = e^{f_{s - r}^{(a)}(\alpha^s, x)} \nu_0^{a, \mathfrak{q}, x, r} \\
\nu_{(\alpha_s, \alpha_{s - 1}, \dotsc, \alpha_{r + 1})}^{a, \mathfrak{q}, x} &= e^{f_{s - r}^{(a)}(\alpha_s, \alpha_{s - 1}, \dotsc, \alpha_{r + 1}, \omega(\alpha_{r + 1}))} \nu_0^{a, \mathfrak{q}, x, r}
\end{align*}
on $\tilde{G}_\mathfrak{q}$ and also for all $H \in C(U, L^2(\tilde{G}_\mathfrak{q}, \mathbb C))$, define the function
\begin{align*}
\phi_{(\alpha_s, \alpha_{s - 1}, \dotsc, \alpha_{r + 1})}^{\mathfrak{q}, H} = \delta_{\mathtt{c}_\mathfrak{q}^{s - r - 1}(\alpha_s, \alpha_{s - 1}, \dotsc, \alpha_{r + 1}, \omega(\alpha_{r + 1}))} * H(\alpha_s, \alpha_{s - 1}, \dotsc, \alpha_{r + 1}, \omega(\alpha_{r + 1}))
%\phi_{(\alpha_s, \alpha_{s - 1}, \dotsc, \alpha_{r + 1})}^{q, H}(g) = H((\alpha_s, \alpha_{s - 1}, \dotsc, \alpha_{r + 1}, \omega(\alpha_{r + 1})), g\tilde{\mathtt{c}}_q^{s - r - 1}(\alpha_s, \alpha_{s - 1}, \dotsc, \alpha_{r + 1}, \omega(\alpha_{r + 1}))^{-1})
\end{align*}
in $L^2(\tilde{G}_\mathfrak{q}, \mathbb C)$ where we note that $\big\|\phi_{(\alpha_s, \alpha_{s - 1}, \dotsc, \alpha_{r + 1})}^{\mathfrak{q}, H}\big\|_2 \leq \|H\|_\infty$. Before we proceed, we record an easy lemma which relate the measures defined above.

\begin{lemma}
\label{lem:muHatLessThanCnu}
There exists $C > 0$ such that for all $\xi = a + ib \in \mathbb C$ with $|a| < a_0'$, for all ideals $\mathfrak{q} \subset \mathcal{O}_{\mathbb K}$ coprime to $\mathfrak{q}_0$, for all $x \in \Sigma^+$, for all integers $0 < r < s$, for all admissible sequences $(\alpha_s, \alpha_{s - 1}, \dotsc, \alpha_{r + 1})$, we have $\left|\mu_{(\alpha_s, \alpha_{s - 1}, \dotsc, \alpha_{r + 1})}^{\xi, \mathfrak{q}, x}\right| \leq \hat{\mu}_{(\alpha_s, \alpha_{s - 1}, \dotsc, \alpha_{r + 1})}^{a, \mathfrak{q}, x}$ and moreover we have both $\hat{\mu}_{(\alpha_s, \alpha_{s - 1}, \dotsc, \alpha_{r + 1})}^{a, \mathfrak{q}, x} \leq C\nu_{(\alpha_s, \alpha_{s - 1}, \dotsc, \alpha_{r + 1})}^{a, \mathfrak{q}, x}$ and $\nu_{(\alpha_s, \alpha_{s - 1}, \dotsc, \alpha_{r + 1})}^{a, \mathfrak{q}, x} \leq C\hat{\mu}_{(\alpha_s, \alpha_{s - 1}, \dotsc, \alpha_{r + 1})}^{a, \mathfrak{q}, x}$.
\end{lemma}

\begin{proof}
Fix $C = e^{\frac{T_0 \theta}{1 - \theta}}$. Let $\xi = a + ib \in \mathbb C$ with $|a| < a_0'$, $\mathfrak{q} \subset \mathcal{O}_{\mathbb K}$ be an ideal coprime to $\mathfrak{q}_0$, $x \in \Sigma^+$, $0 < r < s$ be integers and $(\alpha_s, \alpha_{s - 1}, \dotsc, \alpha_{r + 1})$ be an admissible sequence. Denote $\mu_{(\alpha_s, \alpha_{s - 1}, \dotsc, \alpha_{r + 1})}^{\xi, \mathfrak{q}, x}$ by $\mu$, $\hat{\mu}_{(\alpha_s, \alpha_{s - 1}, \dotsc, \alpha_{r + 1})}^{a, \mathfrak{q}, x}$ by $\hat{\mu}$, and $\nu_{(\alpha_s, \alpha_{s - 1}, \dotsc, \alpha_{r + 1})}^{a, \mathfrak{q}, x}$ by $\nu$. It is an easy computation to check $f_s^{(a)}(\alpha^s, x) = f_{s - r}^{(a)}(\alpha^s, x) + f_r^{(a)}(\alpha^r, x)$ and the first inequality of the lemma $|\mu| \leq \hat{\mu}$ from definitions. We also have
\begin{align*}
&\left|f_s^{(a)}(\alpha^s, x) - (f_{s - r}^{(a)}(\alpha_s, \alpha_{s - 1}, \dotsc, \alpha_{r + 1}, \omega(\alpha_{r + 1})) + f_r^{(a)}(\alpha^r, x))\right| \\
={}&\left|f_{s - r}^{(a)}(\alpha^s, x) - f_{s - r}^{(a)}(\alpha_s, \alpha_{s - 1}, \dotsc, \alpha_{r + 1}, \omega(\alpha_{r + 1}))\right| \\
\leq{}&\sum_{k = 0}^{s - r - 1}\left|f^{(a)}(\sigma^k(\alpha^s, x)) - f^{(a)}(\sigma^k(\alpha_s, \alpha_{s - 1}, \dotsc, \alpha_{r + 1}, \omega(\alpha_{r + 1})))\right| \\
\leq{}&\sum_{k = 0}^{s - r - 1}\Lip_{d_\theta}(f^{(a)}) \cdot d_\theta(\sigma^k(\alpha^s, x), \sigma^k(\alpha_s, \alpha_{s - 1}, \dotsc, \alpha_{r + 1}, \omega(\alpha_{r + 1}))) \\
\leq{}&\Lip_{d_\theta}(f^{(a)}) \sum_{k = 0}^{s - r - 1} \theta^{s - r - k} \\
\leq{}&\frac{T_0 \theta}{1 - \theta}.
\end{align*}
Hence the rest of the lemma follows by comparing $\hat{\mu}$ and $\nu$ and using the above computed estimate.
\end{proof}

Now we return to our goal of approximating the transfer operator.
%Might need |a| < a_0. Depends on whether the normalization as in MOW paper works.
\begin{lemma}
\label{lem:TransferOperatorConvolutionApproximation}
For all $\xi = a + ib \in \mathbb C$ with $|a| < a_0'$, for all ideals $\mathfrak{q} \subset \mathcal{O}_{\mathbb K}$ coprime to $\mathfrak{q}_0$, for all $x \in \Sigma^+$, for all integers $0 < r < s$, for all $H \in \mathcal{V}_\mathfrak{q}(U)$, we have
\begin{multline*}
\left\|\mathcal{M}_{\xi, \mathfrak{q}}^s(H)(x) - \sum_{\alpha_{r + 1}, \alpha_{r + 2}, \dotsc, \alpha_s} \mu_{(\alpha_s, \alpha_{s - 1}, \dotsc, \alpha_{r + 1})}^{\xi, \mathfrak{q}, x} * \phi_{(\alpha_s, \alpha_{s - 1}, \dotsc, \alpha_{r + 1})}^{\mathfrak{q}, H}\right\|_2 \\
\leq C_f \Lip_{d_\theta}(H)\theta^{s - r}.
\end{multline*}
\end{lemma}

\begin{proof}
Let $\xi = a + ib \in \mathbb C$ with $|a| < a_0'$, $\mathfrak{q} \subset \mathcal{O}_{\mathbb K}$ be an ideal coprime to $\mathfrak{q}_0$, $x \in \Sigma^+$, $0 < r < s$ be integers and $H \in \mathcal{V}_\mathfrak{q}(U)$. We calculate that
\begin{align*}
&\sum_{\alpha_{r + 1}, \alpha_{r + 2}, \dotsc, \alpha_s} \mu_{(\alpha_s, \alpha_{s - 1}, \dotsc, \alpha_{r + 1})}^{\xi, \mathfrak{q}, x} * \phi_{(\alpha_s, \alpha_{s - 1}, \dotsc, \alpha_{r + 1})}^{\mathfrak{q}, H} \\
={}& \sum_{\alpha_{r + 1}, \alpha_{r + 2}, \dotsc, \alpha_s} \left(\sum_{\alpha^r} e^{(f_s^{(a)} + ib\tau_s)(\alpha^s, x)} \delta_{\mathtt{c}_\mathfrak{q}^{r + 1}(\alpha_{r + 1}, \alpha^r, x)}\right) \\
{}&* \left(\delta_{\mathtt{c}_\mathfrak{q}^{s - r - 1}(\alpha_s, \alpha_{s - 1}, \dotsc, \alpha_{r + 1}, \omega(\alpha_{r + 1}))} * H(\alpha_s, \alpha_{s - 1}, \dotsc, \alpha_{r + 1}, \omega(\alpha_{r + 1}))\right) \\
={}& \sum_{\alpha_{r + 1}, \alpha_{r + 2}, \dotsc, \alpha_s} \sum_{\alpha^r} e^{(f_s^{(a)} + ib\tau_s)(\alpha^s, x)} \delta_{\mathtt{c}_\mathfrak{q}^{s - r - 1}(\alpha_s, \alpha_{s - 1}, \dotsc, \alpha_{r + 1}, \omega(\alpha_{r + 1})) \mathtt{c}_\mathfrak{q}^{r + 1}(\alpha_{r + 1}, \alpha^r, x)} \\
{}&* H(\alpha_s, \alpha_{s - 1}, \dotsc, \alpha_{r + 1}, \omega(\alpha_{r + 1})) \\
={}& \sum_{\alpha^s} e^{(f_s^{(a)} + ib\tau_s)(\alpha^s, x)} \delta_{\mathtt{c}_\mathfrak{q}^s(\alpha^s, x)} * H(\alpha_s, \alpha_{s - 1}, \dotsc, \alpha_{r + 1}, \omega(\alpha_{r + 1})).
\end{align*}
\begin{comment}
={}& \sum_{\alpha^s} e^{(f_s^{(a)} + ib\tau_s)(\alpha^s, x)} \left(\delta_{\tilde{\mathtt{c}}_q^{r + 1}(\alpha_{r + 1}, \alpha^r, x)} * \phi_{(\alpha_s, \alpha_{s - 1}, \dotsc, \alpha_{r + 1})}^{q, H}\right)(g) \\
={}& \sum_{\alpha^s} e^{(f_s^{(a)} + ib\tau_s)(\alpha^s, x)} H((\alpha_s, \alpha_{s - 1}, \dotsc, \alpha_{r + 1}, \omega(\alpha_{r + 1})), g\tilde{\mathtt{c}}_q^{r + 1}(\alpha_{r + 1}, \alpha^r, x)^{-1} \tilde{\mathtt{c}}_q^{s - r - 1}(\alpha_s, \alpha_{s - 1}, \dotsc, \alpha_{r + 1}, \omega(\alpha_{r + 1}))^{-1}) \\
={}& \sum_{\alpha^s} e^{(f_s^{(a)} + ib\tau_s)(\alpha^s, x)} H((\alpha_s, \alpha_{s - 1}, \dotsc, \alpha_{r + 1}, \omega(\alpha_{r + 1})), g\tilde{\mathtt{c}}_q^s(\alpha^s, x)^{-1}).
\end{comment}
Thus
\begin{align*}
&\left\|\mathcal{M}_{\xi, \mathfrak{q}}^s(H)(x) - \sum_{\alpha_{r + 1}, \alpha_{r + 2}, \dotsc, \alpha_s} \mu_{(\alpha_s, \alpha_{s - 1}, \dotsc, \alpha_{r + 1})}^{\xi, \mathfrak{q}, x} * \phi_{(\alpha_s, \alpha_{s - 1}, \dotsc, \alpha_{r + 1})}^{\mathfrak{q}, H}\right\|_2 \\
={}& \left\|\sum_{\alpha^s} e^{(f_s^{(a)} + ib\tau_s)(\alpha^s, x)} \delta_{\mathtt{c}_\mathfrak{q}^s(\alpha^s, x)} * H(\alpha^s, x) \right.\\
{}&\left.- \sum_{\alpha^s} e^{(f_s^{(a)} + ib\tau_s)(\alpha^s, x)} \delta_{\mathtt{c}_\mathfrak{q}^s(\alpha^s, x)} * H(\alpha_s, \alpha_{s - 1}, \dotsc, \alpha_{r + 1}, \omega(\alpha_{r + 1}))\right\|_2 \\
={}& \sum_{\alpha^s} \left\|e^{(f_s^{(a)} + ib\tau_s)(\alpha^s, x)}\delta_{\mathtt{c}_\mathfrak{q}^s(\alpha^s, x)} * (H(\alpha^s, x) - H(\alpha_s, \alpha_{s - 1}, \dotsc, \alpha_{r + 1}, \omega(\alpha_{r + 1})))\right\|_2 \\
\leq{}& \sum_{\alpha^s} e^{f_s^{(a)}(\alpha^s, x)} \left\|H(\alpha^s, x) - H(\alpha_s, \alpha_{s - 1}, \dotsc, \alpha_{r + 1}, \omega(\alpha_{r + 1}))\right\|_2 \\
\leq{}& \sum_{\alpha^s} e^{f_s^{(a)}(\alpha^s, x)} \Lip_{d_\theta}(H) \cdot d_\theta((\alpha^s, x), (\alpha_s, \alpha_{s - 1}, \dotsc, \alpha_{r + 1}, \omega(\alpha_{r + 1}))) \\
\leq{}& C_f\Lip_{d_\theta}(H)\theta^{s - r}
\end{align*}
where we used the fact that convolutions with $\delta_{\mathtt{c}_\mathfrak{q}^s(\alpha^s, x)}$ preserves the norm and also \cref{lem:SumExpf^aBound}.
\end{proof}

We will use this approximation to study the convolution, rather than dealing with the transfer operator directly, and obtain strong bounds. This is the objective of \cref{subsec:L2FlatteningLemma} but first we need to establish an important fact in \cref{subsec:ZariskiDensityofH} which will be used in \cref{subsec:L2FlatteningLemma}.

\subsection{Zariski density of \texorpdfstring{$\tilde{H}^p(y, z)$}{Hp(y, z)} in \texorpdfstring{$\tilde{\mathbf{G}}$}{G}}
\label{subsec:ZariskiDensityofH}
In this subsection, we prove \cref{lem:Z-DenseInU-CoverG} which will be required to use the expander machinery of Golsefidy-Varj\'{u} \cite{GV12}.

For all $p \in \mathbb Z_{>0}$, for all $(y, z) \in \mathcal{A}^2$, define $H^p(y, z) = \langle S^p(y, z) \rangle < \Gamma < G$ where
\begin{align*}
&S^p(y, z) \\
={}&\left\{\prod_{j = 0}^p \mathtt{c}(\alpha_{p + 1 - j}, \alpha_{p - j}) \prod_{j = 0}^p \mathtt{c}(\tilde{\alpha}_{1 + j}, \tilde{\alpha}_j)^{-1} \in \Gamma: \alpha_{p + 1} = \tilde{\alpha}_{p + 1} = y, \alpha_0 = \tilde{\alpha}_0 = z\right\}.
\end{align*}
Define $\tilde{S}^p(y, z) = \tilde{\pi}^{-1}(S^p(y, z)) \subset \tilde{\Gamma} < \tilde{G}$ and $\tilde{H}^p(y, z) = \langle \tilde{S}^p(y, z) \rangle < \tilde{\Gamma} < \tilde{G}$. Since $\ker(\tilde{\pi}) = \{e, -e\}$, we note that $\tilde{H}^p(y, z) = \tilde{\pi}^{-1}(H^p(y, z))$.

\begin{lemma}
\label{lem:Z-DenseInU-CoverG}
There exists $p_0 \in \mathbb Z_{> 0}$ such that for all integers $p > p_0$, for all $(y, z) \in \mathcal A^2$, the subgroup $\tilde{H}^p(y, z)$ is Zariski dense in $\tilde{\mathbf{G}}$.
\end{lemma}

To prove this, we first make some reductions.

\begin{lemma}
\label{lem:Z-DenseInG}
For all $(y, z) \in \mathcal A^2$, there exists $p_0 \in \mathbb Z_{> 0}$ such that for all integers $p > p_0$, the subgroup $H^p(y, z)$ is Zariski dense in $\mathbf{G}$.
\end{lemma}

%Similar to Lemma 2.8 in the arXiv version of Mixing of frame flow for rank one locally symmetric spaces and measure classification by Dale Winter
\begin{proof}[Proof that \cref{lem:Z-DenseInG} implies \cref{lem:Z-DenseInU-CoverG}]
Let $(y, z) \in \mathcal A^2$. By \cref{lem:Z-DenseInG}, there is a $p_{(y, z)} \in \mathbb Z_{> 0}$ such that for all integers $p > p_{(y, z)}$, the subgroup $H^p(y, z) < G$ is Zariski dense in $\mathbf{G}$. If $\tilde{H}^p(y, z)$ is not Zariski dense in $\tilde{\mathbf{G}}$, then consider $\tilde{\pi}\Big(\overline{\tilde{H}^p(y, z)}^{\mathrm{Z}}\Big) \subset \mathbf{G}$ which is the image of a proper subvariety of $\tilde{\mathbf{G}}$ and hence must be a Zariski constructible set. But it cannot contain any Zariski open subset of $\mathbf{G}$ for dimensional reasons. Thus it is contained in a finite union of proper subvarieties of $\mathbf{G}$ which contradicts the Zariski density of $H^p(y, z)$ in $\mathbf{G}$. Now we can simply choose $p_0 = \max_{(y, z) \in \mathcal A^2} p_{(y, z)}$.
\end{proof}

Through an isometry, we will now view the hyperbolic space in the upper half space model $\mathbb H^n \cong \{(x_1, x_2, \dotsc, x_n) \in \mathbb R^n: x_n > 0\}$ with boundary at infinity $\partial_\infty(\mathbb H^n) \cong \{(x_1, x_2, \dotsc, x_n) \in \mathbb R^n: x_n = 0\} \cup \{\infty\} \cong \mathbb R^{n - 1} \cup \{\infty\}$ and we also use the isometry $\T(\mathbb H^n) \cong \mathbb H^n \times \mathbb R^n$. By a slight abuse of notation, sometimes we will make the distinction and sometimes we will make the identification which will always be clear from context. Let $(e_1, e_2, \dotsc, e_n)$ be the standard basis on $\mathbb R^n$. Let $B^{\mathrm{E}}_\epsilon(u) \subset \mathbb R^{n - 1}$ denote the open Euclidean ball of radius $\epsilon > 0$ centered at $u \in \mathbb R^{n - 1}$. Let $d_{\mathrm{E}}$ denote the Euclidean distance. By a sphere $\tilde{V} \subset \partial_\infty(\mathbb H^n)$ we mean that $\tilde{V} \subset \mathbb R^{n - 1} \cup \{\infty\}$ is a $(n - 2)$-dimensional sphere or an affine hyperplane when we use the upper half space model. Let $\pi_1: \mathbb H^n \times \mathbb R^n \to \mathbb H^n$ be the tangent bundle projection map and $\pi_2: \mathbb H^n \times \mathbb R^n \to \mathbb R^n$ be the projection onto the directional component. For convenience, we denote $\overline{\Omega} \subset \T^1(\mathbb H^n)$ to be the \emph{full preimage} of $\Omega$ under the covering map $\T^1(\mathbb H^n) \to \T^1(X)$. For all $j \in \mathcal{A}$, denote by the $\hat{\overline{R}}_j, \hat{\overline{U}}_j$ and $\hat{\overline{S}}_j$ the cores of $\overline{R}_j, \overline{U}_j$ and $\overline{S}_j$ respectively.

\begin{comment}
\begin{lemma}
\label{lem:LimitSetInSphereNearOriginLargeP_NZ}
Suppose $P_{\mathrm{NZ}}(y, z) = \{p \in \mathbb Z_{> 0}: H^p(y, z) \text{ is not Zariski dense}\}$ is unbounded for some $(y, z) \in \mathcal A^2$. Applying any appropriate isometry, we assume that the vectors in $\hat{\overline{U}}_y$ have direction $\pi_2(\hat{\overline{U}}_y) = -e_n$ and their basepoints lie on the hyperplane $\langle \pi_1(\hat{\overline{U}}_y), e_n\rangle = 1$. Then for all $\epsilon > 0$, there is a $p_0 \in \mathbb Z_{> 0}$ such that for all integers $p \in P_{\mathrm{NZ}}(y, z)$ with $p > p_0$, for all $u \in \hat{\overline{U}}_y$, we have that $\Lambda(H^p(y, z))$ is contained in a $0$ or $n - 2$ dimensional sphere $\tilde{V} \subset \partial \mathbb H^n$ such that $\tilde{V} \cap B^{\mathrm{E}}_\epsilon(u^+) \neq \varnothing$.
\end{lemma}
\end{comment}

\begin{lemma}
\label{lem:LimitSetInSphereNearOriginLargeP_NZ}
Let $(y, z) \in \mathcal A^2$. Applying any appropriate isometry, we assume that the vectors in $\overline{U}_y$ have direction $\pi_2(\overline{U}_y) = -e_n$ and their basepoints lie on the hyperplane $\langle \pi_1(\overline{U}_y), e_n\rangle = 1$. Then for all $\epsilon > 0$, there exists $p_0 \in \mathbb Z_{> 0}$ such that for all integers $p > p_0$, if $H^p(y, z)$ is not Zariski dense in $\mathbf{G}$, then for all $u \in \overline{U}_y$, we have that $\Lambda(H^p(y, z))$ is contained in a $0$-dimensional or $(n - 2)$-dimensional sphere $\tilde{V} \subset \partial_\infty(\mathbb H^n)$ such that $\tilde{V} \cap B^{\mathrm{E}}_\epsilon(u^+) \neq \varnothing$.
\end{lemma}

\begin{proof}[Proof that \cref{lem:LimitSetInSphereNearOriginLargeP_NZ} implies \cref{lem:Z-DenseInG}]
Suppose \cref{lem:Z-DenseInG} is false. Then the set
\begin{align*}
P_{\mathrm{NZ}}(y, z) = \{p \in \mathbb Z_{> 0}: H^p(y, z) \text{ is not Zariski dense in } \mathbf{G}\}
\end{align*}
is unbounded for some $(y, z) \in \mathcal A^2$. Apply any appropriate isometry to assume $\pi_2(\overline{U}_y) = -e_n$ and $\langle \pi_1(\overline{U}_y), e_n\rangle = 1$. By properness of $\overline{U}_y$, there are $w_0 \in \overline{\Omega}$ and $\delta_1 > 0$ such that $W_{\delta_1}^{\mathrm{su}}(w_0) \cap \overline{\Omega} \subset \overline{U}_y$. Consider the diffeomorphism $\Phi: W_{\delta_1}^{\mathrm{su}}(w_0) \to B_{\delta_2}^{\mathrm{E}}(w_0^+)$ defined by $\Phi(w) = w^+$ for all $w \in W_{\delta_1}^{\mathrm{su}}(w_0)$. Since $\Gamma$ is Zariski dense in $\mathbf{G}$, recalling that the set of attracting fixed points of hyperbolic elements of $\Gamma$ is dense in $\Lambda(\Gamma)$, we use the contrapositive of \cite[Proposition 3.12]{Win15} to conclude that $\Lambda(\Gamma) \cap B_{\delta_2}^{\mathrm{E}}(w_0^+) \subset \mathbb R^{n - 1}$ is not contained in any $(n - 2)$-dimensional sphere in $\partial_\infty(\mathbb H^n)$. Thus there is a set of distinct points $\{u_1^+, u_2^+, \dotsc, u_{n + 1}^+\} \subset \Lambda(\Gamma) \cap B_{\delta_2}^{\mathrm{E}}(w_0^+)$ such that it is not contained in any $(n - 2)$-dimensional sphere in $\partial_\infty(\mathbb H^n)$. But $\{(w_1, w_2, \dotsc, w_{n + 1}) \in (\mathbb R^{n - 1})^{n + 1}: \{w_1, w_2, \dotsc, w_{n + 1}\} \text{ is not contained in any } (n - 2)\text{-dimensional sphere in } \partial_\infty(\mathbb H^n)\} \subset (\mathbb R^{n - 1})^{n + 1}$ is an open subset. Hence there is a $\epsilon > 0$ such that if $w_j \in \Lambda(\Gamma) \cap B^{\mathrm{E}}_\epsilon(u_j^+) \neq \varnothing$ for all integers $1 \leq j \leq n + 1$, then $\{w_1, w_2, \dotsc, w_{n + 1}\}$ is also not contained in any $(n - 2)$-dimensional sphere in $\partial_\infty(\mathbb H^n)$. Now there are corresponding vectors $u_j = \Phi^{-1}(u_j^+) \in \overline{U}_y$ for all integers $1 \leq j \leq n + 1$.

\cref{lem:LimitSetInSphereNearOriginLargeP_NZ} provides a corresponding $p_0 \in \mathbb Z_{>0}$ for this $\epsilon$. Since $P_{\mathrm{NZ}}(y, z)$ is unbounded, we can fix some $p \in P_{\mathrm{NZ}}(y, z)$ with $p > p_0$ so that $H^p(y, z)$ is not Zariski dense in $\mathbf{G}$. Then \cref{lem:LimitSetInSphereNearOriginLargeP_NZ} implies that $\Lambda(H^p(y, z))$ is contained in a $0$-dimensional or $(n - 2)$-dimensional sphere $\tilde{V} \subset \partial_\infty(\mathbb H^n)$ such that $\tilde{V} \cap B^{\mathrm{E}}_\epsilon(u_j^+) \neq \varnothing$ for all integers $1 \leq j \leq n + 1$ which is a contradiction.
\end{proof}

\cref{lem:LimitSetInSphereNearOriginLargeP_NZ} is a consequence of the following lemmas put together.

\begin{lemma}
\label{lem:ProducingElementsofH}
Let $p \in \mathbb Z_{>0}$ and $(y, z) \in \mathcal A^2$. If $u_0 \in \hat{\overline{R}}_y$ such that $\mathcal P^{p + 1}(u_0) \in g_z\hat{\overline{R}}_z$ for some $g_z \in \Gamma$ and $v_0 \in g_z\hat{\overline{R}}_z$ such that $\mathcal P^{-(p + 1)}(v_0) \in h\hat{\overline{R}}_y$ for some $h \in \Gamma$, then in fact $h \in S^p(y, z) \subset H^p(y, z)$.
\end{lemma}

\begin{proof}
Let $p \in \mathbb Z_{>0}$ and $(y, z) \in \mathcal A^2$. Suppose $u_0 \in \hat{\overline{R}}_y$ such that $\mathcal P^{p + 1}(u_0) \in g_z\hat{\overline{R}}_z$ for some $g_z \in \Gamma$ and $v_0 \in g_z\hat{\overline{R}}_z$ such that $\mathcal P^{-(p + 1)}(v_0) \in h\hat{\overline{R}}_y$ for some $h \in \Gamma$. Let $u_0' = h^{-1} \mathcal P^{-(p + 1)}(v_0) \in \hat{\overline{R}}_y$. Then $\mathcal P^{p + 1}(u_0') \in g_z'\hat{\overline{R}}_z$ where $g_z' = h^{-1}g_z$. Now by definitions, there are admissible sequences $(\alpha_{p + 1}, \alpha_p, \dotsc, \alpha_0)$ and $(\tilde{\alpha}_{p + 1}, \tilde{\alpha}_p, \dotsc, \tilde{\alpha}_0)$ with $\alpha_{p + 1} = \tilde{\alpha}_{p + 1} = y$ and $\alpha_0 = \tilde{\alpha}_0 = z$ where $g_z = \mathtt{c}^{p + 1}(\alpha_{p + 1}, \alpha_p, \dotsc, \alpha_0) = \prod_{j = 0}^p \mathtt{c}(\alpha_{p + 1 - j}, \alpha_{p - j})$ and $g_z' = \mathtt{c}^{p + 1}(\tilde{\alpha}_{p + 1}, \tilde{\alpha}_p, \dotsc, \tilde{\alpha}_0) = \prod_{j = 0}^p \mathtt{c}(\tilde{\alpha}_{p + 1 - j}, \tilde{\alpha}_{p - j})$. Thus $h = g_z{g_z'}^{-1} = \prod_{j = 0}^p \mathtt{c}(\alpha_{p + 1 - j}, \alpha_{p - j}) \prod_{j = 0}^p \mathtt{c}(\tilde{\alpha}_{1 + j}, \tilde{\alpha}_j)^{-1}$ where $\alpha_{p + 1} = \tilde{\alpha}_{p + 1} \allowbreak = y$ and $\alpha_0 = \tilde{\alpha}_0 = z$ which shows that $h \in S^p(y, z) \subset H^p(y, z)$ by definition.
\end{proof}

We will use the above procedure to produce elements of $S^p(y, z) \subset H^p(y, z)$.

\begin{lemma}
\label{lem:P^p_in_R_z_ForLarge_p}
Let $(y, z) \in \mathcal A^2$. For all $\epsilon > 0$, there exists $p_1 \in \mathbb Z_{>0}$ such that for all integers $p > p_1$, for all $u \in \overline{U}_y$, there exists $u_0 \in \hat{\overline{U}}_y$ with $d_{\mathrm{su}}(u, u_0) < \epsilon$ such that $\mathcal P^{p + 1}(u_0) \in \Gamma\hat{\overline{R}}_z$.
\end{lemma}

\begin{proof}
Fix $c_0, \lambda_0, \epsilon_0$ to be the constants $c, \log(\lambda), q$ respectively from the Anosov property in \cite{Rat73}. Recall the trajectory isomorphism $\psi$ from \cite{Rat73}. Also recall $\underline{\tau} = \inf_{u' \in R} \tau(u') = \inf_{u' \in \overline{R}} \tau(u')$. Let $(y, z) \in \mathcal A^2$. Let $\epsilon > 0$. Fix an integer $p_1 \geq N_T$ such that ${c_0}^2 \hat{\delta} e^{\underline{\tau} - \lambda_0 \underline{\tau} (p + 1 - N_T)} < \frac{\epsilon}{2}$ and let $p > p_1$ be an integer. Let $u \in \overline{U}_y$. Since $\hat{\overline{U}}_y \subset \overline{U}_y$ is a residual set, there is a $u_1 \in \hat{\overline{U}}_y$ with $d_{\mathrm{su}}(u_1, u) < \frac{\epsilon}{2}$. Consider $u_2 = \mathcal P^{p + 1 - N_T}(u_1) \in g_{y'} \hat{\overline{R}}_{y'}$ for some $g_{y'} \in \Gamma$. Since $T$ is topologically mixing and $g_{y'}\hat{\overline{R}}_{y'} \subset g_{y'}\overline{R}_{y'}$ is a residual set, there is a $u_3 \in g_{y'}\hat{\overline{R}}_{y'}$ such that $v = \mathcal P^{N_T}(u_3) \in g_z\hat{\overline{R}}_z$ for some $g_z \in \Gamma$. Let
\begin{align*}
u_0 = \mathcal P^{-(p + 1 - N_T)}([u_3, u_2]) \in \mathcal P^{-(p + 1 - N_T)}([g_{y'}\hat{\overline{U}}_{y'}, u_2]) \subset [\hat{\overline{U}}_y, u_1] = \hat{\overline{U}}_y.
\end{align*}
Denote $w = [w_0, u_2] \in \overline{S}_{y'}$ where $w_0 \in \overline{R}_{y'}$ is the center. Then
\begin{align*}
\mathcal P^{p + 1}(u_0) &= (\mathcal P^{N_T} \circ \mathcal P^{p + 1 - N_T})(\mathcal P^{-(p + 1 - N_T)}([u_3, u_2])) = \mathcal P^{N_T}([u_3, u_2]) \\
&\in \mathcal P^{N_T}([u_3, g_{y'}\hat{\overline{S}}_{y'}]) \subset [v, g_z\hat{\overline{S}}_z] \subset g_z\hat{\overline{R}}_z
\end{align*}
as desired. Also
\begin{align}
\label{eqn:HyperbolicityComputation}
d_{\mathrm{su}}(u_1, u_0) &< c_0e^{-\lambda_0 \tau_{p + 1 - N_T}(u_1)}d_{\mathrm{su}}(u_2, {\psi_{u_2}}^{-1}([u_3, u_2])) \\
&< {c_0}^2e^{\underline{\tau} - \lambda_0 \tau_{p + 1 - N_T}(u_1)}d_{\mathrm{su}}({\psi_w}^{-1}(u_2), {\psi_w}^{-1}([u_3, u_2])) \\
&< {c_0}^2 \hat{\delta} e^{\underline{\tau} - \lambda_0 \underline{\tau} (p + 1 - N_T)} < \frac{\epsilon}{2}
\end{align}
since $p > p_1$. This implies $d_{\mathrm{su}}(u, u_0) < \epsilon$. See the construction of Markov sections in \cite{Rat73} for more details for the bounds.
\end{proof}

\begin{lemma}
\label{lem:Far_v_0}
Let $(y, z) \in \mathcal A^2$. There exist $B > 0$ and $p_2 \in \mathbb Z_{>0}$ such that for all $g_z \in \Gamma$, for all $v \in g_z\hat{\overline{R}}_z$, for all integers $p > p_2$, the following holds. Applying any appropriate isometry, we assume that $v = (e_n, -e_n)$. Then there exists $v_0 \in [v, g_z\hat{\overline{S}}_z]$ such that
\begin{enumerate}
\item \label{itm:BackwardLimitPointBound} $v_0^- \in B_B^{\mathrm{E}}(0)$
\item \label{itm:BackwardFlowToR_y} $\mathcal P^{-(p + 1)}(v_0) \in \Gamma\hat{\overline{R}}_y$.
\end{enumerate}
\end{lemma}

\begin{proof}
Let $(y, z) \in \mathcal A^2$. Let $w_0 \in \overline{R}_z$ be the center. Fix some distinct $w_1, w_2 \in \overline{S}_z \neq \varnothing$ which is possible using the definition of $\overline{S}_z$ and the fact that $\Lambda(\Gamma)$ is a nonempty perfect set (since $\Gamma$ is nonelementary). Let $\delta_1 = \frac{1}{5}d_{\mathrm{ss}}(w_1, w_2) > 0$. Now we make a modification to ensure that \cref{itm:BackwardFlowToR_y} of the lemma holds. By a similar argument as in the proof of \cref{lem:P^p_in_R_z_ForLarge_p} with minor modifications, we conclude that there is a $p_2 \in \mathbb Z_{>0}$ such that for any chosen $p > p_2$, for all $j \in \{1, 2\}$, there is a $\tilde{w}_j \in \hat{\overline{S}}_z$ with $d_{\mathrm{ss}}(\tilde{w}_j, w_j) < \delta_1$ such that $\mathcal P^{-(p + 1)}(\tilde{w}_j) \in \Gamma \hat{\overline{R}}_y$.

Now suppose $v \in g_z\hat{\overline{R}}_z$ for some $g_z \in \Gamma$. If $[w_0, {g_z}^{-1}v] \in W_{\delta_1}^{\mathrm{ss}}(\tilde{w}_1)$, then set $\tilde{v}_0 = \tilde{w}_2$ and otherwise set $\tilde{v}_0 = \tilde{w}_1$. Then $d_{\mathrm{ss}}([w_0, {g_z}^{-1}v], \tilde{v}_0) > \delta_1$ and also $\mathcal P^{-(p + 1)}(\tilde{v}_0) \in \Gamma \hat{\overline{R}}_y$. Set $v_0 = [v, g_z\tilde{v}_0] \in [v, g_z \hat{\overline{S}}_z]$. Now by left $G$-invariance we have $d_{\mathrm{ss}}([g_z w_0, v], g_z \tilde{v}_0) > \delta_1$ and we need to extend the bound to obtain a uniform lower bound for $d_{\mathrm{ss}}(v, v_0)$. Using the homeomorphism $\overline{U}_z \times \overline{S}_z \to \overline{R}_z$ defined by $(v', w') \mapsto [v', w']$ and the compactness of the sets $\overline{U}_z$ and $\overline{S}_z$, there is a $\delta \in (0, \delta_1)$ such that
\begin{align*}
\inf\left\{d_{\mathrm{ss}}([v', u], [v', u']): v' \in \overline{R}_z \text{ and } u, u' \in \overline{S}_z \text{ with } d_{\mathrm{ss}}(u, u') \geq \delta_1\right\} > \delta
\end{align*}
where $\delta$ only depends on the rectangle $\overline{R}_z$. This implies that $d_{\mathrm{ss}}(v, v_0) > \delta$. We can also check that
\begin{align*}
\mathcal P^{-(p + 1)}(v_0) = \mathcal P^{-(p + 1)}([v, g_z\tilde{v}_0]) \in \mathcal P^{-(p + 1)}([g_z \hat{\overline{U}}_z, g_z\tilde{v}_0]) \subset \Gamma \hat{\overline{R}}_y
\end{align*}
by the Markov property since $\mathcal P^{-(p + 1)}(g_z \tilde{v}_0) \in \Gamma \hat{\overline{R}}_y$. So \cref{itm:BackwardFlowToR_y} still holds.

Apply any appropriate isometry to assume $v = (e_n, -e_n)$. Then $v_0 \in A_{\delta, \hat{\delta}}^{\mathrm{ss}}(v) = W_{\hat{\delta}}^{\mathrm{ss}}(v) \setminus \overline{W_{\delta}^{\mathrm{ss}}(v)}^{\mathrm{ss}}$. Note that by homogeneity of $\mathbb H^n$ and compactness of $\overline{A_{\delta, \hat{\delta}}^{\mathrm{ss}}(v)}^{\mathrm{ss}} \subset \mathbb H^n \times \mathbb R^n$, the distance function $d_{\mathrm{ss}}|_{\overline{A_{\delta, \hat{\delta}}^{\mathrm{ss}}(v)}^{\mathrm{ss}}}$ and the \emph{Euclidean} distance function $d_{\mathrm{E}}|_{\overline{A_{\delta, \hat{\delta}}^{\mathrm{ss}}(v)}^{\mathrm{ss}}}$ (measuring the distance between basepoints of tangent vectors along the horosphere) are equivalent, where the implied constant is uniform in $g_z \in \Gamma$ and $v \in g_z \overline{R}_z$ (since we are assuming $v = (e_n, -e_n)$). Since $v^- = \infty$, there is a diffeomorphism $\Phi: \overline{A_{\delta, \hat{\delta}}^{\mathrm{ss}}(v)}^{\mathrm{ss}} \to \overline{A_{\delta_2, B}^{\mathrm{E}}(0)}$ for some $0 < \delta_2 < B$ defined by $\Phi(w) = w^-$ for all $w \in \overline{A_{\delta, \hat{\delta}}^{\mathrm{ss}}(v)}^{\mathrm{ss}}$, where $A_{\delta_2, B}^{\mathrm{E}}(0) = B_B^{\mathrm{E}}(0) \setminus \overline{B_{\delta_2}^{\mathrm{E}}(0)} \subset \mathbb R^{n - 1}$ and $B$ only depends on $\delta$. Thus $v_0^- \in B_B^{\mathrm{E}}(0)$ which is \cref{itm:BackwardLimitPointBound}.
\end{proof}

In the proof of \cref{lem:LimitSetInSphereNearOrigin}, we will also denote by $B^{\mathrm{E}}_\epsilon(u) \subset \mathbb R^n$ the open Euclidean ball of radius $\epsilon > 0$ centered at $u \in \mathbb R^n$ where we will explicitly show the containment to be clear and otherwise the notation will mean the open Euclidean ball in $\mathbb R^{n - 1}$ as defined before.

\begin{comment}
\begin{lemma}
\label{lem:LimitSetInSphereNearOrigin}
Suppose $u_0 \in \hat{\overline{R}}_y$ such that $v = \mathcal P^{p + 1}(u_0) \in \Gamma\hat{\overline{R}}_z$ for some $p \in \mathbb Z_{>0}$. Applying any appropriate isometry, we assume that $u_0 = (e_n, -e_n)$. Then for all $\epsilon > 0$, there is a $p_3 \in \mathbb Z_{> 0}$ such that if $p > p_3$ and $H^p(y, z)$ is not Zariski dense, then $\Lambda(H^p(y, z))$ is contained in a $n - 2$ dimensional sphere $\tilde{V} \subset \partial \mathbb H^n$ such that $d_{\mathrm{E}}(\tilde{V} \cap \mathbb R^{n - 1}, 0) < \epsilon$.
\end{lemma}
\end{comment}

\begin{lemma}
\label{lem:LimitSetInSphereNearOrigin}
Let $(y, z) \in \mathcal A^2$. For all $\epsilon > 0$, there exists $p_3 \in \mathbb Z_{> 0}$ such that the following holds. Suppose that $H^p(y, z)$ is not Zariski dense in $\mathbf{G}$ for some integer $p > p_3$. Suppose further that there exists $u_0 \in \hat{\overline{R}}_y$ such that $v = \mathcal P^{p + 1}(u_0) \in \Gamma\hat{\overline{R}}_z$. Applying any appropriate isometry, we assume that $u_0 = (e_n, -e_n)$. Then $\Lambda(H^p(y, z))$ is contained in a $0$-dimensional or $(n - 2)$-dimensional sphere $\tilde{V} \subset \partial_\infty(\mathbb H^n)$ such that $\tilde{V} \cap B_\epsilon^{\mathrm{E}}(0) \neq \varnothing$.
\end{lemma}

\begin{proof}
%S^{n - 2}(c_S, r_S)
First we fix some notations and constants. Let $(y, z) \in \mathcal A^2$. Without loss of generality, let $\epsilon \in (0, 1)$. We will fix a $\epsilon_{1, 1} > 0$ which will be required for the case when $n > 2$ and $H^p(y, z)$ is elementary and a $\epsilon_{1, 2} > 0$ for the other cases as follows. Let $C = \partial_\infty(\mathbb H^n) \setminus B_\epsilon^{\mathrm{E}}(0)$. Define $\Banana(c_1, c_2) \subset \mathbb H^n$ to be the unique banana neighborhood of the geodesic with limit points $c_1, c_2 \in \partial_\infty(\mathbb H^n)$ such that $e_n \in \partial(\Banana(c_1, c_2))$. Consider the function $\varphi: C \times C \to \mathbb R$ defined by
\begin{align*}
&\varphi(c_1, c_2) \\
={}&\inf\Big\{u_n \in \mathbb R_{>0}: (u_1, u_2, \dotsc, u_n) \in \partial(\Banana(c_1, c_2)) \cap \Big(\overline{B_{\epsilon/2}^{\mathrm{E}}(0)} \times \mathbb R_{>0}\Big)\Big\}
\end{align*}
for all $(c_1, c_2) \in C \times C$. Fix $\epsilon_{1, 1}' = \inf_{(c_1, c_2) \in C \times C} \varphi(c_1, c_2)$. Then $\epsilon_{1, 1}' > 0$ because $C \times C$ is compact and $\varphi$ is positive continuous. Fix $\epsilon_{1, 1} = \min\left(\frac{\epsilon}{2}, \epsilon_{1, 1}'\right)$. Now fix
\begin{align*}
\epsilon_{1, 2} = \min\left(\frac{1}{2}, \frac{\epsilon^2}{3}, \frac{\epsilon^2}{12}, \frac{\epsilon}{\sqrt{6}}\right)
\end{align*}
where the bounds will become clear later. Fix $\epsilon_1 = \min(\epsilon_{1, 1}, \epsilon_{1, 2})$. Fix $B > 0$ and $p_2 \in \mathbb Z_{>0}$ to be the constants provided by \cref{lem:Far_v_0}. Fix $L_0 > 0$ such that $\frac{B}{L_0} < \frac{\epsilon_1}{2}$ and a $\epsilon_2 \in (0, B)$.

%See https://www.aimath.org/WWN/surfacegroups/strubel.pdf (specially Theorem 2.2.4 to see the connection to totally geodesic submanifolds from Mostow's theorem)
Suppose $H^p(y, z)$ is not Zariski dense in $\mathbf{G}$ for some $p \in \mathbb Z_{>0}$. Suppose $u_0 \in \hat{\overline{R}}_y$ such that $v = \mathcal P^{p + 1}(u_0) \in g_z\hat{\overline{R}}_z$ for some $g_z \in \Gamma$ and apply any appropriate isometry to assume $u_0 = (e_n, -e_n)$. First consider the case that $H^p(y, z)$ is elementary with the limit set $\Lambda(H^p(y, z)) = \tilde{V} = \{\tilde{V}_1, \tilde{V}_2\}$. For all $h' \in H^p(y, z)$, since $h'$ must be hyperbolic ($\Gamma$ is torsion-free convex cocompact), it preserves the hypersurface $V = \partial(\Banana({h'}^-, {h'}^+)) = \partial(\Banana(\tilde{V}_1, \tilde{V}_2)) \subset \mathbb H^n$ with boundary at infinity $\partial_\infty(V) = \tilde{V} \subset \partial_\infty(\mathbb H^n)$ which is a $0$-dimensional sphere. Hence the orbit $H^p(y, z)\pi_1(u_0)$ is contained in $V$. Now consider the case that $H^p(y, z)$ is nonelementary, and hence $\#\Lambda(H^p(y, z)) = \infty$, if $n > 2$ (recall that if $n = 2$, then being nonelementary is equivalent to being Zariski dense in $\mathbf{G}$). Since $\overline{H^p(y, z)}^{\mathrm{Z}} < \mathbf{G}$ is a nontrivial proper connected subgroup, the Karpelevi\v{c}-Mostow theorem \cite{DSO09,Kar53,Mos55} implies that $\overline{H^p(y, z)}^{\mathrm{Z}}$, and in particular $H^p(y, z)$, preserves a proper totally geodesic submanifold which in this case is a $(n - 1)$-dimensional Euclidean half sphere or a half affine hyperplane in $\mathbb H^n$ perpendicular to $\mathbb R^{n - 1}$. Consequently, $\Lambda(H^p(y, z))$ is contained in a unique $(n - 2)$-dimensional sphere $\tilde{V} \subset \partial_\infty(\mathbb H^n)$. This further implies that $H^p(y, z)$ preserves \emph{any} spherical hypersurface or half affine hyperplane in $\mathbb H^n$ with boundary at infinity $\tilde{V}$. In particular, the orbit $H^p(y, z)\pi_1(u_0)$ is contained in $V \subset \mathbb H^n$ which is the unique spherical hypersurface or half affine hyperplane with the corresponding spherical boundary at infinity $\partial_\infty(V) = \tilde{V} \subset \partial_\infty(\mathbb H^n)$.

For any of the cases, let $V \subset \mathbb H^n$ be the hypersurface containing the orbit $H^p(y, z)\pi_1(u_0)$ as described above. We will use \cref{lem:Far_v_0} to find a point of the orbit close to $0 \in \mathbb R^{n - 1}$. In order to apply \cref{lem:Far_v_0}, first suppose $p > p_2$ and define the isometry $\iota$ which simply scales by some factor $L > 0$ such that $\iota(v) = (e_n, -e_n)$.
%Take \tau_0 = inf_{u \in \overline{R}} \tau(u). Then the n^th base point coordinate of v is at least most e^{-(p + 1)\tau_0/2} and so the scaling to move v to (e_n, -e_n) is at least L = e^{(p + 1)\tau_0/2}.
It is clear that there is a $p_{3, 1} \in \mathbb Z_{>0}$ satisfying $e^{-\underline{\tau}(p_{3, 1} + 1)} < \frac{1}{L_0}$, so that for all $p > p_{3, 1}$ we have $L > L_0$. Since $v \in g_z\hat{\overline{R}}_z$ with $\iota(v) = (e_n, -e_n)$, we can apply \cref{lem:Far_v_0} to obtain a vector $v_0 \in g_z\hat{\overline{S}}_z$ which satisfies the properties listed in the lemma. Since $\pi_1(\overline{R}_y)$ is contained in some hyperbolic ball, so there is a $\delta_1 \in \left(0, \frac{\epsilon_2}{3}\right)$ such that if $d_{\mathrm{E}}(g\pi_1(\overline{R}_y), \partial_\infty(\mathbb H^n)) < \delta_1$ for some $g \in G$, then $g\pi_1(\overline{R}_y) \subset B_{\epsilon_2/3}^{\mathrm{E}}(c_g) \subset \mathbb R^n$ for some $c_g \in \mathbb R^n$. Note that $\overline{A_{\delta, \hat{\delta}}^{\mathrm{ss}}(\iota(v))}^{\mathrm{ss}}$ is compact and the geodesic flow is smooth and $\Phi\left(\overline{A_{\delta, \hat{\delta}}^{\mathrm{ss}}(\iota(v))}^{\mathrm{ss}}\right) = \overline{A_{\delta_2, B}^{\mathrm{E}}(0)}$, where we use the notations from the proof of \cref{lem:Far_v_0}. Hence there is a $p_{3, 2} \in \mathbb Z_{>0}$ such that if $p > p_{3, 2}$, then $d_{\mathrm{E}}(\pi_1(\mathcal P^{-(p + 1)}(\iota(v_0))), \iota(v_0^-)) < \delta_1$. Fix $p_3 = \max(p_2, p_{3, 1}, p_{3, 2})$ and assume $p > p_3$ henceforth. But recalling that $\mathcal P^{-(p + 1)}(\iota(v_0)) \in h\iota(\hat{\overline{R}}_y)$ for some $h \in H^p(y, z)$ by \cref{itm:BackwardFlowToR_y} of \cref{lem:Far_v_0} and \cref{lem:ProducingElementsofH}, we have $h\iota(\pi_1(\overline{R}_y)) \subset B_{\epsilon_2/3}^{\mathrm{E}}(c_h) \subset B_{\epsilon_2}^{\mathrm{E}}(\iota(v_0^-)) \subset \mathbb R^n$ for some $c_h \in \mathbb R^n$. In particular $h\iota(\pi_1(u_0)) \in B_{\epsilon_2}^{\mathrm{E}}(\iota(v_0^-)) \cap \iota(V) \subset \mathbb R^n$ with nonempty intersection. By \cref{itm:BackwardLimitPointBound} of \cref{lem:Far_v_0}, we have $v_0^- \in B_B^{\mathrm{E}}(0)$, so rescaling back by applying the inverse $\iota^{-1}$ gives $v_0^- \in B_{B/L}^{\mathrm{E}}(0)$ and $h\pi_1(u_0) \in B_{\epsilon_2/L}^{\mathrm{E}}(v_0^-) \cap V \subset \mathbb R^n$ which implies $h\pi_1(u_0) \in B_{\epsilon_1}^{\mathrm{E}}(0) \cap V \subset \mathbb R^n$. For ease of notation, let $w = h\pi_1(u_0)$ and write $w = (\tilde{w}, w_n) \in \mathbb R^{n - 1} \times \mathbb R_{>0}$.

Now consider the case that $H^p(y, z)$ is elementary so that $V$ is the hypersurface $\partial(\Banana(\tilde{V}_1, \tilde{V}_2))$. If $\tilde{V}_1, \tilde{V}_2 \notin B_\epsilon^{\mathrm{E}}(0)$, then the definition of $\epsilon_{1, 1}$ gives a contradiction to $w \in B_{\epsilon_1}^{\mathrm{E}}(0) \cap V \subset \mathbb R^n$. Hence either $\tilde{V}_1 \in B_\epsilon^{\mathrm{E}}(0)$ or $\tilde{V}_2 \in B_\epsilon^{\mathrm{E}}(0)$ which proves the lemma in this case.

Now consider the case that $H^p(y, z)$ is nonelementary if $n > 2$, and $V = S_{\mathrm{E}}^{n - 1}(c, r_{u_0}) \cap \mathbb H^n$ is a spherical hypersurface with center $c = (\tilde{c}, c_n) \in \mathbb R^{n - 1} \times \mathbb R_{>0}$ and radius $r_{u_0} > 0$ and with corresponding spherical boundary at infinity $\partial_\infty(V) = \tilde{V} = S_{\mathrm{E}}^{n - 2}(\tilde{c}, r_{\tilde{V}})$ with center $\tilde{c}$ and radius $r_{\tilde{V}} > 0$. Recalling that $\pi_1(u_0) = e_n = (0, 1) \in \mathbb R^{n - 1} \times \mathbb R_{>0}$, we have $e_n, w \in S_{\mathrm{E}}^{n - 1}(c, r_{u_0})$ with $\|w\| < \epsilon_1$ and $S_{\mathrm{E}}^{n - 1}(c, r_{u_0}) \cap \mathbb R^{n - 1} = \tilde{V} = S_{\mathrm{E}}^{n - 2}(\tilde{c}, r_{\tilde{V}})$. Now, there is a point $k \in \tilde{V}$ closest to $0 \in \mathbb R^{n - 1}$, i.e., satisfying $d_{\mathrm{E}}(k, 0) = \inf_{k' \in \tilde{V}} d_{\mathrm{E}}(k', 0)$ and such a point must be on the ray emanating from the center $\tilde{c}$ passing through $0 \in \mathbb R^{n - 1}$. Thus, using ${r_{u_0}}^2 = d_{\mathrm{E}}(c, e_n)^2 = \|\tilde{c}\|^2 + (1 - c_n)^2$ gives
\begin{align}
d_{\mathrm{E}}(k, 0) &= |d_{\mathrm{E}}(\tilde{c}, k) - d_{\mathrm{E}}(\tilde{c}, 0)| = |r_{\tilde{V}} - \|\tilde{c}\|| = \left|\sqrt{{r_{u_0}}^2 - {c_n}^2} - \|\tilde{c}\|\right| \\
&= \left|\sqrt{\|\tilde{c}\|^2 + (1 - c_n)^2 - {c_n}^2} - \|\tilde{c}\|\right| \\
\label{eqn:DistanceToOrigin}
&= \left|\sqrt{\|\tilde{c}\|^2 + 1 - 2c_n} - \|\tilde{c}\|\right|.
\end{align}
Now, to bound the right hand side, we will use the constraint on the center $c$ that it must lie on the hyperplane which is the perpendicular bisector of the line segment from $w$ to $e_n$. This forces the center to be of the form
\begin{align*}
c = (\tilde{c}, c_n) = \left(\tilde{c}, \frac{1 + w_n}{2} + \frac{\langle \tilde{w}, \tilde{c} \rangle}{1 - w_n} - \frac{\|\tilde{w}\|^2}{2(1 - w_n)}\right).
\end{align*}
If $\|\tilde{c}\| \geq 1$, then
\begin{align*}
d_{\mathrm{E}}(k, 0) &= \frac{|1 - 2c_n|}{\sqrt{\|\tilde{c}\|^2 + 1 - 2c_n} + \|\tilde{c}\|} \leq \frac{|1 - 2c_n|}{\|\tilde{c}\|} \\
&= \left|\frac{w_n}{\|\tilde{c}\|} + \frac{2\langle \tilde{w}, \tilde{c} \rangle}{\|\tilde{c}\|(1 - w_n)} - \frac{\|\tilde{w}\|^2}{\|\tilde{c}\|(1 - w_n)}\right| \leq |w_n| + \frac{2\|\tilde{w}\|}{1 - w_n} + \frac{\|\tilde{w}\|^2}{1 - w_n} \\
&< \epsilon_1 + \frac{2\epsilon_1}{1 - \epsilon_1} + \frac{{\epsilon_1}^2}{1 - \epsilon_1} < \frac{\epsilon^2}{3} + \frac{\epsilon^2}{3} + \frac{\epsilon^2}{3} = \epsilon^2 < \epsilon
\end{align*}
using $\|\tilde{w}\| < \epsilon_1, w_n \in (0, \epsilon_1)$ and the definition of $\epsilon_1$. Now suppose $\|\tilde{c}\| < 1$. Going back to \cref{eqn:DistanceToOrigin}, if either $\|\tilde{c}\| = 0$ or $\|\tilde{c}\|^2 + 1 - 2c_n = 0$, then $d_{\mathrm{E}}(k, 0) = \sqrt{|1 - 2c_n|}$. Now assuming both $\|\tilde{c}\|^2 + 1 - 2c_n > 0$ and $\|\tilde{c}\| > 0$, we have
\begin{align*}
d_{\mathrm{E}}(k, 0) = \frac{|1 - 2c_n|}{\sqrt{\|\tilde{c}\|^2 + 1 - 2c_n} + \|\tilde{c}\|}.
\end{align*}
If $1 - 2c_n = 0$, then trivially $d_{\mathrm{E}}(k, 0) = 0$. If $\|\tilde{c}\|^2 > 2c_n - 1 > 0$, then
\begin{align*}
d_{\mathrm{E}}(k, 0) = \frac{|1 - 2c_n|}{\sqrt{\|\tilde{c}\|^2 + 1 - 2c_n} + \|\tilde{c}\|} < \frac{|1 - 2c_n|}{\|\tilde{c}\|} < \frac{|1 - 2c_n|}{\sqrt{2c_n - 1}} = \sqrt{|1 - 2c_n|}.
\end{align*}
If $1 - 2c_n > 0$, then
\begin{align*}
d_{\mathrm{E}}(k, 0) = \frac{|1 - 2c_n|}{\sqrt{\|\tilde{c}\|^2 + 1 - 2c_n} + \|\tilde{c}\|} < \frac{|1 - 2c_n|}{\sqrt{\|\tilde{c}\|^2 + 1 - 2c_n}} < \frac{|1 - 2c_n|}{\sqrt{1 - 2c_n}} = \sqrt{|1 - 2c_n|}.
\end{align*}
So in any case,
\begin{align*}
d_{\mathrm{E}}(k, 0)^2 \leq |1 - 2c_n| = \left|w_n + \frac{2\langle \tilde{w}, \tilde{c} \rangle}{1 - w_n} - \frac{\|\tilde{w}\|^2}{1 - w_n}\right| \leq |w_n| + \frac{2\|\tilde{w}\| \cdot \|\tilde{c}\|}{1 - w_n} + \frac{\|\tilde{w}\|^2}{1 - w_n}.
\end{align*}
Finally, using $\|\tilde{c}\| < 1, \|\tilde{w}\| < \epsilon_1$ and $w_n \in (0, \epsilon_1)$ gives us the bound
\begin{align*}
d_{\mathrm{E}}(k, 0) < \sqrt{\epsilon_1 + \frac{2\epsilon_1}{1 - \epsilon_1} + \frac{{\epsilon_1}^2}{1 - \epsilon_1}} < \sqrt{\frac{\epsilon^2}{3} + \frac{\epsilon^2}{3} + \frac{\epsilon^2}{3}} = \epsilon
\end{align*}
which proves the lemma in this case.

Now consider the case that $H^p(y, z)$ is nonelementary if $n > 2$, and $V \subset \mathbb H^n$ is a half affine hyperplane containing both $\pi_1(u_0) = e_n$ and $w = h\pi_1(u_0)$, with corresponding boundary at infinity $\partial_\infty(V) = \tilde{V}$ where $\tilde{V} \cap \mathbb R^{n - 1} \subset \mathbb R^{n - 1}$ is also an affine hyperplane. Any affine hyperplane containing $e_n$ and $w$ must also contain the straight line through $e_n$ and $w$. Hence $\tilde{V}$ contains the intersection of that straight line with $\mathbb R^{n - 1}$ which we call $k$. It is easily to calculate that the point is $k = \frac{1}{1 - w_n}\tilde{w}$ and hence
\begin{align*}
d_{\mathrm{E}}(k, 0) = \|k\| = \frac{\|\tilde{w}\|}{1 - w_n} < \frac{\epsilon_1}{1 - \epsilon_1} < \frac{\epsilon^2}{3} < \epsilon
\end{align*}
which proves the lemma in this case.
\end{proof}

\begin{proof}[Proof of \cref{lem:LimitSetInSphereNearOriginLargeP_NZ}]
Let $(y, z) \in \mathcal A^2$. Applying any appropriate isometry, we assume that the vectors in $\overline{U}_y$ have direction $\pi_2(\overline{U}_y) = -e_n$ and their basepoints lie on the hyperplane $\langle \pi_1(\overline{U}_y), e_n\rangle = 1$. Let $\epsilon > 0$. There are $C, \delta > 0$ such that for all $u_1, u_2 \in \overline{U}_y$ with $d_{\mathrm{su}}(u_1, u_2) < \delta$, we have $\frac{1}{C}d_{\mathrm{su}}(u_1, u_2) \leq d_{\mathrm{E}}(u_1^+, u_2^+) \leq Cd_{\mathrm{su}}(u_1, u_2)$. Let $\epsilon' = \min(\frac{\epsilon}{2C}, \delta)$. Fix $p_1, p_3 \in \mathbb Z_{>0}$ to be the constants provided by \cref{lem:P^p_in_R_z_ForLarge_p,lem:LimitSetInSphereNearOrigin} corresponding to $\epsilon'$ and $\frac{\epsilon}{2}$ respectively. Fix $p_0 = \max(p_1, p_3)$ and let $p > p_0$ be an integer such that $H^p(y, z)$ is not Zariski dense in $\mathbf{G}$. Let $u \in \overline{U}_y$. Using \cref{lem:P^p_in_R_z_ForLarge_p}, we obtain the vector $u_0 \in \hat{\overline{U}}_y$ such that $d_{\mathrm{su}}(u_0, u) < \epsilon'$, which implies $d_{\mathrm{E}}(u_0^+, u^+) < \frac{\epsilon}{2}$, and $v = \mathcal P^{p + 1}(u_0) \in \Gamma \hat{\overline{R}}_z$. We remark that translations parallel to $\mathbb R^{n - 1}$ are both hyperbolic and Euclidean isometries and so applying such a translation preserves all hyperbolic and Euclidean properties that we are dealing with. Hence, without loss of generality, we assume that $u_0 = (e_n, -e_n)$. Now the hypotheses of \cref{lem:LimitSetInSphereNearOrigin} are satisfied and we may apply it to conclude that $\Lambda(H^p(y, z))$ is contained in a $0$-dimensional or $(n - 2)$-dimensional sphere $\tilde{V} \subset \partial_\infty(\mathbb H^n)$ such that $\tilde{V} \cap B_{\epsilon/2}^{\mathrm{E}}(u_0^+) \neq \varnothing$ and hence $\tilde{V} \cap B_\epsilon^{\mathrm{E}}(u^+) \neq \varnothing$.
\end{proof}

\subsection{\texorpdfstring{$L^2$}{L-2}-flattening lemma}
\label{subsec:L2FlatteningLemma}
In this subsection we fix any $p > p_0$ from \cref{lem:Z-DenseInU-CoverG} so that the lemma applies when it is needed in \cref{lem:GV_Expander}. The aim of this subsection is to prove the following $L^2$-flattening type lemma. The arguments here are expanded on \cite[Appendix]{MOW17}.

\begin{lemma}
\label{lem:L2FlatteningLemma}
There exist $C > 0, C_0 > 0$ and $l \in \mathbb Z_{>0}$ such that for all $\xi = a + ib \in \mathbb C$ with $|a| < a_0'$, for all square free ideals $\mathfrak{q} \subset \mathcal{O}_{\mathbb K}$ coprime to $\mathfrak{q}_0$, for all $x \in \Sigma^+$, for all integers $C_0 \log(N_{\mathbb K}(\mathfrak{q})) \leq r < s$ with $r \in l\mathbb Z$, for all admissible sequences $(\alpha_s, \alpha_{s - 1}, \dotsc, \alpha_{r + 1})$, for all $\phi \in E_\mathfrak{q}^\mathfrak{q}$ with $\|\phi\|_2 = 1$, we have both
\begin{align*}
\left\|\mu_{(\alpha_s, \alpha_{s - 1}, \dotsc, \alpha_{r + 1})}^{\xi, \mathfrak{q}, x} * \phi\right\|_2 \leq C N_{\mathbb K}(\mathfrak{q})^{-\frac{1}{2}} \left\|\nu_{(\alpha_s, \alpha_{s - 1}, \dotsc, \alpha_{r + 1})}^{a, \mathfrak{q}, x}\right\|_1
\end{align*}
and
\begin{align*}
\left\|\hat{\mu}_{(\alpha_s, \alpha_{s - 1}, \dotsc, \alpha_{r + 1})}^{a, \mathfrak{q}, x} * \phi\right\|_2 \leq C N_{\mathbb K}(\mathfrak{q})^{-\frac{1}{2}} \left\|\nu_{(\alpha_s, \alpha_{s - 1}, \dotsc, \alpha_{r + 1})}^{a, \mathfrak{q}, x}\right\|_1.
\end{align*}
\end{lemma}

The proof uses two tools. The first is \cref{lem:ConvolutionBoundOnE_q^q} derived from lower bounds of nontrivial irreducible representations of Chevalley groups. The second is the expander machinery of Golsefidy-Varj\'{u} \cite{GV12} which we cannot use directly in its raw form but culminates in \cref{lem:ExpanderMachineryBound}. Due to \cref{lem:muHatLessThanCnu}, we focus on $\nu_0^{a, \mathfrak{q}, x, r}$ and our goal is to use \cite{GV12} to obtain bounds on the operator norm but this requires the measure to be what we call ``nearly flat''. Although $\nu_0^{a, \mathfrak{q}, x, r}$ is not nearly flat, it suffices to estimate $\nu_0^{a, \mathfrak{q}, x, r}$ by $\nu_1^{a, \mathfrak{q}, x, r}$ which \emph{breaks up} into convolutions of nearly flat measures. The following is the procedure to do exactly that.

Let $r \in \mathbb Z_{> 0}$ with factorization $r = r'l$ for some $r' \in \mathbb Z_{> 0}$ and some integer $l > 1$, and $\alpha^r$ be an admissible sequence. To better facilitate manipulations of sequences, we introduce the following additional notations. Define
\begin{align*}
\alpha_j^l &= (\alpha_{jl}, \alpha_{jl - 1}, \dotsc, \alpha_{(j - 1)l + 1}) \\
\alpha_j^{(k)_1} &= (\alpha_{jl}, \alpha_{jl - 1}, \dotsc, \alpha_{jl - k + 1}) \\
\alpha_j^{(k)_2} &= (\alpha_{jl - (l - k)}, \alpha_{jl - (l - k) - 1}, \dotsc, \alpha_{(j - 1)l + 1})
\end{align*}
for all integers $1 \leq j \leq r'$ and $1 \leq k \leq l - 1$. For example, with these notations and conventions we have $\alpha^r = (\alpha_{r'}^l, \alpha_{r' - 1}^l, \dotsc,  \alpha_1^l) = (\alpha_r, \alpha_{r - 1}, \dotsc, \alpha_1)$ and $\alpha_j^l = (\alpha_j^{(k)_1}, \alpha_j^{(l - k)_2})$ for all integers $1 \leq j \leq r'$. We also have $\sigma^k(\alpha^j) = \alpha^{j - k}$ for all integers $1 \leq j \leq r$ and $0 \leq k \leq j - 1$.

For all $a \in \mathbb R$, for all $x \in \Sigma^+$, for all $r \in \mathbb Z_{>0}$ with factorization $r = r'l$ with $l > p$, for all admissible sequences $\alpha^r$, we compute that
\begin{align*}
f_r^{(a)}(\alpha^r, x) ={}&\sum_{k = 0}^{r - 1} f^{(a)}(\sigma^k(\alpha^r, x)) = \sum_{k = 0}^{r - 1} f^{(a)}(\alpha^{r - k}, x) \\
={}&\sum_{k = 0}^{p - 1} f^{(a)}(\alpha^{r - j}, x) + \sum_{j = 0}^{r' - 3} \sum_{k = 0}^{l - 1} f^{(a)}(\alpha^{r - p - (jl + k)}, x) \\
{}&+ \sum_{k = 0}^{2l - p - 1} f^{(a)}(\alpha^{2l - p - k}, x) \\
={}&\sum_{k = 0}^{p - 1} f^{(a)}(\sigma^k(\alpha^r, x)) + \sum_{j = 0}^{r' - 3} \sum_{k = 0}^{l - 1} f^{(a)}(\sigma^k(\alpha^{r - p - jl}, x)) \\
{}&+ \sum_{k = 0}^{2l - p - 1} f^{(a)}(\sigma^k(\alpha^{2l - p}, x)) \\
={}&f_p^{(a)}(\alpha^r, x) + \sum_{j = 0}^{r' - 3} f_l^{(a)}(\alpha^{r - p - jl}, x) + f_{2l - p}^{(a)}(\alpha^{2l - p}, x) \\
={}&f_{2l - p}^{(a)}(\alpha^{2l - p}, x) + \sum_{j = 2}^{r' - 1} f_l^{(a)}(\alpha^{(j + 1)l - p}, x) + f_p^{(a)}(\alpha^r, x).
\end{align*}
We can estimate each term in the sum above so that in the $j$\textsuperscript{th} term, we remove dependence on $\alpha_k^{(p)_1}$ for all distinct integers $1 \leq j, k \leq r'$.

\begin{remark}
Such an estimate is not required for $j = 1$ since the first term does not have any dependence on $\alpha_k^{(p)_1}$ for all $2 \leq k \leq r'$.
\end{remark}

\begin{lemma}
\label{lem:Estimate_f_ToRemoveDependence}
There exists $C > 0$ such that for all $|a| < a_0'$, for all $x \in \Sigma^+$, for all $r \in \mathbb Z_{>0}$ with factorization $r = r'l$ with $l > p$, for all admissible sequences $\alpha^r$, we have
\begin{align*}
\left|f_l^{(a)}(\alpha^{(j + 1)l - p}, x) - f_l^{(a)}(\alpha_{j + 1}^{(l - p)_2}, \alpha_j^l, \omega(\alpha_j^l))\right| \leq C \theta^l
\end{align*}
for all integers $2 \leq j \leq r' - 1$ and
\begin{align*}
\left|f_p^{(a)}(\alpha^r, x) - f_p^{(a)}(\alpha_{r'}^l, \omega(\alpha_{r'}^l))\right| \leq C \theta^l.
\end{align*}
\end{lemma}

\begin{proof}
Fix $C = \frac{T_0 \theta^{1 - p}}{1 - \theta}$. Let $|a| < a_0'$, $x \in \Sigma^+$, $r \in \mathbb Z_{>0}$ with factorization $r = r'l$ with $l > p$, and $\alpha^r$ be an admissible sequence. We make the estimate
\begin{align*}
&\left|f_l^{(a)}(\alpha^{(j + 1)l - p}, x) - f_l^{(a)}(\alpha_{j + 1}^{(l - p)_2}, \alpha_j^l, \omega(\alpha_j^l))\right| \\
={}&\left|f_l^{(a)}(\alpha_{j + 1}^{(l - p)_2}, \alpha_j^l, \alpha_{j - 1}^{l}, \dotsc, \alpha_1^{l}, x) - f_l^{(a)}(\alpha_{j + 1}^{(l - p)_2}, \alpha_j^l, \omega(\alpha_j^l))\right| \\
\leq{}&\sum_{k = 0}^{l - 1} \left|f^{(a)}(\sigma^k(\alpha_{j + 1}^{(l - p)_2}, \alpha_j^l, \alpha_{j - 1}^{l}, \dotsc, \alpha_1^{l}, x)) - f^{(a)}(\sigma^k(\alpha_{j + 1}^{(l - p)_2}, \alpha_j^l, \omega(\alpha_j^l)))\right| \\
\leq{}&\sum_{k = 0}^{l - 1} \Lip_{d_\theta}(f^{(a)}) \cdot d_\theta(\sigma^k(\alpha_{j + 1}^{(l - p)_2}, \alpha_j^l, \alpha_{j - 1}^{l}, \dotsc, \alpha_1^{l}, x), \sigma^k(\alpha_{j + 1}^{(l - p)_2}, \alpha_j^l, \omega(\alpha_j^l))) \\
\leq{}&\Lip_{d_\theta}(f^{(a)}) \sum_{k = 0}^{l - 1} \theta^{2l - p - k} \\
\leq{}&\frac{T_0 \theta^{l - p + 1}}{1 - \theta} \\
={}&C \theta^l
\end{align*}
for all integers $2 \leq j \leq r' - 1$. Similarly
\begin{align*}
&\left|f_p^{(a)}(\alpha^r, x) - f_p^{(a)}(\alpha_{r'}^l, \omega(\alpha_{r'}^l))\right| \\
={}&\left|f_p^{(a)}(\alpha_{r'}^l, \alpha^{(r' - 1)l}, x) - f_p^{(a)}(\alpha_{r'}^l, \omega(\alpha_{r'}^l))\right| \\
\leq{}&\sum_{k = 0}^{p - 1} \left|f^{(a)}(\sigma^k(\alpha_{r'}^l, \alpha^{(r' - 1)l}, x)) - f^{(a)}(\sigma^k(\alpha_{r'}^l, \omega(\alpha_{r'}^l)))\right| \\
\leq{}&\sum_{k = 0}^{p - 1} \Lip_{d_\theta}(f^{(a)}) \cdot d_\theta(\sigma^k(\alpha_{r'}^l, \alpha^{(r' - 1)l}, x), \sigma^k(\alpha_{r'}^l, \omega(\alpha_{r'}^l))) \\
\leq{}&\Lip_{d_\theta}(f^{(a)}) \sum_{k = 0}^{p - 1} \theta^{l - k} \\
\leq{}&\frac{T_0 \theta^{l - p + 1}}{1 - \theta} \\
={}&C \theta^l.
\end{align*}
\end{proof}

To make sense of the notations in what follows, we make the convention that $\alpha_j^{(l - p)_2}$ is the empty sequence for all admissible sequences $\alpha^r$, for all $j \in \{0, r' + 1\}$. In light of the calculations and \cref{lem:Estimate_f_ToRemoveDependence} above, for all $a \in \mathbb R$, for all ideals $\mathfrak{q} \subset \mathcal{O}_{\mathbb K}$ coprime to $\mathfrak{q}_0$, for all $x \in \Sigma^+$, for all $r \in \mathbb Z_{>0}$ with factorization $r = r'l$ with $l > p$, for all integers $0 \leq j \leq r'$, for all admissible sequences $\alpha^r$, define the coefficients
\begin{align*}
E_{j, (\alpha_{j + 1}^{(l - p)_2}, \alpha_j^l)}^{a, x, r, r'} =
\begin{cases}
e^{f_{2l - p}^{(a)}(\alpha^{2l - p}, x)}, & j = 1 \\
e^{f_l^{(a)}(\alpha_{j + 1}^{(l - p)_2}, \alpha_j^l, \omega(\alpha_j^l))}, & 2 \leq j \leq r' - 1 \\
e^{f_p^{(a)}(\alpha_{r'}^l, \omega(\alpha_{r'}^l))}, & j = r'
\end{cases}
\end{align*}
and the measures
\begin{align*}
\eta_{j, (\alpha_{j + 1}^{(l - p)_2}, \alpha_j^{(l - p)_2})}^{a, \mathfrak{q}, x, r, r'} =
\begin{cases}
\delta_{\mathtt{c}_\mathfrak{q}(\alpha_1, x)}, & j = 0 \\
\sum_{\alpha_j^{(p)_1}} E_{j, (\alpha_{j + 1}^{(l - p)_2}, \alpha_j^l)}^{a, x, r, r'} \delta_{\mathtt{c}_\mathfrak{q}^l(\alpha_{jl + 1}, \alpha^{jl}, x)}, & 1 \leq j \leq r'
\end{cases}
\end{align*}
where we show the dependence of the admissible choices of $\alpha_j^{(p)_1}$ on $\alpha_{j + 1}^{(l - p)_2}$ and $\alpha_j^{(l - p)_2}$ (or more precisely only on the last entry of $\alpha_{j + 1}^{(l - p)_2}$ and the first entry of $\alpha_j^{(l - p)_2}$). The measures above satisfy a property as shown in \cref{lem:NearlyFlat} which we call \emph{nearly flat}.

\begin{lemma}
\label{lem:NearlyFlat}
There exists $C > 1$ such that for all $|a| < a_0'$, for all $x \in \Sigma^+$, for all $r \in \mathbb Z_{>0}$ with factorization $r = r'l$ with $l > p$, for all integers $0 \leq j \leq r'$, for all pairs of admissible sequences $(\alpha_{j + 1}^{(l - p)_2}, \alpha_j^l)$ and $(\tilde{\alpha}_{j + 1}^{(l - p)_2}, \tilde{\alpha}_j^l)$ with $(\alpha_{j + 1}^{(l - p)_2}, \alpha_j^{(l - p)_2}) = (\tilde{\alpha}_{j + 1}^{(l - p)_2}, \tilde{\alpha}_j^{(l - p)_2})$, we have
\begin{align*}
\frac{E_{j, (\tilde{\alpha}_{j + 1}^{(l - p)_2}, \tilde{\alpha}_j^l)}^{a, x, r, r'}}{E_{j, (\alpha_{j + 1}^{(l - p)_2}, \alpha_j^l)}^{a, x, r, r'}} \leq C.
\end{align*}
\end{lemma}

\begin{proof}
Fix $C = e^{T_0\left(\frac{\theta}{1 - \theta} + p\right)} > 1$. Let $|a| < a_0'$, $x \in \Sigma^+$, $r \in \mathbb Z_{>0}$ with factorization $r = r'l$ with $l > p$. Let $0 \leq j \leq r'$ be an integer and $(\alpha_{j + 1}^{(l - p)_2}, \alpha_j^l)$ and $(\tilde{\alpha}_{j + 1}^{(l - p)_2}, \tilde{\alpha}_j^l)$ be two pairs of admissible sequences with $(\alpha_{j + 1}^{(l - p)_2}, \alpha_j^{(l - p)_2}) = (\tilde{\alpha}_{j + 1}^{(l - p)_2}, \tilde{\alpha}_j^{(l - p)_2})$. We have the elementary calculation
\begin{align*}
&\left|f_{2l - p}^{(a)}(\alpha^{2l - p}, x) - f_{2l - p}^{(a)}(\tilde{\alpha}^{2l - p}, x)\right| \\
\leq{}&\sum_{k = 0}^{2l - p - 1} \left|f^{(a)}(\sigma^k(\alpha^{2l - p}, x)) - f^{(a)}(\sigma^k(\tilde{\alpha}^{2l - p}, x))\right| \\
\leq{}&\sum_{k = 0}^{l - p - 1} \Lip_{d_\theta}(f^{(a)}) \cdot d_\theta(\sigma^k(\alpha^{2l - p}, x), \sigma^k(\tilde{\alpha}^{2l - p}, x)) \\
{}&+ \sum_{k = l - p}^{l - 1} \Lip_{d_\theta}(f^{(a)}) \cdot d_\theta(\sigma^k(\alpha^{2l - p}, x), \sigma^k(\tilde{\alpha}^{2l - p}, x)) \\
{}&+ \sum_{k = l}^{2l - p - 1} \Lip_{d_\theta}(f^{(a)}) \cdot d_\theta(\sigma^k(\alpha^{2l - p}, x), \sigma^k(\tilde{\alpha}^{2l - p}, x)) \\
\leq{}&\Lip_{d_\theta}(f^{(a)}) \left(\sum_{k = 0}^{l - p - 1} \theta^{l - p - k} + \sum_{k = l - p}^{l - 1} 1 + 0\right) \\
\leq{}&T_0\left(\frac{\theta}{1 - \theta} + p\right) \\
={}&\log(C).
\end{align*}
Similarly
\begin{align*}
&\left|f_l^{(a)}(\alpha_{j + 1}^{(l - p)_2}, \alpha_j^l, \omega(\alpha_j^l)) - f_l^{(a)}(\tilde{\alpha}_{j + 1}^{(l - p)_2}, \tilde{\alpha}_j^l, \omega(\tilde{\alpha}_j^l))\right| \\
\leq{}&\sum_{k = 0}^{l - 1} \left|f^{(a)}(\sigma^k(\alpha_{j + 1}^{(l - p)_2}, \alpha_j^l, \omega(\alpha_j^l))) - f^{(a)}(\sigma^k(\tilde{\alpha}_{j + 1}^{(l - p)_2}, \tilde{\alpha}_j^l, \omega(\tilde{\alpha}_j^l)))\right| \\
\leq{}&\sum_{k = 0}^{l - p - 1} \Lip_{d_\theta}(f^{(a)}) \cdot d_\theta(\sigma^k(\alpha_{j + 1}^{(l - p)_2}, \alpha_j^l, \omega(\alpha_j^l)), \sigma^k(\tilde{\alpha}_{j + 1}^{(l - p)_2}, \tilde{\alpha}_j^l, \omega(\tilde{\alpha}_j^l))) \\
&{}+ \sum_{k = l - p}^{l - 1} \Lip_{d_\theta}(f^{(a)}) \cdot d_\theta(\sigma^k(\alpha_{j + 1}^{(l - p)_2}, \alpha_j^l, \omega(\alpha_j^l)), \sigma^k(\tilde{\alpha}_{j + 1}^{(l - p)_2}, \tilde{\alpha}_j^l, \omega(\tilde{\alpha}_j^l))) \\
=&\Lip_{d_\theta}(f^{(a)}) \left(\sum_{k = 0}^{l - p - 1} \theta^{l - p - k} + \sum_{k = l - p}^{l - 1} 1\right) \\
\leq{}&T_0\left(\frac{\theta}{1 - \theta} + p\right) \\
={}&\log(C).
\end{align*}
Similarly again
\begin{align*}
&\left|f_p^{(a)}(\alpha_{r'}^l, \omega(\alpha_{r'}^l)) - f_p^{(a)}(\tilde{\alpha}_{r'}^l, \omega(\tilde{\alpha}_{r'}^l))\right| \\
\leq{}&\sum_{k = 0}^{p - 1} \left|f^{(a)}(\sigma^k(\alpha_{r'}^l, \omega(\alpha_{r'}^l))) - f^{(a)}(\sigma^k(\tilde{\alpha}_{r'}^l, \omega(\tilde{\alpha}_{r'}^l)))\right| \\
\leq{}&\Lip_{d_\theta}(f^{(a)}) \sum_{k = 0}^{p - 1} d_\theta(\sigma^k(\alpha_{r'}^l, \omega(\alpha_{r'}^l)), \sigma^k(\tilde{\alpha}_{r'}^l, \omega(\tilde{\alpha}_{r'}^l))) \\
\leq{}&\Lip_{d_\theta}(f^{(a)}) \sum_{k = 0}^{p - 1} 1 \\
={}&T_0 p \\
<{}&\log(C).
\end{align*}
The lemma now follows from definitions.
\end{proof}

For all $a \in \mathbb R$, for all ideals $\mathfrak{q} \subset \mathcal{O}_{\mathbb K}$ coprime to $\mathfrak{q}_0$, for all $x \in \Sigma^+$, for all $r \in \mathbb Z_{>0}$ with factorization $r = r'l$ with $l > p$, we also define the measure
\begin{align*}
\nu_1^{a, \mathfrak{q}, x, r, r'} = \sum_{\alpha_1^{(l - p)_2}, \alpha_2^{(l - p)_2}, \dotsc, \alpha_{r'}^{(l - p)_2}} \mathop{\bigast}\limits_{j = 0}^{r'} \eta_{j, (\alpha_{j + 1}^{(l - p)_2}, \alpha_j^{(l - p)_2})}^{a, \mathfrak{q}, x, r, r'}
\end{align*}
which in particular consists of convolutions of nearly flat measures. \cref{lem:EstimateNu} shows that we can estimate $\nu_0^{a, \mathfrak{q}, x, r}$ with $\nu_1^{a, \mathfrak{q}, x, r, r'}$ up to a multiplicative constant depending on the factorization $r = r'l$ and vice versa.

\begin{lemma}
\label{lem:EstimateNu}
There exists $C > 0$ such that for all $|a| < a_0'$, for all ideals $\mathfrak{q} \subset \mathcal{O}_{\mathbb K}$ coprime to $\mathfrak{q}_0$, for all $x \in \Sigma^+$, for all $r \in \mathbb Z_{>0}$ with factorization $r = r'l$ with $l > p$, we have $\nu_0^{a, \mathfrak{q}, x, r} \leq e^{r' C\theta^l}\nu_1^{a, \mathfrak{q}, x, r, r'}$ and $\nu_1^{a, \mathfrak{q}, x, r, r'} \leq e^{r' C\theta^l} \nu_0^{a, \mathfrak{q}, x, r}$.
\end{lemma}

\begin{proof}
Fix $C > 0$ to be the one from \cref{lem:Estimate_f_ToRemoveDependence}. Let $|a| < a_0'$, $\mathfrak{q} \subset \mathcal{O}_{\mathbb K}$ be an ideal coprime to $\mathfrak{q}_0$, $x \in \Sigma^+$, and $r \in \mathbb Z_{>0}$ with factorization $r = r'l$ with $l > p$. Denote $\nu_0^{a, \mathfrak{q}, x, r}$ by $\nu_0$, $\nu_1^{a, \mathfrak{q}, x, r, r'}$ by $\nu_1$, $\eta_{j, (\alpha_{j + 1}^{(l - p)_2}, \alpha_j^{(l - p)_2})}^{a, \mathfrak{q}, x, r, r'}$ by $\eta_{j, (\alpha_{j + 1}^{(l - p)_2}, \alpha_j^{(l - p)_2})}$, and $E_{j, (\alpha_{j + 1}^{(l - p)_2}, \alpha_j^l)}^{a, x, r, r'}$ by $E_{j, (\alpha_{j + 1}^{(l - p)_2}, \alpha_j^l)}$. First we have
\begin{align*}
\nu_0 &= \sum_{\alpha^r} e^{f_r^{(a)}(\alpha^r, x)} \delta_{\mathtt{c}_\mathfrak{q}^{r + 1}(\alpha_{r + 1}, \alpha^r, x)} \\
&= \sum_{\alpha^r} e^{f_{2l - p}^{(a)}(\alpha^{2l - p}, x) + \sum_{j = 2}^{r' - 1} f_l^{(a)}(\alpha^{(j + 1)l - p}, x) + f_p^{(a)}(\alpha^r, x)} \delta_{\mathtt{c}_\mathfrak{q}^{r + 1}(\alpha_{r + 1}, \alpha^r, x)}.
\end{align*}
Now by construction, for fixed sequences $\alpha_1^{(l - p)_2}, \alpha_2^{(l - p)_2}, \dotsc, \alpha_{r'}^{(l - p)_2}$, the terms $E_{j, (\alpha_{j + 1}^{(l - p)_2}, \alpha_j^l)}$ and
\begin{align*}
\mathtt{c}_\mathfrak{q}^l(\alpha_{jl + 1}, \alpha^{jl}, x) = \mathtt{c}_\mathfrak{q}^l(\alpha_{jl + 1}, \alpha_j^{(p)_1}, \alpha_j^{(l - p)_2}, \omega(\alpha_j^{(l - p)_2}))
\end{align*}
depend only on the choice of $\alpha_j^{(p)_1}$ and not on $\alpha_k^{(p)_1}$ for all distinct integers $1 \leq j, k \leq r'$. We also note that $\mathtt{c}_\mathfrak{q}(\alpha_1, x)$ depends only on the choice of $\alpha_1^{(l - p)_2}$ since $l - p \geq 1$. So we can do the manipulations
\begin{align*}
\nu_1 ={}&\sum_{\alpha_1^{(l - p)_2}, \alpha_2^{(l - p)_2}, \dotsc, \alpha_{r'}^{(l - p)_2}} \mathop{\bigast}\limits_{j = 0}^{r'} \eta_{j, (\alpha_{j + 1}^{(l - p)_2}, \alpha_j^{(l - p)_2})} \\
={}&\sum_{\alpha_1^{(l - p)_2}, \alpha_2^{(l - p)_2}, \dotsc, \alpha_{r'}^{(l - p)_2}} \delta_{\mathtt{c}_\mathfrak{q}(\alpha_1, x)} * \left(\mathop{\bigast}\limits_{j = 1}^{r'} \sum_{\alpha_j^{(p)_1}} E_{j, (\alpha_{j + 1}^{(l - p)_2}, \alpha_j^l)} \delta_{\mathtt{c}_\mathfrak{q}^l(\alpha_{jl + 1}, \alpha^{jl}, x)}\right) \\
={}&\sum_{\alpha_1^{(l - p)_2}, \alpha_2^{(l - p)_2}, \dotsc, \alpha_{r'}^{(l - p)_2}} \delta_{\mathtt{c}_\mathfrak{q}(\alpha_1, x)} \\
{}&* \left(\sum_{\alpha_1^{(p)_1}, \alpha_2^{(p)_1}, \dotsc, \alpha_{r'}^{(p)_1}} \left(\mathop{\bigast}\limits_{j = 1}^{r'} E_{j, (\alpha_{j + 1}^{(l - p)_2}, \alpha_j^l)} \delta_{\mathtt{c}_\mathfrak{q}^l(\alpha_{jl + 1}, \alpha^{jl}, x)} \right)\right) \\
={}&\sum_{\alpha_1^{(l - p)_2}, \alpha_2^{(l - p)_2}, \dotsc, \alpha_{r'}^{(l - p)_2}} \left(\sum_{\alpha_1^{(p)_1}, \alpha_2^{(p)_1}, \dotsc, \alpha_{r'}^{(p)_1}} \left(\prod_{j = 1}^{r'} E_{j, (\alpha_{j + 1}^{(l - p)_2}, \alpha_j^l)} \right)\right. \\
{}&\left.\cdot \delta_{\mathtt{c}_\mathfrak{q}(\alpha_1, x)} * \delta_{\mathtt{c}_\mathfrak{q}^l(\alpha_{l + 1}, \alpha^{l}, x)} * \delta_{\mathtt{c}_\mathfrak{q}^l(\alpha_{2l + 1}, \alpha^{2l}, x)} * \dotsb * \delta_{\mathtt{c}_\mathfrak{q}^l(\alpha_{r + 1}, \alpha^{r'l}, x)} \rule{0cm}{0.8cm}\right) \\
={}&\sum_{\alpha^r} \left(\prod_{j = 1}^{r'} E_{j, (\alpha_{j + 1}^{(l - p)_2}, \alpha_j^l)}\right) \delta_{\mathtt{c}_\mathfrak{q}^{r + 1}(\alpha_{r + 1}, \alpha^r, x)}.
\end{align*}
Hence the lemma follows by comparing the above two expressions for $\nu_0$ and $\nu_1$ and using \cref{lem:Estimate_f_ToRemoveDependence}.
\end{proof}

For all ideals $\mathfrak{q} \subset \mathcal{O}_{\mathbb K}$ coprime to $\mathfrak{q}_0$, for all $x \in \Sigma^+$, for all $r \in \mathbb Z_{>0}$ with factorization $r = r'l$ with $l > p$, for all integers $0 \leq j \leq r'$, for all pairs of admissible sequences $(\alpha_{j + 1}^{(l - p)_2}, \alpha_j^l)$ and $(\tilde{\alpha}_{j + 1}^{(l - p)_2}, \tilde{\alpha}_j^l)$ with $(\alpha_{j + 1}^{(l - p)_2}, \alpha_j^{(l - p)_2}) = (\tilde{\alpha}_{j + 1}^{(l - p)_2}, \tilde{\alpha}_j^{(l - p)_2})$, we calculate that
\begin{align*}
\mathtt{c}_\mathfrak{q}^l(\alpha_{jl + 1}, \alpha^{jl}, x) &= \mathtt{c}_\mathfrak{q}(\alpha_{jl + 1}, \alpha_{jl}) \mathtt{c}_\mathfrak{q}(\alpha_{jl}, \alpha_{jl - 1}) \dotsb \mathtt{c}_\mathfrak{q}(\alpha_{(j - 1)l + 2}, \alpha_{(j - 1)l + 1}) \\
&= \prod_{k = 0}^{l - 1} \mathtt{c}_\mathfrak{q}(\alpha_{jl + 1 - k}, \alpha_{jl - k}) \\
&= \prod_{k = 0}^p \mathtt{c}_\mathfrak{q}(\alpha_{jl + 1 - k}, \alpha_{jl - k}) \prod_{k = p + 1}^{l - 1} \mathtt{c}_\mathfrak{q}(\alpha_{jl + 1 - k}, \alpha_{jl - k}).
\end{align*}
We write the product in the final form because
\begin{align*}
\mathtt{c}_\mathfrak{q}(\tilde{\alpha}_{jl + 1 - k}, \tilde{\alpha}_{jl - k}) = \mathtt{c}_\mathfrak{q}(\alpha_{jl + 1 - k}, \alpha_{jl - k})
\end{align*}
for all $p + 1 \leq k \leq l - 1$. We use this to obtain
\begin{align*}
&\mathtt{c}_\mathfrak{q}^l(\alpha_{jl + 1}, \alpha^{jl}, x) \mathtt{c}_\mathfrak{q}^l(\tilde{\alpha}_{jl + 1}, \tilde{\alpha}^{jl}, x)^{-1} \\
={}&\left(\prod_{k = 0}^p \mathtt{c}_\mathfrak{q}(\alpha_{jl + 1 - k}, \alpha_{jl - k}) \prod_{k = p + 1}^{l - 1} \mathtt{c}_\mathfrak{q}(\alpha_{jl + 1 - k}, \alpha_{jl - k})\right) \\
{}&\cdot \left(\prod_{k = 0}^p \mathtt{c}_\mathfrak{q}(\tilde{\alpha}_{jl + 1 - k}, \tilde{\alpha}_{jl - k}) \prod_{k = p + 1}^{l - 1} \mathtt{c}_\mathfrak{q}(\tilde{\alpha}_{jl + 1 - k}, \tilde{\alpha}_{jl - k})\right)^{-1} \\
={}&\prod_{k = 0}^p \mathtt{c}_\mathfrak{q}(\alpha_{jl + 1 - k}, \alpha_{jl - k}) \prod_{k = p + 1}^{l - 1} \mathtt{c}_\mathfrak{q}(\alpha_{jl + 1 - k}, \alpha_{jl - k}) \\
{}&\cdot \left(\prod_{k = p + 1}^{l - 1} \mathtt{c}_\mathfrak{q}(\alpha_{jl + 1 - k}, \alpha_{jl - k})\right)^{-1} \left(\prod_{k = 0}^p \mathtt{c}_\mathfrak{q}(\tilde{\alpha}_{jl + 1 - k}, \tilde{\alpha}_{jl - k})\right)^{-1} \\
={}&\prod_{k = 0}^p \mathtt{c}_\mathfrak{q}(\alpha_{jl + 1 - k}, \alpha_{jl - k}) \prod_{k = 0}^p \mathtt{c}_\mathfrak{q}(\tilde{\alpha}_{jl - p + 1 + k}, \tilde{\alpha}_{jl - p + k})^{-1}.
\end{align*}

Now by \cref{lem:Z-DenseInU-CoverG}, we can use the result of Golsefidy-Varj\'{u} \cite{GV12} to obtain the following lemma regarding spectral gap.

\begin{lemma}
\label{lem:GV_Expander}
There exists $\epsilon \in (0, 1)$ such that for all $r \in \mathbb Z_{>0}$ with factorization $r = r'l$ with $l > p$, for all integers $0 \leq j \leq r'$, for all pairs of admissible sequences $(\alpha_{j + 1}^{(l - p)_2}, \alpha_j^{(l - p)_2})$, for all square free ideals $\mathfrak{q} \subset \mathcal{O}_{\mathbb K}$ coprime to $\mathfrak{q}_0$, for all $\phi \in L_0^2(\tilde{G}_\mathfrak{q}, \mathbb C)$ with $\|\phi\|_2 = 1$, there exist admissible sequences $(\beta_{jl + 1}, \beta_j^{(p)_1}, \beta_{jl - p})$ and $(\tilde{\beta}_{jl + 1}, \tilde{\beta}_j^{(p)_1}, \tilde{\beta}_{jl - p})$ with $\beta_{jl + 1} = \tilde{\beta}_{jl + 1} = \alpha_{jl + 1}$ and $\beta_{jl - p} = \tilde{\beta}_{jl - p} = \alpha_{jl - p}$ such that
\begin{align*}
\|\delta_g * \phi - \phi\|_2 \geq \epsilon
\end{align*}
where $g = \prod_{k = 0}^p \mathtt{c}_\mathfrak{q}(\beta_{jl + 1 - k}, \beta_{jl - k}) \prod_{k = 0}^p \mathtt{c}_\mathfrak{q}(\tilde{\beta}_{jl - p + 1 + k}, \tilde{\beta}_{jl - p + k})^{-1}$.
\end{lemma}

\begin{proof}
Uniformity of $\epsilon$ with respect to $r \in \mathbb Z_{>0}$ with factorization $r = r'l$ with $l > p$, integers $0 \leq j \leq r'$, and pairs of admissible sequences $(\alpha_{j + 1}^{(l - p)_2}, \alpha_j^{(l - p)_2})$ is trivial since it only depends on the first entry $\alpha_{jl + 1} \in \mathcal A$ of $\alpha_{j + 1}^{(l - p)_2}$ and the last entry $\alpha_{jl - p} \in \mathcal A$ of $\alpha_j^{(l - p)_2}$ and there are only a finite number of such pairs. So let $r \in \mathbb Z_{>0}$ with factorization $r = r'l$ with $l > p$, $0 \leq j \leq r'$ be an integer and $(\alpha_{j + 1}^{(l - p)_2}, \alpha_j^{(l - p)_2})$ be a pair of admissible sequences. Let $\tilde{S}^p = \tilde{S}^p(\alpha_{jl + 1}, \alpha_{jl - p})$ and $\tilde{H}^p = \langle \tilde{S}^p \rangle = \tilde{H}^p(\alpha_{jl + 1}, \alpha_{jl - p})$. for all ideals $\mathfrak{q} \subset \mathcal{O}_{\mathbb K}$, let $\tilde{S}^p_\mathfrak{q} = \{e, -e\} \backslash \pi_\mathfrak{q}(\tilde{S}^p)$ and $\tilde{H}^p_\mathfrak{q} = \{e, -e\} \backslash \pi_\mathfrak{q}(\tilde{H}^p) = \langle\tilde{S}^p_\mathfrak{q}\rangle$. Recalling the strong approximation theorem, \cref{lem:Z-DenseInU-CoverG} implies that $\tilde{H}^p_\mathfrak{q} = \tilde{G}_\mathfrak{q}$ for all ideals $\mathfrak{q} \subset \mathcal{O}_{\mathbb K}$ coprime to $\mathfrak{q}_0$. Hence again by \cref{lem:Z-DenseInU-CoverG}, we can use \cite[Corollary 6]{GV12} to further conclude that the Cayley graphs $\Cay(\tilde{G}_\mathfrak{q}, \tilde{S}^p_\mathfrak{q}) = \Cay(\{e, -e\} \backslash \tilde{\mathbf{G}}(\mathcal{O}_\mathbb{K}/\mathfrak{q}), \{e, -e\} \backslash \pi_\mathfrak{q}(\tilde{S}^p))$ form a family of expanders with respect to square free ideals $\mathfrak{q} \subset \mathcal{O}_{\mathbb K}$ coprime to $\mathfrak{q}_0$. For all square free ideals $\mathfrak{q} \subset \mathcal{O}_{\mathbb K}$ coprime to $\mathfrak{q}_0$, let $A_\mathfrak{q}: L^2(\tilde{G}_\mathfrak{q}, \mathbb C) \to L^2(\tilde{G}_\mathfrak{q}, \mathbb C)$ be the adjacency operator defined by
\begin{align*}
A_\mathfrak{q}(\phi)(g) = \sum_{h \in \tilde{S}^p_\mathfrak{q}} \phi(gh) = \sum_{h \in \tilde{S}^p_\mathfrak{q}} \phi(gh^{-1}) = \sum_{h \in \tilde{S}^p_\mathfrak{q}} (\delta_h * \phi)(g)
\end{align*}
for all $g \in \tilde{G}_\mathfrak{q}$ which we rewrite as $A_\mathfrak{q}(\phi) = \sum_{h \in \tilde{S}^p_\mathfrak{q}} \delta_h * \phi$ and note that it is self-adjoint. Its largest eigenvalue is $\lambda_1(A_\mathfrak{q}) = |\tilde{S}^p_\mathfrak{q}| = k$ with corresponding eigenvectors being the constant functions. We can choose $\epsilon \in (0, 1)$ coming from the expander property so that for all square free ideals $\mathfrak{q} \subset \mathcal{O}_{\mathbb K}$ coprime to $\mathfrak{q}_0$, the next largest eigenvalue is $\lambda_2(A_\mathfrak{q}) \leq (1 - \epsilon)k$. This corresponds to the graph Laplacian $\Delta_\mathfrak{q}  = \Id_{L^2(\tilde{G}_\mathfrak{q}, \mathbb C)} - \frac{1}{k}A_\mathfrak{q}$ having smallest eigenvalue $\lambda_1(\Delta_\mathfrak{q}) = 0$ with corresponding eigenvectors being the constant functions and having the next smallest eigenvalue $\lambda_2(\Delta_\mathfrak{q}) \geq \epsilon$ for all square free ideals $\mathfrak{q} \subset \mathcal{O}_{\mathbb K}$ coprime to $\mathfrak{q}_0$. Thus for all square free ideals $\mathfrak{q} \subset \mathcal{O}_{\mathbb K}$ coprime to $\mathfrak{q}_0$, for all $\phi \in L_0^2(\tilde{G}_\mathfrak{q}, \mathbb C)$ with $\|\phi\|_2 = 1$, we have $\|\Delta_\mathfrak{q}(\phi)\|_2 = \left\|\frac{1}{k}A_\mathfrak{q}(\phi) - \phi\right\|_2 \geq \epsilon$ which implies
\begin{align*}
\sum_{h \in \tilde{S}^p_\mathfrak{q}} \|\delta_h * \phi - \phi\|_2 \geq \left\|\sum_{h \in \tilde{S}^p_\mathfrak{q}} (\delta_h * \phi - \phi)\right\|_2 \geq k\epsilon
\end{align*}
and so there is a $g \in \tilde{S}^p_\mathfrak{q}$ such that $\|\delta_g * \phi - \phi\|_2 \geq \epsilon$. But $\tilde{S}^p \subset \tilde{H}^p < \tilde{\Gamma}$ and recall the induced isomorphisms $\overline{\pi_{\mathfrak{q}}|_{\tilde{\Gamma}}}: \tilde{\Gamma}_\mathfrak{q} \backslash \tilde{\Gamma} \to \tilde{G}_\mathfrak{q}$ and $\overline{\tilde{\pi}}: \tilde{\Gamma}_\mathfrak{q} \backslash \tilde{\Gamma} \to \Gamma_\mathfrak{q} \backslash \Gamma$. Following these isomorphisms, we can see that $g$ is in fact of the form
\begin{align*}
g = \prod_{k = 0}^p \mathtt{c}_\mathfrak{q}(\beta_{jl + 1 - k}, \beta_{jl - k}) \prod_{k = 0}^p \mathtt{c}_\mathfrak{q}(\tilde{\beta}_{jl - p + 1 + k}, \tilde{\beta}_{jl - p + k})^{-1}.
\end{align*}
\end{proof}

\begin{remark}
Although in the above proof we only know a priori that the graphs $\Cay(\tilde{\mathbf{G}}(\mathcal{O}_\mathbb{K}/\mathfrak{q}), \pi_\mathfrak{q}(\tilde{S}^p))$ form a family of expanders with respect to square free ideals $\mathfrak{q} \subset \mathcal{O}_{\mathbb K}$ coprime to $\mathfrak{q}_0$, it is easy to see using the Cheeger constant formulation for a family of expanders that left quotients by $\{e, -e\}$ preserve this property.
\end{remark}

Let $\mathfrak{q} \subset \mathcal{O}_{\mathbb K}$ be an ideal coprime to $\mathfrak{q}_0$. For all measures $\eta$ on $\tilde{G}_\mathfrak{q}$, we define $\tilde{\eta}$ to be the operator $\tilde{\eta}: L^2(\tilde{G}_\mathfrak{q}, \mathbb C) \to L^2(\tilde{G}_\mathfrak{q}, \mathbb C)$ acting by the convolution $\tilde{\eta}(\phi) = \eta * \phi$ for all $\phi \in L^2(\tilde{G}_\mathfrak{q}, \mathbb C)$ and $\tilde{\eta}^*$ will be its adjoint operator as usual. We also define $\eta^*$ to be the measure on $\tilde{G}_\mathfrak{q}$ defined by $\eta^*(g) = \overline{\eta(g^{-1})}$ for all $g \in \tilde{G}_\mathfrak{q}$. It is then easy to see that these two operations commute, i.e., $\tilde{\eta}^* = \widetilde{\eta^*}$ or $\tilde{\eta}^*(\phi) = \eta^* * \phi$ for all $\phi \in L^2(\tilde{G}_\mathfrak{q}, \mathbb C)$ because
\begin{align*}
\langle \tilde{\eta}^*(\phi), \psi \rangle &= \langle \phi, \tilde{\eta}(\psi) \rangle = \sum_{g \in \tilde{G}_\mathfrak{q}} \phi(g) \overline{(\eta * \psi)(g)} = \sum_{g, h \in \tilde{G}_\mathfrak{q}} \phi(g) \overline{\eta(h^{-1})\psi(gh)} \\
&= \sum_{g, h \in \tilde{G}_\mathfrak{q}} \overline{\eta(h^{-1})} \phi(gh^{-1}) \overline{\psi(g)} = \sum_{g \in \tilde{G}_\mathfrak{q}} (\eta^* * \phi)(g) \overline{\psi(g)} \\
&= \langle \eta^* * \phi, \psi \rangle = \langle \widetilde{\eta^*}(\phi), \psi \rangle
\end{align*}
for all $\phi, \psi \in L^2(\tilde{G}_\mathfrak{q}, \mathbb C)$.

\begin{lemma}
\label{lem:EtaOperatorBound}
There exists $C \in (0, 1)$ such that for all $|a| < a_0'$, for all square free ideals $\mathfrak{q} \subset \mathcal{O}_{\mathbb K}$ coprime to $\mathfrak{q}_0$, for all $x \in \Sigma^+$, for all $r \in \mathbb Z_{>0}$ with factorization $r = r'l$ with $l > p$, for all integers $1 \leq j \leq r'$, for all pairs of admissible sequences $(\alpha_{j + 1}^{(l - p)_2}, \alpha_j^{(l - p)_2})$, for all $\phi \in L_0^2(\tilde{G}_\mathfrak{q}, \mathbb C)$ with $\|\phi\|_2 = 1$, we have
\begin{align*}
\left\|\eta_{j, (\alpha_{j + 1}^{(l - p)_2}, \alpha_j^{(l - p)_2})}^{a, \mathfrak{q}, x, r, r'} * \phi\right\|_2 \leq C \left\|\eta_{j, (\alpha_{j + 1}^{(l - p)_2}, \alpha_j^{(l - p)_2})}^{a, \mathfrak{q}, x, r, r'}\right\|_1.
\end{align*}
\end{lemma}

\begin{proof}
Fix $\epsilon \in (0, 1)$ to be the one from \cref{lem:GV_Expander} and $C_0 > 1$ to be the $C$ from \cref{lem:NearlyFlat}. Fix $C = \sqrt{1 - \frac{\epsilon^2}{2{C_0}^2N^{2p}}} \in (0, 1)$. Let $|a| < a_0'$, $\mathfrak{q} \subset \mathcal{O}_{\mathbb K}$ be a square free ideal coprime to $\mathfrak{q}_0$, $x \in \Sigma^+, r \in \mathbb Z_{>0}$ with factorization $r = r'l$ with $l > p$, $1 \leq j \leq r'$ be an integer, $(\alpha_{j + 1}^{(l - p)_2}, \alpha_j^{(l - p)_2})$ be a pair of admissible sequences and $\phi \in L_0^2(\tilde{G}_\mathfrak{q}, \mathbb C)$ with $\|\phi\|_2 = 1$. Denote $\eta_{j, (\alpha_{j + 1}^{(l - p)_2}, \alpha_j^{(l - p)_2})}^{a, \mathfrak{q}, x, r, r'}$ by $\eta$ and $E_{j, (\alpha_{j + 1}^{(l - p)_2}, \alpha_j^l)}^{a, x, r, r'}$ by $E_{\alpha_j^{(p)_1}}$ for all admissible sequences $(\alpha_{jl + 1}, \alpha_j^{(p)_1}, \alpha_{jl - p})$. Recalling that $E_{\alpha_j^{(p)_1}} \in \mathbb R$, we can define the measure
\begin{align*}
A &= \eta^* * \eta = \left(\sum_{\alpha_j^{(p)_1}} \overline{E_{\alpha_j^{(p)_1}}} \delta_{\mathtt{c}_\mathfrak{q}^l(\alpha_{jl + 1}, \alpha^{jl}, x)^{-1}}\right) * \left(\sum_{\alpha_j^{(p)_1}} E_{\alpha_j^{(p)_1}} \delta_{\mathtt{c}_\mathfrak{q}^l(\alpha_{jl + 1}, \alpha^{jl}, x)}\right) \\
&= \sum_{\alpha_j^{(p)_1}, \tilde{\alpha}_j^{(p)_1}} E_{\alpha_j^{(p)_1}} E_{\tilde{\alpha}_j^{(p)_1}} \delta_{\mathtt{c}_\mathfrak{q}^l(\tilde{\alpha}_{jl + 1}, \tilde{\alpha}^{jl}, x)^{-1}} * \delta_{\mathtt{c}_\mathfrak{q}^l(\alpha_{jl + 1}, \alpha^{jl}, x)} \\
&= \sum_{\alpha_j^{(p)_1}, \tilde{\alpha}_j^{(p)_1}} E_{\alpha_j^{(p)_1}} E_{\tilde{\alpha}_j^{(p)_1}} \delta_{\mathtt{c}_\mathfrak{q}^l(\alpha_{jl + 1}, \alpha^{jl}, x) \mathtt{c}_\mathfrak{q}^l(\tilde{\alpha}_{jl + 1}, \tilde{\alpha}^{jl}, x)^{-1}}
\end{align*}
where $(\tilde{\alpha}_{jl + 1}, \tilde{\alpha}_j^{(p)_1}, \tilde{\alpha}_j^{(l - p)_2}, \tilde{\alpha}^{(j - 1)l}) = (\alpha_{jl + 1}, \tilde{\alpha}_j^{(p)_1}, \alpha_j^{(l - p)_2}, \alpha^{(j - 1)l})$ henceforth. To begin estimating $\|\eta * \phi\|_2$, we first use definitions and properties mentioned above to get
\begin{align*}
0 \leq {\|\eta * \phi\|_2}^2 = \langle \eta * \phi, \eta * \phi \rangle = \langle \tilde{\eta}^*(\eta * \phi), \phi \rangle = \langle \eta^* * \eta * \phi, \phi \rangle = \langle\tilde{A}(\phi), \phi \rangle.
\end{align*}
The calculations also show that $\tilde{A} = \tilde{\eta}^* \tilde{\eta}$ is a self-adjoint positive semidefinite operator. Moreover, it suffices to show $\langle\tilde{A}(\phi), \phi \rangle \leq C^2\|A\|_1$ since
\begin{align*}
\|A\|_1 = \sum_{\alpha_j^{(p)_1}, \tilde{\alpha}_j^{(p)_1}} E_{\alpha_j^{(p)_1}} E_{\tilde{\alpha}_j^{(p)_1}} = \left(\sum_{\alpha_j^{(p)_1}} E_{\alpha_j^{(p)_1}}\right)^2 = {\|\eta\|_1}^2.
\end{align*}
Before we estimate $\langle\tilde{A}(\phi), \phi \rangle$, we first use \cref{lem:GV_Expander} to obtain a
\begin{align*}
g &= \mathtt{c}_\mathfrak{q}^l(\beta_{jl + 1}, \beta^{jl}, x) \mathtt{c}_\mathfrak{q}^l(\tilde{\beta}_{jl + 1}, \tilde{\beta}^{jl}, x)^{-1} \\
&= \prod_{k = 0}^p \mathtt{c}_\mathfrak{q}(\beta_{jl + 1 - k}, \beta_{jl - k}) \prod_{k = 0}^p \mathtt{c}_\mathfrak{q}(\tilde{\beta}_{jl - p + 1 + k}, \tilde{\beta}_{jl - p + k})^{-1}
\end{align*}
for some admissible sequences $(\beta_{jl + 1}, \beta_j^{(p)_1}, \beta_{jl - p})$ and $(\tilde{\beta}_{jl + 1}, \tilde{\beta}_j^{(p)_1}, \tilde{\beta}_{jl - p})$ with $\beta_{jl + 1} = \tilde{\beta}_{jl + 1} = \alpha_{jl + 1}$ and $\beta_{jl - p} = \tilde{\beta}_{jl - p} = \alpha_{jl - p}$ such that $\|\delta_g * \phi - \phi\|_2 \geq \epsilon$. Expanding the norm, we have ${\|\delta_g * \phi\|_2}^2 - 2\Re \langle \delta_g * \phi, \phi\rangle + {\|\phi\|_2}^2 \geq \epsilon^2$. Thus
\begin{align*}
\Re \langle \delta_g * \phi, \phi\rangle \leq 1 - \frac{\epsilon^2}{2} \in (0, 1)
\end{align*}
using the fact that $\|\delta_g * \phi\|_2 = \|\phi\|_2 = 1$. Now we use this inequality to begin estimating $\langle\tilde{A}(\phi), \phi \rangle$. Since $\tilde{A}$ is self-adjoint, we have
\begin{align*}
&\langle \tilde{A}(\phi), \phi\rangle = \Re \langle \tilde{A}(\phi), \phi\rangle \\
={}& \Re \left\langle \sum_{\alpha_j^{(p)_1}, \tilde{\alpha}_j^{(p)_1}} E_{\alpha_j^{(p)_1}} E_{\tilde{\alpha}_j^{(p)_1}} \delta_{\mathtt{c}_\mathfrak{q}^l(\alpha_{jl + 1}, \alpha^{jl}, x) \mathtt{c}_\mathfrak{q}^l(\tilde{\alpha}_{jl + 1}, \tilde{\alpha}^{jl}, x)^{-1}} * \phi, \phi\right\rangle \\
={}&\sum_{\alpha_j^{(p)_1}, \tilde{\alpha}_j^{(p)_1}} E_{\alpha_j^{(p)_1}} E_{\tilde{\alpha}_j^{(p)_1}} \Re \langle \delta_{\mathtt{c}_\mathfrak{q}^l(\alpha_{jl + 1}, \alpha^{jl}, x) \mathtt{c}_\mathfrak{q}^l(\tilde{\alpha}_{jl + 1}, \tilde{\alpha}^{jl}, x)^{-1}} * \phi, \phi\rangle \\
\leq{}&E_{\beta_j^{(p)_1}} E_{\tilde{\beta}_j^{(p)_1}} \Re \langle \delta_g * \phi, \phi\rangle \\
{}&+ \sum_{(\alpha_j^{(p)_1}, \tilde{\alpha}_j^{(p)_1}) \neq (\beta_j^{(p)_1}, \tilde{\beta}_j^{(p)_1})} E_{\alpha_j^{(p)_1}} E_{\tilde{\alpha}_j^{(p)_1}} |\langle \delta_{\mathtt{c}_\mathfrak{q}^l(\alpha_{jl + 1}, \alpha^{jl}, x) \mathtt{c}_\mathfrak{q}^l(\tilde{\alpha}_{jl + 1}, \tilde{\alpha}^{jl}, x)^{-1}} * \phi, \phi\rangle| \\
\leq{}&\left(1 - \frac{\epsilon^2}{2}\right) E_{\beta_j^{(p)_1}} E_{\tilde{\beta}_j^{(p)_1}} \\
{}&+ \sum_{(\alpha_j^{(p)_1}, \tilde{\alpha}_j^{(p)_1}) \neq (\beta_j^{(p)_1}, \tilde{\beta}_j^{(p)_1})} E_{\alpha_j^{(p)_1}} E_{\tilde{\alpha}_j^{(p)_1}} \|\delta_{\mathtt{c}_\mathfrak{q}^l(\alpha_{jl + 1}, \alpha^{jl}, x) \mathtt{c}_\mathfrak{q}^l(\tilde{\alpha}_{jl + 1}, \tilde{\alpha}^{jl}, x)^{-1}} * \phi\|_2 \|\phi\|_2 \\
={}&\left(1 - \frac{\epsilon^2}{2}\right) E_{\beta_j^{(p)_1}} E_{\tilde{\beta}_j^{(p)_1}} + \sum_{(\alpha_j^{(p)_1}, \tilde{\alpha}_j^{(p)_1}) \neq (\beta_j^{(p)_1}, \tilde{\beta}_j^{(p)_1})} E_{\alpha_j^{(p)_1}} E_{\tilde{\alpha}_j^{(p)_1}} \\
={}&\sum_{\alpha_j^{(p)_1}, \tilde{\alpha}_j^{(p)_1}} E_{\alpha_j^{(p)_1}} E_{\tilde{\alpha}_j^{(p)_1}} - \frac{\epsilon^2}{2} E_{\beta_j^{(p)_1}} E_{\tilde{\beta}_j^{(p)_1}} \\
={}&\|A\|_1 - \frac{\epsilon^2}{2} E_{\beta_j^{(p)_1}} E_{\tilde{\beta}_j^{(p)_1}}.
\end{align*}
Finally, \cref{lem:NearlyFlat} gives
\begin{align*}
\|A\|_1 = \sum_{\alpha_j^{(p)_1}, \tilde{\alpha}_j^{(p)_1}} E_{\alpha_j^{(p)_1}} E_{\tilde{\alpha}_j^{(p)_1}} \leq {C_0}^2\sum_{\alpha_j^{(p)_1}, \tilde{\alpha}_j^{(p)_1}} E_{\beta_j^{(p)_1}} E_{\tilde{\beta}_j^{(p)_1}} \leq {C_0}^2 N^{2p}E_{\beta_j^{(p)_1}} E_{\tilde{\beta}_j^{(p)_1}}
\end{align*}
and thus
\begin{align*}
\langle \tilde{A}(\phi), \phi\rangle \leq \left(1 - \frac{\epsilon^2}{2{C_0}^2N^{2p}}\right)\|A\|_1 = C^2\|A\|_1.
\end{align*}
\end{proof}

\begin{lemma}
\label{lem:ExpanderMachineryBound}
There exists $l_0 \in \mathbb Z_{>0}$ such that for all integers $l > l_0$, there exists $C \in (0, 1)$ such that for all $|a| < a_0'$, for all square free ideals $\mathfrak{q} \subset \mathcal{O}_{\mathbb K}$ coprime to $\mathfrak{q}_0$, for all $x \in \Sigma^+$, for all $r \in \mathbb Z_{>0}$ with factorization $r = r'l$, for all admissible sequences $(\alpha_s, \alpha_{s - 1}, \dotsc, \alpha_{r + 1})$, for all $\phi \in L_0^2(\tilde{G}_\mathfrak{q}, \mathbb C)$ with $\|\phi\|_2 = 1$, we have
\begin{align*}
\left\|\nu_{(\alpha_s, \alpha_{s - 1}, \dotsc, \alpha_{r + 1})}^{a, \mathfrak{q}, x} * \phi\right\|_2 \leq C^r \left\|\nu_{(\alpha_s, \alpha_{s - 1}, \dotsc, \alpha_{r + 1})}^{a, \mathfrak{q}, x}\right\|_1.
\end{align*}
\end{lemma}

\begin{proof}
Fix $C_1 > 0$ to be the $C$ from \cref{lem:EstimateNu} and $C_2 \in (0, 1)$ to be the $C$ from \cref{lem:EtaOperatorBound}. Fix $C_3 = -\log(C_2) > 0$. There is an integer $l_0 \geq p$ such that $2C_1 \theta^l - C_3 < 0$ for all integers $l > l_0$. Fix an integer $l > l_0$ and $C = e^{\frac{1}{l}(2C_1 \theta^l - C_3)} \in (0, 1)$. Let $|a| < a_0'$, $\mathfrak{q} \subset \mathcal{O}_{\mathbb K}$ be a square free ideal coprime to $\mathfrak{q}_0$, $x \in \Sigma^+$, $r \in \mathbb Z_{>0}$ with factorization $r = r'l$, $(\alpha_s, \alpha_{s - 1}, \dotsc, \alpha_{r + 1})$ be an admissible sequence and $\phi \in L_0^2(\tilde{G}_\mathfrak{q}, \mathbb C)$ with $\|\phi\|_2 = 1$. Denote $\nu_{(\alpha_s, \alpha_{s - 1}, \dotsc, \alpha_{r + 1})}^{a, \mathfrak{q}, x}$ by $\nu$, $\nu_0^{a, \mathfrak{q}, x, r}$ by $\nu_0$, $\nu_1^{a, \mathfrak{q}, x, r, r'}$ by $\nu_1$, $\eta_{j, (\alpha_{j + 1}^{(l - p)_2}, \alpha_j^{(l - p)_2})}^{a, \mathfrak{q}, x, r, r'}$ by $\eta_{j, (\alpha_{j + 1}^{(l - p)_2}, \alpha_j^{(l - p)_2})}$, and $E_{j, (\alpha_{j + 1}^{(l - p)_2}, \alpha_j^l)}^{a, x, r, r'}$ by $E_{j, (\alpha_{j + 1}^{(l - p)_2}, \alpha_j^l)}$. Using \cref{lem:EstimateNu} and then \cref{lem:EtaOperatorBound} repeatedly $r'$ times, we have
\begin{align*}
&\|\nu_0 * \phi\|_2 \\
\leq{}&e^{r'C_1 \theta^l}\|\nu_1 * \phi\|_2 \leq e^{r'C_1 \theta^l}\sum_{\alpha_1^{(l - p)_2}, \alpha_2^{(l - p)_2}, \dotsc, \alpha_{r'}^{(l - p)_2}} \left\|\mathop{\bigast}\limits_{j = 0}^{r'} \eta_{j, (\alpha_{j + 1}^{(l - p)_2}, \alpha_j^{(l - p)_2})} * \phi \right\|_2 \\
\leq{}&e^{r'C_1 \theta^l} {C_2}^{r'}\sum_{\alpha_1^{(l - p)_2}, \alpha_2^{(l - p)_2}, \dotsc, \alpha_{r'}^{(l - p)_2}} \prod_{j = 1}^{r'} \left\|\eta_{j, (\alpha_{j + 1}^{(l - p)_2}, \alpha_j^{(l - p)_2})}\right\|_1 \\
={}&\left(e^{C_1 \theta^l - C_3}\right)^{r'} \sum_{\alpha_1^{(l - p)_2}, \alpha_2^{(l - p)_2}, \dotsc, \alpha_{r'}^{(l - p)_2}} \prod_{j = 1}^{r'} \sum_{\alpha_j^{(p)_1}} E_{j, (\alpha_{j + 1}^{(l - p)_2}, \alpha_j^l)}.
\end{align*}
Note that in the above calculations $\eta_{0, (\alpha_1^{(l - p)_2}, \alpha_0^{(l - p)_2})} = \delta_{\mathtt{c}_\mathfrak{q}(\alpha_1, x)}$ which preserves the norm when taking convolutions. As mentioned earlier, for fixed sequences $\alpha_1^{(l - p)_2}, \alpha_2^{(l - p)_2}, \dotsc, \alpha_{r'}^{(l - p)_2}$, the term $E_{j, (\alpha_{j + 1}^{(l - p)_2}, \alpha_j^l)}$ depends only on the choice of $\alpha_j^{(p)_1}$ and not on $\alpha_k^{(p)_1}$ for all distinct integers $1 \leq j, k \leq r'$. Hence we can commute the inner sum and product to get
\begin{align*}
&\|\nu_0 * \phi\|_2 \\
\leq{}&\left(e^{\frac{1}{l}(C_1 \theta^l - C_3)}\right)^r \sum_{\alpha_1^{(l - p)_2}, \alpha_2^{(l - p)_2}, \dotsc, \alpha_{r'}^{(l - p)_2}} \left(\sum_{\alpha_1^{(p)_1}, \alpha_2^{(p)_1}, \dotsc, \alpha_{r'}^{(p)_1}} \prod_{j = 1}^{r'} E_{j, (\alpha_{j + 1}^{(l - p)_2}, \alpha_j^l)}\right) \\
={}&\left(e^{\frac{1}{l}(C_1 \theta^l - C_3)}\right)^r \sum_{\alpha^r} \prod_{j = 1}^{r'} E_{j, (\alpha_{j + 1}^{(l - p)_2}, \alpha_j^l)}.
\end{align*}
Recognizing this sum to be $\|\nu_1\|_1$, we use \cref{lem:EstimateNu} once again to get
\begin{align*}
\|\nu_0 * \phi\|_2 \leq \left(e^{\frac{1}{l}(C_1 \theta^l - C_3)}\right)^r \|\nu_1\|_1 \leq \left(e^{\frac{1}{l}(2C_1 \theta^l - C_3)}\right)^r \|\nu_0\|_1 = C^r \|\nu_0\|_1.
\end{align*}
Since $e^{f_{s - r}^{(a)}(\alpha_s, \alpha_{s - 1}, \dotsc, \alpha_{r + 1}, \omega(\alpha_{r + 1}))} > 0$, it follows that $\|\nu * \phi\|_2 \leq C^r \|\nu\|_1$.
\end{proof}

Let $\mathfrak{q} \subset \mathcal{O}_{\mathbb K}$ be a square free ideal coprime to $\mathfrak{q}_0$ and $\mu$ be a complex measure on $\tilde{G}_\mathfrak{q}$. We note that $E_\mathfrak{q}^\mathfrak{q}$ is a $\tilde{\mu}$-invariant submodule of the left $\mathbb C[\tilde{G}_\mathfrak{q}]$-module $L^2(\tilde{G}_\mathfrak{q}, \mathbb C)$. Let $\dot{\mu}$ denote the measure on $\tilde{\mathbf{G}}(\mathcal{O}_{\mathbb K}/\mathfrak{q})$ defined by $\dot{\mu}(g) = \mu(\{e, -e\}g)$ for all $g \in \tilde{\mathbf{G}}(\mathcal{O}_{\mathbb K}/\mathfrak{q})$ and $\tilde{\dot{\mu}}: L^2(\tilde{\mathbf{G}}(\mathcal{O}_{\mathbb K}/\mathfrak{q}), \mathbb C) \to L^2(\tilde{\mathbf{G}}(\mathcal{O}_{\mathbb K}/\mathfrak{q}), \mathbb C)$ denote the operator acting by convolution as before. Similar to above, $\dot{E}_\mathfrak{q}^\mathfrak{q}$ is a $\tilde{\dot{\mu}}$-invariant submodule of the left $\mathbb C[\tilde{\mathbf{G}}(\mathcal{O}_{\mathbb K}/\mathfrak{q})]$-module $L^2(\tilde{\mathbf{G}}(\mathcal{O}_{\mathbb K}/\mathfrak{q}), \mathbb C)$.

%For sufficiently large q! Add this later. This is from the irrep bound from Seitz. Our approximation is for large q. Also may be sufficient to just use q^{-1/2} for all n.
\begin{lemma}
\label{lem:ConvolutionBoundOnE_q^q}
There exists $C > 0$ such that for all square free ideals $\mathfrak{q} \subset \mathcal{O}_{\mathbb K}$ coprime to $\mathfrak{q}_0$, for all complex measures $\mu$ on $\tilde{G}_\mathfrak{q}$, we have
\begin{align*}
\|\tilde{\mu}|_{E_\mathfrak{q}^\mathfrak{q}}\|_{\mathrm{op}} \leq C N_{\mathbb K}(\mathfrak{q})^{-\frac{1}{2}} (\#\tilde{G}_\mathfrak{q})^{\frac{1}{2}} \|\mu\|_2.
\end{align*}
\end{lemma}

\begin{proof}
Let $\mathfrak{q} \subset \mathcal{O}_{\mathbb K}$ be a square free ideal coprime to $\mathfrak{q}_0$ and $\mu$ be a complex measure on $\tilde{G}_\mathfrak{q}$. It will be fruitful to first work with $\tilde{\dot{\mu}}$. One way to calculate the operator norm is using the equation
\begin{align*}
\|\tilde{\dot{\mu}}|_{\dot{E}_\mathfrak{q}^\mathfrak{q}}\|_{\mathrm{op}} = \max_{\lambda \in \Lambda(\tilde{\dot{\mu}}^*\tilde{\dot{\mu}}|_{\dot{E}_\mathfrak{q}^\mathfrak{q}})} \sqrt{\lambda}
\end{align*}
where $\Lambda(\tilde{\dot{\mu}}^*\tilde{\dot{\mu}}|_{\dot{E}_\mathfrak{q}^\mathfrak{q}})$ is the set of eigenvalues of the self-adjoint positive semidefinite operator $\tilde{\dot{\mu}}^*\tilde{\dot{\mu}}|_{\dot{E}_\mathfrak{q}^\mathfrak{q}}$ which is diagonalizable with nonnegative eigenvalues. Since $\tilde{\dot{\mu}}^*\tilde{\dot{\mu}}|_{\dot{E}_\mathfrak{q}^\mathfrak{q}}: \dot{E}_\mathfrak{q}^\mathfrak{q} \to \dot{E}_\mathfrak{q}^\mathfrak{q}$ is a left $\mathbb C[\tilde{\mathbf{G}}(\mathcal{O}_{\mathbb K}/\mathfrak{q})]$-module homomorphism, its eigenspaces are submodules of $\dot{E}_\mathfrak{q}^\mathfrak{q}$ and hence must contain at least one irreducible submodule $V$ which corresponds to an irreducible representation $\rho: \tilde{\mathbf{G}}(\mathcal{O}_{\mathbb K}/\mathfrak{q}) \to \GL(V)$. Now suppose we have the prime ideal factorization $\mathfrak{q} = \prod_{j = 1}^k \mathfrak{p}_j$ for some $k \in \mathbb Z_{>0}$ and for some prime ideals $\mathfrak{p}_1, \mathfrak{p}_2, \dotsc, \mathfrak{p}_k \subset \mathcal{O}_{\mathbb K}$. Then the Chinese remainder theorem gives $\mathcal{O}_{\mathbb K}/\mathfrak{q} \cong \prod_{j = 1}^k \mathcal{O}_{\mathbb K}/\mathfrak{p}_j \cong \prod_{j = 1}^k \mathbb F_{N_{\mathbb K}(\mathfrak{p}_j)}$ and hence $\tilde{\mathbf{G}}(\mathcal{O}_{\mathbb K}/\mathfrak{q}) \cong \prod_{j = 1}^k \tilde{\mathbf{G}}(\mathcal{O}_{\mathbb K}/\mathfrak{p}_j) \cong \prod_{j = 1}^k \tilde{\mathbf{G}}(\mathbb F_{N_{\mathbb K}(\mathfrak{p}_j)})$. Thus we have $\rho = \bigotimes_{j = 1}^k \rho_j$ where $\rho_j: \tilde{\mathbf{G}}(\mathbb F_{N_{\mathbb K}(\mathfrak{p}_j)}) \to \GL(V_j)$ for some complex vector space $V_j$ is an irreducible representation for all integers $1 \leq j \leq k$. The significance of using $\dot{E}_\mathfrak{q}^\mathfrak{q}$ is that its definition and $V \subset \dot{E}_\mathfrak{q}^\mathfrak{q}$ forces $\rho_j$ to be nontrivial for all integers $1 \leq j \leq k$. Now we use bounds directly from \cite{KS13}. However, we note that the bounds originate from \cite{Lan72,LS74,SZ93} which gives lower bounds on the degrees of nontrivial irreducible projective representations of Chevalley groups. By Schur's lemma, projective irreducible representations correspond to irreducible representations of central group extensions as in our case.
%In our case the extension from $\PSO(n + 1, \mathbb F_{{p_j}^{w_j}})$ to $\tilde{G}_{{p_j}^{w_j}}$ is indeed central since the kernel of the cover $\tilde{\mathbf G} \to \mathbf G = \SO(n + 1, \mathbb C)$ is the fundamental group of $G$ which is abelian for $n > 6$. Otherwise use exceptional Lie isomorphisms (see Wikipedia end of page of Spin groups), to reduce to $\PSL_2({p_j}^{w_j})$ case or $\PSU_4({p_j}^{w_j})$. Then follow as in Kelmer and Silberman.
From the proof of \cite[Proposition 4.2]{KS13}, there is a $C_1 > 0$ independent of anything such that $\deg(\rho_j) = \dim(V_j) > C_1 N_{\mathbb K}(\mathfrak{p}_j)^{n - 2} > C_1 N_{\mathbb K}(\mathfrak{p}_j)$ for all integers $1 \leq j \leq k$ if $n \geq 6$ and $\deg(\rho_j) = \dim(V_j) > C_1 N_{\mathbb K}(\mathfrak{p}_j)$ for all integers $1 \leq j \leq k$ if $n < 6$. In any case, we conclude that $\deg(\rho) = \dim(V) > C_1 N_{\mathbb K}(\mathfrak{q})$. Thus for all $\lambda \in \Lambda(\tilde{\dot{\mu}}^*\tilde{\dot{\mu}}|_{\dot{E}_\mathfrak{q}^\mathfrak{q}})$, we have
\begin{align*}
C_1 N_{\mathbb K}(\mathfrak{q})\lambda &\leq \tr(\tilde{\dot{\mu}}^*\tilde{\dot{\mu}}) = \sum_{g \in \tilde{\mathbf{G}}(\mathcal{O}_{\mathbb K}/\mathfrak{q})} \langle \tilde{\dot{\mu}}^*\tilde{\dot{\mu}}(\delta_g), \delta_g \rangle = \sum_{g \in \tilde{\mathbf{G}}(\mathcal{O}_{\mathbb K}/\mathfrak{q})} {\|\dot{\mu} * \delta_g\|_2}^2 \\
&= \#\tilde{\mathbf{G}}(\mathcal{O}_{\mathbb K}/\mathfrak{q}) \cdot {\|\dot{\mu}\|_2}^2.
\end{align*}
Hence,
\begin{align*}
{\|\tilde{\dot{\mu}}|_{\dot{E}_\mathfrak{q}^\mathfrak{q}}\|_{\mathrm{op}}}^2 = \max_{\lambda \in \Lambda(\tilde{\dot{\mu}}^*\tilde{\dot{\mu}}|_{\dot{E}_\mathfrak{q}^\mathfrak{q}})} \lambda \leq {C_1}^{-1} N_{\mathbb K}(\mathfrak{q})^{-1} \#\tilde{\mathbf{G}}(\mathcal{O}_{\mathbb K}/\mathfrak{q}) \cdot {\|\dot{\mu}\|_2}^2.
\end{align*}
Now we convert this to a bound for $\|\tilde{\mu}|_{E_\mathfrak{q}^\mathfrak{q}}\|_{\mathrm{op}}$. Let $\phi \in E_\mathfrak{q}^\mathfrak{q}$. Then $\dot{\phi} \in \dot{E}_\mathfrak{q}^\mathfrak{q}$ and also $\tilde{\dot{\mu}}(\dot{\phi}) = 2\dot{\psi}$ if $e \pmod{\mathfrak{q}} \neq -e \pmod{\mathfrak{q}}$ in $\tilde{\mathbf{G}}(\mathcal{O}_{\mathbb K}/\mathfrak{q})$ and $\tilde{\dot{\mu}}(\dot{\phi}) = \dot{\psi}$ otherwise, where $\psi = \tilde{\mu}(\phi)$. In any case, $\|\tilde{\mu}(\phi)\|_2 = \|\psi\|_2 \leq \|\dot{\psi}\|_2 \leq \|\tilde{\dot{\mu}}(\dot{\phi})\|_2$. Being careful of similar cases, the above bound gives
\begin{align*}
&{\|\tilde{\dot{\mu}}(\dot{\phi})\|_2}^2 \leq {C_1}^{-1} N_{\mathbb K}(\mathfrak{q})^{-1} \#\tilde{\mathbf{G}}(\mathcal{O}_{\mathbb K}/\mathfrak{q}) \cdot {\|\dot{\mu}\|_2}^2  {\|\dot{\phi}\|_2}^2 \\
\implies{}&{\|\tilde{\mu}(\phi)\|_2}^2 \leq {C_1}^{-1} N_{\mathbb K}(\mathfrak{q})^{-1} \cdot 2\#\tilde{G}_\mathfrak{q} \cdot 2{\|\mu\|_2}^2 \cdot 2{\|\phi\|_2}^2 \\
\implies{}&\|\tilde{\mu}(\phi)\|_2 \leq C N_{\mathbb K}(\mathfrak{q})^{-\frac{1}{2}} (\#\tilde{G}_\mathfrak{q})^{\frac{1}{2}} \|\mu\|_2 \|\phi\|_2
\end{align*}
where $C = \sqrt{\frac{8}{C_1}}$ which is independent of $\mathfrak{q}$ and $\mu$. Hence
\begin{align*}
\|\tilde{\mu}|_{E_\mathfrak{q}^\mathfrak{q}}\|_{\mathrm{op}} \leq C N_{\mathbb K}(\mathfrak{q})^{-\frac{1}{2}} (\#\tilde{G}_\mathfrak{q})^{\frac{1}{2}} \|\mu\|_2.
\end{align*}
\end{proof}

Now we prove \cref{lem:L2FlatteningLemma} by starting with \cref{lem:ConvolutionBoundOnE_q^q} obtained from the lower bounds of nontrivial irreducible representations of Chevalley groups, and then using \cref{lem:ExpanderMachineryBound} obtained from the exapansion machinery to continue to bound the right hand side by the $L^1$ norm and also remove the growth contributed by $\#\tilde{G}_\mathfrak{q}$ essentially by fiat.

\begin{proof}[Proof of \cref{lem:L2FlatteningLemma}]
Fix $C_1 > 0$ to be the $C$ from \cref{lem:ConvolutionBoundOnE_q^q} and $C_2 > 0$ to be the $C$ from \cref{lem:muHatLessThanCnu}. Fix any integer $l > l_0$ where $l_0 \in \mathbb Z_{>0}$ is the constant provided by \cref{lem:ExpanderMachineryBound} and fix $C_3 \in (0, 1)$ to be the corresponding $C$ from the same lemma. Fix $C = 2C_1C_2 > 0$ and $C_0 = -\frac{c}{2\log(C_3)} > 0$ where $c > 0$ depends on $n$ and is such that $\#\tilde{G}_\mathfrak{q} \leq N_{\mathbb K}(\mathfrak{q})^{c}$ for all nontrivial ideals $\mathfrak{q} \subset \mathcal{O}_{\mathbb K}$. Let $\xi = a + ib \in \mathbb C$ with $|a| < a_0'$, $\mathfrak{q} \subset \mathcal{O}_{\mathbb K}$ be a square free ideal coprime to $\mathfrak{q}_0$, $x \in \Sigma^+$, $C_0 \log(N_{\mathbb K}(\mathfrak{q})) \leq r < s$ be integers with $r \in l\mathbb Z$, $(\alpha_s, \alpha_{s - 1}, \dotsc, \alpha_{r + 1})$ be an admissible sequence and $\phi \in E_\mathfrak{q}^\mathfrak{q}$ with $\|\phi\|_2 = 1$. Let $\mu$ denote either $\mu_{(\alpha_s, \alpha_{s - 1}, \dotsc, \alpha_{r + 1})}^{\xi, \mathfrak{q}, x}$ or $\hat{\mu}_{(\alpha_s, \alpha_{s - 1}, \dotsc, \alpha_{r + 1})}^{a, \mathfrak{q}, x}$ and $\nu$ denote $\nu_{(\alpha_s, \alpha_{s - 1}, \dotsc, \alpha_{r + 1})}^{a, \mathfrak{q}, x}$. Applying \cref{lem:ConvolutionBoundOnE_q^q} to $\mu$ and then using \cref{lem:muHatLessThanCnu} gives
\begin{align*}
\|\mu * \phi\|_2 \leq C_1C_2 N_{\mathbb K}(\mathfrak{q})^{-\frac{1}{2}}(\#\tilde{G}_\mathfrak{q})^{\frac{1}{2}} \|\nu\|_2.
\end{align*}
By choice of $r$ and the function $\varphi = \delta_e - \frac{1}{\#\tilde{G}_\mathfrak{q}}\chi_{\tilde{G}_\mathfrak{q}} \in L_0^2(\tilde{G}_\mathfrak{q}, \mathbb C)$, which satisfies $\|\varphi\|_2 \leq 1$, we can use \cref{lem:ExpanderMachineryBound} to get
\begin{align*}
\|\nu\|_2 &= \|\nu * \delta_e\|_2 \leq \left\|\nu * \frac{1}{\#\tilde{G}_\mathfrak{q}}\chi_{\tilde{G}_\mathfrak{q}}\right\|_2 +\|\nu * \varphi\|_2 \\
&\leq \frac{\|\nu\|_1}{(\#\tilde{G}_\mathfrak{q})^{\frac{1}{2}}} + {C_3}^r\|\nu\|_1 \leq 2\frac{\|\nu\|_1}{(\#\tilde{G}_\mathfrak{q})^{\frac{1}{2}}}.
\end{align*}
Using this bound in the previous inequality and recalling $C = 2C_1C_2$, we have
\begin{align*}
\|\mu * \phi\|_2 \leq C N_{\mathbb K}(\mathfrak{q})^{-\frac{1}{2}} \|\nu\|_1.
\end{align*}
\end{proof}

\subsection{\texorpdfstring{$L^\infty$}{L-infinity} and Lipschitz bounds and proof of \texorpdfstring{\cref{thm:ReducedTheoremSmall|b|}}{\autoref{thm:ReducedTheoremSmall|b|}}}
\label{subsec:SupremumAndLipschitzBounds}
In this subsection we use \cref{lem:L2FlatteningLemma} to prove \cref{lem:ReducedTheoremEstimate1,lem:ReducedTheoremEstimate2} which is then used to prove \cref{thm:ReducedTheoremSmall|b|} by induction as in \cite{OW16}. We start with fixing some notations and easy bounds.

Let $\mathfrak{q} \subset \mathcal{O}_{\mathbb K}$ be a nontrivial proper ideal. Fix integers
\begin{align*}
r_\mathfrak{q} \in [C_0 \log(N_{\mathbb K}(\mathfrak{q})), C_0 \log(N_{\mathbb K}(\mathfrak{q})) + l)
\end{align*}
with $r_\mathfrak{q} \in l\mathbb Z$ and
\begin{align*}
s_\mathfrak{q} \in \left(r_\mathfrak{q} - \frac{\log(N_{\mathbb K}(\mathfrak{q})) + \log(4C_1C_f)}{\log(\theta)}, C_s\log(N_{\mathbb K}(\mathfrak{q}))\right)
\end{align*}
where we fix $C_0$ and $l$ to be constants from \cref{lem:L2FlatteningLemma}, $C_1$ to be the constant from \cref{lem:Small|b|Bound} and $C_s = C_0 - \frac{1}{\log(\theta)} + \frac{l}{\log(2)} - \frac{\log(4C_1C_f)}{\log(\theta)\log(2)} + \frac{1}{\log(2)}$ so that there is enough room for the integer $s_\mathfrak{q}$ to exist. These definitions of constants ensure that $C_0 \log(N_{\mathbb K}(\mathfrak{q})) \leq r_\mathfrak{q} < s_\mathfrak{q}$ and $4C_1C_f \theta^{s_\mathfrak{q} - r_\mathfrak{q}} \leq N_{\mathbb K}(\mathfrak{q})^{-1}$. For all $\xi = a + ib \in \mathbb C$ with $|a| < a_0'$, for all square free ideals $\mathfrak{q} \subset \mathcal{O}_{\mathbb K}$ coprime to $\mathfrak{q}_0$, for all $x \in \Sigma^+$, for all integers $C_0 \log(N_{\mathbb K}(\mathfrak{q})) \leq r < s$ with $r \in l\mathbb Z$, for all admissible sequences $(\alpha_s, \alpha_{s - 1}, \dotsc, \alpha_{r + 1})$, we have
\begin{align*}
\left\|\nu_{(\alpha_s, \alpha_{s - 1}, \dotsc, \alpha_{r + 1})}^{a, \mathfrak{q}, x}\right\|_1 &= e^{f_{s - r}^{(a)}(\alpha_s, \alpha_{s - 1}, \dotsc, \alpha_{r + 1}, \omega(\alpha_{r + 1}))} \left(\sum_{\alpha^r} e^{f_r^{(a)}(\alpha^r, x)}\right) \\
&\leq C_f e^{f_{s - r}^{(a)}(\alpha_s, \alpha_{s - 1}, \dotsc, \alpha_{r + 1}, \omega(\alpha_{r + 1}))}
\end{align*}
by \cref{lem:SumExpf^aBound} and hence \cref{lem:L2FlatteningLemma} implies that for all $\phi \in E_\mathfrak{q}^\mathfrak{q}$ we have
\begin{align*}
\|\mu * \phi\|_2 \leq CC_f N_{\mathbb K}(\mathfrak{q})^{-\frac{1}{2}} e^{f_{s - r}^{(a)}(\alpha_s, \alpha_{s - 1}, \dotsc, \alpha_{r + 1}, \omega(\alpha_{r + 1}))} \|\phi\|_2
\end{align*}
where $\mu$ denotes either $\mu_{(\alpha_s, \alpha_{s - 1}, \dotsc, \alpha_{r + 1})}^{\xi, \mathfrak{q}, x}$ or $\hat{\mu}_{(\alpha_s, \alpha_{s - 1}, \dotsc, \alpha_{r + 1})}^{a, \mathfrak{q}, x}$ and $C$ is the constant from the same lemma. We will use this in \cref{lem:ReducedTheoremEstimate1,lem:ReducedTheoremEstimate2}. We now start with the $L^\infty$ bound.

\begin{lemma}
\label{lem:ReducedTheoremEstimate1}
There exist $\kappa_1 \in (0, 1)$ and $q_{1,1} \in \mathbb Z_{>0}$ such that for all $\xi = a + ib \in \mathbb C$ with $|a| < a_0'$, for all square free ideals $\mathfrak{q} \subset \mathcal{O}_{\mathbb K}$ coprime to $\mathfrak{q}_0$ with $N_{\mathbb K}(\mathfrak{q}) > q_{1,1}$, for all $H \in \mathcal{W}_\mathfrak{q}^\mathfrak{q}(U)$, we have
\begin{align*}
\big\|\mathcal{M}_{\xi, \mathfrak{q}}^{s_\mathfrak{q}}(H)\big\|_\infty \leq \frac{1}{2}N_{\mathbb K}(\mathfrak{q})^{-\kappa_1}(\|H\|_\infty + \Lip_{d_\theta}(H)).
\end{align*}
\end{lemma}

\begin{proof}
There are $q_{1,1} \in \mathbb Z_{>0}$ and $\epsilon \in (0, 1)$ such that $\frac{1}{2} - \frac{\log(2C{C_f}^2)}{\log(q)} > \epsilon$ for all integers $q > q_{1,1}$. Fix any $\kappa_1 \in (0, \epsilon)$. Let $\xi = a + ib \in \mathbb C$ with $|a| < a_0'$, $\mathfrak{q} \subset \mathcal{O}_{\mathbb K}$ be a square free ideal coprime to $\mathfrak{q}_0$ with $N_{\mathbb K}(\mathfrak{q}) > q_{1,1}$, $H \in \mathcal{W}_\mathfrak{q}^\mathfrak{q}(U)$ and $x \in \Sigma^+$. Denote $r_\mathfrak{q}$ by $r$ and $s_\mathfrak{q}$ by $s$. Now using the approximation from \cref{lem:TransferOperatorConvolutionApproximation} and then \cref{lem:L2FlatteningLemma}, we have
\begin{align*}
&\left\|\mathcal{M}_{\xi, \mathfrak{q}}^s(H)(x)\right\|_2 \\
\leq{}&\left\|\sum_{\alpha_{r + 1}, \alpha_{r + 2}, \dotsc, \alpha_s} \mu_{(\alpha_s, \alpha_{s - 1}, \dotsc, \alpha_{r + 1})}^{\xi, \mathfrak{q}, x} * \phi_{(\alpha_s, \alpha_{s - 1}, \dotsc, \alpha_{r + 1})}^{\mathfrak{q}, H}\right\|_2 + C_f \Lip_{d_\theta}(H)\theta^{s - r} \\
\leq{}&\sum_{\alpha_{r + 1}, \alpha_{r + 2}, \dotsc, \alpha_s} \left\|\mu_{(\alpha_s, \alpha_{s - 1}, \dotsc, \alpha_{r + 1})}^{\xi, \mathfrak{q}, x} * \phi_{(\alpha_s, \alpha_{s - 1}, \dotsc, \alpha_{r + 1})}^{\mathfrak{q}, H}\right\|_2 + C_f \Lip_{d_\theta}(H)\theta^{s - r} \\
\leq{}&\sum_{\alpha_{r + 1}, \alpha_{r + 2}, \dotsc, \alpha_s} CC_f N_{\mathbb K}(\mathfrak{q})^{-\frac{1}{2}} e^{f_{s - r}^{(a)}(\alpha_s, \alpha_{s - 1}, \dotsc, \alpha_{r + 1}, \omega(\alpha_{r + 1}))} \left\|\phi_{(\alpha_s, \alpha_{s - 1}, \dotsc, \alpha_{r + 1})}^{\mathfrak{q}, H}\right\|_2 \\
{}&+ C_f \Lip_{d_\theta}(H)\theta^{s - r} \\
\leq{}&CC_f N_{\mathbb K}(\mathfrak{q})^{-\frac{1}{2}} \|H\|_\infty \left(\sum_{\alpha_{r + 1}, \alpha_{r + 2}, \dotsc, \alpha_s} e^{f_{s - r}^{(a)}(\alpha_s, \alpha_{s - 1}, \dotsc, \alpha_{r + 1}, \omega(\alpha_{r + 1}))}\right) \\
{}&+ C_f \Lip_{d_\theta}(H)\theta^{s - r} \\
\leq{}&C{C_f}^2 N_{\mathbb K}(\mathfrak{q})^{-\frac{1}{2}} \|H\|_\infty + C_f \Lip_{d_\theta}(H)\theta^{s - r} \\
\leq{}&\frac{1}{2}N_{\mathbb K}(\mathfrak{q})^{-\kappa_1}(\|H\|_\infty + \Lip_{d_\theta}(H)).
\end{align*}
since $C{C_f}^2 N_{\mathbb K}(\mathfrak{q})^{-\frac{1}{2}} \leq \frac{1}{2}N_{\mathbb K}(\mathfrak{q})^{-\kappa_1}$ and
\begin{align*}
C_f \theta^{s - r} \leq C_1C_f \theta^{s - r} \leq \frac{1}{4} N_{\mathbb K}(\mathfrak{q})^{-1} \leq \frac{1}{2} N_{\mathbb K}(\mathfrak{q})^{-\kappa_1}
\end{align*}
by definitions of the constants.
\end{proof}

Recalling that we already fixed $b_0 = 1$, we now record an estimate.

\begin{lemma}
\label{lem:Small|b|Bound}
There exists $C > 1$ such that for all $\xi = a + ib \in \mathbb C$ with $|a| < a_0'$ and $|b| \leq b_0$, for all $x, y \in \Sigma^+$, for all $s \in \mathbb Z_{>0}$, for all admissible sequences $\alpha^s$, we have
\begin{align*}
\left|1 - e^{(f_s^{(a)} + ib\tau_s)(\alpha^s, y) - (f_s^{(a)} + ib\tau_s)(\alpha^s, x)}\right| \leq C d_\theta(x, y).
\end{align*}
\end{lemma}

\begin{proof}
Fix $C > \max\left(1, (1 + b_0)\frac{T_0 \theta}{1 - \theta}e^{\frac{T_0 \theta}{1 - \theta}}\right)$. Let $\xi = a + ib \in \mathbb C$ with $|a| < a_0'$ and $|b| \leq b_0$. Let $x, y \in \Sigma^+$, $s \in \mathbb Z_{>0}$ and $\alpha^s$ be an admissible sequence. We calculate that
\begin{align*}
\left|f_s^{(a)}(\alpha^s, y) - f_s^{(a)}(\alpha^s, x)\right| &\leq \sum_{j = 0}^{s - 1} \left|f^{(a)}(\sigma^j(\alpha^s, y)) - f^{(a)}(\sigma^j(\alpha^s, x))\right| \\
&\leq \sum_{j = 0}^{s - 1} \Lip_{d_\theta}(f^{(a)}) \cdot d_\theta(\sigma^j(\alpha^s, y), \sigma^j(\alpha^s, x)) \\
&\leq \Lip_{d_\theta}(f^{(a)}) \sum_{j = 0}^{s - 1} \theta^{s - j} d_\theta(x, y) \\
&\leq \frac{T_0 \theta}{1 - \theta}d_\theta(x, y).
\end{align*}
In the same way, we have a similar bound $|\tau_s(\alpha^s, y) - \tau_s(\alpha^s, x)| \leq \frac{T_0 \theta}{1 - \theta}d_\theta(x, y)$. Thus, using $d_\theta(x, y) \leq 1$, we have
\begin{align*}
&\left|1 - e^{(f_s^{(a)} + ib\tau_s)(\alpha^s, y) - (f_s^{(a)} + ib\tau_s)(\alpha^s, x)}\right| \\
\leq{}&e^{\left|f_s^{(a)}(\alpha^s, y) - f_s^{(a)}(\alpha^s, x)\right|} \left|(f_s^{(a)} + ib\tau_s)(\alpha^s, y) - (f_s^{(a)} + ib\tau_s)(\alpha^s, x)\right| \\
\leq{}&e^{\frac{T_0 \theta}{1 - \theta}}\left(\frac{T_0 \theta}{1 - \theta}d_\theta(x, y) + \frac{b_0T_0 \theta}{1 - \theta}d_\theta(x, y)\right) \\
\leq{}&Cd_\theta(x, y).
\end{align*}
\end{proof}

\begin{remark}
This is the reason the approach of Bourgain-Gamburd-Sarnak is restricted to small $|b|$.
\end{remark}

Now we can take care of the Lipschitz bound and although \cref{lem:TransferOperatorConvolutionApproximation} cannot be used directly, we can use similar albeit more intricate estimates.

\begin{lemma}
\label{lem:ReducedTheoremEstimate2}
There exist $\kappa_2 \in (0, 1)$ and $q_{1,2} \in \mathbb Z_{>0}$ such that for all $\xi = a + ib \in \mathbb C$ with $|a| < a_0'$ and $|b| \leq b_0$, for all square free ideals $\mathfrak{q} \subset \mathcal{O}_{\mathbb K}$ coprime to $\mathfrak{q}_0$ with $N_{\mathbb K}(\mathfrak{q}) > q_{1,2}$, for all $H \in \mathcal{W}_\mathfrak{q}^\mathfrak{q}(U)$, we have
\begin{align*}
\Lip_{d_\theta}(\mathcal{M}_{\xi, \mathfrak{q}}^{s_\mathfrak{q}}(H)) \leq \frac{1}{2}N_{\mathbb K}(\mathfrak{q})^{-\kappa_2}(\|H\|_\infty + \Lip_{d_\theta}(H)).
\end{align*}
\end{lemma}

\begin{proof}
There are $q_{1,2} \in \mathbb Z_{>0}$ and $\epsilon \in (0, 1)$ such that $\frac{1}{2} - \frac{\log(4CC_1{C_f}^2)}{\log(q)} > \epsilon$ for all integers $q > q_{1,2}$. Fix any $\kappa_2 \in (0, \epsilon)$. Let $\xi = a + ib \in \mathbb C$ with $|a| < a_0'$ and $|b| \leq b_0$. Let $\mathfrak{q} \subset \mathcal{O}_{\mathbb K}$ be a square free ideal coprime to $\mathfrak{q}_0$ with $N_{\mathbb K}(\mathfrak{q}) > q_{1,2}$, $H \in \mathcal{W}_\mathfrak{q}^\mathfrak{q}(U)$ and $x, y \in \Sigma^+$. Denote $r_\mathfrak{q}$ by $r$ and $s_\mathfrak{q}$ by $s$. First suppose that $d_\theta(x, y) = 1$. Then from the proof of \cref{lem:ReducedTheoremEstimate1}, we can simply estimate as
\begin{align*}
&\left\|\mathcal{M}_{\xi, \mathfrak{q}}^s(H)(x) - \mathcal{M}_{\xi, \mathfrak{q}}^s(H)(y)\right\|_2 \\
\leq{}&\left\|\mathcal{M}_{\xi, \mathfrak{q}}^s(H)(x)\right\|_2 + \left\|\mathcal{M}_{\xi, \mathfrak{q}}^s(H)(y)\right\|_2 \\
\leq{}&\left(2C{C_f}^2 N_{\mathbb K}(\mathfrak{q})^{-\frac{1}{2}} \|H\|_\infty + 2C_f \Lip_{d_\theta}(H)\theta^{s - r}\right)d_\theta(x, y).
\end{align*}
Now suppose $d_\theta(x, y) < 1$. Then of course $x_0 = y_0$ and hence all the sums which will appear are over the same set of admissible sequences and moreover $\delta_{\mathtt{c}_\mathfrak{q}^s(\alpha^s, x)} = \delta_{\mathtt{c}_\mathfrak{q}^s(\alpha^s, y)}$. Thus we have
\begin{align*}
&\left\|\mathcal{M}_{\xi, \mathfrak{q}}^s(H)(x) - \mathcal{M}_{\xi, \mathfrak{q}}^s(H)(y)\right\|_2 \\
\leq{}&\left\|\sum_{\alpha^s} e^{(f_s^{(a)} + ib\tau_s)(\alpha^s, x)} \delta_{\mathtt{c}_\mathfrak{q}^s(\alpha^s, x)} * H(\alpha^s, x) \right. \\
{}&\left.- \sum_{\alpha^s} e^{(f_s^{(a)} + ib\tau_s)(\alpha^s, y)} \delta_{\mathtt{c}_\mathfrak{q}^s(\alpha^s, y)} * H(\alpha^s, y)\right\|_2 \\
\leq{}& \left\|\sum_{\alpha^s} e^{(f_s^{(a)} + ib\tau_s)(\alpha^s, y)}\delta_{\mathtt{c}_\mathfrak{q}^s(\alpha^s, x)} * \left(H(\alpha^s, x) - H(\alpha^s, y)\right)\right\|_2 \\
&{}+ \left\|\sum_{\alpha^s} \left(e^{(f_s^{(a)} + ib\tau_s)(\alpha^s, x)} - e^{(f_s^{(a)} + ib\tau_s)(\alpha^s, y)}\right) \right. \\
{}&\left.\cdot\delta_{\mathtt{c}_\mathfrak{q}^s(\alpha^s, x)} * (H(\alpha^s, x) - H(\alpha_s, \alpha_{s - 1}, \dotsc, \alpha_{r + 1}, \omega(\alpha_{r + 1})))\rule{0cm}{0.6cm}\right\|_2 \\
&{}+ \left\|\sum_{\alpha^s} \left(e^{(f_s^{(a)} + ib\tau_s)(\alpha^s, x)} - e^{(f_s^{(a)} + ib\tau_s)(\alpha^s, y)}\right) \right. \\
{}&\left.\cdot\delta_{\mathtt{c}_\mathfrak{q}^s(\alpha^s, x)} * H(\alpha_s, \alpha_{s - 1}, \dotsc, \alpha_{r + 1}, \omega(\alpha_{r + 1}))\rule{0cm}{0.6cm}\right\|_2 \\
={}&K_1 + K_2 + K_3.
\end{align*}
We easily estimate the first term $K_1$ as
\begin{align*}
K_1 &\leq \sum_{\alpha^s} e^{f_s^{(a)}(\alpha^s, x)} \|H(\alpha^s, x) - H(\alpha^s, y)\|_2 \\
&\leq \Lip_{d_\theta}(H) \theta^s d_\theta(x, y) \sum_{\alpha^s} e^{f_s^{(a)}(\alpha^s, x)} \leq C_f \Lip_{d_\theta}(H) \theta^s d_\theta(x, y).
\end{align*}
Next we estimate the second term $K_2$ as
\begin{align*}
K_2 &\leq \sum_{\alpha^s} \left|e^{(f_s^{(a)} + ib\tau_s)(\alpha^s, x)}\right| \cdot \left|1 - e^{(f_s^{(a)} + ib\tau_s)(\alpha^s, y) - (f_s^{(a)} + ib\tau_s)(\alpha^s, x)}\right| \\
&\cdot \|H(\alpha^s, x) - H(\alpha_s, \alpha_{s - 1}, \dotsc, \alpha_{r + 1}, \omega(\alpha_{r + 1}))\|_2 \\
&\leq C_1 \Lip_{d_\theta}(H) \theta^{s - r} d_\theta(x, y) \sum_{\alpha^s} e^{f_s^{(a)}(\alpha^s, x)} \\
&\leq C_1C_f \Lip_{d_\theta}(H) \theta^{s - r} d_\theta(x, y).
\end{align*}
Finally, using \cref{lem:L2FlatteningLemma}, we estimate the third and last term $K_3$ as
\begin{align*}
&K_3 \\
\leq{}&\left\|\, \left|\sum_{\alpha^s} \left(e^{(f_s^{(a)} + ib\tau_s)(\alpha^s, x)} - e^{(f_s^{(a)} + ib\tau_s)(\alpha^s, y)}\right) \right.\right. \\
{}&\left.\left.\cdot\delta_{\mathtt{c}_\mathfrak{q}^s(\alpha^s, x)} * H(\alpha_s, \alpha_{s - 1}, \dotsc, \alpha_{r + 1}, \omega(\alpha_{r + 1}))\rule{0cm}{0.6cm}\right| \,\right\|_2 \\
\leq{}&\left\|\sum_{\alpha^s} \left|e^{(f_s^{(a)} + ib\tau_s)(\alpha^s, x)}\right| \cdot \left|1 - e^{(f_s^{(a)} + ib\tau_s)(\alpha^s, y) - (f_s^{(a)} + ib\tau_s)(\alpha^s, x)}\right| \right. \\
{}&\left.\cdot\delta_{\mathtt{c}_\mathfrak{q}^s(\alpha^s, x)} * |H|(\alpha_s, \alpha_{s - 1}, \dotsc, \alpha_{r + 1}, \omega(\alpha_{r + 1}))\rule{0cm}{0.6cm}\right\|_2 \\
\leq{}&C_1 d_\theta(x, y) \sum_{\alpha_{r + 1}, \alpha_{r + 2}, \dotsc, \alpha_s} \left\|\sum_{\alpha^r} e^{f_s^{(a)}(\alpha^s, x)} \delta_{\mathtt{c}_\mathfrak{q}^{r + 1}(\alpha_{r + 1}, \alpha^r, x)} * \phi_{(\alpha_s, \alpha_{s - 1}, \dotsc, \alpha_{r + 1})}^{\mathfrak{q}, |H|}\right\|_2 \\
\leq{}&C_1 d_\theta(x, y) \sum_{\alpha_{r + 1}, \alpha_{r + 2}, \dotsc, \alpha_s} \left\|\hat{\mu}_{(\alpha_s, \alpha_{s - 1}, \dotsc, \alpha_{r + 1})}^{a, \mathfrak{q}, x} * \phi_{(\alpha_s, \alpha_{s - 1}, \dotsc, \alpha_{r + 1})}^{\mathfrak{q}, |H|}\right\|_2 \\
\leq{}&C_1 d_\theta(x, y) \sum_{\alpha_{r + 1}, \alpha_{r + 2}, \dotsc, \alpha_s} CC_f N_{\mathbb K}(\mathfrak{q})^{-\frac{1}{2}} e^{f_{s - r}^{(a)}(\alpha_s, \alpha_{s - 1}, \dotsc, \alpha_{r + 1}, \omega(\alpha_{r + 1}))} \\
{}&\cdot\left\|\phi_{(\alpha_s, \alpha_{s - 1}, \dotsc, \alpha_{r + 1})}^{\mathfrak{q}, |H|}\right\|_2 \\
\leq{}&CC_1C_f N_{\mathbb K}(\mathfrak{q})^{-\frac{1}{2}} \|\, |H| \,\|_\infty d_\theta(x, y) \left(\sum_{\alpha_{r + 1}, \alpha_{r + 2}, \dotsc, \alpha_s} e^{f_{s - r}^{(a)}(\alpha_s, \alpha_{s - 1}, \dotsc, \alpha_{r + 1}, \omega(\alpha_{r + 1}))}\right) \\
\leq{}&CC_1{C_f}^2 N_{\mathbb K}(\mathfrak{q})^{-\frac{1}{2}} \|H\|_\infty d_\theta(x, y).
\end{align*}
So combining all three estimates, we have
\begin{align*}
&\left\|\mathcal{M}_{\xi, \mathfrak{q}}^s(H)(x) - \mathcal{M}_{\xi, \mathfrak{q}}^s(H)(y)\right\|_2 \\
\leq{}&\left(CC_1{C_f}^2 N_{\mathbb K}(\mathfrak{q})^{-\frac{1}{2}} \|H\|_\infty + (C_f \theta^s + C_1C_f \theta^{s - r})\Lip_{d_\theta}(H)\right) d_\theta(x, y).
\end{align*}
Thus, for both cases $d_\theta(x, y) = 1$ and $d_\theta(x, y) < 1$, we have
\begin{align*}
&\left\|\mathcal{M}_{\xi, \mathfrak{q}}^s(H)(x) - \mathcal{M}_{\xi, \mathfrak{q}}^s(H)(y)\right\|_2 \\
\leq{}&\left(2CC_1{C_f}^2 N_{\mathbb K}(\mathfrak{q})^{-\frac{1}{2}} \|H\|_\infty + 2C_1C_f\Lip_{d_\theta}(H)\theta^{s - r}\right) d_\theta(x, y).
\end{align*}
Since the above holds for all $x, y \in \Sigma^+$, we have
\begin{align*}
\Lip_{d_\theta}(\mathcal{M}_{\xi, \mathfrak{q}}^s(H)) &\leq 2CC_1{C_f}^2 N_{\mathbb K}(\mathfrak{q})^{-\frac{1}{2}} \|H\|_\infty + 2C_1C_f\Lip_{d_\theta}(H)\theta^{s - r} \\
&\leq \frac{1}{2}N_{\mathbb K}(\mathfrak{q})^{-\kappa_2}(\|H\|_\infty + \Lip_{d_\theta}(H))
\end{align*}
since $2CC_1{C_f}^2 N_{\mathbb K}(\mathfrak{q})^{-\frac{1}{2}} \leq \frac{1}{2}N_{\mathbb K}(\mathfrak{q})^{-\kappa_2}$ and $2C_1C_f\theta^{s - r} \leq \frac{1}{2} N_{\mathbb K}(\mathfrak{q})^{-1} \leq \frac{1}{2} N_{\mathbb K}(\mathfrak{q})^{-\kappa_2}$ by definitions of the constants.
\end{proof}

\begin{proof}[Proof of \cref{thm:ReducedTheoremSmall|b|}]
Fix $\kappa_1, \kappa_2 \in (0, 1)$ and $q_{1,1}, q_{1,2} \in \mathbb Z_{>0}$ to be the constants from \cref{lem:ReducedTheoremEstimate1,lem:ReducedTheoremEstimate2}. Recall the constant $C_s$ and that we already fixed $b_0 = 1$. Fix $a_0 = a_0'$, $\kappa = \min(\kappa_1, \kappa_2) \in (0, 1)$ and $q_1 = \max(q_{1,1}, q_{1,2}) \in \mathbb Z_{>0}$. Let $\xi = a + ib \in \mathbb C$ with $|a| < a_0$ and $|b| \leq b_0$. Let $\mathfrak{q} \subset \mathcal{O}_{\mathbb K}$ be a square free ideal coprime to $\mathfrak{q}_0$ with $N_{\mathbb K}(\mathfrak{q}) > q_1$. Denote $s_\mathfrak{q}$ by $s$. Note that \cref{lem:ReducedTheoremEstimate1,lem:ReducedTheoremEstimate2} together give
\begin{align*}
\left\|\mathcal{M}_{\xi, \mathfrak{q}}^s(H)\right\|_\infty + \Lip_{d_\theta}(\mathcal{M}_{\xi, \mathfrak{q}}^s(H)) \leq N_{\mathbb K}(\mathfrak{q})^{-\kappa}(\|H\|_\infty + \Lip_{d_\theta}(H))
\end{align*}
for all $H \in \mathcal{W}_\mathfrak{q}^\mathfrak{q}(U)$. Now let $j \in \mathbb Z_{\geq 0}$ and $H \in \mathcal{W}_\mathfrak{q}^\mathfrak{q}(U)$. Then by induction we have
\begin{align*}
\big\|\mathcal{M}_{\xi, \mathfrak{q}}^{js}(H)\big\|_2 &\leq \big\|\mathcal{M}_{\xi, \mathfrak{q}}^{js}(H)\big\|_\infty \leq N_{\mathbb K}(\mathfrak{q})^{-j\kappa}(\|H\|_\infty + \Lip_{d_\theta}(H)) \\
&= N_{\mathbb K}(\mathfrak{q})^{-j\kappa}\|H\|_{\Lip(d_\theta)}.
\end{align*}
\end{proof}

\section{Spectral bounds for large $|b|$ using Dolgopyat's method}
\label{sec:GeodesicFlowLarge|b|}
In this section we will use Dolgopyat's method \cite{Dol98} to prove the following \cref{thm:TheoremLarge|b|}. We will closely follow \cite{OW16,Sto11}.

\begin{theorem}
\label{thm:TheoremLarge|b|}
There exist $\eta > 0, C > 0, a_0 > 0$ and $b_0 > 0$ such that for all $\xi = a + ib \in \mathbb C$ with $|a| < a_0$ and $|b| > b_0$, for all nontrivial ideals $\mathfrak{q} \subset \mathcal{O}_{\mathbb K}$, for all $k \in \mathbb Z_{\geq 0}$, for all $H \in \mathcal{V}_\mathfrak{q}(U)$, we have
\begin{align*}
\left\|\mathcal{M}_{\xi, \mathfrak{q}}^k(H)\right\|_2 \leq C e^{-\eta k} \|H\|_{1, b}.
\end{align*}
\end{theorem}

%Define both open and closed cylinders. Define them on U!!!! not on hat{U}!! Follow Styanov.
\subsection{Dolgopyat's method, preliminary lemmas, and constants}
As in \cite[Section 5]{Sto11}, we define a new distance function which is crucial for the argument. Let $D$ be a new distance function on $U$ defined by
\begin{align*}
D(u, u') =
\begin{cases}
\displaystyle\min_{\substack{u, u' \in \overline{\mathtt{C}}\\ \mathtt{C} \text{ is a cylinder}}} \diam_d(\overline{\mathtt{C}}), & u, u' \in U_j \text{ for some } j \in \mathcal{A} \\
1, & \text{otherwise}
\end{cases}
\end{align*}
for all $u, u' \in U$.

\begin{remark}
The above definition makes sense since $\diam_d(U_j) \leq \hat{\delta} < 1$ for all $j \in \mathcal{A}$. Note that $d(u, u') \leq D(u, u')$ for all $u, u' \in U$. Finally, it is important to observe that $C^{\Lip(D)}(U, \mathbb R) \subset B(U, \mathbb R)$ but $C^{\Lip(D)}(U, \mathbb R) \not\subset C(U, \mathbb R)$. Moreover, $h$ is measurable with respect to $\nu_U$ for all $h \in C^{\Lip(D)}(U, \mathbb R)$.
\end{remark}

We define the cones
\begin{align*}
K_B(U) ={}&\{h \in C^{\Lip(D)}(U, \mathbb R): h > 0, |h(u) - h(u')| \leq Bh(u)D(u, u') \\
{}&\text{ for all } u, u' \in U_j, \text{ for all } j \in \mathcal A\} \\
\tilde{K}_B(U) ={}&\Big\{h \in C^{\Lip(D)}(U, \mathbb R): h > 0, e^{-B D(u, u')} \leq \frac{h(u)}{h(u')} \leq e^{B D(u, u')} \\
{}&\text{ for all } u, u' \in U_j, \text{ for all } j \in \mathcal A\Big\}.
\end{align*}

\begin{remark}
If $h \in K_B(U)$, then using the convexity of $-\log$, we can derive that $h$ is $\log$-Lipschitz with respect to $D$, i.e., $|(\log \circ h)(u) - (\log \circ h)(u')| \leq B D(u, u')$ for all $u, u' \in U_j$, for all $j \in \mathcal{A}$. It follows that $K_B(U) \subset \tilde{K}_B(U)$, however the reverse containment is not true.
\end{remark}

The following \cref{thm:Dolgopyat} captures the mechanism of Dolgopyat's method \cite{Dol98}.

\begin{theorem}
\label{thm:Dolgopyat}
There exist $m \in \mathbb Z_{>0}, \eta \in (0, 1), a_0 > 0, b_0 > 0, E > \max\left(1, \frac{1}{b_0}\right)$ and a set of operators $\{\mathcal{N}_{a, J}: C^{\Lip(D)}(U, \mathbb R) \to C^{\Lip(D)}(U, \mathbb R): |a| < a_0, J \in \mathcal{J}(b), |b| > b_0\}$, where $\mathcal{J}(b)$ is some finite set for all $|b| > b_0$, such that
\begin{enumerate}
\item\label{itm:DolgopyatProperty1}	$\mathcal{N}_{a, J}(K_{E|b|}(U)) \subset K_{E|b|}(U)$ for all $|a| < a_0$, for all $J \in \mathcal{J}(b)$, for all $|b| > b_0$
\item\label{itm:DolgopyatProperty2}	$\|\mathcal{N}_{a, J}(h)\|_2 \leq \eta \|h\|_2$ for all $h \in K_{E|b|}(U)$, for all $|a| < a_0$, for all $J \in \mathcal{J}(b)$, for all $|b| > b_0$
\item\label{itm:DolgopyatProperty3}	for all $\xi = a + ib \in \mathbb C$ with $|a| < a_0$ and $|b| > b_0$, for all nontrivial ideals $\mathfrak{q} \subset \mathcal{O}_{\mathbb K}$, if $H \in C(U, L^2(F_\mathfrak{q}, \mathbb C))$ and $h \in K_{E|b|}(U)$ satisfy
\begin{enumerate}[label=(1\alph*), ref=\theenumi(1\alph*)]
\item\label{itm:DominatedByh}	$\|H(u)\|_2 \leq h(u)$ for all $u \in U$
\item\label{itm:LogLipschitzh}	$\|H(u) - H(u')\|_2 \leq E|b|h(u)D(u, u')$ for all $u, u' \in U_j$, for all $j \in \mathcal{A}$,
\end{enumerate}
then there exists $J \in \mathcal{J}(b)$ such that
\begin{enumerate}[label=(2\alph*), ref=\theenumi(2\alph*)]
\item\label{itm:DominatedByDolgopyat}	$\big\|\mathcal{M}_{\xi, \mathfrak{q}}^m(H)(u)\big\|_2 \leq \mathcal{N}_{a, J}(h)(u)$ for all $u \in U$
\item\label{itm:LogLipschitzDolgopyat}	$\big\|\mathcal{M}_{\xi, \mathfrak{q}}^m(H)(u) - \mathcal{M}_{\xi, \mathfrak{q}}^m(H)(u')\big\|_2 \leq E|b|\mathcal{N}_{a, J}(h)(u)D(u, u')$ for all $u, u' \in U_j$, for all $j \in \mathcal{A}$.
\end{enumerate}
\end{enumerate}
\end{theorem}

\begin{proof}[Proof that \cref{thm:Dolgopyat} implies \cref{thm:TheoremLarge|b|}]
Fix $m \in \mathbb Z_{>0}, a_0 > 0, b_0 > 0, E > \max\left(1, \frac{1}{b_0}\right)$ to be the ones from \cref{thm:Dolgopyat}. Let $\tilde{\eta} \in (0, 1)$ be the $\eta$ in \cref{thm:Dolgopyat}. Fix $\eta = \frac{-\log(\tilde{\eta})}{m} > 0$. Let $\xi = a + ib \in \mathbb C$ with $|a| < a_0$ and $|b| > b_0$. Let $\mathfrak{q} \subset \mathcal{O}_{\mathbb K}$ be a nontrivial ideal, $k \in \mathbb Z_{\geq 0}$, and $H \in \mathcal{V}_\mathfrak{q}(U)$. The theorem is trivial if $H = 0$, so suppose that $H \neq 0$. First set $h_0 \in K_{E|b|}(U)$ to be the positive constant function defined by $h_0(u) = \|H\|_{1, b}$ for all $u \in U$. Then $H$ and $h_0$ satisfy \cref{itm:DominatedByh,itm:LogLipschitzh} in \cref{thm:Dolgopyat}. Thus, given $h_j \in K_{E|b|}(U)$ for any $j \in \mathbb Z_{\geq 0}$, \cref{thm:Dolgopyat} provides a $J_j \in \mathcal{J}(b)$ and we inductively obtain $h_{j + 1} = \mathcal{N}_{a, J_j}(h_j) \in K_{E|b|}(U)$. Then $\big\|\mathcal{M}_{\xi, \mathfrak{q}}^{jm}(H)(u)\big\|_2 \leq h_j(u)$ for all $u \in U$ and hence $\big\|\mathcal{M}_{\xi, \mathfrak{q}}^{jm}(H)\big\|_2 \leq \|h_j\|_2 \leq \tilde{\eta}^j\|h_0\|_2 = \tilde{\eta}^j\|H\|_{1, b}$ for all $j \in \mathbb Z_{\geq 0}$. Fix
\begin{align*}
C_{\mathcal{M}} &= \max\left(1, \sup_{|\Re(\xi)| \leq a_0, \{0\} \subsetneq \mathfrak{q} \subset \mathcal{O}_{\mathbb K}} \left\|\mathcal{M}_{\xi, \mathfrak{q}}\right\|_{\mathrm{op}}\right) \\
&\leq \max\left(1, \sup_{|\Re(\xi)| \leq a_0} \left\|\mathcal{L}_\xi\right\|_{\mathrm{op}}\right) \leq \max(1, Ne^{T_0})
\end{align*}
where we use operator norms for operators on $L^2(U, L^2(F_\mathfrak{q}, \mathbb C))$ and $L^2(U, \mathbb R)$ respectively. Fix $C = {C_{\mathcal{M}}}^m \tilde{\eta}^{-1}$. Then writing $k = jm + l$ for some integers $j \in \mathbb Z_{\geq 0}$ and $0 \leq l < m$, we have
\begin{align*}
\left\|\mathcal{M}_{\xi, \mathfrak{q}}^k(H)\right\|_2 &= \big\|\mathcal{M}_{\xi, \mathfrak{q}}^{jm + l}(H)\big\|_2 \\
&\leq {C_{\mathcal{M}}}^l\big\|\mathcal{M}_{\xi, \mathfrak{q}}^{jm}(H)\big\|_2 \\
&\leq {C_{\mathcal{M}}}^l\tilde{\eta}^j\|H\|_{1, b} \\
&\leq {C_{\mathcal{M}}}^l\tilde{\eta}^{-1} \cdot e^{-\eta(j + 1)m}\|H\|_{1, b} \\
&\leq Ce^{-\eta k} \|H\|_{1, b}.
\end{align*}
\end{proof}

The subsequent subsections are devoted to the proof of \cref{thm:Dolgopyat}. We continue with the preliminaries. The following \cref{lem:SigmaHyperbolicity} is derived from the hyperbolicity of the geodesic flow akin to \cref{eqn:HyperbolicityComputation}.

\begin{lemma}
\label{lem:SigmaHyperbolicity}
There exist $c_0 \in (0, 1)$ and $\kappa_1 > \kappa_2 > 1$ such that for all $j \in \mathbb Z_{\geq 0}$, we have both
%used the gather environment so that both equations are centered.
\begin{gather*}
c_0 {\kappa_2}^j d(u, u') \leq d(\sigma^j(u), \sigma^j(u')) \leq {c_0}^{-1}{\kappa_1}^j d(u, u') \\
c_0 {\kappa_2}^j D(u, u') \leq D(\sigma^j(u), \sigma^j(u')) \leq {c_0}^{-1}{\kappa_1}^j D(u, u')
\end{gather*}
for all $u, u' \in \mathtt{C}$, for all cylinders $\mathtt{C} \subset U$ with $\len(\mathtt{C}) = j$.
\end{lemma}

\begin{remark}
The second line of inequalities follow from the first. We fix constants $c_0 \in (0, 1)$ and $\kappa_1 > \kappa_2 > 1$ as in \cref{lem:SigmaHyperbolicity} for the rest of the section and use these inequalities without further comments.
\end{remark}

Here we recall \cite[Proposition 3.3]{Sto11}.

\begin{lemma}
\label{lem:CylinderDiameterBound}
There exist $p_0 \in \mathbb Z_{>0}, \rho \in (0, 1)$ such that for all $l \in \mathbb Z_{\geq 0}$, for all cylinders $\mathtt{C}$ with $\len(\mathtt{C}) = l$, for all subcylinders $\mathtt{C}', \mathtt{C}'' \subset \mathtt{C}$ with $\len(\mathtt{C}') = l + 1$ and $\len(\mathtt{C}'') = l + p_0$ respectively, we have
\begin{align*}
\diam_d(\mathtt{C}'') \leq \rho\diam_d(\mathtt{C}) \leq \diam_d(\mathtt{C}').
\end{align*}
\end{lemma}

Anytime Dolgopyat's method is employed, some form of non-integrability condition is required. We will utilize the modern \emph{local non-integrability condition} as defined by Stoyanov in \cite[Section 2]{Sto11}. More precisely, we will record the main lemma of Stoyanov \cite[Lemma 4.2]{Sto11} in \cref{lem:LNIC} which is the output of the local non-integrability condition in a convenient form for Dolgopyat's method.

\begin{remark}
We recall the well known but important facts that for all Riemannian manifolds, the geodesic flow is a contact flow and moreover, if it has pinched negative sectional curvature, then the geodesic flow is also Anosov. Hence in our case, the geodesic flow for $X$ is a \emph{contact Anosov} flow which ensures that the local non-integrability condition is satisfied by \cite[Corollary 6.2]{Sto11}. See \cite[Subsection 1.1]{Sto11} and \cite[Section 6]{Sto11} for more details.
\end{remark}

We prepare with some notations and definitions. Let the Whitney sum $\T(X) = E^+ \oplus E^0 \oplus E^-$ be the Anosov splitting of the tangent bundle corresponding to the expanding, flow, and contracting directions. We recall \cite[Definition 4.1]{Sto11}. Let $z_1 \in R_1$ be the center. We have the diffeomorphism $\exp_{z_1}^{\mathrm{su}}: V_0 \to W_0$ for some convex open set $V_0 \subset E^+(z_1)$ and open set $W_0 \subset W_{\epsilon_0}^{\mathrm{su}}(z_1)$ such that $\overline{U_0'}^{\mathrm{su}} \subset \interior^{\mathrm{su}}(U_1)$ where $U_0' = W_0 \cap \Omega$. Fix $\theta_0 \in (0, 1)$ from the local non-integrability condition in \cite[Section 2]{Sto11}, fix any $\theta_1 \in (\theta_0, 1)$ and finally fix $p_1 \in \mathbb Z_{>0}$ such that $\theta_0 < \theta_1 - 32\rho^{p_1 - 1}$.

\begin{definition}
Let $z_1 \in R_1$ be the center. We define the following.
\begin{enumerate}
\item
Let $\eta \in E^+(z_1)$ with $\|\eta\| = 1$. For all closed cylinders $\mathtt{C} \subset U_0'$, we say that \emph{a separation by an $\eta$-plane occurs} in $\mathtt{C}$ if there exist $u, v \in \mathtt{C}$ with $d(u, v) \geq \frac{1}{2} \diam_d(\mathtt{C})$ such that
\begin{align*}
\left\langle\frac{(\exp_{z_1}^{\mathrm{su}})^{-1}(v) - (\exp_{z_1}^{\mathrm{su}})^{-1}(u)}{\|(\exp_{z_1}^{\mathrm{su}})^{-1}(v) - (\exp_{z_1}^{\mathrm{su}})^{-1}(u)\|}, \eta\right\rangle \geq \theta_1.
\end{align*}
Define $\mathcal{S}_\eta$ to be the set of closed cylinders $\mathtt{C} \subset U_0'$ such that a separation by an $\eta$-plane occurs.
\item Let $V \subset U_0'$ be a finite union of cylinders, $\delta > 0$, and $\eta \in E^+(z_1)$ with $\|\eta\| = 1$. Let $\mathcal{C}$ be the set of maximal closed cylinders $\mathtt{C} \subset \overline{V}$ with $\diam_d(\mathtt{C}) \leq \delta$. Then we define
\begin{align*}
M_\eta^\delta(V) = \cup(\mathcal{C} \cap \mathcal{S}_\eta) \subset U_1.
\end{align*}
\end{enumerate}
\end{definition}

%We refer the reader to that paper for the proof and \cite[Definition 4.1]{Sto11} for notations not defined here. We do not included them here because we do not need to use them directly.

%Intersection with \Omega is redundant since U_0 subset U_1 which is subset of \Omega.
%LNIC = local nonintegrability condition. The lemma is the output of LNIC called main lemma in Stoyanov.
\begin{lemma}
\label{lem:LNIC}
Let $z_1 \in R_1$ be the center. There exist integers $1 \leq m_1 \leq m_0$, tangent vectors $\eta_1, \eta_2, \dotsc, \eta_{\ell_0} \in E^+(z_1)$, and $U_0 \subset U_1$ which is a finite nonempty union of cylinders of length $m_1$ such that $\mathcal U = \sigma^{m_1}(U_0)$ is dense in $U$ and we have
\begin{enumerate}
\item for all integers $m \geq m_0$, for all $j \in \{1, 2\}$, for all integers $1 \leq \ell \leq \ell_0$, there exist Lipschitz sections $v_j^{\ell}: U \to U$ of $\sigma^m$ (i.e., $\sigma^m(v_j^{\ell}(u)) = u$ for all $u \in \mathcal U$) and $v_j^{\ell}(\mathcal U)$ is a finite union of cylinders of length $m$
\item $\overline{v_j^{\ell}(U)} \cap \overline{v_{j'}^{\ell'}(U)} = \varnothing$ for all $(j, \ell), (j', \ell') \in \{1, 2\} \times \{1, 2, \dotsc, \ell_0\}$ with $(j, \ell) \neq (j', \ell')$
\item there exits $\delta_0 > 0$ such that for all integers $1 \leq \ell \leq \ell_0$, for all $s \in (\exp_{z_1}^{\mathrm{su}})^{-1}(U_0)$, for all $0 < |t| \leq \delta_0$, for all $\eta \in E^+(z_1)$ with $\|\eta\| = 1$ and $\langle \eta, \eta_\ell \rangle \geq \theta_0$ such that $s + t\eta \in (\exp_{z_1}^{\mathrm{su}})^{-1}(U_0)$, we have
\begin{multline*}
\frac{1}{t}\left((\tau_m \circ  v_2^{\ell} \circ \sigma^{m_1} - \tau_m \circ  v_1^{\ell} \circ \sigma^{m_1})(\exp_{z_1}^{\mathrm{su}}(s + t\eta)) \right.\\
\left.- (\tau_m \circ  v_2^{\ell} \circ \sigma^{m_1} - \tau_m \circ  v_1^{\ell} \circ \sigma^{m_1})(\exp_{z_1}^{\mathrm{su}}(s))\right) \geq \frac{\delta_0}{2}
\end{multline*}
\item \label{itm:CylinderInM} for all cylinders $\mathtt{C} \subset U_0$, there exists $\delta' > 0$ such that
\begin{align*}
\mathtt{C} \subset \bigcup_{\ell = 1}^{\ell_0} M_{\eta_\ell}^\delta(\mathtt{C})
\end{align*}
for all $\delta \in (0, \delta']$.
\end{enumerate}
\end{lemma}

Fix $m_0, m_1, \ell_0, U_0, \mathcal{U}, \delta_0$ as in \cref{lem:LNIC}. Since $U_0$ is a finite union of cylinders, by \cref{itm:CylinderInM} in \cref{lem:LNIC}, we can fix a $\delta_1 > 0$ such that
\begin{align}
\label{eqn:U_0InUnionM_eta_ell}
U_0 \subset \bigcup_{\ell = 1}^{\ell_0} M_{\eta_\ell}^\delta(U_0)
\end{align}
for all $\delta \in (0, \delta_1]$.

The following is a useful Lasota--Yorke \cite{LY73} type of lemma.

\begin{lemma}
\label{lem:PreliminaryLogLipschitz}
There exists $A_0 > 0$ such that for all $\xi = a + ib \in \mathbb C$ with $|a| < a_0'$ and $|b| > 1$, for all nontrivial ideals $\mathfrak{q} \subset \mathcal{O}_{\mathbb K}$, for all $k \in \mathbb Z_{\geq 0}$, we have
\begin{enumerate}
\item\label{itm:PreliminaryLogLipschitzProperty1}	if $h \in K_B(U)$ for some $B > 0$, then we have $\mathcal{L}_a^k(h) \in K_{B'}(U)$ where $B' = A_0\left(\frac{B}{{\kappa_2}^k} + 1\right)$
\item\label{itm:PreliminaryLogLipschitzProperty1_tilde}	if $h \in \tilde{K}_B(U)$ for some $B > 0$, then we have $\mathcal{L}_a^k(h) \in \tilde{K}_{B'}(U)$ where $B' = A_0\left(\frac{B}{{\kappa_2}^k} + 1\right)$
\item\label{itm:PreliminaryLogLipschitzProperty2}	if $H \in C(U, L^2(F_\mathfrak{q}, \mathbb C))$ and $h \in B(U, \mathbb R)$ satisfy
\begin{align*}
\|H(u) - H(u')\|_2 \leq Bh(u)D(u, u')
\end{align*}
for all $u, u' \in U_j$, for all $j \in \mathcal A$, for some $B > 0$, then for all $j \in \mathcal A$, for all $u, u' \in U_j$, we have
\begin{align*}
\left\|\mathcal{M}_{\xi, \mathfrak{q}}^k(H)(u) - \mathcal{M}_{\xi, \mathfrak{q}}^k(H)(u')\right\|_2 \leq A_0\left(\frac{B}{{\kappa_2}^k}\mathcal{L}_a^k(h)(u) + |b|\mathcal{L}_a^k\|H\|(u)\right)D(u, u').
\end{align*}
\end{enumerate}
\end{lemma}

\begin{proof}
Fix $A_0 > 2{c_0}^{-1} e^{\frac{T_0}{c_0(\kappa_2 - 1)}} \max\left(1, \frac{T_0}{\kappa_2 - 1}\right)$. Let $\xi = a + ib \in \mathbb C$ with $|a| < a_0'$ and $|b| > 1$, $\mathfrak{q} \subset \mathcal{O}_{\mathbb K}$ be a nontrivial ideal, and $k \in \mathbb Z_{\geq 0}$. To prove \cref{itm:PreliminaryLogLipschitzProperty1}, let $h \in K_B(U)$ for some $B > 0$. We consider $u, u' \in \interior(U_j)$ for some $j \in \mathcal A$ and refer the reader to \cref{def:TransferOperatorOriginal} for the case when $u \in \partial(U_j)$ or $u' \in \partial(U_j)$. Let $v \in \sigma^{-k}(u)$ and $\mathtt{C}$ be the cylinder with $\len(\mathtt{C}) = k$ containing $v$. Since $\sigma^k|_{\mathtt{C}}: \mathtt{C} \to \interior(U_j)$ is a homeomorphism, there is a unique element which we denote by $\varphi(v) = (\sigma^k|_{\mathtt{C}})^{-1}(u') \in \mathtt{C}$ with $\sigma^k(\varphi(v)) = u'$. Since $v \in \sigma^{-k}(u)$ was arbitrary, we see that $\varphi$ defines a bijective map $\varphi: \sigma^{-k}(u) \to \sigma^{-k}(u')$. Using \cref{lem:SigmaHyperbolicity}, for all integers $0 \leq l \leq k$, we have $d(\sigma^l(v), \sigma^l(\varphi(v))) \leq \frac{1}{c_0 {\kappa_2}^{k - l}}d(u, u')$ and so
\begin{align*}
\left|f_k^{(a)}(v) - f_k^{(a)}(\varphi(v))\right| &\leq \sum_{l = 0}^{k - 1} \left|f^{(a)}(\sigma^l(v)) - f^{(a)}(\sigma^l(\varphi(v)))\right| \\
&\leq \sum_{l = 0}^{k - 1} \Lip_d^{\mathrm{e}}(f^{(a)}) \cdot d(\sigma^l(v), \sigma^l(\varphi(v))) \\
&\leq \sum_{l = 0}^{k - 1} \frac{T_0}{c_0 {\kappa_2}^{k - l}}d(u, u') \\
&\leq \frac{T_0}{c_0 (\kappa_2 - 1)}D(u, u').
\end{align*}
In the same way, we have a similar bound $|\tau_k(v) - \tau_k(\varphi(v))| \leq \frac{T_0}{c_0 (\kappa_2 - 1)}D(u, u')$. Using $\diam_d(U_j) \leq \hat{\delta} < 1$, we note that
\begin{align}
\label{eqn:faExponentialEstimate}
\left|1 - e^{f_k^{(a)}(\varphi(v)) - f_k^{(a)}(v)}\right| \leq e^{\left|f_k^{(a)}(v) - f_k^{(a)}(\varphi(v))\right|}\left|f_k^{(a)}(v) - f_k^{(a)}(\varphi(v))\right| \leq A_0D(u, u')
\end{align}
and similarly
\begin{align}
\begin{aligned}
\label{eqn:faplusTauExponentialEstimate}
&\left|1 - e^{(f_k^{(a)} + ib\tau_k)(\varphi(v)) - (f_k^{(a)} + ib\tau_k)(v)}\right| \\
\leq{}&e^{\left|f_k^{(a)}(v) - f_k^{(a)}(\varphi(v))\right|}\left|(f_k^{(a)} + ib\tau_k)(v) - (f_k^{(a)} + ib\tau_k)(\varphi(v))\right| \leq |b|A_0D(u, u').
\end{aligned}
\end{align}
Thus, again using $\diam_d(U_j) \leq \hat{\delta} < 1$, we have
\begin{align*}
&\left|\mathcal{L}_a^k(h)(u) - \mathcal{L}_a^k(h)(u')\right| \\
={}&\left|\sum_{v \in \sigma^{-k}(u)} e^{f_k^{(a)}(v)}h(v) - \sum_{v' \in \sigma^{-k}(u')} e^{f_k^{(a)}(v')}h(v')\right| \\
\leq{}&\sum_{v \in \sigma^{-k}(u)} \left|e^{f_k^{(a)}(v)}h(v) - e^{f_k^{(a)}(\varphi(v))}h(\varphi(v))\right| \\
\leq{}&\sum_{v \in \sigma^{-k}(u)} e^{f_k^{(a)}(\varphi(v))}|h(v) - h(\varphi(v))| + \sum_{v \in \sigma^{-k}(u)}\left|e^{f_k^{(a)}(v)} - e^{f_k^{(a)}(\varphi(v))}\right|h(v) \\
\leq{}&\sum_{v \in \sigma^{-k}(u)} c_0A_0e^{f_k^{(a)}(v)}Bh(v)D(v, \varphi(v)) \\
{}&+ \sum_{v \in \sigma^{-k}(u)}\left|1 - e^{f_k^{(a)}(\varphi(v)) - f_k^{(a)}(v)}\right|e^{f_k^{(a)}(v)}h(v).
\end{align*}
Now we use \cref{eqn:faExponentialEstimate} and continue the bound as
\begin{align*}
&\left|\mathcal{L}_a^k(h)(u) - \mathcal{L}_a^k(h)(u')\right| \\
\leq{}&c_0A_0B \cdot \frac{1}{c_0 {\kappa_2}^k}D(u, u')\sum_{v \in \sigma^{-k}(u)} e^{f_k^{(a)}(v)}h(v) + A_0D(u, u')\sum_{v \in \sigma^{-k}(u)} e^{f_k^{(a)}(v)}h(v) \\
={}&\frac{A_0B}{{\kappa_2}^k}\mathcal{L}_a^k(h)(u)D(u, u') + A_0\mathcal{L}_a^k(h)(u)D(u, u') \\
={}&A_0\left(\frac{B}{{\kappa_2}^k} + 1\right)\mathcal{L}_a^k(h)(u)D(u, u').
\end{align*}

To prove \cref{itm:PreliminaryLogLipschitzProperty1_tilde}, let $h \in \tilde{K}_B(U)$ for some $B > 0$. We again consider $u, u' \in \interior(U_j)$ for some $j \in \mathcal A$. Define the bijective map $\varphi: \sigma^{-k}(u) \to \sigma^{-k}(u')$ as before. We have
\begin{align*}
\mathcal{L}_a^k(h)(u') &= \sum_{v' \in \sigma^{-k}(u')} e^{f_k^{(a)}(v')}h(v') = \sum_{v \in \sigma^{-k}(u)} e^{f_k^{(a)}(\varphi(v))}h(\varphi(v)) \\
&\leq \sum_{v \in \sigma^{-k}(u)} e^{\left|f_k^{(a)}(v) - f_k^{(a)}(\varphi(v))\right|} e^{f_k^{(a)}(v)} e^{BD(v, \varphi(v))} h(v) \\
&\leq e^{A_0 D(u, u')} e^{\frac{B}{c_0{\kappa_2}^k}D(u, u')} \sum_{v \in \sigma^{-k}(u)} e^{f_k^{(a)}(v)} h(v) \\
&\leq e^{A_0\left(\frac{B}{{\kappa_2}^k} + 1\right)D(u, u')}\mathcal{L}_a^k(h)(u).
\end{align*}

To prove \cref{itm:PreliminaryLogLipschitzProperty2}, suppose $H \in C(U, L^2(F_\mathfrak{q}, \mathbb C))$ and $h \in B(U, \mathbb R)$ satisfy $\|H(u) - H(u')\|_2 \leq Bh(u)D(u, u')$ for all $u, u' \in U_j$, for all $j \in \mathcal A$, for some $B > 0$. We again consider $u, u' \in \interior(U_j)$ for some $j \in \mathcal A$. Define the bijective map $\varphi: \sigma^{-k}(u) \to \sigma^{-k}(u')$ as before. Similar to above, noticing that $\mathtt{c}^k(v) = \mathtt{c}^k(\varphi(v))$, we can use the same estimates to get
\begin{align*}
&\left\|\mathcal{M}_{\xi, \mathfrak{q}}^k(H)(u) - \mathcal{M}_{\xi, \mathfrak{q}}^k(H)(u')\right\|_2 \\
\leq{}&\left\|\sum_{v \in \sigma^{-k}(u)} e^{(f_k^{(a)} + ib\tau_k)(v)} \mathtt{c}^k(v)^{-1}H(v) - \sum_{v' \in \sigma^{-k}(u')} e^{(f_k^{(a)} + ib\tau_k)(v')} \mathtt{c}^k(v')^{-1}H(v')\right\|_2 \\
\leq{}&\sum_{v \in \sigma^{-k}(u)} \left\|e^{(f_k^{(a)} + ib\tau_k)(v)} \mathtt{c}^k(v)^{-1}H(v) - e^{(f_k^{(a)} + ib\tau_k)(\varphi(v))} \mathtt{c}^k(\varphi(v))^{-1}H(\varphi(v))\right\|_2 \\
\leq{}&\sum_{v \in \sigma^{-k}(u)} \left|e^{(f_k^{(a)} + ib\tau_k)(\varphi(v))}\right| \cdot \|H(v) - H(\varphi(v))\|_2 \\
{}&+ \sum_{v \in \sigma^{-k}(u)} \left|e^{(f_k^{(a)} + ib\tau_k)(v)} - e^{(f_k^{(a)} + ib\tau_k)(\varphi(v))}\right| \cdot \|H(v)\|_2 \\
\leq{}&\sum_{v \in \sigma^{-k}(u)} c_0A_0e^{f_k^{(a)}(v)} Bh(v)D(v, \varphi(v)) \\
{}&+ \sum_{v \in \sigma^{-k}(u)} \left|1 - e^{(f_k^{(a)} + ib\tau_k)(\varphi(v)) - (f_k^{(a)} + ib\tau_k)(v)}\right|e^{f_k^{(a)}(v)} \|H\|(v).
\end{align*}
Using \cref{eqn:faplusTauExponentialEstimate}, we can continue the bound as
\begin{align*}
&\left\|\mathcal{M}_{\xi, \mathfrak{q}}^k(H)(u) - \mathcal{M}_{\xi, \mathfrak{q}}^k(H)(u')\right\|_2 \\
\leq{}&c_0A_0B \cdot \frac{1}{c_0 {\kappa_2}^k}D(u, u') \sum_{v \in \sigma^{-k}(u)} e^{f_k^{(a)}(v)}h(v) \\
{}&+ |b|A_0D(u, u')\sum_{v \in \sigma^{-k}(u)} e^{f_k^{(a)}(v)} \|H\|(v) \\
\leq{}&\frac{A_0B}{{\kappa_2}^k}\mathcal{L}_a^k(h)(u)D(u, u') + |b|A_0\mathcal{L}_a^k\|H\|(u)D(u, u') \\
={}&A_0\left(\frac{B}{{\kappa_2}^k}\mathcal{L}_a^k(h)(u) + |b|\mathcal{L}_a^k\|H\|(u)\right)D(u, u').
\end{align*}
\end{proof}

Fix a $A_0 > 0$ provided by \cref{lem:PreliminaryLogLipschitz}. Now we fix some positive constants
\begin{align}
\label{eqn:Constant_b_0}
b_0 &= 1 \\
\label{eqn:ConstantE}
E &> \max(1, 2A_0) \\
\label{eqn:Constantepsilon1}
\epsilon_1 &< \min\left(\frac{c_0r_0}{{\kappa_1}^{m_1}}, \delta_1, \frac{\pi {c_0}^2(\kappa_2 - 1)}{2T_0 {\kappa_1}^{m_1}}\right) \\
\label{eqn:Constantm}
m &> m_0 \text{ such that } {\kappa_2}^m > \max\left(8A_0, \frac{4E{\rho}^{p_1}{\kappa_1}^{m_1} \epsilon_1}{{c_0}^2}, \frac{4\cdot 128E{\kappa_1}^{m_1}}{{c_0}^2\delta_0\rho}\right) \\
\label{eqn:Constantmu}
\mu &< \min\left(\frac{2E\epsilon_1{c_0}^2 \rho^{p_0p_1 + 1}{\kappa_2}^{m_1}}{{\kappa_1}^m}, \frac{1}{4}, \frac{1}{16 \cdot 16e^{2mT_0}}\left(\frac{\delta_0\rho\epsilon_1}{64}\right)^2\right).
\end{align}
For all $j \in \{1, 2\}$, for all integers $1 \leq \ell \leq \ell_0$, fix $v_j^\ell$ to be the corresponding sections of $\sigma^m$ provided by \cref{lem:LNIC}.

After constructing the Dolgopyat operators in \cref{subsec:ConstructionOfDolgopyatOperators}, \cref{thm:Dolgopyat} is a simple consequence of \cref{lem:DolgopyatProperty1,lem:LogLipschitzDolgopyat,lem:DolgopyatProperty2,lem:DominatedByDolgopyat}.

\subsection{Construction of Dolgopyat operators}
\label{subsec:ConstructionOfDolgopyatOperators}
For all $|b| > b_0$, we define the set $\{\mathtt{C}_1(b), \mathtt{C}_2(b), \dotsc, \mathtt{C}_{c_b}(b)\}$ of maximal cylinders $\mathtt{C} \subset U_0$ with $\diam_d(\mathtt{C}) \leq \frac{\epsilon_1}{|b|}$ so that $\overline{U_0} = \bigcup_{j = 1}^{c_b} \overline{\mathtt{C}_j(b)}$.

\begin{remark}
As a consequence of \cref{lem:SigmaHyperbolicity} and \cref{eqn:Constantepsilon1}, we have $\len(\mathtt{C}_j(b)) \geq m_1 + 1$ for all integers $1 \leq j \leq c_b$, for all $|b| > b_0$.
\end{remark}

Now, \cref{cor:CocyclesLocallyConstantCorollary} implies the following lemma.

\begin{lemma}
\label{lem:CocyclesConstantOnsigma^m1C_l}
For all $|b| > b_0$, for all integers $1 \leq l \leq c_b$, if $u, u' \in \mathtt{C}_l(b)$, then $\mathtt{c}^m(v_j^\ell(\sigma^{m_1}(u))) = \mathtt{c}^m(v_j^\ell(\sigma^{m_1}(u')))$ for all $j \in \{1, 2\}$, for all integers $1 \leq \ell \leq \ell_0$.
\end{lemma}

Let $|b| > b_0$. By \cref{lem:CocyclesConstantOnsigma^m1C_l}, we can define $\mathtt{c}_{l, j, \ell}(b) = \mathtt{c}^m(v_j^\ell(\sigma^{m_1}(u)))$ for any choice of $u \in \mathtt{C}_l(b)$, for all integers $1 \leq l \leq c_b$, for all $j \in \{1, 2\}$, for all integers $1 \leq \ell \leq \ell_0$. We define $\{\mathtt{D}_1(b), \mathtt{D}_2(b), \dotsc, \mathtt{D}_{p_b}(b)\}$ to be the set of subcylinders $\mathtt{D} \subset \mathtt{C}_l(b)$ for some integer $1 \leq l \leq c_b$ with $\len(\mathtt{D}) = \len(\mathtt{C}_l(b)) + p_0p_1$. We say $\mathtt{D}_k(b)$ and $\mathtt{D}_{k'}(b)$ are \emph{adjacent} if $\mathtt{D}_k(b), \mathtt{D}_{k'}(b) \subset \mathtt{C}_l(b)$ for some integer $1 \leq l \leq c_b$. We define $\Xi(b) = \{1, 2, \dotsc, p_b\} \times \{1, 2\} \times \{1, 2, \dotsc, \ell_0\}$ and we also define $\mathtt{Z}_k(b) = \sigma^{m_1}(\mathtt{D}_k(b))$ and $\mathtt{X}_{j, k}^\ell(b) = \overline{v_j^\ell(\mathtt{Z}_k(b))}$ for all $(k, j, \ell) \in \Xi(b)$. Note that $\mathtt{X}_{j, k}^\ell(b) \cap \mathtt{X}_{j', k'}^{\ell'}(b) = \varnothing$ for all $(k, j, \ell), (k', j', \ell') \in \Xi(b)$ with $(k, j, \ell) \neq (k', j', \ell')$. For all $J \subset \Xi(b)$, we define the function $\beta_J = \chi_{U} - \mu \sum_{(k, j, \ell) \in J} \chi_{\mathtt{X}_{j, k}^\ell(b)}$, and it can be checked that in fact $\beta_J \in C^{\Lip(D)}(U, \mathbb R)$.

Let $|b| > b_0$. Here we record a number of basic facts derived from \cref{lem:SigmaHyperbolicity,lem:CylinderDiameterBound} and \cite[Lemma 5.2]{Sto11} regarding the above definitions which will be required later.

\begin{enumerate}
\item We have the diameter bounds
\begin{gather}
\label{eqn:DiameterBoundC_l}
\rho\frac{\epsilon_1}{|b|} \leq \diam_d(\mathtt{C}_l(b)) \leq \frac{\epsilon_1}{|b|} \\
\label{eqn:DiameterBoundD_k}
\rho^{p_0p_1 + 1}\frac{\epsilon_1}{|b|} \leq \diam_d(\mathtt{D}_k(b)) \leq \rho^{p_1}\frac{\epsilon_1}{|b|} \\
\label{eqn:DiameterBoundZ_k}
\frac{\epsilon_1c_0\rho^{p_0p_1 + 1}{\kappa_2}^{m_1}}{|b|} \leq \diam_d(\mathtt{Z}_k(b)) \leq \frac{\epsilon_1\rho^{p_1}{\kappa_1}^{m_1}}{|b|c_0} \\
\label{eqn:DiameterBoundX_jk}
\frac{\epsilon_1{c_0}^2\rho^{p_0p_1 + 1}{\kappa_2}^{m_1}}{|b|{\kappa_1}^m} \leq \diam_d(\mathtt{X}_{j, k}^\ell(b)) \leq \frac{\epsilon_1\rho^{p_1}{\kappa_1}^{m_1}}{|b|{c_0}^2{\kappa_2}^m} \\
\end{gather}
for all integers $1 \leq l \leq c_b$, for all $(k, j, \ell) \in \Xi(b)$.
\item	Let $J \subset \Xi(b)$. We have $\beta_J \in C^{\Lip(D)}(U, \mathbb R)$ with Lipschitz constant
\begin{align}
\label{eqn:LipschitzConstantbeta_J}
\Lip_D(\beta_J) \leq \frac{\mu}{\min_{(k, j, \ell) \in J} \diam_d(\mathtt{X}_{j, k}^\ell(b))} \leq \frac{\mu|b|{\kappa_1}^m}{\epsilon_1{c_0}^2\rho^{p_0p_1 + 1}{\kappa_2}^{m_1}}.
\end{align}
\item	Let $1 \leq l \leq c_b$ be an integer. If $u, u' \in \sigma^{m_1}(\mathtt{C}_l(b))$, then
\begin{align}
\label{eqn:DBoundOnsigma^m1C_l}
D(v_j^\ell(u), v_j^\ell(u')) \leq \frac{\epsilon_1{\kappa_1}^{m_1}}{|b|{c_0}^2{\kappa_2}^m}
\end{align}
and
\begin{align}
\label{eqn:LNIC_OutputReverseBound}
|b| \cdot |(\tau_m(v_2^{\ell}(u)) - \tau_m(v_1^{\ell}(u))) - (\tau_m(v_2^{\ell}(u')) - \tau_m(v_1^{\ell}(u')))| \leq \pi
\end{align}
for all $j \in \{1, 2\}$, for all integers $1 \leq \ell \leq \ell_0$, by using \cref{eqn:Constantepsilon1}.
\end{enumerate}

\begin{definition}
For all $\xi = a + ib \in \mathbb C$ with $|a| < a_0'$ and $|b| > b_0$, for all $J \subset \Xi(b)$, we define the \emph{Dolgopyat operators} $\mathcal{N}_{a, J}: C^{\Lip(D)}(U, \mathbb R) \to C^{\Lip(D)}(U, \mathbb R)$ by
\begin{align*}
\mathcal{N}_{a, J}(h) = \mathcal{L}_{a}^m(\beta_J h)
\end{align*}
for all $h \in C^{\Lip(D)}(U, \mathbb R)$.
\end{definition}

\begin{definition}
For all $|b| > b_0$, a subset $J \subset \Xi(b)$ is said to be \emph{dense} if for all integers $1 \leq l \leq c_b$, there exists $(k, j, \ell) \in J$ such that $\mathtt{D}_k(b) \subset \mathtt{C}_l(b)$. Denote $\mathcal{J}(b)$ to be the set of all dense subsets of $\Xi(b)$.
\end{definition}

\subsection{Proof of \texorpdfstring{\cref{itm:DolgopyatProperty1,itm:LogLipschitzDolgopyat}}{Properties \ref{itm:DolgopyatProperty1} and \ref{itm:LogLipschitzDolgopyat}} in \texorpdfstring{\cref{thm:Dolgopyat}}{\autoref{thm:Dolgopyat}}}
\begin{lemma}
\label{lem:DolgopyatProperty1}
There exits $a_0 > 0$ such that for all $\xi = a + ib \in \mathbb C$ with $|a| < a_0$ and $|b| > b_0$, for all $J \in \mathcal J(b)$, we have $\mathcal{N}_{a, J}(K_{E|b|}(U)) \subset K_{E|b|}(U)$.
\end{lemma}

\begin{proof}
Fix $a_0 = a_0'$ and recall that we already fixed $b_0 = 1$. Let $\xi = a + ib \in \mathbb C$ with $|a| < a_0$ and $|b| > b_0$ and $J \in \mathcal J(b)$. Let $h \in K_{E|b|}(U)$ and $u, u' \in U_j$ for some $j \in \mathcal{A}$. Using \cref{eqn:LipschitzConstantbeta_J,eqn:Constantmu}, we have
\begin{align*}
\left|(\beta_J h)(u) - (\beta_J h)(u')\right| &\leq |h(u) - h(u')| + h(u)\left|\beta_J(u) - \beta_J(u')\right| \\
&\leq E|b|h(u)D(u, u') + h(u)\left(\frac{|b|\mu{\kappa_1}^m}{\epsilon_1{c_0}^2\rho^{p_0p_1 + 1}{\kappa_2}^{m_1}}\right)D(u, u') \\
&\leq |b|\left(E + \frac{\mu{\kappa_1}^m}{\epsilon_1{c_0}^2\rho^{p_0p_1 + 1}{\kappa_2}^{m_1}}\right)h(u)D(u, u') \\
&\leq |b|\left(\frac{E + 2E}{1 - \mu}\right)(\beta_J h)(u)D(u, u') \\
&\leq 4E|b|(\beta_J h)(u)D(u, u').
\end{align*}
Hence $\beta_J h \in K_{4E|b|}(U)$. Now applying \cref{lem:PreliminaryLogLipschitz,eqn:ConstantE,eqn:Constantm}, we have
\begin{align*}
\left|\mathcal{N}_{a, J}(h)(u) - \mathcal{N}_{a, J}(h)(u')\right| &= \left|\mathcal{L}_a^m(\beta_J h)(u) - \mathcal{L}_a^m(\beta_J h)(u')\right| \\
&\leq A_0\left(\frac{4E|b|}{{\kappa_2}^m} + 1\right) \mathcal{L}_a^m(\beta_J h)(u) D(u, u') \\
&\leq A_0\left(\frac{4E|b|}{8A_0} + \frac{E|b|}{2A_0}\right) \mathcal{L}_a^m(\beta_J h)(u) D(u, u') \\
&= E|b|\mathcal{N}_{a, J}(h)(u)D(u, u').
\end{align*}
\end{proof}

\begin{lemma}
\label{lem:LogLipschitzDolgopyat}
There exists $a_0 > 0$ such that for all $\xi = a + ib \in \mathbb C$ with $|a| < a_0$ and $|b| > b_0$, for all nontrivial ideals $\mathfrak{q} \subset \mathcal{O}_{\mathbb K}$, if $H \in C(U, L^2(F_\mathfrak{q}, \mathbb C))$ and $h \in B(U, \mathbb R)$ satisfy \cref{itm:DominatedByh,itm:LogLipschitzh} in \cref{thm:Dolgopyat}, then for all $J \in \mathcal{J}(b)$ we have
\begin{align*}
\left\|\mathcal{M}_{\xi, \mathfrak{q}}^m(H)(u) - \mathcal{M}_{\xi, \mathfrak{q}}^m(H)(u')\right\|_2 \leq E|b|\mathcal{N}_{a, J}(h)(u)D(u, u')
\end{align*}
for all $u, u' \in U_j$, for all $j \in \mathcal{A}$.
\end{lemma}

\begin{proof}
Fix $a_0 = a_0'$ and recall that we already fixed $b_0 = 1$. Let $\xi = a + ib \in \mathbb C$ with $|a| < a_0$ and $|b| > b_0$. Let $\mathfrak{q} \subset \mathcal{O}_{\mathbb K}$ be a nontrivial ideal and suppose $H \in C(U, L^2(F_\mathfrak{q}, \mathbb C))$ and $h \in B(U, \mathbb R)$ satisfy \cref{itm:DominatedByh,itm:LogLipschitzh} in \cref{thm:Dolgopyat}. Let $J \in \mathcal{J}(b)$ and $u, u' \in U_j$ for some $j \in \mathcal{A}$. Applying \cref{lem:PreliminaryLogLipschitz,eqn:ConstantE,eqn:Constantm,eqn:Constantmu}, we have
\begin{align*}
\left\|\mathcal{M}_{\xi, \mathfrak{q}}^m(H)(u) - \mathcal{M}_{\xi, \mathfrak{q}}^m(H)(u')\right\|_2 &\leq A_0\left(\frac{E|b|}{{\kappa_2}^m}\mathcal{L}_a^m(h)(u) + |b|\mathcal{L}_a^m\|H\|(u)\right)D(u, u') \\
&\leq A_0\left(\frac{E|b|}{8A_0} + \frac{E|b|}{2A_0}\right)\mathcal{L}_a^m(h)(u)D(u, u') \\
&\leq \left(\frac{E|b|}{8(1 - \mu)} + \frac{E|b|}{2(1 - \mu)}\right)\mathcal{L}_a^m(\beta_Jh)(u)D(u, u') \\
&\leq \left(\frac{E|b|}{6} + \frac{2E|b|}{3}\right)\mathcal{N}_{a, J}(h)(u)D(u, u') \\
&\leq E|b|\mathcal{N}_{a, J}(h)(u)D(u, u').
\end{align*}
\end{proof}

\subsection{Proof of \texorpdfstring{\cref{itm:DolgopyatProperty2}}{Property \ref{itm:DolgopyatProperty2}} in \texorpdfstring{\cref{thm:Dolgopyat}}{\autoref{thm:Dolgopyat}}}
\begin{definition}
\label{def:tSDenseSubset}
We say that a subset  $W \subset U$ is \emph{$(t, C)$-dense} for some $t > 0$ and $C \geq 1$, if there exists a set of mutually disjoint cylinders $\{\mathtt{B}_1, \mathtt{B}_2, \dotsc, \mathtt{B}_k\}$ for some $k \in \mathbb Z_{>0}$, with $\bigcup_{j = 1}^k \overline{\mathtt{B}_j} = U$ such that for all integers $1 \leq j \leq k$, we have
\begin{enumerate}
\item	$\diam_d(\mathtt{B}_j) \leq tC$
\item	there exists a subcylinder $\mathtt{B}_j' \subset W \cap \mathtt{B}_j$ with $\diam_d(\mathtt{B}_j') \geq t$.
\end{enumerate}
\end{definition}

\begin{lemma}
\label{lem:WDenseInequality}
Let $B > 0, C \geq 1$. There exists $\eta \in (0, 1)$ such that for all $t > 0$, for all $(t, C)$-dense subset $W \subset U$, for all $h \in \tilde{K}_{B/t}(U)$, we have $\int_W h^2 \, d\nu_U \geq \eta \int_{U} h^2 \, d\nu_U$.
\end{lemma}

\begin{proof}
Let $B > 0, C \geq 1$. Fix $\eta = e^{-2BC} \cdot \frac{c_1^U}{c_2^U}e^{-p_0 \delta_\Gamma \overline{\tau}\left(1 - \frac{\log(C)}{\log(\rho)}\right)} \in (0, 1)$. Let $t > 0$, $W \subset U$ be a $(t, C)$-dense subset and $\{\mathtt{B}_1, \mathtt{B}_2, \dotsc, \mathtt{B}_k\}$ be the set of mutually disjoint cylinders provided by \cref{def:tSDenseSubset}. Then $\bigcup_{j = 1}^k \overline{\mathtt{B}_j} = U$ and hence $\sum_{j = 1}^k \nu_U(\mathtt{B}_j) = 1$ and $\diam_d(\mathtt{B}_j) \leq tC$ for all integers $1 \leq j \leq k$. Let $h \in \tilde{K}_{B/t}(U)$. Setting $l_j = \inf_{u \in \overline{\mathtt{B}_j}} h(u)$ and $L_j = \sup_{u \in \overline{\mathtt{B}_j}} h(u)$, we have $l_j \geq L_j e^{-BC}$. Let $1 \leq j \leq k$ be an integer. Now, since $W$ is $(t, C)$-dense, there is a subcylinder $\mathtt{B}_j' \subset W \cap \mathtt{B}_j$ such that $\diam_d(\mathtt{B}_j') \geq t$. Write $\len(\mathtt{B}_j') = \len(\mathtt{B}_j) + r_jp_0 + s_j$ for some $r_j \in \mathbb Z_{\geq 0}$ and some integer $0 \leq s_j < p_0$. By \cref{lem:CylinderDiameterBound} we have $t \leq \diam_d(\mathtt{B}_j') \leq \rho^{r_j}\diam(\mathtt{B}_j) \leq \rho^{r_j}Ct$ which implies $r_j \leq -\frac{\log(C)}{\log(\rho)}$. Hence by the property of Gibbs measures in \cref{eqn:PropertyOfGibbsMeasures}, we have $\frac{\nu_U(\mathtt{B}_j')}{\nu_U(\mathtt{B}_j)} \geq \frac{c_1^U}{c_2^U}e^{-(r_jp_0 + s_j)\delta_\Gamma\overline{\tau}} \geq \frac{c_1^U}{c_2^U}e^{-p_0 \delta_\Gamma \overline{\tau}\left(1 - \frac{\log(C)}{\log(\rho)}\right)}$. Thus for all integers $1 \leq j \leq k$, we have $\nu_U(W \cap \mathtt{B}_j) \geq \nu_U(\mathtt{B}_j') \geq \frac{c_1^U}{c_2^U}e^{-p_0 \delta_\Gamma \overline{\tau}\left(1 - \frac{\log(C)}{\log(\rho)}\right)} \nu_U(\mathtt{B}_j)$. Thus we calculate that
\begin{align*}
\int_W h(u)^2 \, d\nu_U(u) &= \sum_{j = 1}^k \int_{W \cap \mathtt{B}_j} h(u)^2 \, d\nu_U(u) \geq \sum_{j = 1}^k {l_j}^2 \nu_U(W \cap \mathtt{B}_j) \\
&\geq \sum_{j = 1}^k {L_j}^2 e^{-2BC} \nu_U(W \cap \mathtt{B}_j) \\
&\geq e^{-2BC} \cdot \frac{c_1^U}{c_2^U}e^{-p_0 \delta_\Gamma \overline{\tau}\left(1 - \frac{\log(C)}{\log(\rho)}\right)} \sum_{j = 1}^k {L_j}^2 \nu_U(\mathtt{B}_j) \\
&\geq \eta\sum_{j = 1}^k \int_{\mathtt{B}_j} h(u)^2 \, d\nu_U(u) \\
&= \eta \int_{U} h(u)^2 \, d\nu_U(u).
\end{align*}
\end{proof}

\begin{lemma}
\label{lem:DolgopyatProperty2}
There exist $a_0 > 0$ and $\eta \in (0, 1)$ such that for all $\xi = a + ib \in \mathbb C$ with $|a| < a_0$ and $|b| > b_0$, for all $J \in \mathcal J(b)$, for all $h \in K_{E|b|}(U)$, we have $\|\mathcal{N}_{a, J}(h)\|_2 \leq \eta \|h\|_2$.
\end{lemma}

\begin{proof}
Fix $B = E\epsilon_1c_0\rho^{p_0p_1 + 1}{\kappa_2}^{m_1} > 0$ and $C = \frac{{\kappa_1}^{m_1}}{{c_0}^2\rho^{p_0p_1 + 1}{\kappa_2}^{m_1}} \geq 1$. Fix $\eta' \in (0, 1)$ to be the $\eta$ provided by \cref{lem:WDenseInequality}. Fix some
\begin{align*}
a_0 \in \left(0, \min\left(a_0', \frac{1}{mA_f}\log\left(\frac{1}{1 - \eta'\mu e^{-mT_0}}\right)\right)\right)
\end{align*}
so that we can also fix $\eta = \sqrt{e^{mA_fa_0}(1 - \eta'\mu e^{-mT_0})} \in (0, 1)$. Recall that we already fixed $b_0 = 1$. Let $\xi = a + ib \in \mathbb C$ with $|a| < a_0$ and $|b| > b_0$. Let $J \in \mathcal J(b)$ and $h \in K_{E|b|}(U)$. We have the estimate $\mathcal{N}_{a, J}(h)^2 \leq e^{mA_fa_0}\mathcal{N}_{0, J}(h)^2$ since $|f^{(a)} - f^{(0)}| \leq A_f|a| < A_fa_0$ and by the Cauchy-Schwarz inequality, we have
\begin{align*}
\mathcal{N}_{0, J}(h)^2 = \mathcal{L}_0^m(\beta_J h)^2 \leq \mathcal{L}_0^m({\beta_J}^2) \mathcal{L}_0^m(h^2).
\end{align*}
We would like to apply \cref{lem:WDenseInequality} on $h$ but first we need to ensure all the hypotheses hold. Let $t = \frac{\epsilon_1 c_0\rho^{p_0p_1 + 1}{\kappa_2}^{m_1}}{|b|}$ and note that $\frac{B}{t} = E|b|$. Let $W = \bigcup_{(k, j, \ell) \in J} \mathtt{Z}_k(b) \subset U$. We will show that $W$ is $(t, C)$-dense. Let $\mathtt{B}_l = \sigma^{m_1}(\mathtt{C}_l(b))$ for all integers $1 \leq l \leq c_b$. Then $\mathtt{B}_l \cap \mathtt{B}_{l'} = \varnothing$ and $\diam_d(\mathtt{B}_l) \leq tC$ by \cref{lem:SigmaHyperbolicity} for all integers $1 \leq l, l' \leq c_b$ with $l \neq l'$, and $\bigcup_{l = 1}^{c_b} \overline{\mathtt{B}_l} = U$. Since $J$ is dense, for all integers $1 \leq l \leq c_b$, there is a $(k, j, \ell) \in J$ such that $\mathtt{D}_k(b) \subset \mathtt{C}_l(b)$ and so $\mathtt{Z}_k(b) = \sigma^{m_1}(\mathtt{D}_k(b)) \subset W \cap \mathtt{B}_l$ with $\diam_d(\mathtt{Z}_k(b)) \geq t$ by \cref{lem:SigmaHyperbolicity}. Hence $W$ is $(t, C)$-dense. Since $h \in K_{E|b|}(U) \subset \tilde{K}_{E|b|}(U)$, we have $h^2 \in \tilde{K}_{2E|b|}(U)$. Applying \cref{lem:PreliminaryLogLipschitz} gives $\mathcal{L}_0^m(h^2) \in \tilde{K}_{B'}(U)$ where $B' = A_0\left(\frac{2E|b|}{{\kappa_2}^m} + 1\right) \leq A_0\left(\frac{2E|b|}{8A_0} + \frac{E|b|}{2A_0}\right) \leq 2E|b|$. Thus $\mathcal{L}_0^m(h^2) \in \tilde{K}_{2E|b|}(U)$ and so $\left(\mathcal{L}_0^m(h^2)\right)^{\frac{1}{2}} \in \tilde{K}_{E|b|}(U) = \tilde{K}_{B/t}(U)$. Now we can apply \cref{lem:WDenseInequality} to get $\int_W \mathcal{L}_0^m(h^2) \, d\nu_U \geq \eta' \int_U \mathcal{L}_0^m(h^2) \, d\nu_U$. Note that $\mathcal{L}_0^m({\beta_J}^2)(u) \leq \mathcal{L}_0^m(\chi_U - \mu \chi_{\mathtt{X}_{j, k}^\ell(b)})(u) \leq 1 - \mu e^{-mT_0}$ for all $u \in W$ by choosing any $(k, j, \ell) \in J$. So putting everything together and using $\mathcal{L}_0^*(\nu_U) = \nu_U$, we have
\begin{align*}
&\int_U \mathcal{N}_{a, J}(h)^2 \, d\nu_U \\
\leq{}&\int_U e^{mA_fa_0}\mathcal{N}_{0, J}(h)^2 \, d\nu_U \\
\leq{}&e^{mA_fa_0}\left(\int_{W} \mathcal{L}_0^m({\beta_J}^2) \mathcal{L}_0^m(h^2) \, d\nu_U + \int_{U \setminus W} \mathcal{L}_0^m({\beta_J}^2) \mathcal{L}_0^m(h^2) \, d\nu_U\right) \\
\leq{}&e^{mA_fa_0}\left((1 - \mu e^{-mT_0})\int_{W} \mathcal{L}_0^m(h^2) \, d\nu_U + \int_{U \setminus W} \mathcal{L}_0^m(h^2) \, d\nu_U\right) \\
={}&e^{mA_fa_0}\left(\int_U \mathcal{L}_0^m(h^2) \, d\nu_U - \mu e^{-mT_0}\int_{W} \mathcal{L}_0^m(h^2) \, d\nu_U\right) \\
\leq{}&e^{mA_fa_0}(1 - \eta'\mu e^{-mT_0})\int_U \mathcal{L}_0^m(h^2) \, d\nu_U \\
={}&\eta^2 \int_U h^2 \, d\nu_U.
\end{align*}
\end{proof}

\subsection{Proof of \texorpdfstring{\cref{itm:DominatedByDolgopyat}}{Property \ref{itm:DominatedByDolgopyat}} in \texorpdfstring{\cref{thm:Dolgopyat}}{\autoref{thm:Dolgopyat}}}
We start with the following lemma which is derived from \cref{lem:LNIC}. It is exactly as in \cite[Lemma 5.9]{Sto11} with the same proof.

\begin{lemma}
\label{lem:LNIC_Output}
Let $z_1 \in R_1$ be the center and $|b| > b_0$. Suppose $\mathtt{D}_k(b), \mathtt{D}_{k'}(b) \subset \mathtt{C}_l(b)$ for some integers $1 \leq k, k' \leq p_b$ and $1 \leq l \leq c_b$ such that $d(u_0, u_0') \geq \frac{1}{2}\diam_d(\mathtt{C}_l(b))$ and
\begin{align*}
\left\langle\frac{(\exp_{z_1}^{\mathrm{su}})^{-1}(u_0) - (\exp_{z_1}^{\mathrm{su}})^{-1}(u_0')}{\|(\exp_{z_1}^{\mathrm{su}})^{-1}(u_0) - (\exp_{z_1}^{\mathrm{su}})^{-1}(u_0')\|}, \eta_{\ell}\right\rangle \geq \theta_1
\end{align*}
for some $u_0 \in \mathtt{D}_k(b), u_0' \in \mathtt{D}_{k'}(b)$ and some integer $1 \leq \ell \leq \ell_0$. Then we have
\begin{align*}
|b| \cdot |(\tau_m(v_2^{\ell}(u)) - \tau_m(v_1^{\ell}(u))) - (\tau_m(v_2^{\ell}(u')) - \tau_m(v_1^{\ell}(u')))| \geq \frac{\delta_0 \rho \epsilon_1}{16}
\end{align*}
for all $u \in \mathtt{Z}_k(b)$, for all $u' \in \mathtt{Z}_{k'}(b)$.
\end{lemma}

Now, for all $\xi = a + ib \in \mathbb C$ with $|a| < a_0'$ and $|b| > b_0$, for all integers $1 \leq \ell \leq \ell_0$, for all $H \in C(U, L^2(F_\mathfrak{q}, \mathbb C))$ for some nontrivial ideal $\mathfrak{q} \subset \mathcal{O}_{\mathbb K}$, for all $h \in K_{E|b|}(U)$, we define the functions $\chi_1^\ell[\xi, H, h], \chi_2^\ell[\xi, H, h]: U \to \mathbb R$ by
\begin{align*}
&\chi_1^\ell[\xi, H, h](u) \\
={}&\frac{\left\|e^{(f_m^{(a)} + ib\tau_m)(v_1^\ell(u))} \mathtt{c}_{l, 1, \ell}(b)^{-1}H(v_1^\ell(u)) + e^{(f_m^{(a)} + ib\tau_m)(v_2^\ell(u))} \mathtt{c}_{l, 2, \ell}(b)^{-1}H(v_2^\ell(u))\right\|_2}{(1 - \mu)e^{f_m^{(a)}(v_1^\ell(u))}h(v_1^\ell(u)) + e^{f_m^{(a)}(v_2^\ell(u))}h(v_2^\ell(u))}
\end{align*}
and
\begin{align*}
&\chi_2^\ell[\xi, H, h](u) \\
={}&\frac{\left\|e^{(f_m^{(a)} + ib\tau_m)(v_1^\ell(u))} \mathtt{c}_{l, 1, \ell}(b)^{-1}H(v_1^\ell(u)) + e^{(f_m^{(a)} + ib\tau_m)(v_2^\ell(u))} \mathtt{c}_{l, 2, \ell}(b)^{-1}H(v_2^\ell(u))\right\|_2}{e^{f_m^{(a)}(v_1^\ell(u))}h(v_1^\ell(u)) + (1 - \mu)e^{f_m^{(a)}(v_2^\ell(u))}h(v_2^\ell(u))}
\end{align*}
for all $u \in \sigma^{m_1}(\mathtt{C}_l(b))$, for all integers $1 \leq l \leq c_b$. We need another lemma before proving \cref{itm:DominatedByDolgopyat} in \cref{thm:Dolgopyat}.

\begin{lemma}
\label{lem:HTrappedByh}
Let $|b| > b_0$ and $\mathfrak{q} \subset \mathcal{O}_{\mathbb K}$ be a nontrivial ideal. Suppose $H \in C(U, L^2(F_\mathfrak{q}, \mathbb C))$ and $h \in K_{E|b|}(U)$ satisfy \cref{itm:DominatedByh,itm:LogLipschitzh} in \cref{thm:Dolgopyat}. Then for all $(k, j, \ell) \in \Xi(b)$, we have
\begin{align*}
\frac{1}{2} \leq \frac{h(v_j^\ell(u))}{h(v_j^\ell(u'))} \leq 2
\end{align*}
for all $u, u' \in \mathtt{Z}_k(b)$ and also either of the alternatives
\begin{alternative}
\item\label{alt:HLessThan3/4h}	$\big\|H(v_j^\ell(u))\big\|_2 \leq \frac{3}{4}h(v_j^\ell(u))$ for all $u \in \mathtt{Z}_k(b)$
\item\label{alt:HGreaterThan1/4h}	$\big\|H(v_j^\ell(u))\big\|_2 \geq \frac{1}{4}h(v_j^\ell(u))$ for all $u \in \mathtt{Z}_k(b)$.
\end{alternative}
\end{lemma}

\begin{proof}
Recall that we already fixed $b_0 = 1$. Let $|b| > b_0$ and $\mathfrak{q} \subset \mathcal{O}_{\mathbb K}$ be a nontrivial ideal. Suppose $H \in C(U, L^2(F_\mathfrak{q}, \mathbb C))$ and $h \in K_{E|b|}(U)$ satisfy \cref{itm:DominatedByh,itm:LogLipschitzh} in \cref{thm:Dolgopyat}. Let $(k, j, \ell) \in \Xi(b)$. We first note that using \cref{eqn:DiameterBoundX_jk,eqn:Constantm}, for all $u, u' \in \mathtt{Z}_k(b)$, we have
\begin{align*}
h(v_j^\ell(u')) &\leq h(v_j^\ell(u)) + E|b|h(v_j^\ell(u))D(v_j^\ell(u), v_j^\ell(u')) \\
&\leq h(v_j^\ell(u))(1 + E|b|\diam_d(\mathtt{X}_{j, k}^\ell)) \\
&\leq h(v_j^\ell(u))\left(1 + E|b| \cdot \frac{\epsilon_1\rho^{p_1}{\kappa_1}^{m_1}}{|b|{c_0}^2{\kappa_2}^m}\right) \\
&\leq 2h(v_j^\ell(u))
\end{align*}
which proves the first sequence of inequalities. If $\big\|H(v_j^\ell(u))\big\|_2 \geq \frac{1}{4}h(v_j^\ell(u))$ for all $u \in \mathtt{Z}_k(b)$, then we are done. Otherwise there is a $u_0 \in \mathtt{Z}_k(b)$ such that $\big\|H(v_j^\ell(u_0))\big\|_2 \leq \frac{1}{4}h(v_j^\ell(u_0))$. Then for all $u \in \mathtt{Z}_k(b)$, using the above gives
\begin{align*}
\big\|H(v_j^\ell(u_0))\big\|_2 \leq \frac{1}{4}h(v_j^\ell(u_0)) \leq \frac{1}{2}h(v_j^\ell(u)).
\end{align*}
Similarly proceeding as above, using \cref{eqn:DiameterBoundX_jk,eqn:Constantm}, for all $u \in \mathtt{Z}_k(b)$, we have
\begin{align*}
\big\|H(v_j^\ell(u))\big\|_2 &\leq \big\|H(v_j^\ell(u_0))\big\|_2 + E|b|h(v_j^\ell(u))D(v_j^\ell(u), v_j^\ell(u_0)) \\
&\leq \frac{1}{2}h(v_j^\ell(u)) + E|b|h(v_j^\ell(u))\diam_d(\mathtt{X}_{j, k}^\ell) \\
&\leq h(v_j^\ell(u))\left(\frac{1}{2} + E|b| \cdot \frac{\epsilon_1\rho^{p_1}{\kappa_1}^{m_1}}{|b|{c_0}^2{\kappa_2}^m}\right) \\
&\leq \frac{3}{4}h(v_j^\ell(u)).
\end{align*}
\end{proof}

For any integer $k \geq 2$, let $\Theta: (\mathbb R^k \setminus \{0\}) \times (\mathbb R^k \setminus \{0\}) \to [0, \pi]$ be the map which gives the angle defined by $\Theta(w_1, w_2) = \arccos\left(\frac{\langle w_1, w_2\rangle}{\|w_1\| \cdot \|w_2\|}\right)$ for all $w_1, w_2 \in \mathbb R^k \setminus \{0\}$, where we use the standard inner product and norm.

\begin{lemma}
\label{lem:StrongTriangleInequality}
Let $k \geq 2$ be an integer. Suppose $w_1, w_2 \in \mathbb R^k \setminus \{0\}$ such that $\Theta(w_1, w_2) \geq \alpha$ and $\frac{\|w_1\|}{\|w_2\|} \leq L$ for some $\alpha \in [0, \pi]$ and $L \geq 1$. Then we have a stronger version of the triangle inequality
\begin{align*}
\|w_1 + w_2\| \leq \left(1 - \frac{\alpha^2}{16L}\right)\|w_1\| + \|w_2\|.
\end{align*}
\end{lemma}

\begin{proof}
Let $k \geq 2$ be an integer. Let $w_1, w_2 \in \mathbb R^k \setminus \{0\}$ such that $\Theta(w_1, w_2) \geq \alpha$ and $\frac{\|w_1\|}{\|w_2\|} \leq L$ for some $\alpha \in [0, \pi]$ and $L \geq 1$. Let $\epsilon = \frac{\alpha^2}{16L} \in [0, 1]$. Using the cosine law, we have
\begin{align*}
\|w_1 + w_2\|^2 &= \|w_1\|^2 + 2\|w_1\| \cdot \|w_2\|\cos(\Theta(w_1, w_2)) + \|w_2\|^2 \\
&\leq \|w_1\|^2 + 2\|w_1\| \cdot \|w_2\|\cos(\alpha) + \|w_2\|^2.
\end{align*}
Hence it suffices to show that $\|w_1\|^2 + 2\|w_1\| \cdot \|w_2\|\cos(\alpha) + \|w_2\|^2 \leq ((1 - \epsilon)\|w_1\| + \|w_2\|)^2$ since $(1 - \epsilon)\|w_1\| + \|w_2\| \geq 0$. Using the double angle formula and some simplification, we see that this is equivalent to showing
\begin{align*}
2\epsilon\left(\frac{\|w_1\|}{\|w_2\|} + 1\right) - \epsilon^2\frac{\|w_1\|}{\|w_2\|} \leq 4\sin^2\left(\frac{\alpha}{2}\right).
\end{align*}
Indeed, by our choice of $\epsilon$ and the simple inequality $\frac{\theta}{2} \leq \sin(\theta)$ for all $\theta \in \left[0, \frac{\pi}{2}\right]$, we have
\begin{align*}
2\epsilon\left(\frac{\|w_1\|}{\|w_2\|} + 1\right) - \epsilon^2\frac{\|w_1\|}{\|w_2\|} \leq 2\epsilon\left(\frac{\|w_1\|}{\|w_2\|} + 1\right) \leq 4\epsilon L = 4 \cdot \frac{\alpha^2}{16} \leq 4\sin^2\left(\frac{\alpha}{2}\right).
\end{align*}
\end{proof}

\begin{lemma}
\label{lem:chiLessThan1}
Let $\xi = a + ib \in \mathbb C$ with $|a| < a_0'$ and $|b| > b_0$. Let $\mathfrak{q} \subset \mathcal{O}_{\mathbb K}$ be a nontrivial ideal and suppose $H \in C(U, L^2(F_\mathfrak{q}, \mathbb C))$ and $h \in K_{E|b|}(U)$ satisfy \cref{itm:DominatedByh,itm:LogLipschitzh} in \cref{thm:Dolgopyat}. For all integers $1 \leq l \leq c_b$, there exists $(k, j, \ell) \in \Xi(b)$ such that $\mathtt{D}_k(b) \subset \mathtt{C}_l(b)$ and such that $\chi_j^\ell[\xi, H, h](u) \leq 1$ for all $u \in \mathtt{Z}_k(b)$.
\end{lemma}

\begin{proof}
Let $z_1 \in R_1$ be the center. Recall that we already fixed $b_0 = 1$. Let $\xi = a + ib \in \mathbb C$ with $|a| < a_0'$ and $|b| > b_0$. Let $\mathfrak{q} \subset \mathcal{O}_{\mathbb K}$ be a nontrivial ideal and suppose $H \in C(U, L^2(F_\mathfrak{q}, \mathbb C))$ and $h \in K_{E|b|}(U)$ satisfy \cref{itm:DominatedByh,itm:LogLipschitzh} in \cref{thm:Dolgopyat}. Let $1 \leq l \leq c_b$ be an integer. Since $\frac{\epsilon_1}{|b|} \in (0, \delta_1]$, so it follows from \cref{eqn:U_0InUnionM_eta_ell} that $\mathtt{C}_l(b) \subset M_{\eta_\ell}^{\epsilon_1/|b|}(U_0)$ for some integer $1 \leq \ell \leq \ell_0$. Hence there are $u_0, u_0' \in \mathtt{C}_l(b)$ such that $d(u_0, u_0') \geq \frac{1}{2}\diam_d(\mathtt{C}_l(b))$ and
\begin{align*}
\left\langle\frac{(\exp_{z_1}^{\mathrm{su}})^{-1}(u_0) - (\exp_{z_1}^{\mathrm{su}})^{-1}(u_0')}{\|(\exp_{z_1}^{\mathrm{su}})^{-1}(u_0) - (\exp_{z_1}^{\mathrm{su}})^{-1}(u_0')\|}, \eta_\ell\right\rangle \geq \theta_1.
\end{align*}
Take integers $1 \leq k, k' \leq p_b$ such that $u_0 \in \mathtt{D}_k(b)$ and $u_0' \in \mathtt{D}_{k'}(b)$ so that the hypotheses of \cref{lem:LNIC_Output} are satisfied. Now, suppose \cref{alt:HLessThan3/4h} in \cref{lem:HTrappedByh} holds for one of $(k, j, \ell), (k', j, \ell) \in \Xi(b)$ for some $j \in \{1, 2\}$. Without loss of generality, we can assume it holds for $(k, j, \ell) \in \Xi(b)$ and then it is a straightforward calculation to check that $\chi_j^\ell[\xi, H, h](u) \leq 1$ for all $u \in \mathtt{Z}_k(b)$, using \cref{eqn:Constantmu}. Otherwise, \cref{alt:HGreaterThan1/4h} in \cref{lem:HTrappedByh} holds for all of $(k, 1, \ell), (k, 2, \ell), (k', 1, \ell), (k', 2, \ell) \in \Xi(b)$. Let $u \in \mathtt{Z}_k(b)$ and $u' \in \mathtt{Z}_{k'}(b)$. Note that $\big\|H(v_j^\ell(u))\big\|_2, \big\|H(v_j^\ell(u'))\big\|_2 > 0$ for all $j \in \{1, 2\}$. We would like to apply \cref{lem:StrongTriangleInequality} but first we need to establish bounds on relative angle and relative size. We start with the former. For all $j \in \{1, 2\}$, let $u_j \in \{u, u'\}$ such that $\big\|H(v_j^\ell(u_j))\big\|_2 = \min\left(\big\|H(v_j^\ell(u))\big\|_2, \big\|H(v_j^\ell(u'))\big\|_2\right)$. Then recalling \cref{eqn:DBoundOnsigma^m1C_l,eqn:Constantm}, for all $j \in \{1, 2\}$, we have
\begin{align*}
\frac{\big\|H(v_j^\ell(u)) - H(v_j^\ell(u'))\big\|_2}{\min\left(\big\|H(v_j^\ell(u))\big\|_2, \big\|H(v_j^\ell(u'))\big\|_2\right)} &\leq \frac{E|b|h(v_j^\ell(u_j))D(v_j^\ell(u), v_j^\ell(u'))}{\big\|H(v_j^\ell(u_j))\big\|_2} \\
&\leq 4E|b| \cdot \frac{\epsilon_1 {\kappa_1}^{m_1}}{|b|{c_0}^2{\kappa_2}^m} \leq \frac{\delta_0 \rho \epsilon_1}{128}.
\end{align*}
Using the isomorphism $L^2(F_\mathfrak{q}, \mathbb C) \cong \mathbb R^{2\#F_\mathfrak{q}}$ of real vector spaces and some elementary geometry, the above shows that $\sin(\Theta(H(v_j^\ell(u)), H(v_j^\ell(u')))) \leq \frac{\delta_0 \rho \epsilon_1}{128}$ with $\Theta(H(v_j^\ell(u)), H(v_j^\ell(u'))) \in [0, \frac{\pi}{2})$, for all $j \in \{1, 2\}$. A simple inequality $\frac{\theta}{2} \leq \sin(\theta)$ for all $\theta \in \left[0, \frac{\pi}{2}\right]$ gives the angular bound $\Theta(H(v_j^\ell(u)), H(v_j^\ell(u'))) \leq 2\sin(\Theta(H(v_j^\ell(u)), H(v_j^\ell(u')))) \leq \frac{\delta_0 \rho \epsilon_1}{64}$ for all $j \in \{1, 2\}$. For notational convenience, we define $\varphi: U \to \mathbb R$ by $\varphi(w) = b(\tau_m(v_1^\ell(w)) - \tau_m(v_2^\ell(w)))$ for all $w \in U$. By \cref{lem:LNIC_Output,eqn:LNIC_OutputReverseBound}, we have $\frac{\delta_0 \rho \epsilon_1}{16} \leq |\varphi(u) - \varphi(u')| \leq \pi$. The second inequality is to ensure that we take the correct branch of angle in the following calculations. We will use these bounds to obtain a lower bound for $\Theta(V_1(u), V_2(u))$ or $\Theta(V_1(u'), V_2(u'))$ where we define
\begin{align*}
V_j(w) = e^{(f_m^{(a)} + ib\tau_m)(v_j^\ell(w))} \mathtt{c}_{l, j, \ell}(b)^{-1}H(v_j^\ell(w))
\end{align*}
for all $w \in U$, for all $j \in \{1, 2\}$. Also note that the triangle inequality on spheres implies that for all $w_1, w_2, w_3 \in \mathbb R^{2\#F_\mathfrak{q}} \setminus \{0\}$ we have $\Theta(w_1, w_3) \leq \Theta(w_1, w_2) + \Theta(w_2, w_3)$. Using the triangle inequality, we have
\begin{align*}
&\Theta\left(e^{(f_m^{(a)} + ib\tau_m)(v_1^\ell(u))} \mathtt{c}_{l, 1, \ell}(b)^{-1}H(v_1^\ell(u)), e^{(f_m^{(a)} + ib\tau_m)(v_2^\ell(u))} \mathtt{c}_{l, 2, \ell}(b)^{-1}H(v_2^\ell(u))\right) \\
={}&\Theta\left(e^{i\varphi(u)} \mathtt{c}_{l, 1, \ell}(b)^{-1}H(v_1^\ell(u)), \mathtt{c}_{l, 2, \ell}(b)^{-1}H(v_2^\ell(u))\right) \\
\geq{}&\Theta\left(e^{i\varphi(u)} \mathtt{c}_{l, 1, \ell}(b)^{-1}H(v_1^\ell(u)), e^{i\varphi(u')} \mathtt{c}_{l, 1, \ell}(b)^{-1}H(v_1^\ell(u))\right) \\
&{}- \Theta\left(e^{i\varphi(u')} \mathtt{c}_{l, 1, \ell}(b)^{-1}H(v_1^\ell(u)), e^{i\varphi(u')} \mathtt{c}_{l, 1, \ell}(b)^{-1}H(v_1^\ell(u'))\right) \\
&{}- \Theta\left(\mathtt{c}_{l, 2, \ell}(b)^{-1}H(v_2^\ell(u)), \mathtt{c}_{l, 2, \ell}(b)^{-1}H(v_2^\ell(u'))\right) \\
&{}- \Theta\left(e^{i\varphi(u')} \mathtt{c}_{l, 1, \ell}(b)^{-1}H(v_1^\ell(u')), \mathtt{c}_{l, 2, \ell}(b)^{-1}H(v_2^\ell(u'))\right).
\end{align*}
Since cocycles act unitarily, we can continue the bound as
\begin{align*}
&\Theta\left(e^{(f_m^{(a)} + ib\tau_m)(v_1^\ell(u))} \mathtt{c}_{l, 1, \ell}(b)^{-1}H(v_1^\ell(u)), e^{(f_m^{(a)} + ib\tau_m)(v_2^\ell(u))} \mathtt{c}_{l, 2, \ell}(b)^{-1}H(v_2^\ell(u))\right) \\
\geq{}&\Theta\left(e^{i(\varphi(u) - \varphi(u'))} \mathtt{c}_{l, 1, \ell}(b)^{-1}H(v_1^\ell(u)), \mathtt{c}_{l, 1, \ell}(b)^{-1}H(v_1^\ell(u))\right) \\
{}&- \Theta\left(H(v_1^\ell(u)), H(v_1^\ell(u'))\right) - \Theta\left(H(v_2^\ell(u)), H(v_2^\ell(u'))\right) \\
&{}- \Theta\left(e^{i\varphi(u')} \mathtt{c}_{l, 1, \ell}(b)^{-1}H(v_1^\ell(u')), \mathtt{c}_{l, 2, \ell}(b)^{-1}H(v_2^\ell(u'))\right) \\
={}&|\varphi(u) - \varphi(u')| - \Theta\left(H(v_1^\ell(u)), H(v_1^\ell(u'))\right) - \Theta\left(H(v_2^\ell(u)), H(v_2^\ell(u'))\right) \\
&{}- \Theta\left(e^{i\varphi(u')} \mathtt{c}_{l, 1, \ell}(b)^{-1}H(v_1^\ell(u')), \mathtt{c}_{l, 2, \ell}(b)^{-1}H(v_2^\ell(u'))\right).
\end{align*}
Using the previously calculated angular bounds, we see that
\begin{align*}
&\Theta\left(e^{(f_m^{(a)} + ib\tau_m)(v_1^\ell(u))} \mathtt{c}_{l, 1, \ell}(b)^{-1}H(v_1^\ell(u)), e^{(f_m^{(a)} + ib\tau_m)(v_2^\ell(u))} \mathtt{c}_{l, 2, \ell}(b)^{-1}H(v_2^\ell(u))\right) \\
\geq{}&\frac{\delta_0 \rho \epsilon_1}{16} - \frac{\delta_0 \rho \epsilon_1}{64} - \frac{\delta_0 \rho \epsilon_1}{64} - \Theta\left(e^{(f_m^{(a)} + ib\tau_m)(v_1^\ell(u'))} \mathtt{c}_{l, 1, \ell}(b)^{-1}H(v_1^\ell(u')), \right. \\
{}&\left.e^{(f_m^{(a)} + ib\tau_m)(v_2^\ell(u'))} \mathtt{c}_{l, 2, \ell}(b)^{-1}H(v_2^\ell(u'))\right) \\
={}&\frac{\delta_0 \rho \epsilon_1}{32} - \Theta\left(e^{(f_m^{(a)} + ib\tau_m)(v_1^\ell(u'))} \mathtt{c}_{l, 1, \ell}(b)^{-1}H(v_1^\ell(u')), \right. \\
{}&\left.e^{(f_m^{(a)} + ib\tau_m)(v_2^\ell(u'))} \mathtt{c}_{l, 2, \ell}(b)^{-1}H(v_2^\ell(u'))\right).
\end{align*}
Hence $\Theta(V_1(u), V_2(u)) + \Theta(V_1(u'), V_2(u')) \geq \frac{\delta_0 \rho \epsilon_1}{32}$ for all $u \in \mathtt{Z}_k(b)$, for all $u' \in \mathtt{Z}_{k'}(b)$. Thus, without loss of generality, we can assume that $\Theta(V_1(u), V_2(u)) \geq \frac{\delta_0 \rho \epsilon_1}{64}$ for all $u \in \mathtt{Z}_k(b)$, which establishes the required bound on relative angle. For the bound on relative size, let $(j, j') \in \{(1, 2), (2, 1)\}$ such that $h(v_j^\ell(u_0)) \leq h(v_{j'}^\ell(u_0))$ for some $u_0 \in \mathtt{Z}_k(b)$. Then by \cref{lem:HTrappedByh}, we have
\begin{align*}
\frac{\|V_j(u)\|_2}{\|V_{j'}(u)\|_2} &= \frac{\big\|e^{(f_m^{(a)} + ib\tau_m)(v_j^\ell(u))} \mathtt{c}_{l, j, \ell}(b)^{-1}H(v_j^\ell(u))\big\|_2}{\big\|e^{(f_m^{(a)} + ib\tau_m)(v_{j'}^\ell(u))} \mathtt{c}_{l, j', \ell}(b)^{-1}H(v_{j'}^\ell(u))\big\|_2} = \frac{e^{f_m^{(a)}(v_j^\ell(u))}\big\|H(v_j^\ell(u))\big\|_2}{e^{f_m^{(a)}(v_{j'}^\ell(u))}\big\|H(v_{j'}^\ell(u))\big\|_2} \\
&\leq \frac{4e^{f_m^{(a)}(v_j^\ell(u)) - f_m^{(a)}(v_{j'}^\ell(u))}h(v_j^\ell(u))}{h(v_{j'}^\ell(u))} \leq \frac{16e^{2mT_0}h(v_j^\ell(u_0))}{h(v_{j'}^\ell(u_0))} \leq 16e^{2mT_0}
\end{align*}
for all $u \in \mathtt{Z}_k(b)$, which establishes the required bound on relative size. Now applying \cref{lem:StrongTriangleInequality,eqn:Constantmu} and $\|H\| \leq h$ on $\|V_j(u) + V_{j'}(u)\|_2$ gives $\chi_j^\ell[\xi, H, h](u) \leq 1$ for all $u \in \mathtt{Z}_k(b)$.
\end{proof}

\begin{lemma}
\label{lem:DominatedByDolgopyat}
There exists $a_0 > 0$ such that for all $\xi = a + ib \in \mathbb C$ with $|a| < a_0$ and $|b| > b_0$, for all nontrivial ideals $\mathfrak{q} \subset \mathcal{O}_{\mathbb K}$, if $H \in C(U, L^2(F_\mathfrak{q}, \mathbb C))$ and $h \in K_{E|b|}(U)$ satisfy \cref{itm:DominatedByh,itm:LogLipschitzh} in \cref{thm:Dolgopyat}, then there exists $J \in \mathcal{J}(b)$ such that
\begin{align*}
\left\|\mathcal{M}_{\xi, \mathfrak{q}}^m(H)(u)\right\|_2 \leq \mathcal{N}_{a, J}(h)(u)
\end{align*}
for all $u \in U$.
\end{lemma}

\begin{proof}
Fix $a_0 = a_0'$ and recall that we already fixed $b_0 = 1$. Let $\xi = a + ib \in \mathbb C$ with $|a| < a_0$ and $|b| > b_0$, and $\mathfrak{q} \subset \mathcal{O}_{\mathbb K}$ be a nontrivial ideal. Suppose $H \in C(U, L^2(F_\mathfrak{q}, \mathbb C))$ and $h \in K_{E|b|}(U)$ satisfy \cref{itm:DominatedByh,itm:LogLipschitzh} in \cref{thm:Dolgopyat}. For all integers $1 \leq l \leq c_b$, we can choose a $(k_l, j_l, \ell_l) \in \Xi(b)$ as guaranteed by \cref{lem:chiLessThan1}. Let $J = \{(k_l, j_l, \ell_l) \in \Xi(b): 1 \leq l \leq c_b\} \subset \Xi(b)$ which is then dense by construction and so $J \in \mathcal{J}(b)$. Now we show that $\big\|\mathcal{M}_{\xi, \mathfrak{q}}^m(H)\big\| \leq \mathcal{N}_{a, J}(h)$ for this choice of $J \in \mathcal{J}(b)$. We consider $u \in \interior(U)$ and refer the reader to \cref{def:TransferOperatorOriginal} for the case when $u \in \partial(U)$. If $u \notin \mathtt{Z}_k(b)$ for all $(k, j, \ell) \in J$, then $\beta_J(v) = 1$ for all branch $v \in \sigma^{-m}(u)$ and hence the bound follows trivially by definitions. Otherwise there is an integer $1 \leq l \leq c_b$ such that $u \in \mathtt{Z}_{k_l}(b)$ corresponding to $(k_l, j_l, \ell_l) \in J$. Note that by definition of $J$, we have $(k_l, j, \ell) \notin J$ for all $(k_l, j, \ell) \in \Xi(b)$ with $(j, \ell) \neq (j_l, \ell_l)$. Let $(j_l, j_l') \in \{(1, 2), (2, 1)\}$. Then by construction of $J$, we have $\chi_{j_l}^{\ell_l}[\xi, H, h](u) \leq 1$ and $\beta_J(v_{j_l}^{\ell_l}(u)) = 1 - \mu$ and $\beta_J(v_j^\ell(u)) = 1$ for all $(k_l, j, \ell) \in \Xi(b)$ with $(j, \ell) \neq (j_l, \ell_l)$. Hence we compute that
\begin{comment}
Let $J_1 = \{(k, 1, \ell) \in \Xi(b): \chi_1^\ell[\xi, H, h](u) \leq 1 \text{ for all } u \in \hat{\mathtt{Z}}_k(b)\}$ and let $J_2 = \{(k, 2, \ell) \in \Xi(b) \setminus J_1: \chi_2^\ell[\xi, H, h](u) \leq 1 \text{ for all } u \in \hat{\mathtt{Z}}_k(b)\}$ and finally let $J = J_1 \cup J_2$. By \cref{lem:chiLessThan1}, $J \in \Xi(b)$ is a dense subset and hence $J \in \mathcal{J}(b)$. Now we show that $\left\|\mathcal{M}_{\xi, q}^m(H)\right\| \leq \mathcal{N}_{J, a}(h)$ for this choice of $J \in \mathcal{J}(b)$. Let $u \in \hat{U}$. If $u \notin \hat{\mathtt{Z}}_k(b)$ for all $(k, j, \ell) \in J$, then $\beta_J(v) = 1$ for all branch $v \in \sigma^{-m}(u)$ and hence the bound follows trivially by definitions. Now suppose that $u \in \hat{\mathtt{Z}}_k(b)$ for some $(k, j, \ell) \in J$. Also suppose $\mathtt{D}_k(b) \subset \mathtt{C}_l(b)$ for some integer $1 \leq l \leq c_b$. Let $j' \in \{1, 2\} \setminus \{j\}$. Then by construction of $J$, we have $\chi_j^\ell[\xi, H, h](u) \leq 1$ for all $u \in \hat{\mathtt{Z}}_k(b)$. Again, $(k, j', \ell) \notin J$ by construction and so we have $\beta_J(v_j^\ell(u)) = 1 - \mu$ and $\beta_J(v_{j'}^\ell(u)) = 1$ for all $u \in \hat{\mathtt{Z}}_k(b)$. Hence for all $u \in \hat{\mathtt{Z}}_k(b)$, we compute that
\end{comment}
\begin{align*}
&\left\|\mathcal{M}_{\xi, \mathfrak{q}}^m(H)(u)\right\|_2 \\
={}& \left\|\sum_{v \in \sigma^{-m}(u)} e^{(f_m^{(a)} + ib\tau_m)(v)} \mathtt{c}^m(v)^{-1} H(v)\right\|_2 \\
\leq{}&\sum_{v \in \sigma^{-m}(u), v \neq v_1^{\ell_l}(u), v \neq v_2^{\ell_l}(u)} \left\|e^{(f_m^{(a)} + ib\tau_m)(v)} \mathtt{c}^m(v)^{-1} H(v)\right\|_2 \\
&{}+ \left\|e^{(f_m^{(a)} + ib\tau_m)(v_{j_l}^{\ell_l}(u))} \mathtt{c}_{l, j_l, \ell_l}(b)^{-1} H(v_{j_l}^{\ell_l}(u))\right. \\
&\left.{}+ e^{(f_m^{(a)} + ib\tau_m)(v_{j_l'}^{\ell_l}(u))} \mathtt{c}_{l, j_l', \ell_l}(b)^{-1} H(v_{j_l'}^{\ell_l}(u))\right\|_2 \\
\leq{}&\sum_{v \in \sigma^{-m}(u), v \neq v_1^{\ell_l}(u), v \neq v_2^{\ell_l}(u)} e^{f_m^{(a)}(v)} h(v) \\
{}&+ \left((1 - \mu)e^{f_m^{(a)}(v_{j_l}^{\ell_l}(u))}h(v_{j_l}^{\ell_l}(u)) + e^{f_m^{(a)}(v_{j_l'}^{\ell_l}(u))}h(v_{j_l'}^{\ell_l}(u))\right) \\
\leq{}&\mathcal{N}_{a, J}(h)(u).
\end{align*}
\end{proof}

\section{Uniform exponential mixing of the geodesic flow}
\label{sec:UniformExponentialMixingOfTheGeodesicFlow}
The aim of this section is to prove \cref{thm:TheoremUniformExponentialMixingOfGeodesicFlow} using the proven spectral bound \cref{thm:TheoremGeodesicFlow}.

%The equivalence relation $\sim$ on $\Sigma \times \mathbb R$ defined by $(u, t + \tau(u)) \sim (\sigma(u), t)$ for all $(u, t) \in \Sigma \times \mathbb R$ gives the suspension space $\Sigma^\tau = (\Sigma \times \mathbb R)/\sim$.

Let $\mathfrak{q} \subset \mathcal{O}_{\mathbb K}$ be an ideal coprime to $\mathfrak{q}_0$. Similar to $R^{\mathfrak{q}, \tau}$, consider the suspension space $U^{\mathfrak{q}, \tau} = (U \times \tilde{G}_\mathfrak{q} \times \mathbb R_{\geq 0})/\mathord{\sim}$ where $\sim$ is the equivalence relation on $U \times \tilde{G}_\mathfrak{q} \times \mathbb R_{\geq 0}$ defined by $(u, g, t + \tau(u)) \sim (\sigma(u), g\mathtt{c}_\mathfrak{q}(u), t)$ for all $(u, g, t) \in U \times \tilde{G}_\mathfrak{q} \times \mathbb R_{\geq 0}$. Define the norm $\|\phi\|_{\mathcal{B}_\mathfrak{q}} = \|\phi\|_\infty + \Lip_{d, |\cdot|}(\phi)$ where the second term is defined to be
\begin{align*}
&\sup\left\{\frac{|\phi(u, g, t) - \phi(u', g, t')|}{d(u, u') + |t - t'|}: u, u' \in U, g \in \tilde{G}_\mathfrak{q}, t \in [0, \tau(u)), t' \in [0, \tau(u')), \right.\\
&\left.(u, t) \neq (u', t')\rule{0cm}{0.5cm}\right\}
\end{align*}
for any function $\phi: U^{\mathfrak{q}, \tau} \to \mathbb R$ and define the associated space $\mathcal{B}_\mathfrak{q} = \{\phi: U^{\mathfrak{q}, \tau} \to \mathbb R: \phi(u, g, \cdot)|_{[0, \tau(u))} \in C^1([0, \tau(u)), \mathbb R), \|\phi\|_{\mathcal{B}_\mathfrak{q}} < \infty\}$. For all $\phi \in \mathcal{B}_\mathfrak{q}$, for all $\xi \in \mathbb C$, define the measurable function $\hat{\phi}_\xi \in B(U, L^2(\tilde{G}_\mathfrak{q}, \mathbb C))$ by
\begin{align*}
\hat{\phi}_\xi(u)(g) = \int_0^{\tau(u)} \phi(u, g, t)e^{-\xi t} \, dt
\end{align*}
for all $g \in \tilde{G}_\mathfrak{q}$, for all $u \in U$.

\begin{remark}
For all $u \in U$, for all $g \in \tilde{G}_\mathfrak{q}$, the map $\mathbb C \to \mathbb C$ defined by $\xi \mapsto \hat{\phi}_\xi(u)(g)$ is entire.
\end{remark}

\begin{comment}
Let $\phi: \hat{U}^{\mathfrak{q}, \tau} \to \mathbb C$. Define the norms
\begin{align*}
\|\phi\|_{\mathcal{B}_0} &= \|\phi\|_\infty + \sup \frac{|\phi(u, g, t) - \phi(u', g', s')|}{d(u, u) + |s - s'|} \\
\|\phi\|_{\mathcal{B}_1} &= \|\phi\|_\infty + \sup\{\Var_{[0, \tau(u)]}\phi(u, g, \cdot): u \in \hat{U}, g \in \tilde{G}_\mathfrak{q}\}
\end{align*}
and the respective spaces $\mathcal{B}_0 = \{\phi: \hat{U}^{\mathfrak{q}, \tau} \to \mathbb C: \|\phi\|_{\mathcal{B}_0} < \infty\}$ and $\mathcal{B}_1 = \{\phi: \hat{U}^{\mathfrak{q}, \tau} \to \mathbb C: \|\phi\|_{\mathcal{B}_1} < \infty\}$. \\
\end{comment}

\begin{remark}
Let $\mathfrak{q} \subset \mathcal{O}_{\mathbb K}$ be an ideal coprime to $\mathfrak{q}_0$, $\phi \in \mathcal{B}_\mathfrak{q}$ and $\xi \in \mathbb C$. Because of $\tau$ involved in the definition of $\hat{\phi}_\xi$, it is not necessarily Lipschitz but it is essentially Lipschitz. However we can see from the proof of \cref{lem:ExtractNormOfLaplaceTransformDecay} that in fact $\mathcal{M}_{\xi, \mathfrak{q}}(\hat{\phi}_{-\overline{\xi}}) \in \mathcal{V}_\mathfrak{q}(U)$. Moreover, we can then deduce that if $\sum_{g \in \tilde{G}_\mathfrak{q}} \phi(u, g, t) = 0$ for all $(u, t) \in U^\tau$, then we have $\mathcal{M}_{\xi, \mathfrak{q}}(\hat{\phi}_{-\overline{\xi}}) \in \mathcal{W}_\mathfrak{q}(U)$.
\end{remark}

\subsection{Correlation function and its Laplace transform}
Let $\mathfrak{q} \subset \mathcal{O}_{\mathbb K}$ be an ideal coprime to $\mathfrak{q}_0$ and $\phi, \psi \in \mathcal{B}_\mathfrak{q}$. Define the continuous correlation function $\Upsilon_{\phi, \psi} \in L^\infty(\mathbb R_{\geq 0}, \mathbb R)$ by
\begin{align*}
\Upsilon_{\phi, \psi}(t) = \sum_{g \in \tilde{G}_\mathfrak{q}} \int_U \int_0^{\tau(u)} \phi(u, g, r + t) \psi(u, g, r) \, dr \, d\nu_U(u)
\end{align*}
for all $t \in \mathbb R_{\geq 0}$. We can decompose this into continuous functions as $\Upsilon_{\phi, \psi} = \Upsilon_{\phi, \psi}^0 + \Upsilon_{\phi, \psi}^1$ where we define
\begin{align*}
\Upsilon_{\phi, \psi}^0(t) &= \sum_{g \in \tilde{G}_\mathfrak{q}} \int_U \int_{\max(0, \tau(u) - t)}^{\tau(u)} \phi(u, g, r + t) \psi(u, g, r) \, dr \, d\nu_U(u) \\
\Upsilon_{\phi, \psi}^1(t) &= \sum_{g \in \tilde{G}_\mathfrak{q}} \int_U \int_0^{\max(0, \tau(u) - t)} \phi(u, g, r + t) \psi(u, g, r) \, dr \, d\nu_U(u)
\end{align*}
for all $t \in \mathbb R_{\geq 0}$.

\begin{remark}
This decomposition is useful because the Laplace transform of $\Upsilon_{\phi, \psi}^0$ has a clean expression in terms of the congruence transfer operator. Combining this with the fact that $\Upsilon_{\phi, \psi}(t) = \Upsilon_{\phi, \psi}^0(t)$ for all $t \geq \overline{\tau}$, it is clear that it suffices to study decay properties of $\Upsilon_{\phi, \psi}^0$.
\end{remark}

Consider the Laplace transform on the half plane $\hat{\Upsilon}_{\phi, \psi}^0: \{\xi \in \mathbb C: \Re(\xi) > 0\} \to \mathbb C$ defined by
\begin{align*}
\hat{\Upsilon}_{\phi, \psi}^0(\xi) = \int_0^\infty \Upsilon_{\phi, \psi}^0(t) e^{-\xi t} \, dt
\end{align*}
for all $\xi = a + ib \in \mathbb C$ with $a > 0$.

\begin{remark}
Clearly $\hat{\Upsilon}_{\phi, \psi}^0$ is holomorphic. However, it is not at all obvious from the above definition whether it can be defined holomorphically on a larger half plane. Nevertheless, using the congruence transfer operator bounds, we will show that $\hat{\Upsilon}_{\phi, \psi}^0$ has a holomorphic extension to a larger half plane $\{\xi \in \mathbb C: \Re(\xi) > -a_0\}$ for some $a_0 > 0$. This will allow us to apply the inverse Laplace transform formula to extract an exponential decay for $\Upsilon_{\phi, \psi}^0$.
\end{remark}

\begin{lemma}
\label{lem:LaplaceTransformOfSymbolicCodingCorrelationFunctionInTermsOfTransferOperator}
For all ideals $\mathfrak{q} \subset \mathcal{O}_{\mathbb K}$ coprime to $\mathfrak{q}_0$, for all $\phi, \psi \in \mathcal{B}_\mathfrak{q}$, for all $\xi = a + ib \in \mathbb C$ with $a > 0$, we have
\begin{align*}
\hat{\Upsilon}_{\phi, \psi}^0(\xi) = \sum_{k = 1}^\infty {\lambda_a}^k \left\langle \hat{\phi}_\xi, \mathcal{M}_{\xi, \mathfrak{q}}^k(\hat{\psi}_{-\overline{\xi}}) \right\rangle.
\end{align*}
\end{lemma}

\begin{proof}
Let $\mathfrak{q} \subset \mathcal{O}_{\mathbb K}$ be an ideal coprime to $\mathfrak{q}_0$ and $\phi, \psi \in \mathcal{B}_\mathfrak{q}$ and $\xi = a + ib \in \mathbb C$ with $a > 0$. Observe that since $\nu_U(\hat{U}) = 1$, we can compute some of the integrals over $\interior(U)$ instead of $U$ when convenient to avoid the boundary. Define the product measure $\nu$ by $d\nu = dt \, dr \, d\nu_U$ for convenience. We calculate that
\begin{align*}
&\hat{\Upsilon}_{\phi, \psi}^0(\xi) \\
={}&\int_0^\infty \left(\sum_{g \in \tilde{G}_\mathfrak{q}} \int_U \int_{\max(0, \tau(u) - t)}^{\tau(u)} \phi(u, g, r + t) \psi(u, g, r) \, dr \, d\nu_U(u)\right) e^{-\xi t} \, dt \\
={}&\sum_{g \in \tilde{G}_\mathfrak{q}} \int_U \int_0^{\tau(u)} \int_{\tau(u) - r}^\infty e^{-\xi t} \phi(u, g, r + t) \psi(u, g, r) \, dt \, dr \, d\nu_U(u) \\
={}&\sum_{g \in \tilde{G}_\mathfrak{q}} \int_U \int_0^{\tau(u)} \int_{\tau(u)}^\infty e^{-\xi (t - r)} \phi(u, g, t) \psi(u, g, r) \, d\nu(t, r, u) \\
={}&\sum_{g \in \tilde{G}_\mathfrak{q}} \int_U \int_0^{\tau(u)} \sum_{k = 1}^\infty \int_{\tau_k(u)}^{\tau_{k + 1}(u)} e^{-\xi (t - r)} \phi(u, g, t) \psi(u, g, r) \, d\nu(t, r, u) \\
={}&\sum_{k = 1}^\infty \sum_{g \in \tilde{G}_\mathfrak{q}} \int_U \int_0^{\tau(u)} \int_0^{\tau(\sigma^k(u))} e^{-\xi (t + \tau_k(u) - r)} \phi(u, g, t + \tau_k(u)) \psi(u, g, r) \, d\nu(t, r, u) \\
={}&\sum_{k = 1}^\infty \sum_{g \in \tilde{G}_\mathfrak{q}} \int_U \int_0^{\tau(u)} \int_0^{\tau(\sigma^k(u))} e^{-\xi (t + \tau_k(u) - r)} \phi(\sigma^k(u), g\mathtt{c}_\mathfrak{q}^k(u), t) \psi(u, g, r) \, d\nu(t, r, u) \\
={}&\sum_{k = 1}^\infty \int_U e^{-\xi \tau_k(u)} \left\langle \mathtt{c}_\mathfrak{q}^k(u)\hat{\phi}_\xi(\sigma^k(u)), \hat{\psi}_{-\overline{\xi}}(u) \right\rangle \, d\nu_U(u).
\end{align*}
Now using the fact that cocycles act unitarily and $\mathcal{L}_0^*(\nu_U) = \nu_U$, we continue the calculation as
\begin{align*}
&\hat{\Upsilon}_{\phi, \psi}^0(\xi) \\
={}&\sum_{k = 1}^\infty \int_U e^{-\xi \tau_k(u)} \left\langle \hat{\phi}_\xi(\sigma^k(u)), \mathtt{c}_\mathfrak{q}^k(u)^{-1}\hat{\psi}_{-\overline{\xi}}(u) \right\rangle \, d(\mathcal{L}_0^*)^k(\nu_U)(u) \\
={}&\sum_{k = 1}^\infty \int_{\interior(U)} \sum_{u' \in \sigma^{-k}(u)} e^{f_k^{(0)}(u')} e^{-\xi \tau_k(u')} \left\langle \hat{\phi}_\xi(\sigma^k(u')), \mathtt{c}_\mathfrak{q}^k(u')^{-1}\hat{\psi}_{-\overline{\xi}}(u') \right\rangle \, d\nu_U(u) \\
={}&\sum_{k = 1}^\infty {\lambda_a}^k \int_{\interior(U)} \left\langle \hat{\phi}_\xi(u), \sum_{u' \in \sigma^{-k}(u)} e^{(f_k^{(a)} + ib\tau_k)(u')} \mathtt{c}_\mathfrak{q}^k(u')^{-1}\hat{\psi}_{-\overline{\xi}}(u') \right\rangle \, d\nu_U(u) \\
={}&\sum_{k = 1}^\infty {\lambda_a}^k \left\langle \hat{\phi}_\xi, \mathcal{M}_{\xi, \mathfrak{q}}^k(\hat{\psi}_{-\overline{\xi}}) \right\rangle.
\end{align*}
\end{proof}

\subsection{Exponential decay of the correlation function}
\begin{lemma}
\label{lem:ExtractNormOfLaplaceTransformDecay}
There exists $C > 0$ such that for all ideals $\mathfrak{q} \subset \mathcal{O}_{\mathbb K}$ coprime to $\mathfrak{q}_0$, for all $\phi \in \mathcal{B}_\mathfrak{q}$, for all $\xi = a + ib \in \mathbb C$ with $|a| \leq a_0'$, we have
\begin{align*}
\left\|\hat{\phi}_\xi\right\|_2 &\leq \left\|\hat{\phi}_\xi\right\|_\infty \leq C N_{\mathbb K}(\mathfrak{q})^C \frac{\|\phi\|_{\mathcal{B}_\mathfrak{q}}}{\max(1, |b|)} & \left\|\mathcal{M}_{\xi, \mathfrak{q}}(\hat{\phi}_{-\overline{\xi}})\right\|_{1, b} &\leq C N_{\mathbb K}(\mathfrak{q})^C \frac{\|\phi\|_{\mathcal{B}_\mathfrak{q}}}{\max(1, |b|)}.
\end{align*}
\end{lemma}

\begin{proof}
Fix $c > 0$ (depending on $n$) such that $\#\tilde{G}_\mathfrak{q} \leq N_{\mathbb K}(\mathfrak{q})^{2c}$ for all nontrivial ideals $\mathfrak{q} \subset \mathcal{O}_{\mathbb K}$. Fix $C_1 = (2 + \overline{\tau})e^{a_0'\overline{\tau}}, C_2 = \max\left(2Ne^{T_0}C_1, \frac{2NT_0e^{2T_0}C_1}{c_0 \kappa_2} + \frac{Ne^{T_0}e^{a_0'\overline{\tau}}(\overline{\tau} + T_0)}{c_0 \kappa_2}\right)$, $C_3 = C_2 + Ne^{T_0}C_1$ and $C = \max(c, C_1, C_3)$. Let $\mathfrak{q} \subset \mathcal{O}_{\mathbb K}$ be an ideal coprime to $\mathfrak{q}_0$, $\phi \in \mathcal{B}_\mathfrak{q}$, $\xi = a + ib \in \mathbb C$ with $|a| \leq a_0'$. For the first inequality, clearly
\begin{align}
\label{eqn:L2BoundByPointwiseSup}
\left\|\hat{\phi}_\xi\right\|_2 \leq \left\|\hat{\phi}_\xi\right\|_\infty = \sup_{u \in U} \left\|\hat{\phi}_\xi(u)\right\|_2 \leq (\#\tilde{G}_\mathfrak{q})^{\frac{1}{2}} \sup_{u \in U} \sup_{g \in \tilde{G}_\mathfrak{q}} \left|\hat{\phi}_\xi(u)(g)\right|.
\end{align}
If $|b| \leq 1$, from the definition of $\hat{\phi}_\xi$ we have $\left\|\hat{\phi}_\xi\right\|_\infty \leq N_{\mathbb K}(\mathfrak{q})^c \overline{\tau}e^{a_0'\overline{\tau}}\|\phi\|_\infty \leq C_1 N_{\mathbb K}(\mathfrak{q})^c \|\phi\|_{\mathcal{B}_\mathfrak{q}}$. If $|b| \geq 1$, integrating by parts gives
\begin{align*}
\hat{\phi}_\xi(u)(g) &= \int_0^{\tau(u)} \phi(u, g, t)e^{-\xi t} \, dt \\
&= \left[-\frac{1}{\xi}\phi(u, g, t)e^{-\xi t}\right]_{t = 0}^{t \nearrow \tau(u)} + \frac{1}{\xi}\int_0^{\tau(u)} \left.\frac{d}{dt'}\right|_{t' = t}\phi(u, g, t') \cdot e^{-\xi t} \, dt
\end{align*}
for all $g \in \tilde{G}_\mathfrak{q}$, for all $u \in U$, and hence
\begin{align*}
\left\|\hat{\phi}_\xi\right\|_\infty \leq N_{\mathbb K}(\mathfrak{q})^c \left(\frac{2}{|b|}\|\phi\|_\infty e^{a_0'\overline{\tau}} + \frac{\overline{\tau}}{|b|}\|\phi\|_{\mathcal{B}_\mathfrak{q}} e^{a_0'\overline{\tau}}\right) \leq C_1 N_{\mathbb K}(\mathfrak{q})^c \frac{\|\phi\|_{\mathcal{B}_\mathfrak{q}}}{|b|}.
\end{align*}

For the second inequality, recall that
\begin{align*}
\left\|\mathcal{M}_{\xi, \mathfrak{q}}(\hat{\phi}_{-\overline{\xi}})\right\|_{1, b} = \left\|\mathcal{M}_{\xi, \mathfrak{q}}(\hat{\phi}_{-\overline{\xi}})\right\|_\infty + \frac{1}{\max(1, |b|)}\Lip_d\left(\mathcal{M}_{\xi, \mathfrak{q}}(\hat{\phi}_{-\overline{\xi}})\right).
\end{align*}
For the first term, we can use the above bound to get
\begin{align*}
\left\|\mathcal{M}_{\xi, \mathfrak{q}}(\hat{\phi}_{-\overline{\xi}})\right\|_\infty \leq Ne^{T_0}\left\|\hat{\phi}_{-\overline{\xi}}\right\|_\infty \leq Ne^{T_0}C_1 N_{\mathbb K}(\mathfrak{q})^c \frac{\|\phi\|_{\mathcal{B}_\mathfrak{q}}}{\max(1, |b|)}.
\end{align*}
For the second term, if $u \in U_j$ and $u' \in U_k$ where $j, k \in \mathcal{A}$ with $j \neq k$, then
\begin{align*}
\left\|\mathcal{M}_{\xi, \mathfrak{q}}(\hat{\phi}_{-\overline{\xi}})(u) - \mathcal{M}_{\xi, \mathfrak{q}}(\hat{\phi}_{-\overline{\xi}})(u')\right\|_2 &\leq 2Ne^{T_0}\left\|\hat{\phi}_{-\overline{\xi}}\right\|_\infty \\
&\leq 2Ne^{T_0}C_1 N_{\mathbb K}(\mathfrak{q})^c \frac{\|\phi\|_{\mathcal{B}_\mathfrak{q}}}{\max(1, |b|)} d(u, u').
\end{align*}
Now let $u, u' \in \interior(U_j)$ for some $j \in \mathcal{A}$. Again we refer the reader to \cref{def:TransferOperatorOriginal} for the case when $u \in \partial(U_j)$ or $u' \in \partial(U_j)$. With a similar computation (using the same notations and constants) as in the proof of \cref{itm:PreliminaryLogLipschitzProperty2} in \cref{lem:PreliminaryLogLipschitz}, we have
\begin{align}
&\left\|\mathcal{M}_{\xi, \mathfrak{q}}(\hat{\phi}_{-\overline{\xi}})(u) - \mathcal{M}_{\xi, \mathfrak{q}}(\hat{\phi}_{-\overline{\xi}})(u')\right\|_2 \\
\leq{}&\sum_{v \in \sigma^{-1}(u)} \left(\left|1 - e^{(f^{(a)} + ib\tau)(\varphi(v)) - (f^{(a)} + ib\tau)(v)}\right|e^{f^{(a)}(v)} \left\|\hat{\phi}_{-\overline{\xi}}(v)\right\|_2\right. \\
{}&\left.+ e^{f^{(a)}(\varphi(v))} \left\|\hat{\phi}_{-\overline{\xi}}(v) - \hat{\phi}_{-\overline{\xi}}(\varphi(v))\right\|_2\right) \\
\leq{}&e^{T_0} \sum_{v \in \sigma^{-1}(u)} \left(e^{\left|f^{(a)}(\varphi(v)) - f^{(a)}(v)\right|} \left|(f^{(a)} + ib\tau)(\varphi(v)) - (f^{(a)} + ib\tau)(v)\right| \cdot \left\|\hat{\phi}_{-\overline{\xi}}\right\|_\infty\right. \\
{}&\left.+ \left\|\hat{\phi}_{-\overline{\xi}}(v) - \hat{\phi}_{-\overline{\xi}}(\varphi(v))\right\|_2\right) \\
\leq{}&\frac{NT_0e^{2T_0}}{c_0 \kappa_2} (1 + |b|) \cdot C_1 N_{\mathbb K}(\mathfrak{q})^c \frac{\|\phi\|_{\mathcal{B}_\mathfrak{q}}}{\max(1, |b|)} d(u, u') \\
{}&+ e^{T_0} \sum_{v \in \sigma^{-1}(u)} \left\|\hat{\phi}_{-\overline{\xi}}(v) - \hat{\phi}_{-\overline{\xi}}(\varphi(v))\right\|_2 \\
\label{eqn:LastLineOfLipschitzCalculation}
\leq{}&\frac{2NT_0e^{2T_0}C_1}{c_0 \kappa_2} N_{\mathbb K}(\mathfrak{q})^c \|\phi\|_{\mathcal{B}_\mathfrak{q}} d(u, u') + e^{T_0} \sum_{v \in \sigma^{-1}(u)} \left\|\hat{\phi}_{-\overline{\xi}}(v) - \hat{\phi}_{-\overline{\xi}}(\varphi(v))\right\|_2.
\end{align}
Let $v \in \sigma^{-1}(u)$. To bound the terms in the sum, recall \cref{eqn:L2BoundByPointwiseSup}. Assume $\tau(v) \leq \tau(\varphi(v))$. Then remembering that $v$ and $\varphi(v)$ are in the same cylinder of length $1$, we have
\begin{align*}
&\left|\hat{\phi}_{-\overline{\xi}}(v)(g) - \hat{\phi}_{-\overline{\xi}}(\varphi(v))(g)\right| \\
\leq{}&\int_0^{\tau(v)} |\phi(v, g, t) - \phi(\varphi(v), g, t)|e^{|a|t} \, dt + \int_{\tau(v)}^{\tau(\varphi(v))} |\phi(\varphi(v), g, t)| e^{|a|t} \, dt \\
\leq{}&\frac{1}{c_0 \kappa_2}\left(\overline{\tau} e^{a_0'\overline{\tau}} \|\phi\|_{\mathcal{B}_\mathfrak{q}} + \Lip_d^{\mathrm{e}}(\tau) e^{a_0'\overline{\tau}} \|\phi\|_\infty\right)d(u, u').
\end{align*}
The case $\tau(\varphi(v)) \leq \tau(v)$ is similar and so
\begin{align*}
\left\|\hat{\phi}_{-\overline{\xi}}(v) - \hat{\phi}_{-\overline{\xi}}(\varphi(v))\right\|_2 \leq \frac{e^{a_0'\overline{\tau}}(\overline{\tau} + T_0)}{c_0 \kappa_2} N_{\mathbb K}(\mathfrak{q})^c \|\phi\|_{\mathcal{B}_\mathfrak{q}} d(u, u').
\end{align*}
Putting this in \cref{eqn:LastLineOfLipschitzCalculation}, gives $\Lip_d\left(\mathcal{M}_{\xi, \mathfrak{q}}(\hat{\phi}_{-\overline{\xi}})\right) \leq C_2 N_{\mathbb K}(\mathfrak{q})^c \|\phi\|_{\mathcal{B}_\mathfrak{q}}$. Hence we conclude that $\left\|\mathcal{M}_{\xi, \mathfrak{q}}(\hat{\phi}_{-\overline{\xi}})\right\|_{1, b} \leq C_3 N_{\mathbb K}(\mathfrak{q})^c \frac{\|\phi\|_{\mathcal{B}_\mathfrak{q}}}{\max(1, |b|)}$.
\end{proof}

\begin{comment}
\begin{remark}
In the above lemma, the reason we prove the second inequality for $\mathcal{M}_{\xi, \mathfrak{q}}(\hat{\phi}_\xi)$ instead of $\hat{\phi}_\xi$ is simply to ensure that we can use the essentially Lipschitz property of $\tau$ since it is not Lipschitz and consequently neither is $\hat{\phi}_\xi$.
\end{remark}
\end{comment}

\begin{lemma}
\label{lem:DecayOfSymbolicCodingCorrelationFunction}
There exist $C > 0, \eta > 0$ and a nontrivial proper ideal $\mathfrak{q}_0' \subset \mathcal{O}_{\mathbb K}$ such that for all square free ideals $\mathfrak{q} \subset \mathcal{O}_{\mathbb K}$ coprime to $\mathfrak{q}_0\mathfrak{q}_0'$, for all $\phi \in \mathcal{B}_\mathfrak{q}$, for all $\psi \in \mathcal{B}_\mathfrak{q}$ such that $\sum_{g \in \tilde{G}_\mathfrak{q}} \psi(u, g, t) = 0$ for all $(u, t) \in U^\tau$, we have
\begin{align*}
\left|\Upsilon_{\phi, \psi}(t)\right| \leq CN_{\mathbb K}(\mathfrak{q})^C e^{-\eta t} \|\phi\|_{\mathcal{B}_\mathfrak{q}} \|\psi\|_{\mathcal{B}_\mathfrak{q}}.
\end{align*}
\end{lemma}

%See quick derivation of inverse Laplace transform in The Laplace Transform by Joel Schiff
\begin{proof}
Fix $C_1 \geq 1, \tilde{\eta} > 0, a_0'' > 0$ and the nontrivial proper ideal $\mathfrak{q}_0' \subset \mathcal{O}_{\mathbb K}$ be the $C, \eta, a_0, \mathfrak{q}_0'$ from \cref{thm:TheoremGeodesicFlow} and $C_2 > c$ be the $C$ from \cref{lem:ExtractNormOfLaplaceTransformDecay} where $c$ is the constant from its proof. Fix $\eta = a_0 \in \left(0, \frac{1}{2}\min(a_0', a_0'')\right)$ such that $\sup_{|a| \leq 2a_0} \log(\lambda_a) \leq \frac{\tilde{\eta}}{2}$. Fix $C_3 = \max(2e^{\tilde{\eta}}C_1{C_2}^2, C_1 + 2C_2)$ and
\begin{align*}
C = \max\left(C_3, C_3\sum_{k = 1}^\infty e^{-\frac{\tilde{\eta}}{2}k}, 2\overline{\tau}e^{\eta \overline{\tau}}\right).
\end{align*}
Let $\mathfrak{q} \subset \mathcal{O}_{\mathbb K}$ be a square free ideal coprime to $\mathfrak{q}_0\mathfrak{q}_0'$, $\phi \in \mathcal{B}_\mathfrak{q}$, and $\psi \in \mathcal{B}_\mathfrak{q}$ such that $\sum_{g \in \tilde{G}_\mathfrak{q}} \psi(u, g, t) = 0$ for all $(u, t) \in U^\tau$. Recall that \cref{lem:LaplaceTransformOfSymbolicCodingCorrelationFunctionInTermsOfTransferOperator} gives
\begin{align*}
\hat{\Upsilon}_{\phi, \psi}^0(\xi) = \sum_{k = 1}^\infty {\lambda_a}^k \left\langle \hat{\phi}_\xi, \mathcal{M}_{\xi, \mathfrak{q}}^k(\hat{\psi}_{-\overline{\xi}}) \right\rangle
\end{align*}
for all $\xi = a + ib \in \mathbb C$ with $a > 0$. Also, recalling previous remarks, we can see that for all integers $k \geq 1$, the map $\xi \mapsto {\lambda_a}^k \left\langle \hat{\phi}_\xi, \mathcal{M}_{\xi, \mathfrak{q}}^k(\hat{\psi}_{-\overline{\xi}}) \right\rangle$ is defined on $\mathbb C$ and is entire. Hence, to show that $\hat{\Upsilon}_{\phi, \psi}^0$ has a holomorphic extension to the half plane $\{\xi \in \mathbb C: \Re(\xi) > -2a_0\}$, it suffices to show that the above sum is absolutely convergent for all $\xi = a + ib \in \mathbb C$ with $|a| < 2a_0$. Recalling $\mathcal{M}_{\xi, \mathfrak{q}}(\hat{\psi}_{-\overline{\xi}}) \in \mathcal{W}_\mathfrak{q}(U)$ and noting $\frac{1}{\max(1, |b|)^2} \leq \frac{2}{1 + b^2}$, we use \cref{thm:TheoremGeodesicFlow,lem:ExtractNormOfLaplaceTransformDecay} to calculate that
\begin{align*}
&\left|{\lambda_a}^k \left\langle \hat{\phi}_\xi, \mathcal{M}_{\xi, \mathfrak{q}}^k(\hat{\psi}_{-\overline{\xi}}) \right\rangle\right| \\
\leq{}&{\lambda_a}^k \left\|\hat{\phi}_\xi\right\|_2 \cdot \left\|\mathcal{M}_{\xi, \mathfrak{q}}^k(\hat{\psi}_{-\overline{\xi}})\right\|_2 \\
\leq{}&{\lambda_a}^k \left\|\hat{\phi}_\xi\right\|_2 \cdot C_1N_{\mathbb K}(\mathfrak{q})^{C_1}e^{-\tilde{\eta} (k - 1)} \left\|\mathcal{M}_{\xi, \mathfrak{q}}(\hat{\psi}_{-\overline{\xi}})\right\|_{1, b} \\
\leq{}&{\lambda_a}^k C_2N_{\mathbb K}(\mathfrak{q})^{C_2} \frac{\|\phi\|_{\mathcal{B}_\mathfrak{q}}}{\max(1, |b|)} \cdot C_1N_{\mathbb K}(\mathfrak{q})^{C_1}e^{-\tilde{\eta} (k - 1)} \cdot C_2N_{\mathbb K}(\mathfrak{q})^{C_2}\frac{\|\psi\|_{\mathcal{B}_\mathfrak{q}}}{\max(1, |b|)} \\
\leq{}&\frac{C_3 N_{\mathbb K}(\mathfrak{q})^{C_3} e^{-\frac{\tilde{\eta}}{2}k}}{1 + b^2} \|\phi\|_{\mathcal{B}_\mathfrak{q}} \|\psi\|_{\mathcal{B}_\mathfrak{q}}
\end{align*}
for all $\xi = a + ib \in \mathbb C$ with $|a| < 2a_0$, whose sum over integers $k \geq 1$ converges as desired. The above calculation also gives the important bound $\left|\hat{\Upsilon}_{\phi, \psi}^0(\xi)\right| \leq \frac{CN_{\mathbb K}(\mathfrak{q})^C}{1 + b^2} \|\phi\|_{\mathcal{B}_\mathfrak{q}} \|\psi\|_{\mathcal{B}_\mathfrak{q}}$ for all $\xi = a + ib \in \mathbb C$ with $|a| < 2a_0$. Since $\Upsilon_{\phi, \psi}^0$ is continuous and in $L^\infty(\mathbb R_{\geq 0}, \mathbb R)$, we use the holomorphic extension and the inverse Laplace transform formula along the line $\{\xi \in \mathbb C: \Re(\xi) = -a_0\}$ to obtain
\begin{align*}
\Upsilon_{\phi, \psi}^0(t) = \frac{1}{2\pi i} \lim_{B \to \infty} \int_{-a_0 - iB}^{-a_0 + iB} \hat{\Upsilon}_{\phi, \psi}^0(\xi) e^{\xi t} \, d\xi = \frac{1}{2\pi} \int_{-\infty}^\infty \hat{\Upsilon}_{\phi, \psi}^0(-a_0 + ib) e^{(-a_0 + ib) t} \, db
\end{align*}
for all $t > 0$ and we already know $\Upsilon_{\phi, \psi}^0(0) = 0$. Then using the above bound, we have
\begin{align*}
\left|\Upsilon_{\phi, \psi}^0(t)\right| &\leq \frac{1}{2\pi} e^{-a_0t} \int_{-\infty}^\infty \left|\hat{\Upsilon}_{\phi, \psi}^0(-a_0 + ib)\right| \, db \\
&\leq \frac{1}{2\pi} e^{-a_0t} \int_{-\infty}^\infty \frac{CN_{\mathbb K}(\mathfrak{q})^C}{1 + b^2} \|\phi\|_{\mathcal{B}_\mathfrak{q}} \|\psi\|_{\mathcal{B}_\mathfrak{q}} \, db \\
&= \frac{C}{2}N_{\mathbb K}(\mathfrak{q})^C e^{-\eta t} \|\phi\|_{\mathcal{B}_\mathfrak{q}} \|\psi\|_{\mathcal{B}_\mathfrak{q}}
\end{align*}
for all $t \geq 0$. Now, $\Upsilon_{\phi, \psi}(t) = \Upsilon_{\phi, \psi}^0(t)$ for all $t \geq \overline{\tau}$ while
\begin{align*}
\left|\Upsilon_{\phi, \psi}^1(t)\right| \leq \overline{\tau} N_{\mathbb K}(\mathfrak{q})^c \|\phi\|_{\mathcal{B}_\mathfrak{q}} \|\psi\|_{\mathcal{B}_\mathfrak{q}} \leq \frac{C}{2}N_{\mathbb K}(\mathfrak{q})^C e^{-\eta t} \|\phi\|_{\mathcal{B}_\mathfrak{q}} \|\psi\|_{\mathcal{B}_\mathfrak{q}}
\end{align*}
for all $t \in [0, \overline{\tau}]$ and hence
\begin{align*}
\left|\Upsilon_{\phi, \psi}(t)\right| \leq CN_{\mathbb K}(\mathfrak{q})^C e^{-\eta t} \|\phi\|_{\mathcal{B}_\mathfrak{q}} \|\psi\|_{\mathcal{B}_\mathfrak{q}}.
\end{align*}
\end{proof}

\subsection{Integrating out the strong stable direction and proof of \texorpdfstring{\cref{thm:TheoremUniformExponentialMixingOfGeodesicFlow}}{\autoref{thm:TheoremUniformExponentialMixingOfGeodesicFlow}}}
Let $\mathfrak{q} \subset \mathcal{O}_{\mathbb K}$ be an ideal coprime to $\mathfrak{q}_0$. Given a $\phi \in C^1(\Gamma_\mathfrak{q} \backslash G/M, \mathbb R)$, we can convert it to a function in $\mathcal{B}_\mathfrak{q}$. By Rokhlin's disintegration theorem with respect to the projection $\proj_U: R \to U$, the probability measure $\nu_R$ disintegrates to give the set of conditional probability measures $\{\nu_u: u \in U\}$. For all $j \in \mathcal{A}$, for all $u \in U_j$, the measure $\nu_u$ is actually define on the fiber ${\proj_U}^{-1}(u) = [u, S_j]$ but we can push forward via the diffeomorphism $[u, S_j] \to S_j$ defined by $[u, s] \mapsto s$ to think of $\nu_u$ as a measure on $S_j$. Thinking of the overline as an isometry $R \to \overline{R}$, for all $t \geq 0$, we define $\phi_t \in \mathcal{B}_\mathfrak{q}$ by
\begin{align*}
\phi_t(u, g, r) = \int_{S_j} \phi(g[\overline{u}, \overline{s}]a_{t + r}) \, d\nu_u(s)
\end{align*}
%There is a mistake. Write U_j \times \tilde{G}_\mathfrak{q} \times [0, \tau(u)) more correctly.
for all $(u, g, r) \in U_j \times \tilde{G}_\mathfrak{q} \times [0, \tau(u))$, for all $j \in \mathcal{A}$, and in order to ensure that indeed $\phi_t \in \mathcal{B}_\mathfrak{q}$, we must define $\phi_t(u, g, r) = \phi_t(\sigma^k(u), g\mathtt{c}_\mathfrak{q}^k(u), r - \tau_k(u))$ for all $r \in [\tau_k(u), \tau_{k + 1}(u))$, for all $k \geq 0$.

\begin{remark}
We need to deal with some technicalities. Let $\mathfrak{q} \subset \mathcal{O}_{\mathbb K}$ be an ideal coprime to $\mathfrak{q}_0$ and $j \in \mathcal{A}$.
%$[u, s], [u', s] \in R_j$ for some $j \in \mathcal{A}$ and $[\overline{u}, \overline{s}], [\overline{u}', \overline{s}] \in \overline{R}_j$ be their lifts. Let $w \in R_j$ be the center.
Firstly, by smoothness of the strong unstable and strong stable foliations and compactness of $R_j$, there is a $C_1 > 1$ such that $d(g[\overline{u}, \overline{s}], g[\overline{u}', \overline{s}]) = d([u, s], [u', s]) \leq C_1d(u, u') \leq C_1d_{\mathrm{su}}(u, u')$ for all $u, u' \in U_j$, for all $s \in S_j$, for all $g \in \tilde{G}_\mathfrak{q}$. Now, for all $u \in U_j$, the Patterson-Sullivan density induces the measure $d\mu^{\mathrm{PS}}_{[u, S_j]}([u, s]) = e^{\delta_\Gamma \beta_{[\overline{u}, \overline{s}]^-}(o, [\overline{u}, \overline{s}])} \, d\mu^{\mathrm{PS}}_o([\overline{u}, \overline{s}]^-)$ on $[u, S_j]$. For all $u \in U_j$, pushing forward via the diffeomorphism $[u, S_j] \to S_j$ mentioned above gives the measure $\mu^{\mathrm{PS}}_u(s) = e^{\delta_\Gamma \beta_{[\overline{u}, \overline{s}]^-}(o, [\overline{u}, \overline{s}])} d\mu^{\mathrm{PS}}_o((\overline{s})^-)$ on $S_j$, using the important fact that $[\overline{u}, \overline{s}]^- = (\overline{s})^-$. In fact, for all $u \in U_j$, comparing with the definition of the BMS measure, it can be checked that $\mu^{\mathrm{PS}}_u$ after normalization is exactly the conditional measures $\nu_u$. The nontrivial consequence is that $\frac{d\nu_{u'}}{d\nu_u} \in C^\infty(S_j, \mathbb R)$ for all $u, u' \in U_j$. Together with compactness of $R_j$, there is a $C_2 > 0$ such that $\left|1 - \frac{d\nu_{u'}}{d\nu_u}(s)\right| \leq C_2d(u, u') \leq C_2d_{\mathrm{su}}(u, u')$ for all $u, u' \in U_j$, for all $s \in S_j$. An easy computation using the two derived inequalities shows that $\|\phi_t\|_{\mathcal{B}_\mathfrak{q}} \leq C\|\phi\|_{C^1}$ where $C = C_1 + C_2$, for all $\phi \in C^1(\Gamma_\mathfrak{q} \backslash G/M, \mathbb R)$, for all $t \geq 0$.
\end{remark}

\begin{lemma}
\label{lem:EstimateFunctionByIntegratingOverS}
There exist $C > 0$ and $\eta > 0$ such that for all ideals $\mathfrak{q} \subset \mathcal{O}_{\mathbb K}$ coprime to $\mathfrak{q}_0$, for all $\phi \in C^1(\Gamma_\mathfrak{q} \backslash G/M, \mathbb R)$, we have
\begin{align*}
|\phi(g[\overline{u}, \overline{s}]a_{2t + r}) - \phi_t(u, g, t + r)| \leq Ce^{-\eta t}\|\phi\|_{C^1}
\end{align*}
for all $[u, s] \in R$, for all $g \in \tilde{G}_\mathfrak{q}$, for all $t, r \geq 0$.
\end{lemma}

\begin{proof}
Fix $C, \eta > 0$ to be the constants $c, \log(\lambda)$ from the Anosov property in \cite{Rat73}. Let $\mathfrak{q} \subset \mathcal{O}_{\mathbb K}$ be an ideal coprime to $\mathfrak{q}_0$ and $\phi \in C^1(\Gamma_\mathfrak{q} \backslash G/M, \mathbb R)$. Let $j \in \mathcal{A}, [u, s] \in R_j, g \in \tilde{G}_\mathfrak{q}$ and $t, r \geq 0$. Let $m \in \mathbb Z_{\geq 0}$ such that $t + r \in [\tau_m(u), \tau_{m + 1}(u))$ and let $u_1 = \sigma^m(u) \in U_k$ for some $k \in \mathcal{A}$. Let $s_1 \in S_k$. Then $[\overline{u}, \overline{s}]a_{\tau_m(u)} = \mathtt{c}^m(u)\overline{\mathcal{P}^m([u, s])}$ and $\mathtt{c}^m(u)[\overline{u}_1, \overline{s}_1]$ are both in $\mathtt{c}^m(u)[\overline{u}_1, \overline{S}_k]$. Noting that $\tau_m(u) \leq t + r$, we have
\begin{align*}
&d(g[\overline{u}, \overline{s}]a_{2t + r}, g\mathtt{c}^m(u)[\overline{u}_1, \overline{s}_1]a_{2t + r - \tau_m(u)}) \\
\leq{}&d_{\mathrm{ss}}(g[\overline{u}, \overline{s}]a_{2t + r}, g\mathtt{c}^m(u)[\overline{u}_1, \overline{s}_1]a_{2t + r - \tau_m(u)}) \\
\leq{}&Ce^{-\eta(2t + r - \tau_m(u))} \leq Ce^{-\eta t}.
\end{align*}
Thus, $|\phi(g[\overline{u}, \overline{s}]a_{2t + r}) - \phi(g\mathtt{c}^m(u)[\overline{u}_1, \overline{s}_1]a_{2t + r - \tau_m(u)})| \leq Ce^{-\eta t}\|\phi\|_{C^1}$. Integrating over $s_1 \in S_k$ with respect to the probability measure $\nu_{u_1}$ gives $|\phi(g[\overline{u}, \overline{s}]a_{2t + r}) - \phi_t(u, g, t + r)| \leq Ce^{-\eta t}\|\phi\|_{C^1}$.
\end{proof}

\begin{corollary}
\label{cor:ApproximatingMixingOfGeodesicFlowBySymbolicCodingCorrelationFunction}
There are $C, \eta > 0$ such that for all ideals $\mathfrak{q} \subset \mathcal{O}_{\mathbb K}$ coprime to $\mathfrak{q}_0$, for all $\phi, \psi \in C^1(\Gamma_\mathfrak{q} \backslash G/M, \mathbb R)$, we have
\begin{multline*}
\left|\int_{\Gamma_\mathfrak{q} \backslash G/M} \phi(xa_{2t})\psi(x) \, dm^{\mathrm{BMS}}_\mathfrak{q}(x) - \frac{m^{\mathrm{BMS}}(\Gamma \backslash G/M)}{\nu_U(\tau)}\Upsilon_{\phi_t, \psi_0}(t)\right| \\
\leq CN_{\mathbb K}(\mathfrak{q})^C e^{-\eta t} \|\phi\|_{C^1} \|\psi\|_{C^1}.
\end{multline*}
\end{corollary}

\begin{proof}
Fix $c > 0$ (depending on $n$) such that $\#\tilde{G}_\mathfrak{q} \leq N_{\mathbb K}(\mathfrak{q})^c$ for all nontrivial ideals $\mathfrak{q} \subset \mathcal{O}_{\mathbb K}$. Fix $C', \eta > 0$ to be the $C$ and $\eta$ from \cref{lem:EstimateFunctionByIntegratingOverS}. Fix $C = \max(C'm^{\mathrm{BMS}}(\Gamma \backslash G/M), c)$. Let $\mathfrak{q} \subset \mathcal{O}_{\mathbb K}$ be an ideal coprime to $\mathfrak{q}_0$, and $\phi, \psi \in C^1(\Gamma_\mathfrak{q} \backslash G/M, \mathbb R)$. We have
\begin{align*}
&\int_{\Gamma_\mathfrak{q} \backslash G/M} \phi(xa_{2t})\psi(x) \, dm^{\mathrm{BMS}}_\mathfrak{q}(x) \\
={}&\frac{m^{\mathrm{BMS}}_\mathfrak{q}(\Gamma_\mathfrak{q} \backslash G/M)}{\#\tilde{G}_\mathfrak{q} \cdot \nu_R(\tau)} \int_{\Omega_\mathfrak{q}} \phi(xa_{2t})\psi(x) \, d\nu^{\mathfrak{q}, \tau}(x) \\
={}&\frac{m^{\mathrm{BMS}}(\Gamma \backslash G/M)}{\nu_R(\tau)} \int_{R^\mathfrak{q}} \int_0^{\tau_\mathfrak{q}(x)} \phi(xa_{r + 2t})\psi(xa_r) \, dr \, d\nu_{R^\mathfrak{q}}(x) \\
={}&\frac{m^{\mathrm{BMS}}(\Gamma \backslash G/M)}{\nu_R(\tau)} \sum_{g \in \tilde{G}_\mathfrak{q}} \int_{R} \int_0^{\tau(x)} \phi(g\overline{x}a_{r + 2t})\psi(g\overline{x}a_r) \, dr \, d\nu_R(x) \\
={}&\frac{m^{\mathrm{BMS}}(\Gamma \backslash G/M)}{\nu_U(\tau)} \sum_{g \in \tilde{G}_\mathfrak{q}} \sum_{j = 1}^N \int_{U_j} \int_{S_j} \\
&\int_0^{\tau(u)} \phi(g[\overline{u}, \overline{s}]a_{r + 2t})\psi(g[\overline{u}, \overline{s}]a_r) \, dr \, d\nu_u(s) \, d\nu_U(u).
\end{align*}
Now recall the definition of $\Upsilon_{\phi_t, \psi_0}$ and use \cref{lem:EstimateFunctionByIntegratingOverS} to finish the proof.
\end{proof}

\begin{proof}[Proof of \cref{thm:TheoremUniformExponentialMixingOfGeodesicFlow}]
A single ideal $\mathfrak{q} = \mathcal{O}_{\mathbb K}$, corresponds to the single manifold $\Gamma \backslash G/M$ and hence exponential mixing of the geodesic flow has already been established by works of Dolgopyat and Stoyanov (essentially using Dolgopyat's method from \cref{sec:GeodesicFlowLarge|b|} for large $|b|$ and the complex RPF theorem for small $|b|$ in a similar proof as in \cref{lem:DecayOfSymbolicCodingCorrelationFunction}). Now, recall the remark before \cref{lem:EstimateFunctionByIntegratingOverS} and fix $C_4 > 1$ to be the constant described there. Fix $C_1, \eta_1, C_2, \eta_2 > 0$ to be the $C$ and $\eta$ from \cref{cor:ApproximatingMixingOfGeodesicFlowBySymbolicCodingCorrelationFunction,lem:DecayOfSymbolicCodingCorrelationFunction} respectively. Fix $C_3 = \frac{m^{\mathrm{BMS}}(\Gamma \backslash G/M)}{\nu_U(\tau)}C_2, \eta = \frac{1}{2}\min(\eta_1, \eta_2)$ and $C = \max(C_1, C_2, C_1 + C_3{C_4}^2)$. Fix the nontrivial proper ideal $\mathfrak{q}_0' \subset \mathcal{O}_{\mathbb K}$ from \cref{lem:DecayOfSymbolicCodingCorrelationFunction}. Let $\mathfrak{q} \subset \mathcal{O}_{\mathbb K}$ be a square free ideal coprime to $\mathfrak{q}_0\mathfrak{q}_0'$. Let $\phi, \psi \in C^1(\Gamma_\mathfrak{q} \backslash G/M, \mathbb R)$. Consider the decomposition $\psi = \psi^{\tilde{G}_\mathfrak{q}} + \psi^0$ where $\psi^{\tilde{G}_\mathfrak{q}}$, defined by $\psi^{\tilde{G}_\mathfrak{q}}(x) = \sum_{g \in \tilde{G}_\mathfrak{q}} \psi(gx)$ for all $x \in \Gamma_\mathfrak{q} \backslash G/M$, is $\tilde{G}_\mathfrak{q}$-invariant and consequently $\psi^0$ satisfies $\sum_{g \in \tilde{G}_\mathfrak{q}} \psi^0(gx) = 0$ for all $x \in \Gamma_\mathfrak{q} \backslash G/M$. Then
\begin{align*}
\left|\int_{\Gamma_\mathfrak{q} \backslash G/M} \phi(xa_t)\psi(x) \, dm^{\mathrm{BMS}}_\mathfrak{q}(x)\right| \leq{}&\left|\int_{\Gamma_\mathfrak{q} \backslash G/M} \phi(xa_t)\psi^{\tilde{G}_\mathfrak{q}}(x) \, dm^{\mathrm{BMS}}_\mathfrak{q}(x)\right| \\
{}&+ \left|\int_{\Gamma_\mathfrak{q} \backslash G/M} \phi(xa_t)\psi^0(x) \, dm^{\mathrm{BMS}}_\mathfrak{q}(x)\right|
\end{align*}
where the first term can be written as
\begin{align*}
\left|\int_{\Gamma_\mathfrak{q} \backslash G/M} \phi(xa_t)\psi^{\tilde{G}_\mathfrak{q}}(x) \, dm^{\mathrm{BMS}}_\mathfrak{q}(x)\right| &= \left|\sum_{g \in \tilde{G}_\mathfrak{q}} \int_{\Gamma \backslash G/M} \phi(gxa_t)\psi^{\tilde{G}_\mathfrak{q}}(gx) \, dm^{\mathrm{BMS}}(x)\right| \\
&= \left|\int_{\Gamma \backslash G/M} \left(\sum_{g \in \tilde{G}_\mathfrak{q}} \phi(gxa_t)\right) \psi^{\tilde{G}_\mathfrak{q}}(x) \, dm^{\mathrm{BMS}}(x)\right|
\end{align*}
using the $\tilde{G}_\mathfrak{q}$-invariance of $\psi^{\tilde{G}_\mathfrak{q}}$. Hence it again reduces to a question of exponential mixing of the geodesic flow on $\Gamma \backslash G/M$ which is known as explained above. Thus it suffices to assume that $\psi = \psi^0$, i.e., $\sum_{g \in \tilde{G}_\mathfrak{q}} \psi(gx) = 0$ for all $x \in \Gamma_\mathfrak{q} \backslash G/M$. Thus we have corresponding functions $\phi_t, \psi_0 \in \mathcal{B}_\mathfrak{q}$ with $\sum_{g} \psi_0(u, g, t) = 0$ for all $(u, t) \in U^\tau$ and $\|\phi_t\|_{\mathcal{B}_\mathfrak{q}} \leq C_4\|\phi\|_{C^1}$ for all $t \geq 0$ and $\|\psi_0\|_{\mathcal{B}_\mathfrak{q}} \leq C_4\|\psi\|_{C^1}$. Hence by \cref{cor:ApproximatingMixingOfGeodesicFlowBySymbolicCodingCorrelationFunction,lem:DecayOfSymbolicCodingCorrelationFunction}, for all $t \geq 0$, letting $t' = \frac{t}{2}$, we have
\begin{align*}
&\left|\int_{\Gamma_\mathfrak{q} \backslash G/M} \phi(xa_t)\psi(x) \, dm^{\mathrm{BMS}}_\mathfrak{q}(x)\right| \\
\leq{}&\frac{m^{\mathrm{BMS}}(\Gamma \backslash G/M)}{\nu_U(\tau)}\left|\Upsilon_{\phi_{t'}, \psi_0}(t')\right| + C_1N_{\mathbb K}(\mathfrak{q})^{C_1} e^{- \eta_1 t'} \|\phi\|_{C^1} \|\psi\|_{C^1} \\
\leq{}&C_3{C_4}^2N_{\mathbb K}(\mathfrak{q})^{C_2} e^{- \eta_2 t'} \|\phi\|_{C^1} \|\psi\|_{C^1} + C_1N_{\mathbb K}(\mathfrak{q})^{C_1} e^{- \eta_1 t'} \|\phi\|_{C^1} \|\psi\|_{C^1} \\
\leq{}&CN_{\mathbb K}(\mathfrak{q})^C e^{-\eta t} \|\phi\|_{C^1} \|\psi\|_{C^1}.
\end{align*}
\end{proof}

%\newpage
\nocite{*}
\bibliographystyle{alpha}
\bibliography{References}

\begin{thebibliography}{MOW17}

\bibitem[AGY06]{AGY06}
Artur Avila, S\'{e}bastien Gou\"{e}zel, and Jean-Christophe Yoccoz.
\newblock Exponential mixing for the {T}eichm\"{u}ller flow.
\newblock {\em Publ. Math. Inst. Hautes \'{E}tudes Sci.}, (104):143--211, 2006.

\bibitem[Bab02]{Bab02}
Martine Babillot.
\newblock On the mixing property for hyperbolic systems.
\newblock {\em Israel J. Math.}, 129:61--76, 2002.

\bibitem[BG08]{BG08}
Jean Bourgain and Alex Gamburd.
\newblock Uniform expansion bounds for {C}ayley graphs of {${\rm SL}_2(\Bbb
  F_p)$}.
\newblock {\em Ann. of Math. (2)}, 167(2):625--642, 2008.

\bibitem[BGS10]{BGS10}
Jean Bourgain, Alex Gamburd, and Peter Sarnak.
\newblock Affine linear sieve, expanders, and sum-product.
\newblock {\em Invent. Math.}, 179(3):559--644, 2010.

\bibitem[BGS11]{BGS11}
Jean Bourgain, Alex Gamburd, and Peter Sarnak.
\newblock Generalization of {S}elberg's {$\frac{3}{16}$} theorem and affine
  sieve.
\newblock {\em Acta Math.}, 207(2):255--290, 2011.

\bibitem[BK14]{BK14}
Jean Bourgain and Alex Kontorovich.
\newblock On {Z}aremba's conjecture.
\newblock {\em Ann. of Math. (2)}, 180(1):137--196, 2014.

\bibitem[BKS10]{BKS10}
Jean Bourgain, Alex Kontorovich, and Peter Sarnak.
\newblock Sector estimates for hyperbolic isometries.
\newblock {\em Geom. Funct. Anal.}, 20(5):1175--1200, 2010.

\bibitem[Bor16]{Bor16}
David Borthwick.
\newblock {\em Spectral theory of infinite-area hyperbolic surfaces}, volume
  318 of {\em Progress in Mathematics}.
\newblock Birkh\"{a}user/Springer, [Cham], second edition, 2016.

\bibitem[Bow70]{Bow70}
Rufus Bowen.
\newblock Markov partitions for {A}xiom {${\rm A}$} diffeomorphisms.
\newblock {\em Amer. J. Math.}, 92:725--747, 1970.

\bibitem[Bow71]{Bow71}
Rufus Bowen.
\newblock Periodic points and measures for {A}xiom {$A$} diffeomorphisms.
\newblock {\em Trans. Amer. Math. Soc.}, 154:377--397, 1971.

\bibitem[Bow73]{Bow73}
Rufus Bowen.
\newblock Symbolic dynamics for hyperbolic flows.
\newblock {\em Amer. J. Math.}, 95:429--460, 1973.

\bibitem[Bow95]{Bow95}
B.~H. Bowditch.
\newblock Geometrical finiteness with variable negative curvature.
\newblock {\em Duke Math. J.}, 77(1):229--274, 1995.

\bibitem[Bow08]{Bow08}
Rufus Bowen.
\newblock {\em Equilibrium states and the ergodic theory of {A}nosov
  diffeomorphisms}, volume 470 of {\em Lecture Notes in Mathematics}.
\newblock Springer-Verlag, Berlin, revised edition, 2008.
\newblock With a preface by David Ruelle. Edited by Jean-Ren\'{e} Chazottes.

\bibitem[But98]{But98}
Jack Button.
\newblock All {F}uchsian {S}chottky groups are classical {S}chottky groups.
\newblock In {\em The {E}pstein birthday schrift}, volume~1 of {\em Geom.
  Topol. Monogr.}, pages 117--125. Geom. Topol. Publ., Coventry, 1998.

\bibitem[BV12]{BV12}
Jean Bourgain and P\'{e}ter~P. Varj\'{u}.
\newblock Expansion in {$SL_d({\bf Z}/q{\bf Z}),\,q$} arbitrary.
\newblock {\em Invent. Math.}, 188(1):151--173, 2012.

\bibitem[CdV85]{Col85}
Yves Colin~de Verdi\`ere.
\newblock Th\'{e}orie spectrale des surfaces de {R}iemann d'aire infinie.
\newblock Number 132, pages 259--275. 1985.
\newblock Colloquium in honor of Laurent Schwartz, Vol. 2 (Palaiseau, 1983).

\bibitem[Che02]{Che02}
N.~Chernov.
\newblock Invariant measures for hyperbolic dynamical systems.
\newblock In {\em Handbook of dynamical systems, {V}ol. 1{A}}, pages 321--407.
  North-Holland, Amsterdam, 2002.

\bibitem[Dol98]{Dol98}
Dmitry Dolgopyat.
\newblock On decay of correlations in {A}nosov flows.
\newblock {\em Ann. of Math. (2)}, 147(2):357--390, 1998.

\bibitem[DSO09]{DSO09}
Antonio~J. Di~Scala and Carlos Olmos.
\newblock A geometric proof of the {K}arpelevich-{M}ostow theorem.
\newblock {\em Bull. Lond. Math. Soc.}, 41(4):634--638, 2009.

\bibitem[GLZ04]{GLZ04}
Laurent Guillop\'{e}, Kevin~K. Lin, and Maciej Zworski.
\newblock The {S}elberg zeta function for convex co-compact {S}chottky groups.
\newblock {\em Comm. Math. Phys.}, 245(1):149--176, 2004.

\bibitem[GN09]{GN09}
Colin Guillarmou and Fr\'{e}d\'{e}ric Naud.
\newblock Wave decay on convex co-compact hyperbolic manifolds.
\newblock {\em Comm. Math. Phys.}, 287(2):489--511, 2009.

\bibitem[GV12]{GV12}
A.~Salehi Golsefidy and P\'{e}ter~P. Varj\'{u}.
\newblock Expansion in perfect groups.
\newblock {\em Geom. Funct. Anal.}, 22(6):1832--1891, 2012.

\bibitem[Kai90]{Kai90}
Vadim~A. Kaimanovich.
\newblock Invariant measures of the geodesic flow and measures at infinity on
  negatively curved manifolds.
\newblock {\em Ann. Inst. H. Poincar\'{e} Phys. Th\'{e}or.}, 53(4):361--393,
  1990.
\newblock Hyperbolic behaviour of dynamical systems (Paris, 1990).

\bibitem[Kai91]{Kai91}
Vadim~A. Kaimanovich.
\newblock Bowen-{M}argulis and {P}atterson measures on negatively curved
  compact manifolds.
\newblock In {\em Dynamical systems and related topics ({N}agoya, 1990)},
  volume~9 of {\em Adv. Ser. Dynam. Systems}, pages 223--232. World Sci. Publ.,
  River Edge, NJ, 1991.

\bibitem[Kar53]{Kar53}
F.~I. Karpelevi\v{c}.
\newblock Surfaces of transitivity of a semisimple subgroup of the group of
  motions of a symmetric space.
\newblock {\em Doklady Akad. Nauk SSSR (N.S.)}, 93:401--404, 1953.

\bibitem[Kat95]{Kat95}
Tosio Kato.
\newblock {\em Perturbation theory for linear operators}.
\newblock Classics in Mathematics. Springer-Verlag, Berlin, 1995.
\newblock Reprint of the 1980 edition.

\bibitem[KS13]{KS13}
Dubi Kelmer and Lior Silberman.
\newblock A uniform spectral gap for congruence covers of a hyperbolic
  manifold.
\newblock {\em Amer. J. Math.}, 135(4):1067--1085, 2013.

\bibitem[Lan72]{Lan72}
Vicente Land\'{a}zuri.
\newblock Bounds for the characters of the {C}hevalley groups.
\newblock {\em Rev. Colombiana Mat.}, 6:125--164, 1972.

\bibitem[LO13]{LO13}
Min Lee and Hee Oh.
\newblock Effective circle count for {A}pollonian packings and closed
  horospheres.
\newblock {\em Geom. Funct. Anal.}, 23(2):580--621, 2013.

\bibitem[LP82]{LP82}
Peter~D. Lax and Ralph~S. Phillips.
\newblock The asymptotic distribution of lattice points in {E}uclidean and
  non-{E}uclidean spaces.
\newblock {\em J. Funct. Anal.}, 46(3):280--350, 1982.

\bibitem[LS74]{LS74}
Vicente Landazuri and Gary~M. Seitz.
\newblock On the minimal degrees of projective representations of the finite
  {C}hevalley groups.
\newblock {\em J. Algebra}, 32:418--443, 1974.

\bibitem[LY73]{LY73}
A.~Lasota and James~A. Yorke.
\newblock On the existence of invariant measures for piecewise monotonic
  transformations.
\newblock {\em Trans. Amer. Math. Soc.}, 186:481--488 (1974), 1973.

\bibitem[Mar04]{Mar04}
Grigoriy~A. Margulis.
\newblock {\em On some aspects of the theory of {A}nosov systems}.
\newblock Springer Monographs in Mathematics. Springer-Verlag, Berlin, 2004.
\newblock With a survey by Richard Sharp: Periodic orbits of hyperbolic flows,
  Translated from the Russian by Valentina Vladimirovna Szulikowska.

\bibitem[MM87]{MM87}
Rafe~R. Mazzeo and Richard~B. Melrose.
\newblock Meromorphic extension of the resolvent on complete spaces with
  asymptotically constant negative curvature.
\newblock {\em J. Funct. Anal.}, 75(2):260--310, 1987.

\bibitem[MMO14]{MMO14}
Gregory Margulis, Amir Mohammadi, and Hee Oh.
\newblock Closed geodesics and holonomies for {K}leinian manifolds.
\newblock {\em Geom. Funct. Anal.}, 24(5):1608--1636, 2014.

\bibitem[MO15]{MO15}
Amir Mohammadi and Hee Oh.
\newblock Matrix coefficients, counting and primes for orbits of geometrically
  finite groups.
\newblock {\em J. Eur. Math. Soc. (JEMS)}, 17(4):837--897, 2015.

\bibitem[Mok78]{Mok78}
Kam~Ping Mok.
\newblock On the differential geometry of frame bundles of {R}iemannian
  manifolds.
\newblock {\em J. Reine Angew. Math.}, 302:16--31, 1978.

\bibitem[Mos55]{Mos55}
G.~D. Mostow.
\newblock Some new decomposition theorems for semi-simple groups.
\newblock {\em Mem. Amer. Math. Soc.}, No. 14:31--54, 1955.

\bibitem[MOW17]{MOW17}
Michael Magee, Hee Oh, and Dale Winter.
\newblock Uniform congruence counting for schottky semigroups in {${\rm
  SL}_2(\Bbb{Z})$}.
\newblock {\em J. Reine Angew. Math.}, arXiv:1601.03705v3 [math.NT]:1--47,
  2017.
\newblock Ahead of Print.

\bibitem[MVW84]{MVW84}
C.~R. Matthews, L.~N. Vaserstein, and B.~Weisfeiler.
\newblock Congruence properties of {Z}ariski-dense subgroups. {I}.
\newblock {\em Proc. London Math. Soc. (3)}, 48(3):514--532, 1984.

\bibitem[Nau05]{Nau05}
Fr\'{e}d\'{e}ric Naud.
\newblock Expanding maps on {C}antor sets and analytic continuation of zeta
  functions.
\newblock {\em Ann. Sci. \'{E}cole Norm. Sup. (4)}, 38(1):116--153, 2005.

\bibitem[OS13]{OS13}
Hee Oh and Nimish~A. Shah.
\newblock Equidistribution and counting for orbits of geometrically finite
  hyperbolic groups.
\newblock {\em J. Amer. Math. Soc.}, 26(2):511--562, 2013.

\bibitem[OW16]{OW16}
Hee Oh and Dale Winter.
\newblock Uniform exponential mixing and resonance free regions for convex
  cocompact congruence subgroups of {${\rm SL}_2(\Bbb{Z})$}.
\newblock {\em J. Amer. Math. Soc.}, 29(4):1069--1115, 2016.

\bibitem[Pat76]{Pat76}
S.~J. Patterson.
\newblock The limit set of a {F}uchsian group.
\newblock {\em Acta Math.}, 136(3-4):241--273, 1976.

\bibitem[Pat88]{Pat88}
S.~J. Patterson.
\newblock On a lattice-point problem in hyperbolic space and related questions
  in spectral theory.
\newblock {\em Ark. Mat.}, 26(1):167--172, 1988.

\bibitem[Pol85]{Pol85}
Mark Pollicott.
\newblock On the rate of mixing of {A}xiom {A} flows.
\newblock {\em Invent. Math.}, 81(3):413--426, 1985.

\bibitem[PP90]{PP90}
William Parry and Mark Pollicott.
\newblock Zeta functions and the periodic orbit structure of hyperbolic
  dynamics.
\newblock {\em Ast\'{e}risque}, (187-188):268, 1990.

\bibitem[PS16]{PS16}
Vesselin Petkov and Luchezar Stoyanov.
\newblock Ruelle transfer operators with two complex parameters and
  applications.
\newblock {\em Discrete Contin. Dyn. Syst.}, 36(11):6413--6451, 2016.

\bibitem[Rat73]{Rat73}
M.~Ratner.
\newblock Markov partitions for {A}nosov flows on {$n$}-dimensional manifolds.
\newblock {\em Israel J. Math.}, 15:92--114, 1973.

\bibitem[Rob03]{Rob03}
Thomas Roblin.
\newblock Ergodicit\'{e} et \'{e}quidistribution en courbure n\'{e}gative.
\newblock {\em M\'{e}m. Soc. Math. Fr. (N.S.)}, (95):vi+96, 2003.

\bibitem[Rud82]{Rud82}
Daniel~J. Rudolph.
\newblock Ergodic behaviour of {S}ullivan's geometric measure on a
  geometrically finite hyperbolic manifold.
\newblock {\em Ergodic Theory Dynam. Systems}, 2(3-4):491--512 (1983), 1982.

\bibitem[Sas58]{Sas58}
Shigeo Sasaki.
\newblock On the differential geometry of tangent bundles of {R}iemannian
  manifolds.
\newblock {\em T\^{o}hoku Math. J. (2)}, 10:338--354, 1958.

\bibitem[Sch05]{Sch05}
Barbara Schapira.
\newblock Equidistribution of the horocycles of a geometrically finite surface.
\newblock {\em Int. Math. Res. Not.}, (40):2447--2471, 2005.

\bibitem[Sto05]{Sto05}
Luchezar Stoyanov.
\newblock On the {R}uelle-{P}erron-{F}robenius theorem.
\newblock {\em Asymptot. Anal.}, 43(1-2):131--150, 2005.

\bibitem[Sto11]{Sto11}
Luchezar Stoyanov.
\newblock Spectra of {R}uelle transfer operators for axiom {A} flows.
\newblock {\em Nonlinearity}, 24(4):1089--1120, 2011.

\bibitem[Sul79]{Sul79}
Dennis Sullivan.
\newblock The density at infinity of a discrete group of hyperbolic motions.
\newblock {\em Inst. Hautes \'{E}tudes Sci. Publ. Math.}, (50):171--202, 1979.

\bibitem[Sul84]{Sul84}
Dennis Sullivan.
\newblock Entropy, {H}ausdorff measures old and new, and limit sets of
  geometrically finite {K}leinian groups.
\newblock {\em Acta Math.}, 153(3-4):259--277, 1984.

\bibitem[SZ93]{SZ93}
Gary~M. Seitz and Alexander~E. Zalesskii.
\newblock On the minimal degrees of projective representations of the finite
  {C}hevalley groups. {II}.
\newblock {\em J. Algebra}, 158(1):233--243, 1993.

\bibitem[Wei84]{Wei84}
Boris Weisfeiler.
\newblock Strong approximation for {Z}ariski-dense subgroups of semisimple
  algebraic groups.
\newblock {\em Ann. of Math. (2)}, 120(2):271--315, 1984.

\bibitem[Win15]{Win15}
Dale Winter.
\newblock Mixing of frame flow for rank one locally symmetric spaces and
  measure classification.
\newblock {\em Israel J. Math.}, 210(1):467--507, 2015.

\bibitem[Win16]{Win16}
Dale Winter.
\newblock Exponential mixing for frame flows for convex cocompact hyperbolic
  manifolds.
\newblock {\em arXiv e-prints}, arXiv:1612.00909v1 [math.DS]:1--47, 2016.

\end{thebibliography}
\end{document}